%% file: 2012_RJ_thesis.tex
\documentclass[11pt,a4paper,makeidx,final]{book}

\include{thesis_preample}

\makeindex

\theoremstyle{plain}
\newtheorem{thm}{Theorem}[chapter]
\newtheorem{prop}[thm]{Proposition}
\newtheorem{lem}[thm]{Lemma}
\newtheorem{cor}[thm]{Corollary}
\newtheorem{conj}[thm]{Conjecture}

\theoremstyle{definition}
\newtheorem{definition}[thm]{Definition}
\newtheorem{example}[thm]{Example}
\newtheorem{prob}[thm]{Problem}
\newtheorem{remark}[thm]{Remark}

\newtheorem*{cond}{Condition}

\begin{document}
\frontmatter

\include{thesis_title}

\setcounter{page}{2}
\include{thesis_address}
\include{thesis_abstract}

\include{thesis_preface}

\tableofcontents
\cleardoublepage
\addcontentsline{toc}{chapter}{List of Figures}
\listoffigures

\mainmatter
\chapter{Shift spaces}
\include{thesis_shift_spaces}
\chapter{Flow Equivalence}
\include{thesis_flow}

\chapter{Structure of covers of sofic shifts}
\include{thesis_covers}

\chapter{Flow equivalence of beta-shifts}
\include{thesis_beta}

\chapter{Flow equivalence of renewal systems}
\include{thesis_renewal}

\appendix
\chapter{Experimental investigation of renewal systems}
\include{thesis_programs}

\backmatter
\bibliographystyle{alpha}
\bibliography{../papers/rune}{}

\cleardoublepage
\addcontentsline{toc}{chapter}{Index}
\printindex

\end{document}

%% file: thesis_preample.tex
\usepackage[danish,english]{babel}
\usepackage{amsmath,amscd}
\usepackage{amssymb}
\usepackage{amsthm}
\usepackage{mathrsfs}
\usepackage{fixme}
\usepackage{enumerate}
\usepackage{graphicx}                    
\usepackage{pgf}
\usepackage{tikz}
\usetikzlibrary{arrows}
\usetikzlibrary{positioning}
\usetikzlibrary{shapes.geometric}
\usetikzlibrary{calc}
\usepackage[notref,notcite]{showkeys}
\usepackage[perpage,para,symbol*]{footmisc}
\usepackage{url}
\usepackage{makeidx}
\usepackage[vmargin={3.7cm,3.7cm},textwidth=.6\paperwidth]{geometry}
\usepackage{fancyhdr}

\fancypagestyle{plain}{%
\fancyhf{} 
\fancyfoot[C]{\scriptsize{\thepage}} 

}

\DeclareMathOperator{\id}{id}
\DeclareMathOperator{\Id}{Id}
\DeclareMathOperator{\rl}{rl}
\DeclareMathOperator{\leftl}{ll}
\DeclareMathOperator{\BF}{BF}
\DeclareMathOperator{\Diag}{Diag}
\DeclareMathOperator{\diag}{diag}
\DeclareMathOperator{\sgn}{sgn}
\DeclareMathOperator{\cov}{cov}

\newcommand*{\threesim}{%
\mathrel{\vcenter{\offinterlineskip
\hbox{$\sim$}\vskip-.35ex\hbox{$\sim$}\vskip-.35ex\hbox{$\sim$}}}}

\newcommand{\N}{\mathbb{N}}
\newcommand{\Z}{\mathbb{Z}}

\newcommand{\R}{\mathbb{R}}

\renewcommand{\AA}{\mathcal{A}}
\newcommand{\BB}{\mathcal{B}}
\newcommand{\CC}{\mathcal{C}}
\newcommand{\DD}{\mathcal{D}}
\newcommand{\FF}{\mathcal{F}}
\newcommand{\LL}{\mathcal{L}}
\newcommand{\OO}{\mathcal{O}}
\newcommand{\PP}{\mathcal{P}}
\newcommand{\GG}{\mathcal{G}}
\newcommand{\EE}{\mathcal{E}}
\newcommand{\HH}{\mathcal{H}}

\newcommand{\X}{\mathsf{X}}

\newcommand{\psc}{W}
\newcommand{\FE}{\sim_{\textrm{FE}}}
\renewcommand{\frac}[2]{#1 / #2}
\renewcommand{\diamond}{\diamondsuit}

\newcommand{\Xd}[1]{\X_{\textrm{diag}(#1) }}

\newcommand{\fracpart}[1]{ \langle #1 \rangle}

\newcommand{\intpart}[1]{ \lfloor #1 \rfloor}

\newcommand{\class}[1]{ \left[ #1 \right]}
\newcommand{\cclass}[1]{\texttt{#1}}
\newcommand{\program}[1]{\texttt{#1}}

\newcommand{\Cs}{C^\ast}


%% file: thesis_title.tex
\begin{titlepage}
\newgeometry{hmargin={.2\paperwidth,.2\paperwidth},vmargin={.15\paperheight,.08\paperheight}, centering} 
\noindent
\rule{\linewidth}{0.5mm}
\begin{center}
\Huge{ \textsc{On Flow Equivalence of Sofic Shifts}}
\end{center}
\rule{\linewidth}{0.5mm}
\vspace{0.5 em}\fxnote{\quad \\ DRAFT \\ \today }

\begin{flushleft}
PhD thesis by
\end{flushleft}
\vspace{0.5 em}

\begin{flushleft}
\textsc{\Large{Rune Johansen}} \\ \vspace{0.5 em}
Department of Mathematical Sciences \\ University of Copenhagen \\ Denmark
\end{flushleft}

\vspace{\stretch{1}}

\begin{flushleft}
PhD School of Science \hspace{\stretch{1}} -- \hspace{\stretch{1}} Faculty of Science \hspace{\stretch{1}} -- \hspace{\stretch{1}} University of Copenhagen 
\end{flushleft}

\restoregeometry
\end{titlepage}

%% file: thesis_address.tex
\thispagestyle{empty}

\begin{flushleft}
Rune Johansen \\
Department of Mathematical Sciences \\
University of Copenhagen \\
Universitetsparken 5 \\
DK-2100 K\o benhavn \O \\
Denmark \\
\vspace{0.5 em}
\url{rune@math.ku.dk} \\
\url{http://www.math.ku.dk/~rune/}
\end{flushleft}

\vspace{\stretch{0.15}}

\noindent
PhD thesis submitted to the PhD School of Science, Faculty of Science, University of Copenhagen, Denmark in May 2011.

\vspace{0.5em}

\begin{flushleft}
\begin{tabular}{l l}
Academic advisor: & S\o ren Eilers \\
                                  & University of Copenhagen, Denmark \\
& \\
Assessment committee:  
& Marie-Pierre B\' eal \\
& Universit\'e Paris-Est Marne-la-Vall\' ee, France\\
&\\
& Erik Christensen (chair) \\
& University of Copenhagen, Denmark \\
&\\
& Takeshi Katsura \\
& Keio University, Japan \\
\end{tabular}
\end{flushleft}

\vspace{\stretch{0.85}}

\begin{flushleft}
\copyright{} Rune Johansen, 2011 -- except for  the material in Chapter \ref{chap_covers} based on: \\
\quad \emph{On the Structure of Covers of Sofic Shifts}  \copyright{} Documenta Mathematica.

\vspace{0.5 em}

ISBN 978-87-91927-61-4
\end{flushleft}

%% file: thesis_abstract.tex
\thispagestyle{empty}

\addcontentsline{toc}{chapter}{Abstract}
\begin{center}
\textbf{\large{Abstract}}

\vspace{1 em}

\begin{minipage}{0.9 \textwidth}
The flow equivalence of sofic shifts is examined using results about the structure of the corresponding covers. A canonical cover generalising the left Fischer cover to arbitrary sofic shifts is introduced and used to prove that the left Krieger cover and the past set cover of a sofic shift can be divided into natural layers. These results are used to find the range of a flow invariant and to investigate the ideal structure of the universal $\Cs$-algebras associated to sofic shifts. The right Fischer covers of sofic beta-shifts are constructed, and it is proved that the covering maps are always $2$ to $1$. This is used to construct the corresponding fiber product covers and to classify these up to flow equivalence. 
Additionally, the flow equivalence of renewal systems is studied, and several partial results are obtained in an attempt to find the range of the Bowen-Franks invariant over the set of  renewal systems of finite type. In particular, it is shown that the Bowen-Franks group is cyclic for every member of a class of renewal systems known to attain all entropies realised by shifts of finite type.
\end{minipage}
\end{center}

\vspace{1.5cm}
\noindent
The following is a Danish translation of the abstract as required by the rules of the University of Copenhagen.
\vspace{1cm}

\begin{otherlanguage}{danish}
\begin{center}
\textbf{\large{Resum\' e}}

\vspace{1 em}

\begin{minipage}{0.9 \textwidth}
Str\o mnings\ae kvivalens af sofiske skiftrum unders\o ges vha. 
 resultater om strukturen af de tilsvarende repr\ae sentationer p\aa{} grafer.
En repr\ae sentation, der generaliserer Fischer-rep\ae sentationen til vilk\aa rlige sofiske skiftrum, introduceres og bruges til at vise, at Krieger-repr\ae sentationen og en anden hyppigt anvendt repr\ae sentation begge kan inddeles i naturlige lag.
Disse resultater benyttes til at finde billedet af en invariant for str\o mnings\ae kviva\-lens og til at unders\o ge idealgitrene i de universelle $\Cs$-algebraer knyttet til sofiske skiftrum.
Fischer-repr\ae sentationerne af sofiske beta-skift konstrueres, og det vises, at de dertilh\o rende afbildninger altid er $2$ til $1$. Dette resultat benyttes til at konstruere fiberprodukt-repr\ae sentationerne af sofiske beta-skift og til at klassificere disse op til str\o mnings\ae kvivalens.
Desuden unders\o ges str\o mnings\ae kvivalens af fornyelsessystemer, og en r\ae kke del\-resultater opn\aa s i et fors\o g p\aa{} at finde billedet af Bowen-Franks-invarianten over m\ae ngden af fornyelsessystemer af endelig type. Specielt vises det, at Bowen-Franks-gruppen er cyklisk for alle medlemmer af en klasse af fornyelsessystemer, der realiserer alle entropier, som kan realiseres af irreducible skift af endelig type.
\end{minipage}
\end{center}
\end{otherlanguage}

\vspace{\stretch{1}}

%% file: thesis_preface.tex
\addcontentsline{toc}{chapter}{Preface}
\section*{Preface}

This text constitutes my thesis for the PhD degree in mathematics from the PhD School of Science at the Faculty of Science, University of Copenhagen where I have been enrolled from May 2007 to May 2011. My time as a PhD student has primarily been spent studying representations and flow invariants of sofic shift spaces. 

In symbolic dynamics, it is generally very difficult to determine whether two objects (shift spaces) are equivalent in the natural way, and this has lead to the development of a large number of powerful invariants which can be used to provide negative answers to such questions. It has also lead to an interest in weaker notions of equivalence, and flow equivalence is one of these. One of the simplest classes of shift spaces is the class of irreducible shifts of finite type, and Franks \cite{franks} has given a very satisfactory classification of these up to flow equivalence in terms of a complete invariant that is both easy to compute and easy to compare. This result has been extended to general shifts of finite type by Boyle and Huang  \cite{boyle,boyle_huang,huang}, but very little is known about the flow equivalence of the class of irreducible sofic shifts even though it constitutes a natural first generalisation of the class of shifts of finite type.

The driving force behind this thesis has been a desire to understand the structure of various standard presentations of sofic shifts on labelled graphs and to use this understanding to derive results about flow equivalence. The complete flow equivalence classification of general irreducible sofic shifts is still a very distant goal, so focus has been on understanding the flow equivalence of various special classes of sofic shifts.

Chapter \ref{chap_shift} gives an introduction to the basic properties of shift spaces, conjugacy, and covers. In particular, the Fischer cover and the Krieger cover, which will play important roles in the following chapters, are introduced, and their basic features are examined. This chapter contains no original work, and for an expert in symbolic dynamics, it will only serve to establish notation.

Chapter \ref{chap_flow} introduces the basic definitions and properties of flow equivalence and symbol expansion. The first part of the chapter contains no original work and can be skipped by experts. 
The chapter also contains a series of simple lemmas about symbol expansion which will be used to study flow equivalence repeatedly in the following chapters, some of these lemmas may not have been seen before, but the results are unsurprising. The chapter concludes with an application of the results to a class of shift spaces known as gap shift to illustrate some of the problems encountered when working with symbol expansion.

Chapter \ref{chap_covers} is a slightly extended version of the paper \emph{On the Structure of Covers of Sofic Shifts} which has recently been published in Documenta Mathematica \cite{johansen_covers}. It investigates the structure of various standard covers of sofic shifts by introducing a new canonical cover generalising the Fischer cover and using it to prove that the left Krieger cover and the past set cover of a sofic shift can be divided into natural layers. This structure is used to find the range of a flow invariant and to investigate the ideal-lattices of the $\Cs$-algebras associated to sofic shifts. The material in this chapter was mainly developed during visits at the University of Wollongong, Australia and the University of Tokyo, Japan in 2008.

Chapter \ref{chap_beta} investigates the flow equivalence of a class of sofic shifts known as beta-shifts. To each sofic beta-shift, one can associate a shift of finite type equipped with a group action, and the main result of the chapter is a classification up to flow equivalence of these group-shifts. With a conjectured result by Boyle, this gives a complete flow classification of sofic beta-shifts. The results in this chapter were developed in the beginning of 2011.

Chapter \ref{chap_rs} concerns a special class of shift spaces called renewal systems. It has long been an open question in symbolic dynamics whether every irreducible shift of finite type is conjugate to a renewal system, and the goal of this work has been to answer the corresponding question for flow equivalence. The aim has been to find this answer by finding the range of the Bowen-Franks invariant, which is a complete invariant for flow equivalence of irreducible shifts of finite type, over the set of renewal systems of finite type. The main results have been achieved by combining insight gained through an experimental investigation with a study of the Fischer covers of renewal systems, and this has allowed a wide variety of values of the invariant to be constructed, but this is still a work in progress, and the range remains unknown. During the investigation, a flow classification is given of the class of renewal systems which Hong and Shin have proved to achieve all the entropies attained by shifts of finite type \cite{hong_shin}.

Appendix \ref{app_programs} describes the experimental approach used in the study of renewal systems. It accounts for the experimental strategy and gives a short description of the computer programs used. Additionally, it contains extra results about the range of the Bowen-Franks invariant which were not suitable for the presentation in Chapter \ref{chap_rs}. 

Where possible, the notation has been chosen to conform with the notation of the textbook \emph{Symbolic Dynamics and Coding} by Lind and Marcus \cite{lind_marcus}. An index is provided to make it easier to navigate between the chapters.

\addcontentsline{toc}{section}{Acknowledgements}
\subsection*{Acknowledgements}
This thesis was made possible by the financial support of the Faculty of Science, University of Copenhagen and the SNF-Research Center in Non-commutative Geometry at the Department of Mathematical Sciences, University of Copenhagen.
Additionally, the work has been supported by the Danish National Research Foundation (DNRF) through the Centre for Symmetry and Deformation. I am grateful to Valdemar Andersen's \emph{Rejselegat for Matematikere} for an exceptional opportunity to spend ten months visiting mathematical departments abroad, and I would like to thank the Fields Institute in Toronto, Canada, the University of Wollongong, Australia and the University of Tokyo, Japan for kind hospitality during these travels.

Among the many individuals who deserve thanks, I would first and foremost thank my advisor S\o ren Eilers for his dedication, for the many interesting discussions, and for his ability to ask the right questions.
I am grateful to David Pask and Toke Meier Carlsen for enlightening conversations and helpful comments about the material in Chapter \ref{chap_covers}, to an anonymous referee for useful comments which improved the exposition in the associated paper, and to Takeshi Katsura whose thorough corrections improved the thesis in general and Section \ref{sec_gap_shifts} in particular.
Additionally, I wish to thank Jes Frellsen for invaluable technical advice concerning the techniques used in the computer programs described in Appendix \ref{app_programs}.
Thanks is also due to my fellow PhD students for creating a great environment these past four years. In particular, I am happy to have worked here at the same time as Sara Arklint and Tarje Bargheer without whom nothing would have been the same.
I would also like to thank Signe -- my love and travel companion -- for her loving support and unwavering confidence, and finally, I am eternally grateful to our daughter Elin for all the smiles that kept me sane.

\vspace{1 cm}

\begin{flushright}
Rune Johansen \\
Copenhagen, \today
\end{flushright}

\vspace{\stretch{1}}

%% file: thesis_shift_spaces.tex
\label{chap_shift}
Here, a short introduction to the definition and properties of shift spaces is given to make the presentation self-contained. For a thorough treatment of shift spaces see \cite{lind_marcus}. 

\section{Basic definitions and properties of shift spaces}
\label{sec_background}\index{full shift}\index{shift map}\index{shift space}\index{alphabet}\index{A@$\AA(X)$}\index{shift|see{shift space}}
Let $\AA$ be a finite set with the discrete topology. The \emph{full shift} over $\AA$ consists of the space $\AA^\Z$ endowed with the product topology and the \emph{shift map} $\sigma \colon \AA^\Z \to \AA^\Z$ defined  by $\sigma(x)_i = x_{i+1}$ for all $i \in \Z$. Let $\AA^*$ be the collection of finite words (also known as blocks) over $\AA$. A closed and shift invariant subset $X \subseteq \AA^\Z$ is called a \emph{shift space} (or sometimes just a \emph{shift}), and $\AA$ is called the \emph{alphabet} of $X$. When the alphabet is not given a priori, $\AA(X)$ will denote the alphabet of $X$.

\subsection{Languages and forbidden words}
\index{forbidden words}\index{factor}\index{empty word}\index{language}\index{word!forbidden}\index{word!empty}\index{$\dashv$}\index{rl@$\rl$}\index{ll@$\leftl$}\index{XF@$\X_\FF$}\index{e*@$\epsilon$} \index{B*@$\BB(X)$}\index{B*n@$\BB_n(X)$}\index{word}\index{prefix}\index{suffix}
For each $\FF \subseteq \AA^*$, define  $\X_\FF$ to be the set of biinfinite sequences in $\AA^\Z$ which do not contain any of the \emph{forbidden words} from $\FF$.  A subset $X \subseteq \AA^\Z$ is a shift space if and only if there exists $\FF \subseteq \AA^*$ such that $X = \X_\FF$
(cf.\ \cite[Proposition 1.3.4]{lind_marcus}).
A word $w \in \AA^*$ is said to be a \emph{factor} (or subword or sub block) of $x \in X$ if there exist $i,j \in \Z$ with $i \leq j$ such that $x_i x_{i+1} \ldots x_j = w$; this will be denoted $w = x_{[i,j]}$, and a similar notation is used for factors of finite words. For $w,u \in \AA^*$ write $u \dashv w$ if $u$ is a factor of $w$. If $u = w_{[1,|u|]}$, then $u$ is said to be a \emph{prefix} of $w$; a \emph{suffix} is defined analogously.
The maps $\rl, \leftl \colon \AA^* \to \AA$ map a word to respectively its rightmost and leftmost symbol. The \emph{empty word} $\epsilon$ is considered a factor of every word and biinfinite sequence. The \emph{language} of a shift space $X$ is the set of all factors of elements of $X$ and it is denoted $\BB(X)$. The set of words of length $n$ is denoted $\BB_n(X)$.  

\begin{prop}[{\cite[Proposition 1.3.4]{lind_marcus}}]
\label{prop_language}
Let $\AA$ be an alphabet. A set $\BB \subset \AA^*$ is the language of a shift space if and only if for every $w \in \BB$ 
\begin{itemize}
\item if $v$ is a factor of $w$, then $v \in \BB$, and
\item there exist $u,v \in \BB \setminus \{ \epsilon \}$ such that $uwv \in \BB$.
\end{itemize}
If $X$ is a shift space, then $X = \X_{\AA^* \setminus \BB(X) }$.
\end{prop}

\noindent 
In particular, two shift spaces are equal if and only if they have the same language.

\index{ray}\index{shift space!one sided}
For each $x \in X$, define the \emph{left-ray} of $x$ to be $x^- = \cdots x_{-2} x_{-1}$ and define the \emph{right-ray} of $x$ to be $x^+ = x_0 x_1 x_2 \cdots$. The sets of all left-rays and all right-rays are, respectively, denoted  $X^-$ and $X^+$. The set $X^+$ is invariant under the shift map, and the pair $(X^+, \sigma)$ is an example of a \emph{one-sided shift space} (cf.\ \cite[p.\ 140]{lind_marcus}).

\index{shift space!irreducible}
A shift space $X$ is said to be \emph{irreducible} if there for every $u,w \in \BB(X)$ exists $v \in \BB(X)$ such that $uvw \in \BB(X)$. In many contexts, irreducible shift spaces can be thought of as the building blocks from which more complicated shift spaces are constructed. $X$ is irreducible if and only if there exists $x \in X$ such that the forward orbit $\{ \sigma^n(x) \mid n \geq 0 \}$ is dense in $X$.

\begin{example} \index{golden mean shift}
\label{ex_golden_mean}
Consider the alphabet $\AA = \{ a, b \}$ and the set of forbidden words $\FF = \{ bb \}$. The shift space $X = \X_\FF$ consists of all biinfinite sequences over $\AA$ where there are no consecutive $b$s. This is known as the \emph{golden mean shift} for reasons that will become apparent later.
\end{example}

\subsection{Shifts of finite type}
\index{shift of finite type}\index{SFT|see{shift of finite type}}\index{shift space!of finite type|\\see{shift of finite type}}\index{shift of finite type!$M$-step}
A shift space $X$ over $\AA$ is said to be a \emph{shift of finite type} (SFT) if $X = \X_\FF$ for some finite set $\FF$; $X$ is said to be \emph{$M$-step} if $\FF$ can be chosen such that $\lvert w \rvert = M+1$ for all $w \in \FF$. Clearly, every SFT is $M$-step for some $M \in \N$.

\begin{thm}[{\cite[Theorem 2.1.8]{lind_marcus}}]
\label{thm_m-step}
A shift space $X$ is an $M$-step SFT if and only if $uv, vw \in \BB(X)$ and $\lvert v \rvert \geq M$ implies that $uvw \in \BB(X)$. 
\end{thm}  

\index{graph}\index{path}\index{circuit}\index{vertex!connected}
\index{vertex!order}
\index{vertex!maximal}\index{graph!irreducible}
\index{graph!root of}\index{graph!essential}
\index{graph!adjacency matrix of}\index{subgraph!induced by subset}\index{irreducible component}
For countable sets $E^0$ and $E^1$ and maps $r,s \colon E^1 \to E^0$, the quadruple $E = (E^0,E^1,r,s)$ is called a \emph{directed graph} or simply a \emph{graph}. The elements of $E^0$ and $E^1$ are, respectively, the vertices and the edges of the graph, while the maps give the directions of the edges: For each edge $e \in E^1$, $s(e)$ is the vertex where $e$ starts, and $r(e)$ is the vertex where $e$ ends. A \emph{path} $\lambda = e_1 \cdots e_n$ is a sequence of edges such that $r(e_i) = s(e_{i+1})$ for all $i \in \{1, \ldots n-1 \}$. The vertices in $E^0$ are considered to be paths of length $0$. For each $n \in \N_0$, the set of paths of length $n$ is denoted $E^n$, and the set of all finite paths is denoted $E^*$. Define $E^+ = \{ e_1 e_2 \cdots \in (E^1)^\N \mid r(e_i) = s(e_{i+1}) \}$ and define the set of left-infinite paths $E^-$ in the same manner.
Extend the maps $r$ and $s$ to $E^*$, $E^+$, and $E^-$ in the natural way, e.g.\ by defining $s(e_1 \cdots e_n) = s(e_1)$. A \emph{circuit} is a path $\lambda$ with $r(\lambda) = s(\lambda)$ and $|\lambda| > 0$.
For $u,v \in E^0$, $u$ is said to be \emph{connected} to $v$ if there is a path $\lambda \in E^*$ such that $s(\lambda) = u$ and $r(\lambda) = v$, and this is denoted by $u \geq v$ \cite[Section 4.4]{lind_marcus}.
A vertex is said to be \emph{maximal}, if it is connected to all other
vertices. $E$ is said to be \emph{irreducible} (or strongly connected or transitive) if all vertices are maximal. If $E$ has a unique maximal vertex, this vertex is said to be the \emph{root} of $E$. $E$ is said to be \emph{essential} if every vertex emits and receives an edge. 
If $H \subseteq E^0$, then the \emph{subgraph of $E$ induced by $H$} is the subgraph of $E$ with vertices $H$ and edges $\{ e \in E^1 \mid s(e), r(e) \in H \}$. An \emph{irreducible component} of $E$ is an irreducible subgraph that is not contained in any larger irreducible subgraph.

\index{adjacency matrix}
\index{edge shift}
\index{X*A@$\X_A$}
For a finite directed graph $E$, the edge shift $\X_E$ is defined by
\begin{displaymath}
  \X_E = \left\{ x \in (E^1)^\Z \mid r(x_i) = s(x_{i+1}) \textrm{ for all }  i \in \Z \right\}.
\end{displaymath}
By construction, an edge shift is a $1$-step SFT over $E^1$ with $\FF = \{ ef \mid r(e) \neq s(f) \}$. If $H$ is the maximal essential subgraph of $E$, then $\X_E = \X_H$. The  edge shift $\X_E$ is irreducible if and only if the maximal essential subgraph of $E$ is an irreducible graph. 
For $E^0 = \{ v_1 , \ldots , v_n \}$, define the \emph{adjacency matrix} of $E$ to be the $n \times n$ matrix $A$ where $A_{ij} = \lvert \{ e\in E^1 \mid s(e) = v_i, r(e) = v_j \}  \rvert$. Many results about edge shifts are conveniently formulated using the adjacency matrix, so to ease the notation, $\X_A$ will denote the edge shift of the graph with adjacency matrix $A$ when $A$ is an integer matrix with non-negative entries.

\subsection{Sofic shifts}
\index{shift intertwining}\index{shift space!homomorphism of}\index{sofic shift!strictly}\index{factor map}\index{shift space!sofic|see{sofic shift}}\index{sofic shift}\index{shift space!factor of}
Let $X_1$ and $X_2$ be shift spaces with shift maps $\sigma_1$ and $\sigma_2$, respectively. A map $\varphi \colon X_1 \to X_2$ is said to be \emph{shift intertwining} if $\varphi \circ \sigma_1 = \sigma_2 \circ  \varphi$, and it is said to be a \emph{homomorphism} if it is both continuous and shift intertwining. A surjective homomorphism $\varphi \colon X_1 \to X_2$ is called a \emph{factor map}, and when such a map exists, $X_2$ is said to be a \emph{factor} of $X_1$. A shift space is called \emph{sofic}\footnote{An English transliteration of a Hebrew word that means \emph{finite}.} \cite{weiss} if it is a factor of an SFT. Every SFT is sofic, and a sofic shift which is not an SFT is called \emph{strictly sofic}. Sofic shift spaces are often simply called \emph{sofic shifts}.
labelling
\index{labelled graph}\index{labelled graph!isomorphism of}\index{adjacency matrix!symbolic}\index{labelled graph!adjacency matrix of} \index{subgraph!induced by subset} 
A \emph{labelled graph} $(E, \LL)$ over an alphabet $\AA$ consists of a directed graph $E$ and a surjective labelling map $\LL \colon E^1 \to \AA$. Extend the labelling map to $E^*$, $E^+$, and $E^-$ in the natural way, e.g.\ by defining $\LL(e_1 \cdots e_n) = \LL(e_1) \cdots \LL(e_n) \in \AA^*$. The labelled graph $(E, \LL)$ is called, respectively, \emph{essential} and \emph{finite} if $E$ is, respectively, essential and finite. Other terminology is similarly inherited by labelled graphs. An isomorphism between labelled graphs $(E, \LL_E)$ and $(F, \LL_F)$ is a pair of bijections $\varphi^i \colon E^i \to F^i$ for $i \in \{0,1\}$ such that  $r(\varphi^1(e)) = \varphi^0(r(e))$, $s(\varphi^1(e)) = \varphi^0(s(e))$, and $\LL_E(e) = \LL_F(\varphi^1(e))$ for all $e \in E^1$. If $E^0 = \{ v_1, \ldots , v_n\}$, then the \emph{symbolic adjacency matrix} of $(E, \LL)$ is the $n \times n$ matrix $A$ where $A_{ij}$ is the formal sum of the labels of all edges from $v_i$ to $v_j$. If $H \subseteq E^0$, then the \emph{subgraph of $(E,\LL)$ induced by $H$} is 
the subgraph of $E$ induced by $H$ with labelling inherited from $(E,\LL)$.

\index{shift space!presentation of|see{presentation}}
\index{sofic shift!presentation of|see{presentation}}
\index{presentation}
\index{sofic shift!irreducible}
\index{representative!of word}\index{representative!of ray}\index{source set}\index{range set}\index{s(x+)@$s(x^+)$}\index{r(x^-)@$r(x^-)$}\index{language!regular}\index{labelled graph!left-resolving}\index{labelled graph!right-resolving}\index{presentation!left-resolving}\index{presentation!right-resolving}
Given a labelled graph $(E, \LL)$, define the shift space $\X_{(E, \LL)}$ by 
\begin{displaymath}
  \X_{(E, \LL)} = \left\{ \left( \LL(x_i) \right)_i \in \AA^\Z \mid 
                          x \in \X_E  \right\}.
\end{displaymath}
The labelled graph $(E, \LL)$ is said to be a \emph{presentation} of the shift space $\X_{(E, \LL)}$, and a \emph{representative} of a word $w \in \BB(\X_{(E, \LL)}) $ is a path $\lambda \in E^*$ such that $\LL(\lambda) = w$. Representatives of rays are defined analogously. For $x \in \BB(\X_{(E, \LL)}) \cup \X_{(E, \LL)}^+$ define  the \emph{source set} by $s(x) = \{ s(\lambda) \mid \lambda \in E^* \cup E^+, \LL(\lambda) = x \}$. For each $x \in \BB(\X_{(E, \LL)}) \cup \X_{(E, \LL)}^-$, the \emph{range set} $r(x)$ is defined analogously. 

Fischer proved that a shift space is sofic if and only if it can be presented by a finite labelled graph \cite{fischer}. Equivalently, a shift space is sofic if and only if the corresponding language is \emph{regular}, i.e.\ it can be recognised by a deterministic finite automaton. A sofic shift space is irreducible if and only if it can be presented by an irreducible labelled graph (see \cite[Section 3.1]{lind_marcus}).
A presentation $(E,\LL)$ of $X$ is said to be \emph{left-resolving} if no vertex in $E^0$ receives two edges with the same label. Right-resolving presentations are defined analogously.

\begin{example}
\label{ex_even_def}
\index{even shift}
Let $\AA = \{ 1, 0 \}$ and let $\FF = \{ 10^{2n-1}1 \mid n \in \N \}$. The \emph{even shift} $\X_\FF$ is the standard example of a strictly sofic shift. Figure \ref{fig_presentation_even_shift} shows a labelled graph presenting $X$, so $X$ is sofic. If $X$ was an SFT, it would be $M$-step for some $M \in \N$, and since $10^n, 0^n1 \in \BB(X)$ for all $n \in \N$, Theorem \ref{thm_m-step} would imply that $10^n1 \in \BB(X)$ for all $n \geq M$ in contradiction with the definition of $X$.
\begin{figure} 
\begin{center}
\begin{tikzpicture}
  [bend angle=45,
   knude/.style = {circle, inner sep = 0pt},
   vertex/.style = {circle, draw, minimum size = 1 mm, inner sep =
      0pt},
   textVertex/.style = {rectangle, draw, minimum size = 6 mm, inner sep =
      1pt},
   to/.style = {->, shorten <= 1 pt, >=stealth', semithick}]
  
  \node[textVertex] (u) at (0,0) {$$};
  \node[textVertex] (v) at (2,0) {$$};

  \draw[to, bend left=45] (u) to node[auto] {0} (v);
  \draw[to, bend left=45] (v) to node[auto] {0} (u);
  \draw[to, loop left] (u) to node[auto] {1} (u);

\end{tikzpicture}
\end{center}
\caption{A presentation of the even shift.}
\label{fig_presentation_even_shift}
\end{figure}
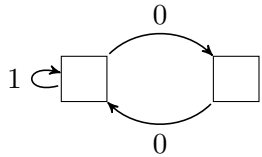
\end{example}

\begin{example}
Let $\AA = \{ 1, 0 \}$ and let $\FF = \{ 10^n 1 \mid n \textrm{ not prime} \}$. The \emph{prime gap shift} $\X_\FF$ is an example of a shift space that is not sofic.
\end{example}

\section{Conjugacy}
\label{sec_conjugacy}
\index{conjugacy}\index{conjugate}
Because shift spaces are compact, a bijective homomorphism automatically has a continuous inverse, so the isomorphisms of shift spaces are the invertible, continuous, and shift intertwining maps. Such a map is called a \emph{conjugacy}.
Two shift spaces are said to be \emph{conjugate} when there exists a conjugacy between them.
\emph{Flow  equivalence} is a weaker equivalence relation generated by conjugacy and \emph{symbol expansion} \cite{parry_sullivan} which will be examined in greater detail in Chapter \ref{chap_flow}.

\subsection{Sliding block codes}
\index{sliding block code}\index{sliding block code!memory of}
\index{sliding block code!anticipation of}\index{*f@$\Phi_\infty$}
\index{*fmn@$\Phi_\infty^{[m,n]}$}
Let $\AA_1$ and $\AA_2$ be alphabets, and let $X_1$ be a shift space over $\AA_1$. A map $\varphi \colon X_1 \to \AA_2^\Z$ is said to be a \emph{sliding block code}  with \emph{memory} $m$ and \emph{anticipation} $n$ if there exists a map $\Phi \colon \BB_{m+n+1}(X_1) \to \AA_2$ such that 
\begin{displaymath}
  (\varphi(x))_i =  \Phi(x_{i-m} x_{i-m+1} \cdots x_{i+n-1} x_{i+n}).
\end{displaymath}
The notation $\varphi = \Phi_\infty^{[m,n]}$ will be used when $\varphi$ is determined by $\Phi$ in this way; if $m=n=0$, this will simply be written $\varphi = \Phi_\infty$. The image $X_2 = \varphi(X_1)$ is a  shift space whenever $\varphi$ is a sliding block code \cite[Theorem 1.5.13]{lind_marcus}.

\begin{prop}[{\cite[Proposition 1.5.8]{lind_marcus}}]
\index{shift space!homomorphism of}
Let $X_1$ and $X_2$ be shift spaces. A map $\varphi \colon X_1 \to X_2$ is a homomorphism if and only if it is a sliding block code.
\end{prop}

\begin{example}
\label{ex_conjugacy_golden_mean}
Let $X$ be the golden mean shift as defined in Example \ref{ex_golden_mean}. Define $\Phi \colon \BB_2(X) \to \{1,2,3 \}$ by
\begin{displaymath}
\Phi(aa) =1, 
\qquad \Phi(ab) = 2, 
\quad \textrm{and} \quad \Phi(ba) = 3,  
\end{displaymath}
and define  $\varphi = \Phi_\infty^{[0,1]} \colon X \to \{ 1,2,3\}^\Z$. Let $Y = \varphi(X)$ and note that $Y$ is the edge shift of the graph with two vertices $\{v_1,v_2\}$ and edges $\{1,2,3\}$ with $s(1) = r(1) = s(2) = r(3) = v_1$, $r(2) = s(3) = v_2$. Define $\Psi \colon \{1,2,3 \} \to \{a,b\}$ by $\Psi(1) = \Psi(2) = a$ and $\Psi(3) =  b$. It is straightforward to check that $\Psi_\infty = \varphi^{-1}$, so $X$ and $Y$ are conjugate.
\end{example}

\subsection{Conjugacy of shifts of finite type}
Even for nicely behaved shifts such as SFTs, it is generally very difficult to determine whether two shift spaces are conjugate.
Williams \cite{williams} introduced an equivalence relation, called strong shift equivalence, on the set of integer matrices with non-negative entries such that two edge shifts are conjugate if and only if the corresponding adjacency matrices are strong shift equivalent, and this result will be examined in the following.

The first proposition demonstrates the usefulness of Williams's result by reducing the conjugacy problem for SFTs to a question about conjugacy of edge shifts. A sketch of the proof is included to show how the associated graph is constructed.

\begin{prop}
\label{prop_higher_block}
Every SFT is conjugate to an edge shift.
\end{prop}

\begin{proof}
Let $X$ be an $M$-step SFT. Define a graph with vertex set $H^0 = \BB_M(X)$, edge set $H^1 = \BB_{M+1}(X)$, and range and source maps $r,s \colon H^1 \to H^0$ such that
$s(w) = w_{[1,M]}$ and $r(w) = w_{[2,M+1]}$. Now,
\begin{displaymath}
\X_H = \{ (w_i)_{i \in \Z} \in \BB_{M+1}(X)^\Z \mid
                          (w_{i+1})_{[1,M]} = (w_{i})_{[2,M+1]}  \},
\end{displaymath}
and it is easy to construct a conjugacy from $X$ to $\X_H$. See \cite[Definition 1.4.1]{lind_marcus} for details.
\end{proof}

\index{higher block shift}
\noindent
The edge shift $\X_H$ constructed in the proof of Proposition \ref{prop_higher_block} is called the \emph{higher block shift of} $X$. The following theorem shows that it makes sense to talk about a conjugacy problem for shifts of finite type.

\begin{thm}[{\cite[Theorem 2.1.10]{lind_marcus}}]
A shift space that is conjugate to an SFT is itself an SFT.
\end{thm}

\begin{definition}[{\cite[Definition 7.2.1]{lind_marcus}}]
\index{elementary equivalence}\index{$\threesim$}\index{strong shift equivalence}
Let $A,B$ be integer matrices with non-negative entries. An \emph{elementary equivalence from $A$ to $B$} is a pair of integer matrices with non-negative entries $(R,S)$ such that $A = RS$ and $B = SR$. This is written $(R,S)\colon A \threesim B$. A \emph{strong shift equivalence of lag $l$} from $A$ to $B$ is a sequence of $l$ elementary equivalences $(R_i,S_i)\colon A_{i-1} \threesim A_i$ with $A_0 = A$ and $A_l = B$. $A$ and $B$ are said to be \emph{strong shift equivalent} if there exists a strong shift equivalence of lag $l$ from $A$ to $B$ for some $l$.
\end{definition}

\index{state-splitting}\index{state-amalgamation}
\noindent
It is straightforward to check that strong shift equivalence is an equivalence relation while elementary shift equivalence is not  transitive. An important special class of elementary equivalences are \emph{state-splittings} and \emph{state-amalgamations}, see \cite[\S 2.4]{lind_marcus} for details.

\begin{thm}[Williams {\cite{williams}}]
\label{thm_williams}
\index{strong shift equivalence!and conjugacy}
Let $A,B$ be integer matrices with non-negative entries. The edge shifts $\X_A$ and $\X_B$ are conjugate if and only if $A$ and $B$ are strong shift equivalent.
\end{thm}

\index{shift equivalence}\index{shift equivalence problem}\index{Williams conjecture}
\noindent
This illustrates the problem of determining whether or not two shift spaces are conjugate: The proof is not constructive and neither the lag $l$ nor the dimensions of the matrices involved in the elementary equivalences are bounded. Hence, there is no way to use this theorem to check whether two SFTs are conjugate. 
A weaker equivalence relation called \emph{shift equivalence} \cite[Definition 7.3.1]{lind_marcus} was introduced by Williams \cite{williams} in an attempt to alleviate this problem, and a long-standing conjecture in symbolic dynamics stated that two matrices are shift equivalent if and only if they are strong shift equivalent. This is known as the \emph{shift equivalence problem} or the \emph{Williams conjecture}, and for a long time it was arguably the most important problem in symbolic dynamics. Eventually, shift equivalence and strong shift equivalence were proved to be different by Kim and Roush who constructed a counterexample \cite{kim_roush_1,kim_roush_2}. 

\subsection{Conjugacy invariants}
\label{sec_conjugacy_ivariants}
As mentioned in the previous section, it is generally difficult to determine whether two shift spaces are conjugate, so it is useful to have a collection of invariants in order to be able to provide negative answers to such questions. The first and arguably most important of these invariants is the entropy:

\begin{definition}[{\cite[Definition 4.1.1]{lind_marcus}}]
\index{entropy}
The \emph{entropy} of a shift space $X$ is $h(X) = \lim_{n \to \infty} \frac{1}{n} \log \left \lvert \BB_n(X) \right \rvert$.
\end{definition}

\index{golden mean shift!entropy of}
\noindent
If $X$ is a shift space with alphabet $\AA$, then $\BB_n(X) \leq \lvert \AA \rvert^n$, so $h(X) \leq  \log \lvert \AA \rvert$. Equality holds if $X$ is the full $\AA$-shift. The golden mean shift gets it name from the fact that it's entropy is $(1+\sqrt 5)/2$ (cf.\ \cite[Example 4.1.4]{lind_marcus}).

\begin{prop}[{\cite[Proposition 4.1.9]{lind_marcus}}]
Let $X$ and $Y$ be shift spaces. If $Y$ is a factor of $X$, then $h(Y) \leq  h(X)$. In particular, $h(X) = h(Y)$ when $X$ and $Y$ are conjugate.
\end{prop}

The entropy of an edge shift (and hence of an arbitrary SFT) can be computed using Perron-Frobenius theory \cite[\S 4.2]{lind_marcus}:

\begin{thm}[{\cite[Theorem 4.4.4]{lind_marcus}}]
If $A$ is an integer matrix with non-negative entries, then $h(\X_A) = \log \lambda_A$ where $\lambda_A >0$ is the maximal eigenvalue of $A$.  
\end{thm}

\index{Perron number}\index{Perron number!weak}
\noindent
A real number $\lambda \geq 1$ is a 
\emph{Perron number} if it is an algebraic integer that strictly dominates all its other algebraic conjugates. A \emph{weak Perron number} is a real number $\lambda^p$ for which $\lambda$ is a Perron number and $p \in \N$. A number $\log \mu$ is the entropy of an SFT if and only if $\mu$ is a weak Perron number \cite[Theorem 11.1.5]{lind_marcus}.
The following proposition shows that these results about the entropies of  SFTs can also be used to compute the entropies of sofic shifts.

\begin{prop}[{\cite[Proposition 4.1.13]{lind_marcus}}]
\label{prop_entropy_graph}
If a labelled graph $(G, \LL)$ is either right- or left-resolving, then $h(\X_{(G, \LL)}) = h(\X_G)$.
\end{prop}

\noindent
Using Examples \ref{ex_even_def} and \ref{ex_conjugacy_golden_mean} and Proposition \ref{prop_entropy_graph}, it follows that the even shift and the golden mean shift have the same entropy.

\index{periodic point}\index{period}
Let $X$ be a shift space. An element $x \in X$ is said to have \emph{period} $p$ if $\sigma^p(x) = x$. It is straightforward to prove that for each $n \in \N$, the number of points with period $n$ and the number of points with least period $n$ are conjugacy invariants.

\begin{definition}[{\cite[Definition 7.4.15]{lind_marcus}}]
\index{Bowen-Franks group}
\index{Bowen-Franks group}
\index{BF@$\BF$|see{Bowen-Franks group}}
Let $A$ be a $n \times n$ integer matrix with non-negative entries. The co-kernel $\BF(A) = \Z^n / \Z^n (\Id - A)$ is called the \emph{Bowen-Franks group} of $A$. 
\end{definition}

\index{Smith normal form}
\index{Bowen-Franks group!computation of}
\noindent
When $A$ is an $n \times n$ integer matrix with non-negative entries, row and column operations over $\Z$ can be used to transform $\Id - A$ into a unique diagonal matrix $\diag(d_1, \ldots, d_n)$ where $d_i \vert d_{i+1}$. This is  called the \emph{Smith normal form} of $A$ (see e.g.\ \cite{newman}). The row and column operations are given by matrices invertible over $\Z$, so $\BF(A) \simeq \Z / d_1\Z \oplus \cdots \oplus \Z / d_n\Z$, and there is an algorithm for computing the Smith normal form, so it is straightforward to compute the Bowen-Franks group when the adjacency matrix is known. 

\begin{thm}[{\cite[Theorem 7.4.17]{lind_marcus}}]
If $A$ and $B$ are strong shift equivalent integer matrices with non-negative entries, then $\BF(A) \simeq \BF(B)$. \label{thm_bf_group_invariant}
\end{thm}

\index{Bowen-Franks group!of general SFT}
\noindent
In fact, the Bowen-Franks group is also an invariant of shift equivalence, but that will not be important in this context; see the proof in \cite{lind_marcus} for details. For an arbitrary SFT $X$, Proposition \ref{prop_higher_block} and Theorem \ref{thm_bf_group_invariant} are used to define the \emph{Bowen-Franks group} $\BF(X)$ to be the unique Bowen-Franks group of the adjacency matrix of an edge shift conjugate to $X$.

\section{Covers}
\label{sec_covers}
Every sofic shift can be presented by infinitely many different labelled graphs, so it is natural to look for general ways to find presentations with nice properties.

\index{cover}
\begin{definition}
Let $S$ be a subclass of the class of sofic shifts. A \emph{cover} defined on $S$ is a pair of maps $(\cov, \pi_{\cov})$ such that for each $X \in S$, $\cov (X)$ is an SFT and $\pi_{\cov} (X) \colon \cov(X) \to X$ is a factor map.
\end{definition}

\index{covering map}
\noindent 
To improve readability, the notation $\pi_{\cov (X)}$ will be used instead of $\pi_{\cov} (X)$ in the following. The map $\pi_{\cov (X)}$ is called the \emph{covering map}.

\index{cover!as graph}
\begin{remark} 
\label{rem_cover_graph}
If $S$ is a subclass of the class of sofic shifts and if each $X \in S$ has a distinguished presentation $(G_X , \LL_X)$, then $\cov(X) = \X_{G_X}$ and $\pi_{\cov (X)} = (\LL_X)_\infty$ defines a cover. 
I.e.\ $\cov(X)$ is the edge shift of the distinguished presentation of $X$, and $\pi_{\cov (X)}$ maps a path $\lambda \in \X_{G_X}$ to the corresponding biinfinite sequence of labels $(\LL_X(\lambda_i))_{i \in \Z}$.
With a slight abuse of terminology, the labelled graph $(G_X , \LL_X)$ itself will be said to be a \emph{cover} of $X$ since this is done in large parts of the literature on the subject. The covers considered in the following will almost exclusively be of this kind.
\end{remark}

\index{predecessor set}\index{P@$P_\infty$}
\begin{definition}
(cf.\ \cite[Sections I and III]{jonoska_marcus} and \cite[Exercise
    3.2.8]{lind_marcus}) Let $X$ be a shift space. For a right-ray or word $x \in X^+ \cup \BB(X)$, define the \emph{predecessor set} of $x$ to be the set of left-rays which may  precede $x$ in $X$, i.e.\ $P_\infty(x^+) = \{ y^- \in X^- \mid y^- x^+ \in X \}$ when $x^+ \in X^+$. 
\end{definition}

\index{F@$F_\infty$}\index{follower set}\index{F@$F(w)$}\index{P@$P(w)$}
\noindent The \emph{follower set} $F_\infty(x)$ of a word or left-ray $x$ is defined analogously. Predecessor and follower sets consisting of finite words instead of rays are useful in some contexts, so for $w \in \BB(X)$, define $P(w) = \{ v \in \BB(X) \mid vw \in \BB(X) \}$ and $F(w) = \{ v \in \BB(X) \mid wv \in \BB(X) \}$.

\index{predecessor set!of vertex}
\begin{definition}
\label{def_predecessor_set_of_vertex}
\index{vertex!predecessor set of}
\index{labelled graph!predecessor-separated}
\index{labelled graph!follower-separated}
Let $(E, \LL)$ be a labelled graph presenting a sofic shift $X$. For each $v \in E^0$, define the \emph{predecessor set} of $v$ to be the set of left-rays in $X$ which have a presentation terminating at $v$. This is denoted $P_\infty^E(v)$, or just $P_\infty(v)$ when $(E, \LL)$ is understood from the context. The presentation $(E, \LL)$ is said to be \emph{predecessor-separated} if $P_\infty^E(u) \neq P_\infty^E(v)$ when $u,v \in E^0$ and $u \neq v$. \emph{Follower-separated} vertices and labelled graphs are defined analogously.
\end{definition}

\noindent
In the following, there will be particular focus on covers of the type discussed in Remark \ref{rem_cover_graph} for which the labelled graphs are predecessor-separated.

\subsection{The Krieger cover}
\label{sec_shift_lkc}
The first, and arguably most important, example of a cover defined on the entire class of sofic shifts is the Krieger cover introduced in \cite{krieger_sofic_I}. This is an example of a cover defined as described in Remark \ref{rem_cover_graph}.

\begin{definition}[{Krieger \cite{krieger_sofic_I}}]
\index{Krieger cover}\index{sofic shift!Krieger cover of|\\see{Krieger cover}}
The \emph{left Krieger cover} of a shift space $X$ is the labelled graph
$(K, \LL_K)$ where $K^0 = \{ P_\infty(x^+) \mid x^+ \in X^+\}$,
and where there is an edge labelled $a \in \AA(X)$ from $P \in K^0$ to
$P' \in K^0$ if and only if there exists $x^+ \in X^+$ such that $P
= P_\infty(a x^+)$ and $P' = P_\infty(x^+)$.
\end{definition}

The right Krieger cover is defined analogously using the follower sets of left-rays as vertices. A shift space is sofic if and only if the number of predecessor sets is finite \cite[\S 2]{krieger_sofic_I}, so the left Krieger cover is a finite labelled graph exactly when the shift space is sofic. It is easy to check that the left Krieger cover of $X$ is a left-resolving and predecessor-separated presentation of $X$, and that $K$ is always an essential graph.

The left Krieger cover is sometimes (\cite{carlsen,carlsen_matsumoto}) defined using predecessor sets consisting of the finite words which can precede a given right-ray. This is also sometimes called the Perron-Frobenius  cover \cite{samuel}. However, there is a natural bijective  correspondence between the predecessor sets consisting of left-rays, and the predecessor sets consisting of finite words, and the two definitions are equivalent. A benefit of using words instead of rays is that the definition can then be used for both one-sided and two-sided shift spaces.

\begin{prop}[{\cite[Exercise 2.2.8]{lind_marcus}}]
\label{prop_lkc_sft}
Let $X$ be a sofic shift with left Krieger cover $(K, \LL)$, then $X$ is an SFT if and only if the sliding block code $\LL_\infty \colon \X_K \to X$ induced by $\LL$ is a conjugacy.
\end{prop}

The following simple facts are needed in the investigation of the
structure of the left Krieger cover. 

\begin{lem} 
\label{lem_K_basic_facts} 
Let $X$ be a sofic shift with left Krieger cover $(K,\LL_K)$.
\begin{enumerate}
\item Each $x^+ \in X^+$ has a
  presentation in $(K,\LL_K)$ starting at $P_\infty(x^+)$.\label{K_facts_start} 
\item If $P \in K^0$, then $P_\infty^K(P) = P$.\label{K_facts_equal} 
\item \label{K_facts_subset} If $P \in K^0$, and there is a presentation of $x^+$ starting at $P$, then $P \subseteq P_\infty(x^+)$.
\end{enumerate}
\end{lem}

\begin{proof}
\quad
\begin{enumerate}
\item This follows directly from the definition of the left Krieger
cover. 
\item Choose $x^+ \in X^+$ such that $P = P_\infty(x^+)$. By (\ref{K_facts_start}) there
is a path with label $x^+$ starting 
at $P$. The left Krieger cover is a presentation of $X$, so if there
is a path with label $y^-$ terminating at $P$, then $y^-x^+ \in X$, and
hence, $y^- \in P_\infty(x^+)$. On the other hand, if $ \ldots y_{-2} y_{-1} =
y^- \in P_\infty(x^+)$, then there is a path 
\begin{displaymath}
  \cdots \stackrel{y_{-3}}{\longrightarrow} P_\infty(y_{-2} y_{-1} x^+)
         \stackrel{y_{-2}}{\longrightarrow} P_\infty(y_{-1} x^+) 
         \stackrel{y_{-1}}{\longrightarrow} P_\infty (x^+)
\end{displaymath} 
in the left Krieger cover.  
\item If $y^- \in P$, then there is a path
labelled $y^-$ terminating at $P$ by (\ref{K_facts_equal}), so $y^-x^+ \in X$ and therefore
$y^- \in P_\infty(x^+)$. 
\end{enumerate} 
\end{proof}

\noindent Consider the right-ray $x^+ = 0^\infty$ in the even shift
(see Examples \ref{ex_even_def} and \ref{ex_even}), to see that a right-ray can start at a vertex $P \neq P_\infty(x^+)$. 

\subsection{The past set cover}

The following definition gives a cover that is closely related to -- but not always isomorphic to -- the left Krieger cover. 

\index{past set cover}
\index{sofic shift!past set cover of|\\see{past set cover}}
\begin{definition}[{\cite[p.\ 73]{lind_marcus}}]
The \emph{past set cover} of a shift space $X$ is the labelled graph $(\psc, \LL_\psc)$ where $\psc^0 = \{ P_\infty(w) \mid w \in \BB(X) \}$, and where there is an edge labelled $a \in \AA(X)$ from $P \in \psc^0$ to $P'\in \psc^0$ if and only if there exists $w \in \BB(X)$ such that $P = P_\infty(a w)$ and $P' = P_\infty(w)$.
\end{definition}

\index{future set cover}
\index{sofic shift!future set cover of|\\see{future set cover}}
It is easy to check that the past set cover of $X$ is a left-resolving and predecessor-separated presentation of $X$. The past set cover is sometimes (e.g.\ \cite{bates_eilers_pask}) defined using predecessor sets consisting of the finite words which can precede a given word. However, as with the Krieger cover, there is a natural bijective correspondence between the predecessor sets consisting of left-rays, and the predecessor sets consisting of finite words, and the two definitions are equivalent. The definition used here was chosen to highlight the similarities with the left Krieger cover, but
despite the similarity of the definitions, the left Krieger cover and the past set cover are not always identical. A later example (Example \ref{ex_difference_lkc_psc}) illustrates this difference.
The \emph{future set cover} is defined analogously to the past set cover by using the follower sets of words as vertices.

\subsection{The Fischer cover}
\label{sec_fischer}
\index{Fischer cover}
\index{sofic shift!Fischer cover of|\\see{Fischer cover}}
Fischer proved that up to labelled graph isomorphism every irreducible sofic shift has a unique left-resolving presentation with fewer vertices than any other left-resolving presentation \cite{fischer}. This defines a cover of the type discussed in Remark \ref{rem_cover_graph} called the \emph{left Fischer cover} of $X$. 

\index{intrinsically synchronizing}
\begin{definition}
(cf.\ \cite[Section III]{jonoska_marcus}) A word $v \in \BB(X)$ is \emph{intrinsically synchronizing} if $uvw \in \BB(X)$ whenever $uv \in \BB(X)$ and $vw \in \BB(X)$. A right-ray $x^+$ is \emph{intrinsically synchronizing} if it contains an intrinsically synchronizing word as a factor. 
\end{definition}

\noindent
By Theorem \ref{thm_m-step}, a shift space $X$ is an $M$-step SFT if and only if every $w \in \BB(X)$ with $\lvert w \rvert \geq M$ is intrinsically synchronizing.

\index{Fischer cover!as subgraph of Krieger cover}\index{Krieger cover!top irreducible component of}
The left Fischer cover of an irreducible sofic shift $X$ is isomorphic to an irreducible subgraph of the left Krieger
cover of $X$ induced by the set of vertices that are predecessor sets of intrinsically synchronizing right-rays. This result can be traced back to \cite[Lemma 2.7]{krieger_sofic_I}. In a left-resolving and predecessor-separated labelled graph, every right-ray can be extended on the left to an intrinsically synchronizing right-ray \cite[Proposition 3.3.16]{lind_marcus}, so if $P, P' \in K^0$ and $P$ is a vertex in the irreducible subgraph of the left Krieger cover identified with the left Fischer cover, then $P \geq P'$. Hence, this irreducible component will be called the \emph{top irreducible component} of the left Krieger cover. This leads to the following two  corollaries to Proposition \ref{prop_lkc_sft}.

\begin{cor}
If $X$ is an irreducible SFT then the left Fischer cover and the left Krieger cover of $X$ are isomorphic.
\end{cor} 

\noindent
The opposite implication is, however, not true in general (see e.g.\ Example \ref{ex_justifying}).

\begin{cor}\label{cor_lfc_conj}
Let $X$ be an irreducible sofic shift with left Fischer cover $(F, \LL)$, then $X$ is an SFT if and only if the sliding block code $\LL_\infty \colon \X_F \to X$ induced by $\LL$ is a conjugacy.
\end{cor} 

The following theorem gives an important characterisation of the left Fischer cover which will be used repeatedly in the following.

\begin{thm}[{\cite[Corollary 3.3.19]{lind_marcus}}]\label{thm_lfc_char}
Let $X$ be an irreducible sofic shift. Then a labelled graph $(G, \LL)$ presenting $X$ is (isomorphic to) the left Fischer cover of $X$ if and only if it is irreducible, left-resolving, and predecessor-separated.  
\end{thm}

\begin{example}
\label{ex_even}
\index{even shift!Krieger cover of}\index{even shift!Fischer cover of}
Let $X$ be the even shift discussed in Example \ref{ex_even_def} and consider the predecessor sets:
\begin{align*}
  P_1 &= P_\infty( 0^{2n}1 x^+) = \{ y^-10^{2k} \in X^- \mid k \in \N_0 \}
                           \cup \{ 0^\infty \} \\
  P_2 &= P_\infty( 0^{2n+1}1 x^+) = \{ y^-10^{2k+1} \in X^- \mid k \in \N_0 \}
                           \cup \{ 0^\infty \}  \\
  P_3 &= P_\infty(0^\infty) = X^-.   
\end{align*}
A word  (and hence a right-ray) is intrinsically synchronizing if and
only if it contains a $1$, so the left Fischer cover is the
irreducible subgraph of the left Krieger cover which contains the
two vertices corresponding to the first two predecessor sets.
Figure \ref{fig_even_shift} shows the left Fischer cover and the
left Krieger cover of the even shift.
\end{example}

\begin{figure} 
\begin{center}
\begin{tikzpicture}
  [bend angle=45,
   knude/.style = {circle, inner sep = 0pt},
   vertex/.style = {circle, draw, minimum size = 1 mm, inner sep =
      0pt,},
   textVertex/.style = {rectangle, draw, minimum size = 6 mm, inner sep =
      1pt},
   to/.style = {->, shorten <= 1 pt, >=stealth', semithick}]
  
  \node[knude] (trans) at (5,0) {} ;

  \node[textVertex] (u) at (0,0) {$P_1$};
  \node[textVertex] (v) at (2,0) {$P_2$};

  \node[textVertex] (P1) at ($(u)+(trans)$) {$P_1$};
  \node[textVertex] (P2) at ($(v)+(trans)$) {$P_2$};
  \node[textVertex] (P3) at ($(0,-2)+(trans)$) {$P_3$};

  \node[knude] (fiscer)  at (1,1.5) {$(F, \LL_F)$};
  \node[knude] (krieger) at ($(fiscer)+(trans)$) {$(K, \LL_K)$};

  \draw[to, bend left=45] (u) to node[auto] {0} (v);
  \draw[to, bend left=45] (v) to node[auto] {0} (u);
  \draw[to, loop left] (u) to node[auto] {1} (u);
  \draw[to, bend left=45] (P1) to node[auto] {0} (P2);
  \draw[to, bend left=45] (P2) to node[auto] {0} (P1);
  \draw[to, loop left] (P1) to node[auto] {1} (P1);
  \draw[to] (P1) to node[auto, swap] {1} (P3);
  \draw[to, loop below] (P3) to node[auto] {0} (P3);
\end{tikzpicture}
\end{center}
\caption[Covers of the even shift.]{Left Fischer cover and left Krieger cover of the even shift. Note that the left Fischer cover is the top irreducible component of the left Krieger cover.}
\label{fig_even_shift}
\end{figure}
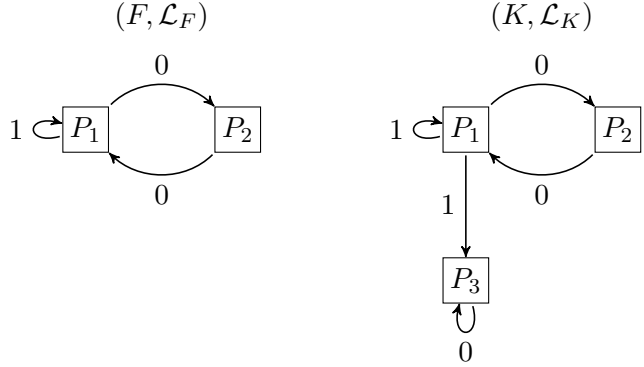

\index{labelled graph!right-closing}\index{labelled graph!right-closing!at vertex}
The introduction of the left Fischer cover allows simple definitions of two nicely behaved classes of sofic shifts called, respectively, almost finite type shifts \cite{marcus,nasu_aft} and near Markov shifts \cite{boyle_krieger}. A labelled graph $(E, \LL)$ is said to be \emph{right-closing with lag $l$ at $v \in E^0$} if whenever $\lambda, \mu \in E^l$, $s(\lambda) = s(\mu) = v$, and $\LL(\lambda) = \LL(\mu)$, then $\lambda_1 = \mu_1$. 
$(E, \LL)$ is said to be \emph{right-closing with lag $l$} if it is right-closing with lag $l$ at every vertex $v \in E^0$. This is a generalisation of the right-resolving graphs introduced earlier since a graph is right-resolving if and only if it is right-closing with lag $1$. Left-closing graphs are defined analogously.

\begin{definition}[{\cite[Definition 5.1.4]{lind_marcus}}]
\index{shift space!of almost finite type}\index{AFT}
An irreducible sofic shift $X$ is said to be of \emph{almost finite type} (AFT) if the left Fischer cover is right-closing.
\end{definition}

\noindent
This definition appears asymmetric, but in fact, an irreducible sofic shift is AFT if and only if the right Fischer cover is left-closing. These are two of many different equivalent definitions of AFT shifts. Note that every SFT is an AFT. The following notation is taken from \cite{boyle_carlsen_eilers}.

\index{Fischer cover!$m$ to $1$}
\begin{definition}
Let $X$ be an irreducible sofic shift with left Fischer cover $(F,\LL_F)$ and let $\pi \colon \X_F \to X$ be the covering map. Define
\begin{itemize}
\item the \emph{multiplicity} $m(\pi) = \sup \{ \lvert \pi^{-1}(x) \rvert \mid x \in X \}$,
\item the \emph{multiplicity set} $M(\pi) = \{ x \in X \mid \lvert \pi^{-1}(x) \rvert > 1 \}  \subseteq X$, and
\item the \emph{multiplicity set} $M^{-1}(\pi) = \pi^{-1}(M(\pi))  \subseteq \X_F$.
\end{itemize}
\end{definition}

\noindent
Since the left Fischer cover is left resolving, $m(\pi) < \lvert F^0 \rvert < \infty$, so the map $\pi$ is $m(\pi)$ to $1$.
An irreducible sofic shift is AFT if and only if $M^{-1}(\pi)$ is closed \cite{marcus,nasu_aft}, so for an AFT, both $M^{-1}(\pi)$ and $M(\pi)$ are shift spaces. 

\index{near Markov shift}
\begin{definition}[{Boyle and Krieger \cite{boyle_krieger}}]
Let $X$ be an irreducible sofic shift with left Fischer cover $(F,\LL_F)$ and let $\pi \colon \X_F \to X$ be the covering map. $X$ is said to be \emph{near Markov} if $M(\pi)$ is finite. 
\end{definition}

\noindent
Note that every near Markov shift is an AFT and that every SFT is near Markov.

\subsection{Canonical covers}
A cover is particularly useful if it respects conjugacy. This is captured by the following definition.

\index{cover!canonical}
\begin{definition}
A cover $(\cov, \pi_{\cov})$ defined on a subclass $S$ of the class of sofic shifts is said to be \emph{canonical} if whenever $X,Y \in S$ and $\varphi \colon X \to Y$ is a conjugacy, there exists a unique conjugacy $\cov(\varphi) \colon \cov(X) \to \cov(Y)$ such that the following diagram commutes
\begin{center}
\begin{tikzpicture}
  [bend angle=10,
   clear/.style = {rectangle, minimum width = 5 mm, minimum height = 5 mm, inner sep = 0pt}, 
   map/.style = {->, shorten >= 4 pt, shorten <= 4 pt,  semithick}]
  
  \node[clear] (X) at (0,0) {$X$};
  \node[clear] (Y) at (4,0) {$Y$};
  \node[clear] (cX) at (0,2) {$\cov (X)$};
  \node[clear] (cY) at (4,2) {$\cov (Y)$};

  \draw[map] (X) to node[auto] {$\varphi$} (Y);
  \draw[map] (cX) to node[auto,swap] {$\pi_{\cov(X)}$} (X);
  \draw[map] (cX) to node[auto] {$\cov(\varphi)$} (cY);
  \draw[map] (cY) to node[auto] {$\pi_{\cov(Y)}$} (Y);

\end{tikzpicture}
\end{center}
\end{definition}

\index{Krieger cover!is canonical}\index{Fischer cover!is canonical}
\begin{thm}[{Krieger \cite{krieger_sofic_I}}]
\label{thm_lkc_lfc_canonical}
The left and right Krieger and Fischer covers are canonical.
\end{thm}

\noindent
This result was also proved by Nasu \cite{nasu}, and the technique of that proof will be discussed in detail in Chapter \ref{chap_covers} where the same method will be used to show that a generalisation of the Fischer cover is also canonical. It should also be possible to use the same technique to prove that the past set cover is canonical.

%% file: thesis_flow.tex
\label{chap_flow}

As mentioned in Section \ref{sec_conjugacy}, it is generally hard to determine when two shift spaces are conjugate. This makes it attractive to consider weaker equivalence relations as well, and flow equivalence is one of these. The definition is purely topological, but a result by Parry and Sullivan \cite{parry_sullivan} gives a description of flow equivalence in terms of symbolic dynamics. Flow equivalence of shift spaces has been studied much less intensely than conjugacy, but there are several interesting results. In particular, Franks \cite{franks} has proved that irreducible SFTs can be classified up to flow equivalence by a complete invariant that is both easy to compute and easy to compare, and Boyle and Huang \cite{boyle,boyle_huang,huang} have extended this to reducible SFTs at the cost of introducing a much less tractable invariant. Much less is known about the flow equivalence of strictly sofic shifts, but Boyle, Carlsen, and Eilers have given a classification of near Markov shifts and AFT shifts where the covering maps of the Fischer covers are $2$ to $1$ \cite{boyle_carlsen_eilers}. The complete flow classification of (irreducible) sofic shifts is still a very distant goal, so the work in this thesis has been focused on investigating invariants and classifying special classes of sofic shift spaces such as the ones considered in Chapters \ref{chap_beta} and \ref{chap_rs}.

Section \ref{sec_fe_se} gives basic definitions and properties of flow equivalence and introduces the fundamental result of Parry and Sullivan. Section \ref{sec_fe_inv} gives a brief account of classification results and invariants which will be used in the following, and Section \ref{sec_se_lemmas} gives a number of lemmas describing concrete flow equivalences which will be used repeatedly in later chapters. Finally, Section \ref{sec_gap_shifts} applies these results to a class of sofic shifts called gap shifts to illustrate some of the problems that arise when working with flow equivalence.

\section{Flow equivalence and symbol expansion}
\label{sec_fe_se}
A short introduction to the basic definitions and properties of flow equivalence will be given in the following. Particular focus will be on a process called symbol expansion and on the important result by Parry and Sullivan which allows a dynamical interpretation of flow equivalence. For a more detailed introduction to flow equivalence of shift spaces, see \cite{thamsborg}. 

\subsection{Flow equivalence}
\label{sec_flow_def}
Let $(X,\sigma)$ be a shift space. Intuitively, the goal of the following definition is to make sense of $\sigma^r(x)$ for $x \in X$ and $r \in \R \setminus \Z$. 

\index{suspension flow}\index{shift space!suspension flow of}\index{S@$SX$}
\begin{definition}
Let $(X ,\sigma)$ be a shift space, equip $X \times \R$ with the product topology, and define an equivalence relation on $X \times \R$ by $(x,t) \sim (\sigma(x), t-1)$. The \emph{suspension flow} $SX$ of $X$ is the quotient space $X \times \R / {\sim}$. The equivalence class of $(x,t)$ in $SX$ will be denoted $[x,t]$. 
\end{definition}

\index{mapping torus}\index{continuous flow}\index{shift space!homomorphism of!suspension of}
\noindent
The suspension flow is also sometimes (e.g.\ in \cite{boyle_carlsen_eilers}) called the \emph{mapping torus}.
The suspension flow is a special case of a more general topological concept called a \emph{continuous flow} (see e.g.\ \cite[p.\ 456]{lind_marcus}). The suspension flow of a shift space is a compact Hausdorff space. A homomorphism $\varphi \colon X \to Y$ between shift spaces induces a homeomorphism $S\varphi \colon SX \to SY$.

\index{flow line}
Let $X$ be a shift space. For each $[x,t] \in SX$ and $r \in \R$, define $[x,t] + r = [x,t+r]$. For each $z \in SX$, the set $\{ z + r \in SX \mid r \in \R \}$ is called a \emph{flow line}.
If $Y$ is a shift space and if $\Phi \colon SX \to SY$ is a homeomorphism, then for each $z \in SX$, there exists a map $\varphi_z \colon \R \to \R$ such that $\Phi(z + r) = \Phi(z) + \varphi_{z}(r)$ for all $r \in \R$ (see e.g.\ \cite[Theorem 6]{thamsborg} for details), i.e.\ a homeomorphism maps flow lines to flow lines.

\index{flow equivalence}\index{shift space!flow equivalent}\index{$\FE$}
\begin{definition}
Let $X$ and $Y$ be shift spaces. A homeomorphism $\Phi \colon SX \to SY$ is said to be a \emph{flow equivalence} if
there for each $z \in SX$ exists a monotonically increasing map $\varphi_z \colon \R \to \R$ such that $\Phi(z + r) = \Phi(z) + \varphi_{z}(r)$. In this case, $X$ and $Y$ are said to be \emph{flow equivalent} and this is denoted $X \FE Y$.
\end{definition}

\noindent
I.e.\ a homeomorphism is a flow equivalence if it maps flow lines to flow lines in a direction preserving manner. It is easy to check that flow equivalence is an equivalence relation and that conjugate shift spaces are also flow equivalent.

\subsection{Symbol expansion}
\index{exp@$\exp_A$}
Let $X$ be a shift space with alphabet $\AA$, let $A \subseteq \AA$, and introduce a new symbol $\diamond_a \notin \AA$ for each $a \in A$.
The goal is to define a new shift space obtained from $X$ by replacing each occurrence of $a \in  A$ by $a \diamond_a$ in each $x \in X$. Let $\AA_{\exp_A} = \AA \cup \{ \diamond_a \mid a \in A \}$, and  define $\tau_A \colon \AA  \to \AA_{\exp_A}^*$ by 
\begin{displaymath}
\tau_A(a) = \left\{ \begin{array}{l c l}
a \diamond_a & , & a \in A \\
a                     & , & a \notin A 
\end{array} \right. .
\end{displaymath}
Extend $\tau_A$ to a map on $\AA^*$ in the natural way. In order to construct a corresponding map of biinfinite sequences, it is necessary to choose how to index the entries in the image, so define $\exp_A \colon X \to \AA_{\exp_A}^\Z$ by
\begin{displaymath}
  \exp_A( \cdots x_{-1}. x_0 x_1 \cdots ) = \cdots \tau_A (x_{-1}). \tau_A (x_{0}) \tau_A (x_{1}) \cdots  	                         
\end{displaymath}
Here, the period denotes the position to the left of entry number 0, so if $x \in X$ with $x_0 = a$ and $y = \exp_A(x)$, then $y_0 = a$ and $y_1 = \diamond_a$. The image $\exp_A(X)$ is closed but not shift invariant. This leads to the following definition.

\begin{definition}
Define $X^{\exp_A} = \exp_A(X) \cup \sigma(\exp_A(X))$. 
\end{definition}

\index{symbol expansion}\index{$a \mapsto a \diamond$}
\noindent
$X^{\exp_A}$ it is said to be \emph{the shift space obtained from $X$ via a symbol expansion of $A$}. This is justified by the following proposition. For $A = \{ a \}$, write $\exp_a$ instead of $\exp_{\{ a\}}$. In this case, the symbol expanded shift $X^{\exp_a}$ will often be denoted $X^{a \mapsto a \diamond}$ to emphasise the procedure.

\begin{prop}
$X^{\exp_A}$ is a shift space over $\AA_{\exp_A}$ with language $\BB_{\exp_A} = \tau_A( \BB(X) ) \cup \{ w  \mid \exists a \in A : aw \in \tau_A( \BB(X) ) \}$.
\end{prop}

\begin{proof}
$\BB_{\exp_A}$ satisfies the conditions in Proposition \ref{prop_language} since $\BB(X)$ satisfies them, and it is easy to check that $w \in \AA_{\exp_A}^*$ is a factor of some $x \in X^{\exp_A}$ if and only if $w \in \BB_{\exp_A}$. 
\end{proof}

\noindent
Similarly, if $\FF$ is a set of forbidden words for $X$, then 
$\FF(X^{\exp_A}) = \tau_A \FF \cup \{ b \diamond_a \mid b \neq a\}$ is 
a set of forbidden words for $X^{\exp_A}$, and $X^{\exp_A}$ is the shift space obtained from $X$ by replacing every $a$ by $a \diamond_a$ in every $x \in X$ for each $a \in A$. Symbol expansion is often thought of as a time-delay: Consider a biinfinite sequence $x \in X$ as being read one symbol at a time by a machine taking equally long reading each symbol. In $X^{\exp_A}$, the word $a \diamond_a$ takes twice as long to read as the individual symbols, but in many ways it still behaves like a single symbol. 

Symbol expansion is important in the study of flow equivalence because of the following fundamental theorem which allows flow equivalence to be interpreted in a natural way in terms of symbolic dynamics.

\begin{thm}[{Parry and Sullivan \cite{parry_sullivan}}]
\label{thm_parry-sullivan}
\index{flow equivalence!and symbol expansion}
Shift spaces $X$ and $Y$ are flow equivalent if and only if there exist shift spaces $X = X_1, \ldots, X_n = Y$ such that  for each $1 \leq i \leq n-1$, either
\begin{itemize}
\item $X_i$ and $X_{i+1}$ are conjugate,
\item $X_{i+1}$ is obtained from $X_{i}$ via a symbol expansion, or  
\item $X_{i}$ is obtained from $X_{i+1}$ via a symbol expansion.
\end{itemize}
\end{thm}

\noindent
In this way, flow equivalence is the coarsest equivalence relation finer than conjugacy and symbol expansion. The proof in \cite{parry_sullivan} is very compact; for a more detailed proof, see  \cite[pp.\ 87]{parry_tuncel} or \cite[Theorem 17]{thamsborg}.

\begin{remark}
In the definition of symbol expansion, a direction was chosen by inserting the new symbol $\diamond_a$ to the right of the original symbol $a$. One could think that a different notion of flow equivalence would be obtained by considering a form of symbol expansion inserting the new symbol to the left of the original symbol, but this is not the case. Indeed, it is straightforward to check that such a left symbol expansion can be obtained by a right symbol expansion followed by a conjugacy. This is expectable, since there is no preferred direction in the topological definition of flow equivalence.
\end{remark}

\section{Flow invariants and classification}
\label{sec_fe_inv}
\index{flow invariant}
Irreducible shifts of finite type have been completely classified up to flow equivalence by Franks \cite{franks}, and this has been extended to a classification of arbitrary shifts of finite type by Boyle and Huang \cite{boyle,boyle_huang, boyle_sullivan, huang}. Boyle, Carlsen, and Eilers have reduced the classification problem of AFT shifts
to a problem involving reducible SFTs equipped with group actions 
and used this to give classifications of near Markov shifts and AFT shifts where the covering maps of the Fischer covers are $2$ to $1$ \cite{boyle_carlsen_eilers}, but very little is known about the classification of general (irreducible) sofic shifts. The purpose of this section is to give an overview of some  classification results and invariants for sofic shifts. Invariants of flow equivalence will be called \emph{flow invariants} in the following.

The following well-known result shows that many of the classes of shift spaces considered so far are invariant under flow equivalence.

\begin{prop}
If $X$ is an SFT, sofic, or irreducible and if $Y$ is flow equivalent to $X$, then $Y$ is, respectively, an SFT, sofic, and irreducible.
\end{prop}

\begin{proof}
The three properties are preserved under conjugacy \cite[Theorem 2.1.10 and Corollary 3.2.3]{lind_marcus}, so by Theorem \ref{thm_parry-sullivan}, it is sufficient to check that they are also preserved under symbol expansions. Let $a \in \AA(X)$. 
If $X = \X_\FF$ for a finite set $\FF$, then $X^{\exp_a}$ can also be described by a finite set of forbidden words.
If $X$ is presented by a labelled graph $(G, \LL)$, then $X^{\exp_a}$ is presented by the graph obtained from $(G, \LL)$ by replacing every edge labelled $a$ with two edges in succession labelled respectively $a$ and $\diamond_a$. The invariance of irreducibility follows in a similar manner.
\end{proof}

\index{entropy!not flow invariant}
\index{periodic point!not flow invariant}
In this context, it is also worth mentioning that many important invariants of conjugacy are not flow invariants. The entropy, in particular, is not a flow invariant because a symbol expansion will change the number of allowed words of a given length. The symbol expansion of a periodic point is again a periodic point, but the period will change, so the number of periodic points of a given length is not a flow invariant.

\subsection{Shifts of finite type}

By Proposition \ref{prop_higher_block}, every SFT is conjugate (and hence also flow equivalent) to an edge shift, so it is sufficient to understand the flow equivalence of edge shifts. Bowen and Franks \cite{bowen_franks} have proved that the Bowen-Franks group of an edge shift $\X_A$ is not only an invariant of conjugacy but also an invariant of flow equivalence, and the same is true for the determinant of $\Id - A$ \cite{parry_sullivan}.
Since $\lvert \det(\Id -A) \rvert$ is determined by the Smith normal form of $\Id - A$, it is sufficient to know the Bowen-Franks group and the sign of the determinant in order to find the determinant. The following important theorem shows that these invariants are sufficient to give a complete flow classification of irreducible SFTs.

\begin{thm}[{Franks \cite{franks}}]
\label{thm_franks}
\index{shift of finite type!irreducible!flow equivalence of}
Let $\X_A, \X_B$ be irreducible edge shifts that are not flow equivalent to the trivial shift with one element. $\X_A \FE \X_B$ if and only if $\BF(A) = \BF(B)$ and $\sgn \det(\Id - A) = \sgn \det(\Id - B)$. 
\end{thm}

\index{Bowen-Franks group!signed|\\see{Bowen-Franks invariant}}\index{BF+@$\BF_+$|see{Bowen-Franks invariant}}\index{Bowen-Franks invariant}
\index{Bowen-Franks invariant!range of}
\index{Bowen-Franks invariant!of full shift}
\index{Bowen-Franks invariant!of general SFT}
\index{full shift!Bowen-Franks invariant of}
\noindent
As mentioned in Section \ref{sec_conjugacy}, it is straightforward to compute the Bowen-Franks group of an edge shift, so the complete invariant is both easy to compare and easy to compute, and this makes the result very useful. The pair consisting of $\sgn \det(\Id - A)$ and $\BF(\X_A)$ is called \emph{the signed Bowen-Franks group} or the \emph{Bowen-Franks invariant} and it is denoted $\BF_+(\X_A)$. If, for example, $\BF(\X_A) = G$ and $\det(\Id - A) < 0$, then $\BF_+(\X_A) = -G$. As for the unsigned Bowen-Franks group, this terminology and notation will be extended to arbitrary SFTs via Proposition \ref{prop_higher_block} and Theorem \ref{thm_franks}.

The Bowen-Franks group is a finitely generated abelian group, so it can be written as a finite direct sum of cyclic groups. 
If $\BF(\X_A) = G$ and if one of the terms in $G$ is $\Z$, then $\det(\Id -A) = 0$, and the Bowen-Franks invariant is written $\BF_+(\X_A) = G$. 
It is easy to check that every possible combination of sign and finitely generated abelian group can be achieved as the Bowen-Franks invariant of an irreducible SFT. If $X^{[n]}$ is the full shift with $n$ symbols, then $\BF_+(X^{[n]}) = - \Z / (n-1) \Z$.

\index{shift of finite type!reducible!flow equivalence of}
In order to classify reducible SFTs, Boyle and Huang \cite{boyle,boyle_huang, boyle_sullivan, huang} use a collection of groups and group homomorphisms called the $K$-web as an invariant. The details of the construction will be omitted here since the result will not be needed in the following.

\subsection{Covers of sofic shifts}
When $(\cov, \pi_{\cov})$ is a cover defined on some subclass $S$ of the sofic shifts, $\cov X$ is an SFT for each $X \in S$. The following will demonstrate when the flow class of this SFT can be used as a flow invariant of $X$.

\index{cover!respecting symbol expansion}
\begin{definition}
A cover $(\cov, \pi_{\cov})$ defined on a subclass $S$ of the class of sofic shifts is said to \emph{respect symbol expansion} if for all $X \in S$, $a \in \AA(X)$, and $\diamond \notin \AA(X)$ there exists a map $\LL_{\cov(X)} \colon \AA(\cov(X)) \to \AA(X)$ such that 
\begin{itemize}
\item $\pi_{\cov(X) } ((y_n)_{n\in \Z}) = ( \LL_{\cov(X)} (y_n))_{n \in \Z}$ for all $(y_n)_{n\in \Z} \in \cov(X)$, and
\item for $A = \LL_{\cov(X)}^{-1}( \{ a \} )$, there exists a conjugacy $\varphi \colon \cov( X^{a \mapsto a \diamond}) \to  (\cov(X))^{\exp_A }$ such that the following diagram commutes
\begin{center}
\begin{tikzpicture}
  [bend angle=10,
   clear/.style = {rectangle, minimum width = 5 mm, minimum height = 5 mm, inner sep = 0pt}, 
   map/.style = {->, shorten >= 4 pt, shorten <= 4 pt,  semithick}]
  
  \node[clear] (cov_exp_aX) at (0,0) {$\cov( X^{a \mapsto a \diamond})$};
  \node[clear] (exp_A_covX) at (5,0) {$(\cov(X))^{\exp_A }$};
  \node[clear] (exp_aX) at (0,-2) {$X^{a \mapsto a \diamond}$};
  \node[clear] (covX) at (5,-2) {$\cov X$};
  
  \draw[map] (cov_exp_aX) to node[auto] {$\varphi$} (exp_A_covX);
  \draw[map] (cov_exp_aX) to node[auto,swap] {$\pi_{\cov(X^{a \mapsto a \diamond})}$} (exp_aX);  
  \draw[map] (exp_A_covX) to node[auto] {$\exp_A^{-1}$} (covX);
  \draw[map] (covX) to node[auto] {$\exp_a \circ \, \pi_{\cov(X)} $} (exp_aX);
\end{tikzpicture}
\end{center} 
\end{itemize}
\end{definition}

\begin{remark}
\label{rem_flow_cover_respect}
Let $S$ be a subclass of the class of sofic shifts
where each $X \in S$ has a distinguished presentation $(G_X , \LL_X)$ as described in Remark \ref{rem_cover_graph}, and consider the cover defined by $\cov(X) = \X_{G_X}$ and $\pi_{\cov}(X) = \LL_X$. This cover respects symbol expansion if $(G_{X^{a \mapsto a \diamond}} , \LL_{X^{a \mapsto a \diamond}})$ is the labelled graph obtained from $(G_X , \LL_X)$ by replacing each edge labelled $a$ by two edges in succession labelled respectively $a$ and $\diamond$. 
\end{remark}

\begin{thm}[{Boyle, Carlsen, and Eilers \cite{boyle_carlsen_eilers}}]
Let $(\cov, \pi_{\cov})$ be a  ca\-nonical cover respecting symbol expansion defined on a subclass $S$ of the class of sofic shifts. If $X,Y \in S$ and $\varphi \colon SX \to SY$ is a flow equivalence, then there exists a unique flow equivalence $\cov(\varphi) \colon S \cov(X) \to S \cov(Y)$ such that the following diagram commutes\label{thm_cover_fe}
\begin{center}
\begin{tikzpicture}
  [bend angle=10,
   clear/.style = {rectangle, minimum width = 5 mm, minimum height = 5 mm, inner sep = 0pt}, 
   map/.style = {->, shorten >= 6 pt, shorten <= 6 pt,  semithick}]
  
  \node[clear] (SX) at (0,0) {$SX$};
  \node[clear] (SY) at (4,0) {$SY$};
  \node[clear] (ScX) at (0,2) {$S \cov (X)$};
  \node[clear] (ScY) at (4,2) {$S \cov (Y)$};

  \draw[map] (SX) to node[auto] {$\varphi$} (SY);
  \draw[map] (ScX) to node[auto,swap] {$S \pi_{\cov}(X)$} (SX);
  \draw[map] (ScX) to node[auto] {$\cov(\varphi)$} (ScY);
  \draw[map] (ScY) to node[auto] {$S \pi_{\cov}(Y)$} (SY);

\end{tikzpicture}
\end{center}
\end{thm}

\index{cover!flow invariant}
\noindent
A cover satisfying the conditions of Theorem \ref{thm_cover_fe} is said to be \emph{flow invariant}.

The following proposition shows that the main covers considered in Section \ref{sec_covers} can all be used as invariants of flow equivalence of sofic shifts. 

\index{Krieger cover!is flow invariant}\index{Fischer cover!is flow invariant}
\begin{prop}[{Boyle, Carlsen, and Eilers \cite{boyle_carlsen_eilers}}]
\label{prop_covers_respect_se}
The left and right  Fischer and Krieger covers as well as the past set and future set covers are all flow invariant.
\end{prop}

\begin{proof}
These covers are all canonical and have the property mentioned in Remark \ref{rem_flow_cover_respect}.
\end{proof}

The previous proposition is useful, because it allows the powerful invariants of SFTs to be used for sofic shifts, but information is obviously lost in this process. Some of this can be reclaimed using the following proposition.

\begin{prop}[{Boyle, Carlsen, and Eilers \cite{boyle_carlsen_eilers}}]
\label{prop_fischer_n_to_1}
Let $X_1$ and $X_2$ be flow equivalent irreducible sofic shifts. For $i \in \{1,2 \}$, let $\pi_i \colon \X_{F_i} \to \X_i$ be the covering map of the left Fischer cover $(F_i, \LL_i)$ of $X_i$, then $m(\pi_1) = m(\pi_2)$.
\end{prop}

\noindent
Naturally, an analogous result holds for the right Fischer cover.

\section{Working with symbol expansion}
\label{sec_se_lemmas}
As seen in Theorem \ref{thm_parry-sullivan}, conjugacies and symbol expansions allow far-reaching deformations of a shift space. This section will show a number of concrete constructions that can be used to construct explicit conjugacies and symbol expansions between flow equivalent shift spaces.
This will allow a more intuitive use of symbol expansion, and the lemmas are used repeatedly in the following chapters. The results are unsurprising and at least some of them have already been used in arguments in the literature, but  there are traps which must be avoided when working with symbol expansion, so full proofs are given to show what can (and cannot) be achieved.

\subsection{Expanding with letters from the original alphabet}
In symbol expansion, a letter that does not belong to the original alphabet is inserted after each occurrence of a specific letter, but often it is undesirable to expand the alphabet in this way, so the first goal is to show how to symbol expand a letter by symbol from the original alphabet.

Let $X$ be a shift space over $\AA$ and let $a,b \in \AA$. Choose a
symbol $\diamond \notin \AA$ and construct $X^{a \mapsto
a \diamond}$. Define a one-block map $\Phi \colon \AA \cup \{ \diamond \} 
\to \AA$ by $\Phi(\diamond) = b$ and $\Phi \vert_{\AA} =
\Id_{\AA}$, and let  
$\varphi = \Phi_\infty \colon X^{a \mapsto a \diamond} \to \AA^\Z$ be the induced
sliding block code with memory and anticipation 0.
Define the shift space $X^{a \mapsto  ab} = \varphi(X^{a \mapsto a \diamond})$.

\index{$a \mapsto ab$}
\index{symbol expansion!with original letter}
\begin{lem} \label{lem_se_of_letter}
If $X$ is a shift space over $\AA$ and $a,b \in \AA$ with $a \neq b$, then $X \FE X^{a \mapsto ab}$.
\end{lem}

\begin{proof}
Choose a symbol $\diamond \notin \AA$ and construct $\varphi \colon X^{a \mapsto a
  \diamond} \to X^{a \mapsto ab}$ as above. Since $X \FE X^{a \mapsto a \diamond}$, it is sufficient to prove that $\varphi$ is a conjugacy.
Define $\Psi: \BB_2(X^{a \mapsto ab}) \to \AA \cup \{\diamond \}$ by 
\begin{displaymath}
\Psi(w) = \left \{ \begin{array}{l c l}
\diamond & , & w = ab \\
\rl(w)        & , & w \neq ab
\end{array} \right. ,
\end{displaymath}
let $\psi = \Psi_\infty^{[1,0]} \colon X^{a \mapsto ab} \to (\AA \cup 
\{\diamond\})^\Z$  be the induced sliding block code with memory 1 and
anticipation 0, and note that $\psi = \varphi^{-1}$.
\end{proof}

The map $\varphi \colon X^{a \mapsto a \diamond} \to X^{a \mapsto ab}$ generally fails to be injective when $a = b$ since
$\varphi( (a \diamond)^\infty ) = a^\infty = \varphi( (\diamond a)^\infty)$.
A symbol can, however, be replaced by two copies of itself in a shift
space where this never happens. 

\begin{lem}
Let $X$ be a shift space and let $a \in \AA(X)$. If there exists $N \in \N$ such that $a^N \notin \BB_N(X)$, then $X \FE X^{a \mapsto aa}$.
\label{lem_a_to_aa}
\end{lem}

\begin{proof}
As above, it is sufficient to prove that the map $\varphi \colon X^{a \mapsto a\diamond} \to X^{a \mapsto aa}$ is injective. Define $\Psi: \BB_{2N+1}(X^{a \mapsto aa}) \to \AA \cup \{ \diamond \}$ by
\begin{displaymath}
\Psi(w) = \left \{ \begin{array}{l c l}
\rl(w)        & , & \rl(w) \neq a \\
a              & , & ba^{2n-1} \textrm{ is a suffix of } w \textrm{ for } b \neq a, n \in \N \\
\diamond & , & ba^{2n} \textrm{ is a suffix of } w \textrm{ for } b \neq a, n \in \N \\
\end{array} \right . ,
\end{displaymath}
and let $\psi = \Psi_\infty^{[2N,0]} \colon X^{a \mapsto aa} \to X^{a \mapsto a \diamond}$. Now $\psi =
\varphi^{-1}$. 
\end{proof}

\subsection{Symbol contraction}
\index{symbol contraction}
In this section, an inverse operation to symbol expansion is constructed.
Let $X$ be a shift space over $\AA \cup \{ \diamond \}$. Define the map $e_\diamond \colon \AA \cup \{ \diamond \} \to \AA^*$ by  $e_\diamond(\diamond) = \epsilon$ and $e_\diamond \vert_{\AA} = \Id_{\AA}$. Extend this to a map $e_\diamond \colon (\AA \cup \{ \diamond \})^* \to \AA^*$ in the natural way. 
Let $\BB_{\diamond \mapsto \varepsilon} = e_\diamond(\BB(X))$.
Assume that there exists $N \in \N$ such that $\diamond^N \notin
\BB(X)$. It is straightforward to prove that $\BB_{\diamond \mapsto \epsilon}$ is then the language of a shift space $X^{\diamond \mapsto \epsilon}$.

\begin{lem} \label{lem_sc_letter}
Let $X$ be a shift space over $\AA \cup \{ \diamond \}$ with $a \in
\AA$ such that $x_i = a$ if and only if $x_{i+1} = \diamond$ for all
$x \in X$ and $i \in \Z$. Then $X$ is obtained from $X^{\diamond
\mapsto \epsilon}$ by a symbol expansion of $a$ to $a \diamond$. In particular $X^{\diamond \mapsto \epsilon} \FE X$. 
\end{lem}

\begin{proof}
Let $Y = X^{\diamond \mapsto \epsilon}$ and check that $Y^{a \mapsto a \diamond}$ has the same language as $X$.
\end{proof}

Let $X$ be a shift space over $\AA$, and assume that there exist $a,b \in \AA$ with $a \neq b$ such that $F(a) = b F(ab)$ (i.e.\ in $X$, the letter $a$ is always followed by the letter $b$). Choose $\diamond \notin \AA$ and define $\Phi \colon \BB_2(X) \to \AA \cup \{\diamond \}$ by
\begin{displaymath}
\Phi(w) = \left \{ \begin{array}{l c l}
\diamond & , & w = ab \\
\rl(w)        & , & w \neq ab \\
\end{array} \right. .
\end{displaymath}
Let $\varphi = \Phi_\infty^{[1,0]} \colon X \to (\AA \cup \{\diamond \})^\Z$ be the induced sliding block code with memory 1 and anticipation 0. Let $Y = \varphi(X)$. Now $\varphi \colon X \to Y$ is a conjugacy by the same argument as in the proof of Lemma \ref{lem_se_of_letter}.  In $Y$, $a$ is always followed by $\diamond$, and clearly $\diamond$
is always preceded by $a$. Let $X^{ab \mapsto a} = Y^{\diamond \mapsto
  \epsilon}$.

\index{$ab \mapsto a$}
\index{symbol contraction}  
\begin{lem} \label{lem_sc_of_letter}
If $X$ is a shift space over $\AA$ with $a,b \in \AA$, $a \neq b$ such that $F(a) = b F(ab)$, then $X \FE X^{ab \mapsto a}$. 
\end{lem}

\begin{proof}
This follows from  the construction above and Lemma \ref{lem_sc_letter}. 
\end{proof}

\index{$aw \mapsto a$}
\noindent
If $X$ is a shift space, $a \in \AA(X)$, and $w \in \BB(X)$ such that $F(a) = wF(aw)$, then Lemma \ref{lem_sc_of_letter} can be applied repeatedly and the resulting shift space is denoted $X^{aw \mapsto a}$.

\subsection{Symbol expansion of words}
The goal of this section is to show how and when it is possible to insert a symbol after each occurrence of a word rather than after each occurrence of a letter. 

\index{symbol expansion!of word}\index{$w \mapsto w \diamond$}
Let $X$ be a shift space over $\AA$, let $\diamond, \heartsuit \notin \AA$, and consider a non-empty word $w \in \BB(X)$. Let $N = \vert w \vert \geq 1$, and define $\Phi \colon \BB_N(X) \to \AA \cup \{\heartsuit \}$ by
\begin{displaymath}
\Phi(v) = \left\{ \begin{array}{l c l}
\heartsuit & , & v = w \\
\rl(v)        & , & v \neq w
\end{array} \right. .
\end{displaymath}
Let $\varphi = \Phi_\infty^{[N-1,0]} \colon X \to (\AA \cup \{\heartsuit \})^\Z$. The map $\varphi$ is injective, so $Y = \varphi(X)$ is a shift space conjugate to $X$. Construct $Y^{\heartsuit \mapsto \heartsuit \diamond}$ and define $\Psi \colon \AA \cup \{ \heartsuit,  \diamond \} \to \AA \cup \{\diamond \}$ by 
\begin{displaymath}
\Psi(a) = \left\{ \begin{array}{l c l}
\rl(w) & , & a = \heartsuit \\
a       & , & a \neq \heartsuit
\end{array} \right. .
\end{displaymath}
Let $\psi = \Psi_\infty \colon Y^{\heartsuit \mapsto \heartsuit \diamond} \to (\AA \cup \{\diamond \})^\Z$ be the induced sliding block code with memory and anticipation 0.
The inverse of $\psi$ is the sliding block code induced by a function which maps $\rl(w)$ to $\heartsuit$ if it is followed by $\diamond$, so $\psi$ is a conjugacy.
The image $\psi(Y^{\heartsuit \mapsto \heartsuit \diamond})$ is denoted $X^{w \mapsto w \diamond}$. Note that the construction can be carried out even if $w$ can overlap with itself. 

\begin{lem} \label{lem_se_word_new}
For any shift space $X$ over $\AA$, any word $w \in \BB(X)$, and any symbol $\diamond \notin \AA$, $X^{w \mapsto w \diamond} \FE X$.
\end{lem}

\begin{proof}
In the construction above, $X$ is conjugate to $Y$ and $Y^{\heartsuit \mapsto \heartsuit \diamond}$ is conjugate to $X^{w \mapsto w \diamond}$. The result follows since $Y \FE Y^{\heartsuit \mapsto \heartsuit \diamond}$ by Theorem \ref{thm_parry-sullivan}.
\end{proof}

The next goal is to define the symbol expansion of a word by a letter that is already in the alphabet. Let $X$ be a shift space over $\AA$, and consider a word $w \in \BB(X)$ and a letter $b \in \AA$. Choose some symbol $\diamond$ not in $\AA$ and construct $X^{w \mapsto w \diamond}$ as above. Define $\Phi \colon \AA \cup \{\diamond \} \to \AA$ by 
\begin{displaymath}
\Phi(a) = \left\{ \begin{array}{l c l}
b & , & a = \diamond \\
a & , & a \neq \diamond
\end{array} \right. . 
\end{displaymath}
Construct the induced sliding block code $\varphi = \Phi_\infty \colon X^{w \mapsto w \diamond} \to (\AA)^\Z$, and define the shift space $X^{w \mapsto w b} = \varphi(X^{w \mapsto w \diamond})$. Note that $X^{w \mapsto wb}$ need not be flow equivalent to $X$, since $\varphi$ is not generally a conjugacy.

\index{overlap}\index{overlap!non-trivial}\index{overlap!maximal}
To formulate conditions for when $X$ and $X^{w \mapsto wb}$ are flow equivalent, it is useful to introduce some new notation.
Consider $w \in \BB_n(X)$. A word $v$ is said to be a \emph{$w$-overlap} if there exist  $1 = k_1 < k_2 <  \ldots < k_m = |v|-n+1$ with $k_{i+1} < k_i + n$ such that $v_{[k_i, k_i + n-1]} = w$ for all $i \in \{ 1, \ldots, m \}$.
For example, the word $ababa$ is a $aba$-overlap, while it is not a $abab$-overlap.
For $x \in X$, a factor $x_{[i,j]}$ is said to be a \emph{maximal $w$-overlap} if there do not exist $i' < i$ and $j' > j$ such that $x_{[i',j']}$ is a $w$-overlap (i.e.\ it cannot be extended to a longer $w$-overlap).
$X$ is said to \emph{admit} (non-trivial) $w$-overlaps if there exists a  $w$-overlap $v \neq w$ in $\BB(X)$. 

\begin{lem} \label{lem_se_word}
Let $X$ be a shift space over $\AA$, let $w \in \BB(X)$, and let $b \in \AA$. If $X$ does not allow non-trivial $w$-overlaps and if for all $v \in F(w)$ with $|v| \leq |w|$, $wbv$ does not have
  $wb$ as a suffix, then $X \FE X^{w \mapsto wb}$.
\end{lem}


\begin{proof}
The conditions guarantee that the sliding block code $\varphi$ used in the construction above is injective. 
\end{proof}


\noindent The following example shows that the conditions in Lemma \ref{lem_se_word} are more restrictive than necessary. 

\begin{example} \label{ex_gap_shift}
\index{run length limited shift}
Consider the run length limited shift $X = \X(1,3)$ (cf.\ \cite[Example 1.2.5]{lind_marcus}). This is the shift space over $\AA = \{ 0, 1\}$ where 1 occurs infinitely often in all elements of $X$ and where the number of 0's between two successive occurrences of 1 is either 1, 2, or 3. Construct $X^{00 \mapsto 000}$ as described above. Note that $X^{00 \mapsto 000}$ is the shift space over $\{ 0,1 \}$ where 1 occurs infinitely often in all elements and where the number of 0's between two successive 1's is an element of $\{1,3,5\}$.
Clearly, $w=00$ overlaps non-trivially with itself in $000 \in \BB(X)$, so the conditions of Lemma \ref{lem_se_word} are not satisfied, but it is also clear that the map $\varphi$ used in the construction of $X^{00 \mapsto 000}$ is in fact injective because there is an upper bound on the length of strings of 0s in $X$ and $X^{00 \mapsto 000}$ which makes it possible to distinguish newly added 0s from the original ones. Hence, $X  \FE X^{00 \mapsto 000}$.
\end{example}

\subsection{Symbol contraction of words}

\index{$wb \mapsto w$}
Let $X$ be a shift space over $\AA$ and let $w \in \BB(X)$. Assume that $F(w) = b F(wb)$ for some $b \in \AA$
and that $w$ is not a power of $b$. The latter only excludes the shifts where a sufficient number of $b$'s is necessarily followed by an infinite string of $b$'s. Let $n = |w|$, let $\diamond \notin \AA$, and define a map $\Phi :\BB_n(X) \to \AA \cup \{\diamond\}$ by 
\begin{displaymath}
\Phi(v) = \left\{ \begin{array}{l c l}
\diamond & , & v = w \\
\rl(v)         & , & v \neq w \\
\end{array} \right. .
\end{displaymath}
Let $\varphi = \Phi_\infty^{[n-1,0]} \colon X \to (\AA \cup \{ \diamond \})^\Z$, and let $Y = \varphi(X)$. By assumption, $\diamond$ is always followed by $b$ in $Y$, so it is possible to define $Y^{\diamond b \mapsto \diamond}$. Define $\Psi \colon \AA \cup \{ \diamond \} \to \AA$ by 
\begin{displaymath}
\Psi(a) = \left\{ \begin{array}{l c l}
\rl(w) & , & a = \diamond \\
a       & , & a \neq \diamond \\
\end{array} \right. ,
\end{displaymath}
let $\psi = \Psi_\infty \colon Y^{\diamond b \mapsto \diamond} \to \AA^\Z$ be the induced sliding block code with memory and anticipation 0, and define $X^{wb \mapsto w} = \psi(Y^{\diamond b \mapsto \diamond})$. Note that $\psi$ need not be a conjugacy, so $X^{wb \mapsto w}$ is not flow equivalent to $X$ in general. In particular, this process will erase letters inside some occurrences of $w$ if $w$ can overlap with itself, so the construction is mainly useful when there are some constraints on $w$.

\begin{lem} \label{lem_sc_words}
Let  $X$ be a shift space for which there exist $w \in \BB(X)$ and 
$b \in \AA$ such that $F(w) = b F(wb)$. If 
\begin{itemize}
  \item $X$ does not admit non-trivial $wb$-overlaps, and
  \item there is no $u \in F(wb)$ with $1 \leq |u| < |w|$ such that  
   $w$ is a suffix of $wu$, 
\end{itemize}
then $X \FE X^{wb \mapsto w}$.
\end{lem}

\begin{proof}
The first condition guarantees that none of the original
$w$'s are altered by the process, while the second condition
guarantees that no new $w$'s are created in $X^{wb \mapsto w}$, so the map $\psi$ constructed above is invertible.     
\end{proof}

\begin{example}
Consider the shift space $X$ over $\{a,b\}$ where the set of forbidden words is $\FF = \{ bab \}$. Clearly, any occurrence of $ba$ must be followed by $a$, and neither $ba$ nor $baa$ overlaps with itself, so the conditions of Lemma \ref{lem_sc_words} are satisfied. Hence, $X \FE X^{baa \mapsto ba}$, and it is easy to check that $X^{baa \mapsto ba} = \{ a, b \}^\Z$.
\end{example}

%
%

\subsection{Replacing words}

\index{flow equivalence!replacing words}\index{$w \mapsto \diamond$}
Let $X$ be a shift space over $\AA$. Consider $w \in \BB_n(X)$ for which $X$ does not admit non-trivial $w$-overlaps, and let $\diamond$ be some symbol not in $\AA$. The goal is to replace $w$ by $\diamond$ in every $x \in \X$.
Define $\Phi \colon \BB_n(X) \to (\AA \cup \{ \diamond \})$ by 
\begin{displaymath}
\Phi(v) = \left\{ \begin{array}{l c l}
\diamond & , & v = w \\
\leftl(v)     & , & v \neq w \\
\end{array} \right. ,
\end{displaymath}
and let $\varphi = \Phi_\infty^{[0,n-1]}$ be the induced sliding block code with memory 0 and anticipation $n-1$. This is clearly a conjugacy. Let $Y = \varphi(X)$. Since $w$ does not overlap with itself in any allowed word in $\BB(X)$, each occurrence of $\diamond$ in an element of $Y$ is followed by the final $n-1$ letters of $w$. Since $\diamond$ is a new symbol, all these letters can be removed using Lemma \ref{lem_sc_of_letter}. The resulting shift space is denoted $X^{w \mapsto \diamond}$.  Note that $X^{w \mapsto \diamond}$ is obtained from $X$ by replacing each occurrence of $w$ by $\diamond$ in every $x \in X$. 

\begin{lem} \label{lem_replace}
For any shift space $X$ and any $w \in \BB(X)$ for which $X$ does not
admit non-trivial $w$-overlaps, $X \FE X^{w \mapsto \diamond}$. 
\end{lem}

\begin{proof}
This is obvious from the construction above.
\end{proof}

\section{Application: Flow equivalence of gap shifts}
\label{sec_gap_shifts}\index{gap shift}
In this section, the preceding lemmas are applied to an investigation of the flow equivalence of a certain class of sofic shifts. 
For $S \subseteq \N_0$, the $S$-gap shift $\X(S)$ (cf.\ \cite[pp.\ 7]{lind_marcus}) is the shift space over $\{ 0, 1 \}$ for which the set of forbidden words is 
\begin{displaymath}
\FF = \left\{ \begin{array}{l c l}
\{ 10^n1 \mid n \notin S \} & , & S \textrm{ infinite} \\ 
\{ 10^n1 \mid n \notin S \} \cup \{ 0^{\max S +1} \} & , & S \textrm{ finite} \\ 
\end{array} \right. .
\end{displaymath}
$\X(S)$ is sofic if and only if there exist $e_1, \ldots , e_k, f_1, \ldots , f_l, N \in \N_0$
such that $S = \{e_1, \ldots, e_k\} \cup (\{f_1, \ldots, f_l\} + N \N_0)$ (see e.g.\ \cite[problem 3.1.10]{lind_marcus}).
Without loss of generality, it can be assumed that $\{e_1, \ldots, e_k\} \cap (\{f_1, \ldots, f_l\} + N \N_0) = \emptyset$, that $f_1 < \cdots < f_l$, and that $f_l-f_1 < N$. $\X(S)$ is an SFT if and only if $N \in \{0,1\}$.  

\begin{lem}\label{lem_gap_fe}
Let $S \subseteq \N_0$.
\begin{enumerate}
\item If $k \in \N$, then $\X(S+k) \FE \X(S)$.\label{lem_gap_fe_trans}
\item If $a,b \in \N_0 \setminus \{S\}$, then $\X(S \cup \{a \}) \FE \X(S \cup \{b \})$.\label{lem_gap_fe_add}
\end{enumerate}
\end{lem}

\begin{proof}
The first statement follows by using Lemma \ref{lem_se_of_letter} $k$ times on $\X(S)$ to expand $1$ by $0^k$. For the second statement, introduce a new symbol $\diamond$ and define $\Phi \colon \BB_{a+2} (\X(S \cup \{a \})) \to \{0,1,\diamond\}$ by 
\begin{displaymath}
\Phi(w) = \left \{ \begin{array}{l c l}
\diamond & , & w = 10^a1 \\
\leftl(w)    & , & \textrm{otherwise} 
\end{array}\right. .
\end{displaymath}
Clearly, $\varphi = \Phi_{[0,a+1]} \colon \X(S \cup \{a \}) \to \{0,1,\diamond\}^\Z$ is injective. Use Lemma \ref{lem_se_of_letter} or \ref{lem_sc_of_letter} on $\varphi(\X(S \cup \{a \}))$ to construct a shift space
where every $\diamond$ is followed by $0^bx$ for some $x \in \{1,\diamond\}$. Finally, construct a conjugacy from the resulting shift to $\X(S \cup \{b \})$ as above. The result follows by Theorem \ref{thm_parry-sullivan}.
\end{proof}

\begin{prop} \label{prop_s-gap_reduction}
Let $N \in \N$ and let 
\begin{displaymath}
S  = \{e_1, \ldots, e_k\} \cup (\{f_1, \ldots, f_l\} + N \N_0) \subseteq \N_0
\end{displaymath}
with $\{e_1, \ldots, e_k\} \cap (\{f_1, \ldots, f_l\} + N \N_0) = \emptyset$, $f_1 < \cdots < f_l$, and $f_l-f_1 < N$. Let $1 \leq j \leq l$, $j \equiv 1-k \pmod{l}$, and 
\begin{displaymath}
S' = \{ 0, f_{j+1} -f_j, \ldots , f_l -f_j, f_1+N-f_j, \ldots , f_{j-1} + N - f_j \}+N\N_0,
\end{displaymath}
then $\X(S) \FE \X(S')$. 
\end{prop}

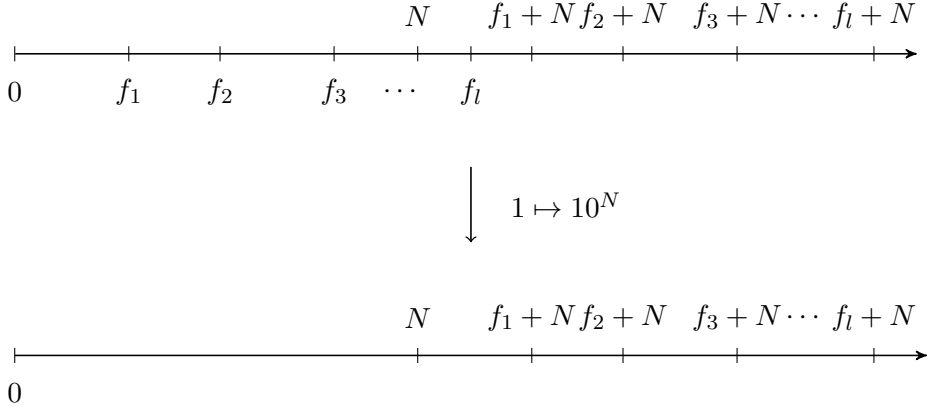
\begin{figure}
\begin{center}
\begin{tikzpicture}
  [bend angle=45,
   clearRound/.style = {circle, inner sep = 0pt, minimum size = 17mm},
   clear/.style = {rectangle, minimum width = 8 mm, minimum height = 5 mm, inner sep = 0pt},  
   greyRound/.style = {circle, draw, minimum size = 1 mm, inner sep =
      0pt, fill=black!10},
   grey/.style = {rectangle, draw, minimum size = 6 mm, inner sep =
      1pt, fill=black!10},
    white/.style = {rectangle, draw, minimum size = 6 mm, inner sep =
      1pt},
   tick/.style = {rectangle, draw, minimum height = 2 mm, minimum width = 0pt, inner sep = 0pt},
   to/.style = {->, >=stealth', semithick},
   mapsto/.style = {->, semithick}]
  
  \node[tick] (O) at (0,0) {};
  \node (max) at (12,0) {};

  \node[tick] (f1) at (1.5,0) {};
  \node[tick] (f2) at (2.7,0) {};
  \node[tick] (f3) at (4.2,0) {};
  \node[tick] (fl) at (6,0) {}; 
  
  \node[clear] (0text) at ($(O)-(0,0.5)$) {$0$};
  \node[clear] (f1text) at ($(f1)-(0,0.5)$) {$f_1$};
  \node[clear] (f2text) at ($(f2)-(0,0.5)$) {$f_2$};
  \node[clear] (f3text) at ($(f3)-(0,0.5)$) {$f_3$};
  \node[clear] (fltext) at ($(fl)-(0,0.5)$) {$f_l$}; 
  \node[clear] (dots) at (5.1,-0.5) {$\cdots$};

  \node[tick] (N) at (5.3,0) {};
  \node[tick] (f1+N) at ($(f1)+(N)$) {};
  \node[tick] (f2+N) at ($(f2)+(N)$) {};
  \node[tick] (f3+N) at ($(f3)+(N)$) {};
  \node[tick] (fl+N) at ($(fl)+(N)$) {}; 
  
  \node[clear] (Ntext) at ($(N)+(0,0.5)$) {$N$};
  \node[clear] (f1Ntext) at ($(f1text)+(N)+(0,1)$) {$f_1+N$};
  \node[clear] (f2Ntext) at ($(f2text)+(N)+(0,1)$) {$f_2+N$};
  \node[clear] (f3Ntext) at ($(f3text)+(N)+(0,1)$) {$f_3+N$};
  \node[clear] (flNtext) at ($(fltext)+(N)+(0,1)$) {$f_l+N$}; 
  \node[clear] (dotsN) at ($(dots)+(N)+(0,1)$) {$\cdots$};

  \draw[to] (O) to (max);

   \node (vt) at (0,-4) {};

  \draw[mapsto] (6,-1.5) to node[auto] {\quad $1 \mapsto 10^N$} (6,-2.5);

  \node[tick] (O') at ($(O)+(vt)$) {};
  \node[tick] (N') at ($(N)+(vt)$) {};
  \node[tick] (f1+N') at ($(f1+N)+(vt)$) {};
  \node[tick] (f2+N') at ($(f2+N)+(vt)$) {};
  \node[tick] (f3+N') at ($(f3+N)+(vt)$) {};
  \node[tick] (fl+N') at ($(fl+N)+(vt)$) {}; 

  \node[clear] (0text') at ($(0text)+(vt)$) {$0$};  
  \node[clear] (Ntext') at ($(Ntext)+(vt)$) {$N$};
  \node[clear] (f1Ntext') at ($(f1Ntext)+(vt)$) {$f_1+N$};
  \node[clear] (f2Ntext') at ($(f2Ntext)+(vt)$) {$f_2+N$};
  \node[clear] (f3Ntext') at ($(f3Ntext)+(vt)$) {$f_3+N$};
  \node[clear] (fltext') at ($(flNtext)+(vt)$) {$f_l+N$}; 
  \node[clear] (dotsN') at ($(dotsN)+(vt)$) {$\cdots$};

  \draw[to] ($(O)+(vt)$) to ($(max)+(vt)$);

\end{tikzpicture}
\end{center}
\caption[Translation of $S$.]{Illustration of the proof of Proposition \ref{prop_s-gap_reduction}. Translation of the periodic structure of $S$ in the case $0 < k \leq l$ where $m=1$.} 
\label{fig_s_translation}
\end{figure}

\begin{proof}
Choose $m \in \N_0$ such that $k = ml-j+1$, and consider $T = S+mN$. By Lemma \ref{lem_gap_fe}(\ref{lem_gap_fe_trans}), $\X(T) \FE \X(S)$. 
Consider Figure \ref{fig_s_translation}, and note how the periodic structure is translated. Use Lemma \ref{lem_gap_fe}(\ref{lem_gap_fe_add}) $k$ times to see that the sets
\begin{displaymath}
T =  \{e_1+mN, \ldots, e_k+mN\} \cup (\{f_1, \ldots, f_l\}+ N(m+ \N_0))
\end{displaymath}
and 
\begin{align*}
T' &= \{ f_j , \ldots, f_l\} \\ 
    &\qquad
          \cup( \{f_1, \ldots, f_l\}+N) \cup \cdots 
          \cup ( \{f_1, \ldots, f_l\}+(m-1)N) \\
    &\qquad \qquad \qquad \qquad \qquad \qquad \qquad \qquad      
          \cup (\{f_1, \ldots, f_l \} + N(m+\N_0)) \\
    &= \{ f_j , f_{j+1}, \ldots, f_l, f_1+N, \ldots, f_{j-1}+N\} + N\N_0 
\end{align*}
generate flow equivalent gap shifts. The result now follows by Lemma \ref{lem_gap_fe}(\ref{lem_gap_fe_trans}).
\end{proof}

\noindent
This result allows a reduction of the generating set of a sofic gap shift to a standard form, and two gap shifts with the same reduced form are clearly flow equivalent.
However, it is still unclear whether two sofic gap shifts can be flow equivalent without having the same reduced form.

The following proposition uses the reduction from Proposition \ref{prop_s-gap_reduction} to give a complete classification of SFT gap shifts.
If $S \subseteq \N_0$ and $\X(S)$ is an SFT, then either $S$ or $\N_0 \setminus  S$ is finite, i.e.\ $S = \{e_1, \ldots, e_k \}$ or $S = \{e_1, \ldots , e_k\} \cup (f+\N_0)$ for some $0 \leq e_1 < e_2 < \ldots < e_k < f$.

\index{gap shift!SFT}
\begin{prop} \label{cor_finite_gap_shift}
Let $S \subseteq \N_0$ such that $\X(S)$ is an SFT.
\begin{itemize}
\item If $\lvert S \rvert = k$, then $\X(S)$ is flow equivalent to the full $k$-shift.
\item
If $S$ is infinite, then $\X(S)$ is flow equivalent to the full 2-shift.
\end{itemize}
\end{prop}

\begin{proof}
Let $S = \{ e_1 , \ldots , e_k \}$ for $0 \leq e_1 < e_2 < \ldots < e_k$, and choose $k$ different symbols $\diamond_1, \ldots , \diamond_k \notin \AA({\X(S)}) = \{0,1\}$.
First, use Lemma \ref{lem_replace} to replace the word $10^{e_k}$ by
$\diamond_n$. Then replace $10^{e_{k-1}}$ by $\diamond_{k-1}$ and
continue so that $10^{e_i}$ gets replaced by $\diamond_i$ for all $i
\in \{1 ,\ldots ,n\}$. Note that it is important to replace the
longest strings of 0's first. The result is the shift space
$\{\diamond_1 , \ldots , \diamond_n \}^\Z$. 

For the second statement, assume that $S = \{e_1, \ldots , e_k\} \cup (f+\N_0)$ with $0 \leq e_1 < e_2 < \ldots < e_k < f$. By Proposition \ref{prop_s-gap_reduction}, $\X(S) \FE \X(\N_0) = \{ 0, 1 \}^\Z$. 
\end{proof}

\begin{lem}
\label{lem_s-gap_rfc}\index{gap shift!Fischer cover of}
If $S = \{s_1, \ldots , s_k\} + n \N_0$ with minimal $n$ and $s_1 = 0 < s_2 < \cdots < s_k < n$, then the right Fischer cover $(F,\LL_F)$ of $\X(S)$ is the labelled graph shown in Figure \ref{fig_rkc_s-gap}, and $\BF_+(\X_F) = - \Z / k\Z$.
\end{lem}

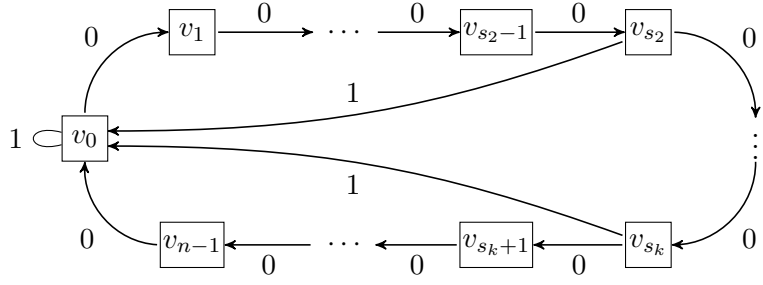
\begin{figure}
\begin{center}
\begin{tikzpicture}
  [bend angle=45,
   clearRound/.style = {circle, inner sep = 0pt, minimum size = 17mm},
   clear/.style = {rectangle, minimum width = 8 mm, minimum height = 5 mm, inner sep = 0pt},  
   greyRound/.style = {circle, draw, minimum size = 1 mm, inner sep =
      0pt, fill=black!10},
   grey/.style = {rectangle, draw, minimum size = 6 mm, inner sep =
      1pt, fill=black!10},
    white/.style = {rectangle, draw, minimum size = 6 mm, inner sep =
      1pt},
   to/.style = {->, shorten <= 1 pt, >=stealth', semithick}]
  
  \node[white] (v0) at (0,0) {$v_0$};
  \node[white] (v1) at (1.41,1.41) {$v_1$};
  \node[clear] (dots1) at (3.41,1.41) { $\, \cdots$};  
  \node[white] (vs1-1) at (5.41,1.41) {$v_{s_2 -1}$};  
  \node[white] (vs1) at (7.41,1.41) {$v_{s_2}$};  
  \node[clear] (dots2) at (8.82,0) {$\vdots$};  
  \node[white] (vsk) at (7.41,-1.41) {$v_{s_k}$};  
  \node[white] (vsk+1) at (5.41,-1.41) {$v_{s_k +1}$};  
  \node[clear] (dots3) at (3.41,-1.41) { $\, \cdots$};  
  \node[white] (vn-1) at (1.41,-1.41) {$v_{n-1}$};
   
  \draw[loop left] (v0) to node[auto] {$1$} (v0);
  \draw[to,bend left] (v0) to node[auto] {$0$} (v1);
  \draw[to] (v1) to node[auto] {$0$} (dots1);
  \draw[to] (dots1) to node[auto] {$0$} (vs1-1);
  \draw[to] (vs1-1) to node[auto] {$0$} (vs1);
  \draw[to,bend left] (vs1) to node[auto] {$0$} (dots2);  
  \draw[to,bend left = 10] (vs1) to node[auto,swap] {$1$} ($(v0) + (0.3,0.1)$);      

  \draw[to,bend left] (vn-1) to node[auto] {$0$} (v0);  
  \draw[to] (dots3) to node[auto] {$0$} (vn-1);
  \draw[to] (vsk+1) to node[auto] {$0$} (dots3);
  \draw[to] (vsk) to node[auto] {$0$} (vsk+1);
  \draw[to,bend left] (dots2) to node[auto] {$0$} (vsk);  
  \draw[to, bend right = 10] (vsk) to node[auto] {$1$} ($(v0) + (0.3,-0.1)$);      
    
\end{tikzpicture}
\end{center}
\caption[Right Fischer covers of sofic gap shifts.]{Right Fischer cover of the gap shift $\X(S)$ with $S = \{s_1, \ldots , s_k\} + n \N_0$ for minimal $n$ and $s_1 = 0 < s_2 < \cdots < s_k < n$.} 
\label{fig_rkc_s-gap}
\end{figure}

\begin{proof}
The graph in Figure \ref{fig_rkc_s-gap} is a presentation of $\X(S)$, and it is both right-resolving and follower-separated, so it is the right Fischer cover of $\X(S)$ by (an analogue of) Theorem \ref{thm_lfc_char}. The (non-symbolic) adjacency matrix of the underlying graph is the $n \times n$ matrix $A$ defined by
\begin{displaymath}
A_{i,j} = \left\{ \begin{array}{l c l}
1 & , & j \equiv i+1 \pmod{n} \\
1 & , & j =1,  i \in \{ s_1, \ldots , s_k\}, \textrm{ and } i < n-1 \\
2 & , & j =1,  i = s_k = n-1 \\
0 & , & \textrm{otherwise } 
\end{array} \right. .
\end{displaymath}
Using row and column additions and subtractions, it is easy to check that the Smith normal form of $\Id - A$ is $\Diag (k, 1,1, \ldots, 1)$ and that $\det( \Id - A) = -k$. 
\end{proof}

\begin{prop}
Let $S = \{s_1, \ldots , s_k\} + n \N_0$ and $T = \{t_1, \ldots , t_l\} +
m \N_0$ with minimal $n,m$. If $\X(S) \FE \X(T)$, then $k = l$ and $n = m$.
\end{prop}

\begin{proof}
Assume that $\X(S) \FE \X(T)$, and let $S'$, $T'$ be the sets obtained by applying Proposition \ref{prop_s-gap_reduction} to $S$ and $T$ respectively. Then $\X(S') \FE \X(T')$. The edge shifts of the underlying graphs of the right Fischer covers of $\X(S')$ and $\X(T')$ are flow equivalent by Proposition \ref{prop_covers_respect_se}, so $k=l$ by Theorem \ref{thm_franks} and Lemma \ref{lem_s-gap_rfc}.
Let $(F, \LL_F)$ be the right Fischer cover of $\X(S')$. The only left-ray that has more than one presentation in $(F, \LL_F)$ is $0^\infty \in \X(S')^-$ which has $n$ presentations. Hence, the covering map of $(F, \LL_F)$ is $n$ to $1$. By symmetry, the covering map of the right Fischer cover of $\X(T')$ is $m$ to $1$, so $n = m$ by Proposition \ref{prop_fischer_n_to_1}.
\end{proof}

According to Boyle \cite{boyle_personal}, the first part of this result can be replaced by a significantly stronger statement: If $N \in \N$, $S,T \subseteq \{0, \ldots, N-1\}$, and $\X(S +N\N_0) \FE \X(T +N\N_0)$ then $S \equiv T \pmod N$. I.e.\ the periodic structure of the generating set is a flow invariant. This can, for example, be used to show that the gap shifts generated by $\{0,1\} +5\N_0$ and $\{0,2\} +5\N_0$ are not flow equivalent. With Lemma \ref{lem_s-gap_rfc}, this shows that the Bowen-Franks invariant of the right Fischer cover is not a complete invariant of flow equivalence of sofic gap shifts.

\begin{example}
Consider $S = \{ 0, 1\} + 3\N_0$ and $T = \{ 0, 2\} + 3\N_0$. All the invariants discussed above have the same value for the two corresponding gap shifts, and this is arguably the simplest example where it is still unknown whether there exists a flow equivalence.
\end{example}

%% file: thesis_covers.tex
\label{chap_covers}

The purpose of this chapter is to investigate the structure of -- and relationships between -- various standard presentations (the Fischer cover, the Krieger cover, and the past set cover) of sofic shift spaces.
These results are used to find the range of the flow invariant introduced in \cite{bates_eilers_pask}, and to investigate the ideal structure of the $\Cs$-algebras associated to sofic shifts. In this way, the work can be seen as a continuation of the strategy applied in
\cite{carlsen_eilers_ergod_2004,carlsen_eilers_doc_2004,matsumoto_2001}, where invariants for shift spaces are extracted from the associated $\Cs$-algebras.  

Section \ref{sec_foundation} introduces the concept of a foundation of a cover, which is used to prove that the left Krieger cover and the past set cover can be divided into natural layers and to show that the left Krieger cover of an arbitrary sofic shift can be identified with a subgraph of the past set cover.

In Section \ref{sec_invariant}, the structure of the layers of the left Krieger cover of an irreducible sofic shift is used to find the range of the flow invariant introduced in \cite{bates_eilers_pask}. 
Section \ref{sec_cs} uses the results about the structure of covers of sofic shifts to investigate ideal lattices of the associated $\Cs$-algebras. Additionally, it is proved that 
Condition $(\ast)$ introduced by Carlsen and Matsumoto \cite{carlsen_matsumoto} holds if and
only if the left Krieger cover is the maximal essential subgraph of
the past set cover.

\index{transpose}
\index{shift space!transpose}
\index{graph!transpose}
\index{labelled graph!transpose}
\index{X*T@$X^\textrm{T}$}
Every result developed for left-resolving covers in this chapter has an analogue for the corresponding right-resolving cover. These results can easily be obtained by considering the transposed shift space $X^\textrm{T} = \{ (x_{-i})_{i \in \Z} \mid x \in X\}$ and the transposed graphs obtained by reversing all edges (see e.g.\ \cite[p. 39]{lind_marcus}).


\section{Generalising the Fischer cover}
\label{sec_gfc} 
Jonoska \cite{jonoska} proved that a reducible sofic shift does not necessarily have a unique minimal left-resolving presentation, so there is no direct analogue of the left Fischer cover for reducible sofic shifts. The aim of this section is to define a generalisation of the left Fischer cover as the subgraph of the left Krieger cover induced by a certain subset of vertices. 

\index{indecomposable}\index{vertex!indecomposable}
\index{predecessor set!indecomposable}
Let $X$ be a sofic shift space, and let $(K, \LL_K)$ be the
left Krieger cover of $X$. 
A predecessor set $P  \in K^0$ is said to be
\emph{indecomposable} if $V \subseteq K^0$ and $P = \bigcup_{Q \in V} Q$ implies that $P \in V$.

\begin{lem}
\label{lem_decomposable}
If a predecessor set $P \in K^0$ is indecomposable, then the  subgraph of $(K, \LL_K)$ induced by $K^0 \setminus \{ P \}$ is not a presentation of $X$.  
\end{lem}

\begin{proof}
Let $E$ be the
subgraph of $K$ induced by $K^0 \setminus \{P\}$.
Choose $x^+ \in X^+$ such that $P = P_\infty(x^+)$.
Let $V \subseteq K^0 \setminus\{ P \}$ be the set of vertices where a
presentation of $x^+$ can start. By Lemma \ref{lem_K_basic_facts}(\ref{K_facts_subset}),
$Q \subseteq P_\infty(x^+) = P$ for each $Q \in V$, and by assumption, there exists $y^- \in P \setminus \bigcup_{Q \in V} Q$. Hence, there is no presentation of $y^-x^+$ in $(E, \LL_K|_E)$.       
\end{proof}

\noindent
Lemma \ref{lem_decomposable} shows that a subgraph of the left Krieger cover which presents the same shift must contain all the indecomposable vertices. The next example shows that this subgraph is not always large enough.

\begin{example}
\label{ex_gfc_justifying}
It is easy to check that the labelled graph in Figure
\ref{fig_gfc_justifying} is the left Krieger cover of a reducible
sofic shift $X$. Note that the predecessor set $P$ is decomposable since $P = P_1
\cup P_2$, and that the graph obtained by removing the vertex $P$ and
all edges starting at or terminating at $P$ is not a presentation of
the same sofic shift since there is no presentation of $f^\infty
dbjk^\infty$ in this graph. Note that there is a path from
$P$ to the vertex $P'$ which is indecomposable. 
\end{example}

\begin{figure} 
\begin{center}
\begin{tikzpicture}
  [bend angle=45,
   knude/.style = {circle, inner sep = 0pt},
   vertex/.style = {circle, draw, minimum size = 1 mm, inner sep =
      0pt},
   textVertex/.style = {rectangle, draw, minimum size = 6 mm, inner sep =
      1pt},
   to/.style = {->, shorten <= 1 pt, >=stealth', semithick}]
  
  \matrix[row sep=5mm, column sep=10mm]{
   & & & & \node[textVertex] (P1) {$P_1$}; & \node[textVertex] (h1)
    {}; & \\
  \node[textVertex] (v2) {}; & \node[textVertex] (v1) {$P'$};  & & 
  \node[textVertex] (O) {}; & & & \node[textVertex] (h3) {}; \\
  &  & \node[textVertex] (P) {$P$}; & & \node[textVertex] (P2) {$P_2$};
  & \node[textVertex] (h2) {}; & \\
  };
  \draw[to, loop above]  (O) to node[auto] {f} (O);
  \draw[to, bend right=20] (O) to node[auto,swap] {c} (v1);
  \draw[to] (O) to node[auto] {d} (P1);
  \draw[to] (O) to node[auto,swap] {e} (P2);
  \draw[to] (P1) to node[auto] {b} (h1);
  \draw[to] (P2) to node[auto] {b} (h2);
  \draw[to, loop above]  (h1) to node[auto] {a} (h1);
  \draw[to, loop below]  (h2) to node[auto] {a} (h2);
  \draw[to] (h1) to node[auto] {g} (h3);
  \draw[to] (h2) to node[auto,swap] {h} (h3);
  \draw[to, loop right]  (h3) to node[auto] {i} (h3);
 
  \draw[to, bend right=20] (O) to node[auto,swap] {d} (P);
  \draw[to, bend left=20] (O) to node[auto] {e} (P);

  \draw[to] (P) to node[auto] {b} (v1);
  \draw[to, loop above]  (v1) to node[auto] {a} (v1);

  \draw[to] (v1) to node[auto] {j} (v2);
  \draw[to, loop left]  (v2) to node[auto] {k} (v2);

\end{tikzpicture}
\end{center}
\caption[Justification for the definition of the generalised Fischer cover.]{Left Krieger cover of the shift considered in Example
  \ref{ex_gfc_justifying}.
  Note that the labelled graph is no longer a presentation of the same
  shift if the decomposable predecessor set $P = P_1 \cup P_2$ is
  removed.
}  
\label{fig_gfc_justifying}
\end{figure}
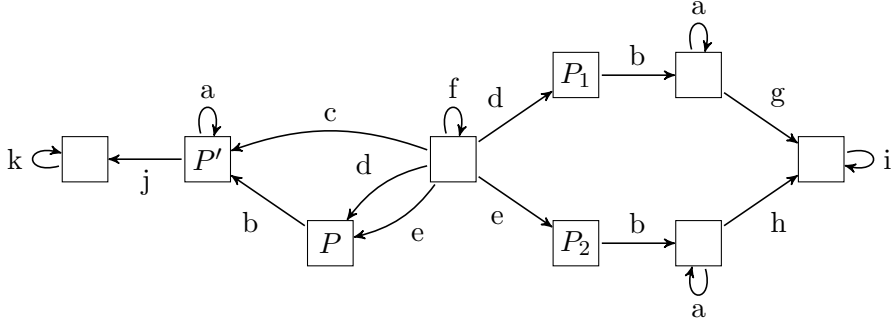

\noindent Together with Lemma \ref{lem_decomposable}, this example
motivates the following definition. 

\index{generalised Fischer cover}
\index{Fischer cover!generalised|\\see{generalised Fischer cover}}
\index{sofic shift!generalised Fischer cover of|\\see{generalised Fischer cover}}
\begin{definition}
The \emph{generalised left Fischer cover} $(G, \LL_G)$ of a sofic
shift $X$ is defined to be the subgraph of the left Krieger
cover induced by $G^0 = \{ P \in K^0 \mid  P \geq P', P' \textrm{ indecomposable} \}$. 
\end{definition}

\noindent The following proposition justifies the term generalised left Fischer cover.

\begin{prop}
\label{prop_justifying}
\qquad
\begin{enumerate}[(i)]
\item The generalised left Fischer cover of a sofic shift $X$
  is a left-resolving and predecessor-separated presentation of $X$. 
\item If $X$ is an irreducible sofic shift, then the generalised left
  Fischer cover is isomorphic to the left Fischer cover. 
\item If $X_1,X_2$ are sofic shifts with
  disjoint alphabets, then the generalised left Fischer cover
  of $X_1 \cup X_2$ is the disjoint union of the generalised left
  Fischer covers of $X_1$ and $X_2$. 
\end{enumerate}
\end{prop}

\begin{proof}
Given $y^- \in X^-$, choose $x^+ \in X^+$ such that $y^- \in
P_\infty(x^+) = P$. By definition of the generalised left Fischer cover,
there exist vertices $P_1, \ldots, P_n \in G^0$ such that $P = \bigcup_{i=1}^n
P_i$. Choose $i$ such that $y^- \in P_i$.
By construction, the left Krieger cover contains a path labelled $y^-$ terminating at
$P_i$. Since $P_i \in G^0$, this is also a path in the generalised
left Fischer cover. This proves that the generalised left Fischer
cover is a presentation of $X^-$, and hence also a presentation of
$X$. Since the left Krieger cover is left-resolving and
predecessor-separated, so is the generalised left Fischer cover. 

Let $X$ be an irreducible sofic shift, and identify the left Fischer cover $(F, \LL_F)$ with the top irreducible component of the left Krieger cover $(K,\LL_K)$.
By the construction of the generalised left Fischer cover, it follows that the left Fischer cover is a subgraph of the generalised left Fischer cover.
Let $x^+ \in X^+$ such that $P = P_\infty(x^+)$ is indecomposable. Let $S \subseteq F^0$ be the set of vertices where a presentation of $x^+$ in $(F,\LL_F)$ can start. Then $P = \bigcup_{v \in S} P_\infty(v)$, so $P \in S \subseteq F^0$ by assumption. Hence, the generalised left Fischer cover is also a subgraph of the left Fischer cover.

Since $X_1$ and $X_2$ have no letters in common, the left Krieger
cover of $X_1 \cup X_2$ is just the disjoint union of the left
Krieger covers of $X_1$ and $X_2$. The generalised left Fischer
cover inherits this property from the left Krieger cover.
\end{proof}

The shift consisting of two non-interacting copies of the even shift (see Example \ref{ex_even}) is a simple reducible example where the generalised left Fischer cover is a proper subgraph of the left Krieger cover.

\begin{lem}
\label{lem_GLFC_essential}
Let $X$ be a sofic shift with left Krieger cover $(K, \LL_K)$. If there
is an edge labelled $a$ from an indecomposable $P \in K^0$ to a
decomposable $Q \in K^0$, then there exists an indecomposable $Q' \in
K^0$ and an edge labelled $a$ from $P$ to $Q'$.   
\end{lem}

\begin{proof}
Choose $x^+ \in X^+$ such that $P = P_\infty(a x^+)$ and $Q =
P_\infty(x^+)$. Since $Q$ is decomposable, there exist $n > 1$ and
indecomposable predecessor sets $Q_1, \ldots, Q_n \in K^0 \setminus \{ Q \}$
such that $Q = Q_1 \cup \cdots \cup Q_n$.
Let $S$ be the set of predecessor sets $P' \in K^0$ for which
there is an edge labelled $a$ from $P'$ to $Q_j$ for some $1 \leq j
\leq n$.
Given $y^- \in P$, $y^-ax^+ \in X$, so $y^-a \in Q$.
Choose $1 \leq i \leq n$ such that $y^-a \in Q_i$.
By construction, there exists $P' \in S$ such that $y^- \in
P'$. Reversely, if $y^- \in P' \in S$, then there is an edge labelled
$a$ from $P'$ to $Q_i$ for some $1 \leq i \leq n$, so $y^- a \in Q_i
\subseteq Q$. This implies that $y^-ax^+ \in X$, so $y^- \in P$.
Thus $P = \bigcup_{P' \in S} P'$, but $P$ is indecomposable, so this
means that $P \in S$.
Hence, there is an edge labelled $a$ from $P$ to $Q_i$ for some $i$,
and $Q_i$ is indecomposable. 
\end{proof}

\noindent The following proposition is an immediate consequence of this
result and the definition of the generalised left Fischer cover.

\begin{prop}
\label{prop_GLFC_essential}
The generalised left Fischer cover is essential.
\end{prop}

The left Fischer cover of an irreducible sofic shift $X$ is minimal
in the sense that no other left-resolving presentation of $X$
has fewer vertices.
The next example shows that this is not true for the generalised left
Fischer cover.

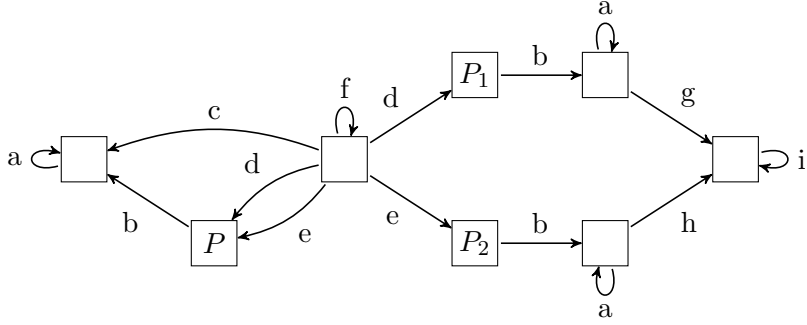
\begin{figure} 
\begin{center}
\begin{tikzpicture}
  [bend angle=45,
   knude/.style = {circle, inner sep = 0pt},
   vertex/.style = {circle, draw, minimum size = 1 mm, inner sep =
      0pt},
   textVertex/.style = {rectangle, draw, minimum size = 6 mm, inner sep =
      1pt},
   to/.style = {->, shorten <= 1 pt, >=stealth', semithick}]
  
  \matrix[row sep=5mm, column sep=11mm]{
  & & & \node[textVertex] (P1) {$P_1$}; & \node[textVertex] (h1)
    {}; & \\
  \node[textVertex] (v1) {};  & & 
  \node[textVertex] (O) {}; & & & \node[textVertex] (h3) {}; \\
  & \node[textVertex] (P) {$P$}; & & \node[textVertex] (P2) {$P_2$};
  & \node[textVertex] (h2) {}; & \\
  };
  \draw[to, loop above]  (O) to node[auto] {f} (O);
  \draw[to, bend right=20] (O) to node[auto,swap] {c} (v1);
  \draw[to] (O) to node[auto] {d} (P1);
  \draw[to] (O) to node[auto,swap] {e} (P2);
  \draw[to] (P1) to node[auto] {b} (h1);
  \draw[to] (P2) to node[auto] {b} (h2);
  \draw[to, loop above]  (h1) to node[auto] {a} (h1);
  \draw[to, loop below]  (h2) to node[auto] {a} (h2);
  \draw[to] (h1) to node[auto] {g} (h3);
  \draw[to] (h2) to node[auto,swap] {h} (h3);
  \draw[to, loop right]  (h3) to node[auto] {i} (h3);
 
  \draw[to, bend right=20] (O) to node[auto,swap] {d} (P);
  \draw[to, bend left=20] (O) to node[auto] {e} (P);

  \draw[to] (P) to node[auto] {b} (v1);
  \draw[to, loop left]  (v1) to node[auto] {a} (v1);
\end{tikzpicture}
\end{center}
\caption[Non-minimality of the generalised left Fischer cover.]{Left Krieger cover of the shift considered in Example
  \ref{ex_gfc_not_minimal}.
}  
\label{fig_gfc_not_minimal}
\end{figure}

\begin{example} \label{ex_gfc_not_minimal}
\index{generalised Fischer cover!non-minimality of}
It is easy to check that the labelled graph in Figure
\ref{fig_gfc_not_minimal} is the left Krieger cover of a reducible
sofic shift $X$. Note the similarity with the shift considered in
Example \ref{ex_gfc_justifying}.
The predecessor set $P$ is decomposable since $P = P_1 \cup P_2$. All
other predecessor sets are indecomposable, so the labelled graph in
Figure \ref{fig_gfc_not_minimal} is also the generalised left Fischer
cover of $X$.
Consider the labelled graph obtained by removing the vertex $P$ and
all edges starting at or terminating at $P$. This is clearly also a
presentation of $X$. 
Hence, there exist sofic shifts for which a proper subgraph of the
generalised left Fischer cover is also a presentation of the shift.
Note that Lemma \ref{lem_decomposable} shows that such a subgraph must
contain all vertices given by indecomposable predecessor sets. 
\end{example}

\subsection{Canonical and flow invariant}
The next goal is to prove that the generalised left Fischer cover is canonical by using results and methods used by Nasu
\cite{nasu} to prove that the left Krieger cover is canonical. With Theorem \ref{thm_cover_fe}, this will be used to prove that the generalised left Fischer cover is also flow invariant.

\begin{definition}
\label{def_bipartite_code}
\index{bipartite code}\index{bipartite expression}
When $\AA, \CC, \DD$ are alphabets, an injective map $f \colon \AA \to \CC \DD$ is called a \emph{bipartite expression}. If $X_1, X_2$ are shift spaces with alphabets $\AA_1$ and $\AA_2$, respectively, and if $f_1 \colon \AA_1 \to \CC \DD$ is a bipartite expression, then a map $\Phi \colon X_1 \to X_2$ is said to be a \emph{bipartite code induced by} $f_1$ if there exists a bipartite expression $f_2 \colon \AA_2 \to \DD \CC$ such that one of the following two conditions is satisfied:
\begin{enumerate}[(i)]
\item If $x \in X_1$, $y = \Phi(x)$, and $f_1(x_i) = c_i d_i$ with $c_i
  \in \CC$ and $d_i \in \DD$ for all $i \in \Z$, then $f_2(y_i) = d_i
  c_{i+1}$ for all $i \in \Z$. 
\item If $x \in X_1$, $y = \Phi(x)$, and $f_1(x_i) = c_i d_i$ with $c_i
  \in \CC$ and $d_i \in \DD$ for all $i \in \Z$, then $f_2(y_i) = d_{i-1}
  c_i$ for all $i \in \Z$.
\end{enumerate}
A mapping $\Phi \colon X_1 \to X_2$ is called a \emph{bipartite
  code}, if it is the bipartite code induced by some bipartite
expression.
\end{definition}

\noindent It is clear that a bipartite code is a conjugacy and that the inverse
of a bipartite code is a bipartite code.
Bipartite codes are of interest in this context because of the following theorem.

\begin{thm}[Nasu {\cite[Theorem 2.4]{nasu}}]
\label{thm_nasu_decomposition}
\index{bipartite code!composition of}\index{conjugacy!as product of bipartite codes}
Any conjugacy
can be decomposed into a product of bipartite codes. 
\end{thm}

\index{recoded shift}\index{shift space!recoded}\index{bipartite code!recoded}
Let $\Phi \colon X_1 \to X_2$ be a bipartite code corresponding to
bipartite expressions $f_1 \colon \AA_1 \to \CC \DD$ and $f_2
\colon \AA_2 \to \DD \CC$, and use the bipartite expressions to recode $X_1$ and $X_2$ to
\begin{align*}
  \hat X_1 &= \{ (f_1(x_i))_i \mid x \in X_1 \} \subseteq (\CC \DD)^\Z \\
  \hat X_2 &= \{ (f_2(x_i))_i \mid x \in X_2 \} \subseteq (\DD \CC)^\Z.
\end{align*}
For $i \in \{1,2\}$, $f_i$ induces a one-block conjugacy from $X_i$
to $\hat X_i$, and $\Phi$ induces a bipartite code $\hat \Phi \colon \hat X_1
\to \hat X_2$ which commutes with these conjugacies.
If $\Phi$ satisfies condition (i) in the definition of a bipartite
code, then $(\hat \Phi (\hat x))_i = d_i c_{i+1}$ when $\hat x = (c_i d_i)_{i \in \Z}
\in \hat X_1$.
If it satisfies condition (ii), then $(\hat \Phi (\hat x))_i = d_{i-1}
c_i$ when $\hat x = (c_i d_i)_{i\in \Z} \in \hat X_1$. 
The shifts $\hat X_1$ and $\hat X_2$ will be called the \emph{recoded shifts}
of the bipartite code, and $\hat \Phi$ will be called the \emph{recoded bipartite code}.


\index{graph!bipartite}\index{labelled graph!bipartite}\index{labelled graph!induced pair of}
A labelled graph $(G, \LL)$ is said to be \emph{bipartite} if $G$ is a
bipartite graph, i.e.\ the vertex set can be partitioned into two sets
$(G^0)_1$ and $(G^0)_2$ such that no edge has it's range and source in
the same set.
When $(G, \LL)$ is a bipartite labelled 
graph over an alphabet $\AA$, define two graphs $G_1$ and $G_2$ as
follows: For $i \in \{1,2\}$, the vertex set of $G_i$ is $(G^0)_i$,
the edge set is 
the set of paths of length 2 in $(G, \LL)$ for which both range
and source are in $(G^0)_i$, and the range and source maps are
inherited from $G$.
For $i \in \{1,2\}$, define $\LL_i \colon G_i^1 \to \AA^2$ by
$\LL_i(ef) = \LL(e) \LL(f)$.
The pair $(G_1, \LL_1)$, $(G_2, \LL_2)$ is called the \emph{induced
pair of labelled graphs} of $(G, \LL)$. This decomposition is not
necessarily unique, but whenever a bipartite labelled graph is
mentioned, it will be assumed that the induced graphs are specified.

\begin{remark}[Nasu {\cite[Remark 4.2]{nasu}}]
\label{rem_standard_code_of_bipartite_graph}
\index{labelled graph!bipartite!standard bipartite code of}
\index{bipartite code!of bipartite graph}
Let $(G, \LL)$ be a bipartite label\-led graph for which the induced
pair of labelled graphs is $(G_1, \LL_1)$, $(G_2, \LL_2)$. Let
$X_1$ and $X_2$ be the sofic shifts presented by these graphs, and
let $\X_{G_1}, 
\X_{G_2}$ be the edge shifts generated by $G_1$, $G_2$.
The natural embedding $f \colon G_1^1 \to (G^1)^2$ is a bipartite
expression which induces two bipartite codes $\varphi_\pm \colon \X_{G_1}
\to \X_{G_2}$ such that $(\varphi_+(x))_i = f_i e_{i+1}$ and
$(\varphi_-(x))_i = f_{i-1} e_i$ when $x = (e_i f_i)_{i \in \Z} \in
\X_{G_1}$.     
Similarly, the embedding $F \colon \LL_1(G_1^1) \to (\LL(G^1))^2$ is a
bipartite expression which induces bipartite codes $\Phi_\pm \colon
X_1 \to X_2$ such that
$(\Phi_+(x))_i = b_i a_{i+1}$ and $(\Phi_-(x))_i = b_{i-1} a_i$ when
$x = (a_i b_i)_{i \in \Z} \in X_1$.
By definition, 
$\Phi_\pm \circ \pi_1 = \pi_2 \circ \varphi_\pm$ when 
$\pi_1 \colon \X_{G_1} \to X_1$, $\pi_2 \colon \X_{G_2} \to X_2$ are
the covering maps.
The bipartite codes $\varphi_\pm$ and $\Phi_\pm$ are called the
\emph{standard bipartite codes induced by} $(G, \LL)$. 
\end{remark}


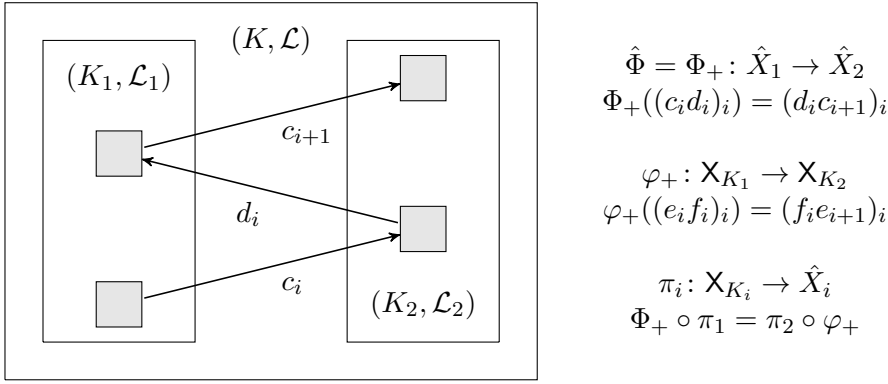
\begin{figure}
\begin{center}
\begin{tabular}{c c}
\begin{tabular}{c}
\begin{tikzpicture}
  [bend angle=45,
   knude/.style = {circle, inner sep = 1pt, minimum size = 6 mm},
   vertex/.style = {circle, draw, minimum size = 1 mm, inner sep =
      0pt, fill=black!10},
   textVertex/.style = {rectangle, draw, minimum size = 6 mm, inner sep =
      1pt, fill=black!10},
   to/.style = {->, shorten <= 1 pt, >=stealth', semithick}]
  

  \node[knude] (P1) at   (-2,-1.5) {};
  \node[knude] (P2) at   (-2,0.5) {};
  \node[knude] (Q1) at   (2,-0.5) {};
  \node[knude] (Q2) at   (2,1.5) {};

  \node[textVertex] (P1text) at ($(P1)$) {};
  \node[textVertex] (P2text) at ($(P2)$) {};

  \draw (-3,-2) -- (-1,-2) -- (-1,2) -- (-3,2) -- (-3,-2); 
  \node[knude] (K1) at (-2,1.5) {$(K_1, \LL_1 )  $};

  \node[textVertex] (Q1text) at ($(Q1)$) {};
  \node[textVertex] (Q2text) at ($(Q2)$) {};

  \draw (3,-2) -- (1,-2) -- (1,2) -- (3,2) -- (3,-2); 
  \node[knude] (K2) at (2,-1.5) {$(K_2, \LL_2 )  $};
   
  \draw (-3.5,-2.5) -- (3.5,-2.5) -- (3.5,2.5) -- (-3.5,2.5) -- (-3.5,-2.5); 
  \node[knude] (K) at (0,2) {$(K, \LL)  $};

  \draw[to] (P1) to node[auto,swap] {$c_i$} (Q1);
  \draw[to] (Q1) to node[auto] {$d_i$} (P2);	 
  \draw[to] (P2) to node[auto,swap] {$c_{i+1}$} (Q2);

\end{tikzpicture}
\end{tabular}
&
\begin{tabular}{c}
$\hat \Phi = \Phi_+ \colon \hat X_1 \to \hat X_2$ \\
$\Phi_+( (c_i d_i)_i ) = (d_i c_{i+1} )_i$ \\
\\
$\varphi_+ \colon \X_{K_1} \to \X_{K_2}$ \\
$\varphi_+( (e_i f_i)_i ) = (f_i e_{i+1} )_i$ \\
\\
$\pi_i \colon \X_{K_i} \to \hat X_i$ \\ 
$\Phi_+ \circ \pi_1 = \pi_2 \circ \varphi_+$
\end{tabular}
\end{tabular}
\end{center}
\caption[Krieger covers of recoded shifts.]{An illustration of Lemma \ref{lem_nasu_4.6}. For $i \in \{1,2\}$, $(K_i, \LL_i)$ is the left Krieger cover of the recoded shift $\hat X_i$. The recoded bipartite code $\hat \Phi$ is equal to the standard bipartite code $\Phi_+$ induced by the bipartite left Krieger cover $(K, \LL)$ of the sofic shift $\hat X$ constructed in Lemma \ref{lem_nasu_4.6}. }
\label{fig_nasu_illustration}
\end{figure}
 
The following lemma is illustrated in Figure \ref{fig_nasu_illustration}. 
 
\begin{lem}[Nasu {\cite[Corollary 4.6(1)]{nasu}}]
\label{lem_nasu_4.6}
Let $\Phi \colon X_1 \to X_2$ be a bipartite code between sofic
shifts $X_1$ and $X_2$.
Let $\hat X_1$ and $\hat X_2$ be the recoded shifts of $X_1$ and $X_2$
respectively, and let $(K_1, \LL_1)$ and $(K_2, \LL_2)$ be the left
Krieger covers of $\hat X_1$ and $\hat X_2$ respectively.
Then there exists a sofic shift $\hat X$ for which the left Krieger cover
is a bipartite labelled graph such that the induced pair of labelled
graphs is $(K_1, \LL_1)$, $(K_2, \LL_2)$ and such that the recoded
bipartite code $\hat \Phi \colon \hat X_1 \to \hat X_2$ of $\Phi$ is
one of the standard bipartite codes $\Phi_\pm$ induced by the left Krieger
cover of $\hat X$ as defined in Remark \ref{rem_standard_code_of_bipartite_graph}. 
\end{lem}

The proof of the following theorem is very similar to the proof of the
corresponding result by Nasu \cite[Theorem 3.3]{nasu} for the left Krieger cover.

\index{generalised Fischer cover!is canonical}
\begin{thm} \label{thm_GFC_canonical}
The generalised left Fischer cover is canonical.
\end{thm}

\begin{proof}
Let $\Phi \colon X_1 \to X_2$ be a bipartite code. 
Let $\hat X_1, \hat X_2$ be the recoded shifts, let $(K_1, \LL_1)$, $(K_2,
\LL_2)$ be the corresponding left Krieger covers, and let $\hat \Phi
\colon \hat X_1 \to \hat X_2$ be the recoded bipartite code. 
Use Lemma \ref{lem_nasu_4.6} to find a sofic shift $\hat X$ such that the
left Krieger cover $(K, \LL)$ of $\hat X$ is a bipartite labelled
graph for which 
the induced pair of labelled graphs is $(K_1, \LL_1)$, $(K_2, \LL_2)$.
Let $(G_1, \LL_1)$, $(G_2, \LL_2)$, and $(G, \LL)$ be the
generalised left Fischer covers of respectively $\hat X_1$, $\hat X_2$, and
$\hat X$.

The labelled graph $(G, \LL)$ is bipartite since $G$ is a
subgraph of $K$. Note that a predecessor set $P$ in $K_1^0$ or $K_2^0$ is
decomposable if and only if the corresponding predecessor set in $K^0$
is decomposable. 
If $i \in \{1,2\}$ and $Q \in G_i^0 \subseteq K_i^0$, then there is a
path in $K_i$ from $Q$ to an indecomposable $P \in K_i^0$.
By considering the corresponding path in $K$, it is clear that the
vertex in $K^0$ corresponding to $Q$ is in $G^0$.
Conversely, if $Q \in G^0$, then there is a path in $K$ from $Q$ to a
indecomposable $P \in K^0$.
If $P$ and $Q$ belong to the same partition $K_i^0$, then the vertex in
$K_i$ corresponding to $Q$ is in $G_i^0$ by definition.
On the other hand, if $Q$ corresponds to a vertex in $K_i$ and if $P$ 
belongs to the other partition, then Lemma \ref{lem_GLFC_essential}
shows that there exists an indecomposable $P'$ in the same partition
as $Q$ and an edge from $P$ to $P'$ in $K$. Hence, there is also a
path in $K_i$ from the vertex corresponding to $Q$ to the vertex
corresponding to $P'$, so $Q \in G_i^0$. This proves that the pair of
induced labelled graphs of $(G, \LL)$ is $(G_1, \LL_1)$, $(G_2,
\LL_2)$.   

Let $\hat \Psi_\pm \colon \hat X_1 \to \hat X_2$ be the standard
bipartite codes induced by $(G, \LL)$.
Remark \ref{rem_standard_code_of_bipartite_graph} shows that
there exist bipartite codes $\hat \psi_\pm \colon \X_{G_1} \to \X_{G_2}$
such that $\hat \Psi_\pm \circ \hat \pi_1|_{\X_{G_1}} = \hat
\pi_2|_{\X_{G_2}} \circ \hat \psi_\pm$.
The labelled graph $(G, \LL)$ presents the same sofic shift as $(K,
\LL)$, so they both induce the same standard bipartite codes from
$\hat X_1$ to $\hat X_2$, and by Lemma \ref{lem_nasu_4.6},
$\hat \Phi$ is one of these standard bipartite codes, so $\hat
\Phi = \hat \Psi_+$ or $\hat \Phi = \hat \Psi_-$. In particular, there
exists a bipartite code $\hat \psi \colon \X_{G_1} \to \X_{G_2}$ such
that $\hat \Phi \circ \hat \pi_1|_{\X_{G_1}} = \hat \pi_2|_{\X_{G_2}}
\circ \hat \psi$. 

By recoding $\hat X_1$ to $X_1$ and $\hat X_2$ to $X_2$ via
the bipartite expressions inducing $\Phi$,
this gives a bipartite code $\psi$ 
such that $\Phi \circ \pi_1 = \pi_2 \circ \psi$ when $\pi_1, \pi_2$
are the covering maps of the generalised left Fischer covers of $X_1$
and $X_2$ respectively.
By Theorem \ref{thm_nasu_decomposition}, any conjugacy can be
decomposed as a product of bipartite codes, so this proves that the
generalised left Fischer cover is canonical. 
\end{proof}


\begin{thm}
\label{thm_GFC_FE}
\index{generalised Fischer cover!is flow invariant}
The generalised left Fischer cover is flow invariant.
\end{thm}

\begin{proof}
By Proposition \ref{prop_covers_respect_se}, the left Krieger cover respects symbol expansion: If $X$ is a sofic shift with alphabet $\AA$, $a \in \AA$, and $\diamond \notin \AA$,
then the left Krieger cover of $X^{a \mapsto a\diamond}$ is obtained by replacing each edge labelled $a$ in the left Krieger cover of $X$ by two edges in sequence labelled $a$ and $\diamond$
respectively. Clearly, the generalised left Fischer cover inherits
this property. By Theorems \ref{thm_cover_fe} and \ref{thm_GFC_canonical}, it follows that the generalised left Fischer cover is flow invariant.
\end{proof}

\section{Foundations and layers of covers}
\label{sec_foundation}

Let $\EE = (E, \LL)$ be a finite left-resolving and predecessor-separated labelled graph.
For each $V \subseteq E^0$ and each word $w$ over the alphabet $\AA$ of $\LL$  define 
\begin{displaymath}
   wV = \{ u \in E^0 \mid u \textrm{ is the source of a path labelled } w \textrm{ terminating in } V \}.
\end{displaymath}

\begin{definition}
\index{past closed}
Let $S$ be a subset of the power set $\PP(E^0)$, and let ${\sim}$ be an equivalence relation on $S$.
The pair $(S,{\sim})$ is said to be \emph{past closed} if for all $u,v \in E^0$, $U,V \in S$, and $a \in \AA$, 
\begin{itemize}
\item $\{ v \} \in S$,
\item $\{ u \} \sim \{ v \}$ implies $u = v$,
\item $aV \neq \emptyset$ implies $aV \in S$, and
\item $U \sim V$ and $aU \neq  \emptyset$ implies $aV \neq \emptyset$ and $aU \sim aV$.
\end{itemize}
\end{definition}

\index{receive}
Let $(S, {\sim})$ be past closed. For each $V \in S$, let $\class V$ denote the equivalence class of $V$ with respect to ${\sim}$.
When $a \in \AA$ and $V \in S$, $\class V$ is said to \emph{receive} $a$ if $aV \neq \emptyset$.
For each $\class V \in S / {\sim}$, define $\lvert \class V \rvert = \min_{V \in \class V} \lvert V \rvert$.

\begin{definition}
\index{layer}\index{foundation}\index{labelled graph!foundation of}
Define $\GG(\EE, S, {\sim})$ to be the labelled graph with vertex set $S / {\sim}$ for which there is an edge labelled $a$ from $\class{aV}$ to $\class V$ whenever $\class V$ receives $a$.
For each $n \in \N$, the $n$\emph{th layer} of $\GG(\EE, S,{\sim})$ is the labelled subgraph induced by $S_n = \{ \class V \in S / {\sim} \mid  n = \lvert \class V \rvert \}$. 
$\EE$ is said to be a \emph{foundation} of any labelled graph isomorphic to $\GG(\EE, S, {\sim})$.
 \end{definition}

\index{labelled graph!layer of}
If a labelled graph $\HH$ is isomorphic to $\GG(\EE, S, {\sim})$, then the subgraph of $\HH$ corresponding to the $n$th layer of $\GG(\EE, S, {\sim})$ is be said to be the $n$\emph{th layer of} $\HH$ \emph{with respect to} $\EE$, or simply the \emph{$n$th layer} if $\EE$ is understood from the context.
The following proposition motivates the use of the term layer by showing that edges can never go from higher to lower layers.

\index{layer!structure of}
\begin{prop}
\label{prop_layers}
If $\class V \in S/ {\sim}$ receives $a \in \AA$, then $\lvert \class{aV} \rvert  \leq$ $\lvert \class V \rvert$\fxnote{check hack!}. If
$\GG(\EE, S,{\sim})$ has an edge from a vertex in the $m$th layer to a vertex in the $n$th layer, then $m \leq n$.
\end{prop}

\begin{proof}
Choose $V \in \class V$ such that $\lvert V \rvert = \lvert \class V \rvert$. Each $u \in aV$ emits at least one edge labelled $a$ terminating in $V$, and $\EE$ is left-resolving, so $\lvert \class{aV} \rvert \leq \lvert aV \rvert \leq \lvert V \rvert = \lvert \class V \rvert$. The second statement follows from the definition of $\GG(\EE, S,{\sim})$.
\end{proof}

\begin{prop}
$\EE$ and $\GG(\EE, S,{\sim})$ present the same sofic shift, and
$\EE$ is labelled graph isomorphic to the first layer of $\GG(\EE, S,{\sim})$. 
\end{prop}

\begin{proof}
By assumption, there is a bijection between $E^0$ and the set of vertices in the first layer of $\GG(\EE, S,{\sim})$. By Proposition \ref{prop_layers}, there is an edge labelled $a$ from $u$ to $v$ in $\EE$ if and only if there is an edge labelled $a$ from $\class{\{u\}}$ to $\class{\{v\}}$ in $\GG(\EE, S,{\sim})$.
Every finite word presented by $\GG(\EE, S, {\sim})$ is also presented by $\EE$, so they present the same sofic shift.
\end{proof}

\begin{example}
\index{multiplicity set cover}
Let $(F, \LL_F)$ be the left Fischer cover of an irreducible sofic shift $X$. 
For each $x^+ \in X^+$, let $s(x^+) \subseteq F^0$ to be the set of 
vertices where a presentation of $x^+$ can start. With the trivial relation $=$, the set $S = \{ s(x^+) \mid x^+ \in X^+ \} \subseteq \PP(F^0)$ is past closed since each vertex in the left Fischer cover is the predecessor set of an intrinsically synchronizing right-ray, so the \emph{multiplicity set cover} of $X$ can be defined to be $\GG((F, \LL_F), S, =)$.
 An analogous cover can be defined by considering the vertices where presentations of finite words can start.
Thomsen \cite{thomsen} constructs the
derived shift space $\partial X$ of $X$
using
right-resolving graphs, but an analogous construction works for left-resolving graphs.
The procedure from \cite[Example 6.10]{thomsen} shows that this
$\partial X$ is presented by the labelled graph obtained by removing the left Fischer cover from the multiplicity set cover.
\end{example}

Let $X$ be a sofic shift, and let $(K, \LL_K)$ be the left Krieger cover of $X$. In order to use the preceding results to investigate the structure of the left Krieger cover and the past set cover, define an equivalence relation on $\PP(K^0)$ by $U \sim_\cup V$ if and only if $\bigcup_{P \in U} P = \bigcup_{Q \in V} Q$. Clearly, $\{ P \} \sim_\cup \{ Q \}$ if and only if $P = Q$. If $U, V \subseteq K^0$, $a \in \AA$, $aV \neq \emptyset$, and $U \sim_\cup V$, then $aU \sim_\cup aV$ by the definition of the left Krieger cover.

\begin{thm}
\label{thm_gfc_foundation_lkc}
\index{generalised Fischer cover!as foundation of \\Krieger cover}\index{Fischer cover!as foundation of \\Krieger cover}
For a sofic shift $X$, the generalised left Fischer cover  is a foundation of the left Krieger cover, and no smaller subgraph is a foundation.
\end{thm}

\begin{proof}
Let $(G, \LL_G)$ and $(K, \LL_K)$ be the generalised left Fischer cover and the left Krieger cover of $X$, respectively. Define
\begin{displaymath}
S = \Big \{ V \subseteq G^0 \mid \exists x^+ \in X^+ \textrm{ such that } P_\infty(x^+) = \bigcup_{P \in V} P \Big \}.
\end{displaymath} 
Note that $\{ P \} \in S$ for every $P \in G^0$.
If $x^+ \in X^+$ with $P_\infty(x^+) = \bigcup_{P \in V} P$ and if 
$aV \neq \emptyset$ for some $a \in \AA$, then $ax^+ \in X^+$ and $P_\infty(ax^+) = \bigcup_{P \in aV} P$.
This proves that  the pair $(S, {\sim_\cup})$ is past closed, so 
$\GG((G,\LL_G), S, {\sim_\cup})$ is well defined.
Since $(G, \LL_G)$ is a presentation of $X$, there is a bijection $\varphi \colon S / {\sim_\cup} \to K^0$ defined by $\varphi( \class V ) = \bigcup_{P \in V} P$.
By construction, there is an edge labelled $a$ from $\class U$ to $\class V$ in $\GG((G,\LL_G), S, {\sim_\cup})$ if and only if there exists $x^+ \in X^+$ such that
$P_\infty(ax^+) = \bigcup_{P \in U} P$ and $P_\infty(x^+) = \bigcup_{Q \in V} Q$, so $\GG((G,\LL_G), S, {\sim_\cup})$ is isomorphic to $(K, \LL_K)$.
It follows from Lemma \ref{lem_decomposable} that no proper subgraph of $(G, \LL_G)$ can be a foundation of the left Krieger cover.
\end{proof}

\begin{example} \label{ex_justifying}
Figure \ref{F_justifying_shift} shows a labelled graph which is predecessor-separated, irreducible, and left-resolving, so by Theorem \ref{thm_lfc_char}, it is the left Fischer cover of an irreducible sofic shift $X$. Note that $P_\infty(a^\infty) = P_\infty(u) \cup P_\infty(v) = P_\infty(u)$, so this is a vertex in the first layer of the Krieger cover with respect to the Fischer cover even though $a^\infty$  has presentations starting at two different vertices. This illustrates the difference between the left Krieger cover and the left multiplicity set cover.
This example was inspired by the example in \cite[Section 4]{carlsen_matsumoto}. 
\end{example}

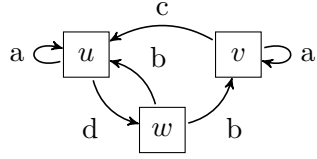
\begin{figure}
\begin{center}
\begin{tikzpicture}
  [bend angle=45,
   knude/.style = {circle, inner sep = 0pt},
   vertex/.style = {circle, draw, minimum size = 1 mm, inner sep =
      0pt},
   textVertex/.style = {rectangle, draw, minimum size = 6 mm, inner sep =
      1pt},
   to/.style = {->, shorten <= 1 pt, >=stealth', semithick}]
  
  \node[textVertex] (u) at (-1, 0) {$u$};
  \node[textVertex] (v) at ( 1, 0) {$v$};  
  \node[textVertex] (w) at (0, -1) {$w$}; 
  
  \draw[to, bend right=30] (v) to node[auto,swap] {c} (u);
  \draw[to, loop left]     (u) to node[auto] {a} (u);
  \draw[to, loop right]    (v) to node[auto] {a} (v);
  \draw[to, bend right=30] (u) to node[auto, swap] {d} (w);
  \draw[to, bend right=30] (w) to node[auto, swap] {b} (u); 
  \draw[to, bend right=30] (w) to node[auto,swap] {b} (v);   
\end{tikzpicture}
\end{center}
\caption[Identical Fischer and Krieger covers.]{Left Fischer cover of an irreducible strictly sofic
  shift. The sequence 
  $a^\infty$ is $1$-synchronizing even though $s(a^\infty) = \{u,v\}$.}
\label{F_justifying_shift}
\end{figure}

The example from \cite[Section 4]{carlsen_matsumoto} shows that the left Krieger cover can be a proper subgraph of the past set cover. The following lemma will be used to further investigate this relationship.

\begin{lem}
\label{lem_prefix_with_same_past}
Let $X$ be a sofic shift. For each $x^+ = x_1 x_2 x_3 \ldots \in X^+$ there exists $n \in \N$ such that $P_\infty(x^+) = P_\infty(x_1 x_2 \ldots x_k)$ for all $k \geq n$. 
\end{lem}

\begin{proof} 
It is clear that $P_\infty(x_1) \supseteq P_\infty(x_1 x_2) \supseteq
\cdots \supseteq P_\infty(x^+)$. Since $X$ is
sofic, there are only finitely many different predecessor sets of
words, so there must exist $n \in \N$ such that $P_\infty(x_1 x_2 \ldots x_k)
= P_\infty(x_1 x_2 \ldots x_n)$ for all $k \geq n$. If $y^- \in
P_\infty(x_1 x_2 \ldots x_n)$ is given, then $y^- x_1 x_2 \ldots x_k \in X$ for
all $k \geq n$, so $y^- x^+$ contains no forbidden words, and
therefore $y^- \in P_\infty(x^+)$. Since $y^-$ was arbitrary,
$P_\infty(x^+) =  P_\infty(x_1 x_2 \ldots x_n)$.
\end{proof}

\begin{thm}
\label{thm_foundations_psc}
\index{generalised Fischer cover!as foundation of \\ past set cover}\index{Fischer cover!as foundation of \\ past set cover}
\index{Krieger cover!as foundation of  \\past set cover}
For a sofic shift $X$, the generalised left Fischer cover and the left Krieger cover  are both foundations of the past set cover.
\end{thm}

\begin{proof}
Let $(G, \LL_G)$ be the generalised left Fischer cover of $X$, let $(K, \LL_K)$ be the left Krieger cover of $X$, and let $(\psc, \LL_\psc)$ be the past set cover of $X$. Define
\begin{displaymath}
S = \Big \{ V \subseteq G^0 \mid \exists w \in \BB(X) \textrm{ such that } P_\infty(w) = \bigcup_{P \in V} P \Big \},
\end{displaymath}
and use Lemma \ref{lem_prefix_with_same_past} to conclude that $S$ contains $\{ P \}$ for every $P \in G^0$. By arguments analogous to the ones used in the proof of Theorem \ref{thm_gfc_foundation_lkc}, it follows that $\GG((G,\LL_G),S, {\sim_\cup})$ is isomorphic to $(\psc, \LL_\psc)$. 
To see that $(K, \LL_K)$ is also a foundation, define 
$T = \{ V \subseteq K^0 \mid \exists w \in \BB(X) \textrm{ such that } P_\infty(w) = \bigcup_{P \in V} P \}$,
 and apply arguments analogous to the ones used above to prove that $(\psc, \LL_\psc)$ is isomorphic to $\GG((K,\LL_K),T, {\sim_\cup})$.
\end{proof}

\index{*1@$\frac{1}{n}$-synchronizing}
\index{word!$\frac{1}{n}$-synchronizing}
\index{ray!$\frac{1}{n}$-synchronizing}
\index{Krieger cover!layer of}
\index{past set cover!layer of}
In the following, the $n$th layer of the left Krieger cover (past set cover) will always refer to the $n$th layer with respect to the generalised left Fischer cover $(G, \LL_G)$. For a right-ray (word) $x$, $P_\infty(x)$ is a vertex in the $n$th layer of the left Kriger cover (past set cover) for some $n \in \N$, and such an $x$ is said to be $\frac{1}{n}$-\emph{synchronizing}. Note that $x$ is $\frac{1}{n}$-synchronizing if and only if $n$ is the smallest number such that there exist $P_1, \ldots, P_n \in G^0$ with $\bigcup_{i=1}^n P_i = P_\infty(x)$.
In an irreducible sofic shift with left Fischer cover $(F, \LL_F)$, this happens if and only if $n$ is the smallest number such that there exist $u_1, \ldots, u_n \in F^0$ with $\bigcup_{i=1}^n P_\infty(u_i) = P_\infty(x)$.

\begin{cor}
\label{cor_PSC_reducible_if_LKC_reducible}
If the left Krieger cover of a sofic shift is reducible, then so is the past set cover.
\end{cor}

\begin{proof}
This follows from Proposition \ref{prop_layers} and Theorem  \ref{thm_foundations_psc}. 
\end{proof}

\begin{example}
Figures \ref{F_LFC_3cc} and \ref{F_LKC_3cc} show, respectively, the left Fischer and the left Krieger cover of the 3-charge constrained 
shift (see e.g.\ \cite[Example 1.2.7]{lind_marcus} for the definition of charge
constrained shifts).
There are 3 vertices in the second layer of the left Krieger cover and
two in the third.
Note how the left Fischer cover can be identified with the first layer
of the left Krieger cover.
Note also that the second layer is the left Fischer cover
of the 2-charge constrained shift and that the third layer is the left
Fischer cover of the 1-charge constrained shift.
\end{example}

\begin{figure}
\begin{center}
\begin{tikzpicture}
  [
   knude/.style = {circle, inner sep = 0pt},
   vertex/.style = {circle, draw, minimum size = 1 mm, inner sep =
      0pt},
   textVertex/.style = {rectangle, draw, minimum size = 6 mm, inner sep =
      1pt},
   to/.style = {->, shorten <= 1 pt, >=stealth', semithick}]
  

    \node[textVertex] (u) at (-3,4) {$u$};  
    \node[textVertex] (v) at (-1,4) {$v$};  
    \node[textVertex] (w) at (1,4) {$w$};  
    \node[textVertex] (x) at (3,4) {$x$}; 

  \draw[to, bend left=30] (u) to node[auto] {$+$} (v);
  \draw[to, bend left=30] (v) to node[auto] {$+$} (w);
  \draw[to, bend left=30] (w) to node[auto] {$+$} (x);
  \draw[to, bend left=30] (x) to node[auto] {$-$} (w);
  \draw[to, bend left=30] (w) to node[auto] {$-$} (v);
  \draw[to, bend left=30] (v) to node[auto] {$-$} (u);

\end{tikzpicture}
\end{center}
\caption{Left Fischer cover of the 3-charge
  constrained shift.} 
\label{F_LFC_3cc}
\end{figure}
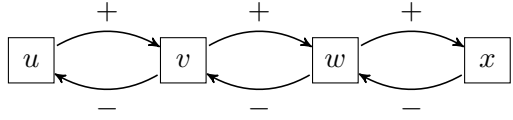
\begin{figure}
\begin{center}
\begin{tikzpicture}
  [
   knude/.style = {circle, inner sep = 0pt},
   vertex/.style = {circle, draw, minimum size = 1 mm, inner sep =
      0pt},
   textVertex/.style = {rectangle, draw, minimum size = 6 mm, inner sep =
      1pt},
   to/.style = {->, shorten <= 1 pt, >=stealth', semithick}]
  
    \node[textVertex] (Pu) at (-4.5,1) {$P_\infty(u)$};
    \node[textVertex] (Pv) at (-1.5,1) {$P_\infty(v)$};
    \node[textVertex] (Pw) at ( 1.5,1) {$P_\infty(w)$}; 
    \node[textVertex] (Px) at ( 4.5,1) {$P_\infty(x)$}; 

    \node[textVertex] (Puv) at (-4,-1) {$P_\infty(u) \cup P_\infty(v)$};
    \node[textVertex] (Pvw) at ( 0,-1) {$P_\infty(v) \cup P_\infty(w)$};
    \node[textVertex] (Pwx) at ( 4,-1) {$P_\infty(w) \cup P_\infty(x)$};

    \node[textVertex] (Puvw) at (-3,-3) {$P_\infty(u) \cup
      P_\infty(v) \cup P_\infty(w)$};
    \node[textVertex] (Pvwx) at (3,-3) {$P_\infty(v) \cup
      P_\infty(w) \cup P_\infty(x)$};

  \draw[to, bend left=20] (Pu) to node[auto] {$+$} (Pv);
  \draw[to, bend left=20] (Pv) to node[auto] {$+$} (Pw);
  \draw[to, bend left=20] (Pw) to node[auto] {$+$} (Px);
  \draw[to, bend left=20] (Px) to node[auto] {$-$} (Pw);
  \draw[to, bend left=20] (Pw) to node[auto] {$-$} (Pv);
  \draw[to, bend left=20] (Pv) to node[auto] {$-$} (Pu);
 
  \draw[to] (Pu) to node[auto,swap] {$+$} (Puv);
  \draw[to] (Px) to node[auto] {$-$} (Pwx);

  \draw[to, bend left=10] (Puv) to node[auto] {$+$} (Pvw);
  \draw[to, bend left=10] (Pvw) to node[auto] {$+$} (Pwx);
  \draw[to, bend left=10] (Pwx) to node[auto] {$-$} (Pvw);
  \draw[to, bend left=10] (Pvw) to node[auto] {$-$} (Puv);

  \draw[to] (Puv) to node[auto,swap] {$+$} (Puvw);
  \draw[to] (Pwx) to node[auto] {$-$} (Pvwx);

  \draw[to, bend left=5] (Puvw) to node[auto] {$+$} (Pvwx);
  \draw[to, bend left=5] (Pvwx) to node[auto] {$-$} (Puvw);
\end{tikzpicture}
\end{center}
\caption{Left Krieger cover of the 3-charge
  constrained shift.} 
\label{F_LKC_3cc}
\end{figure}
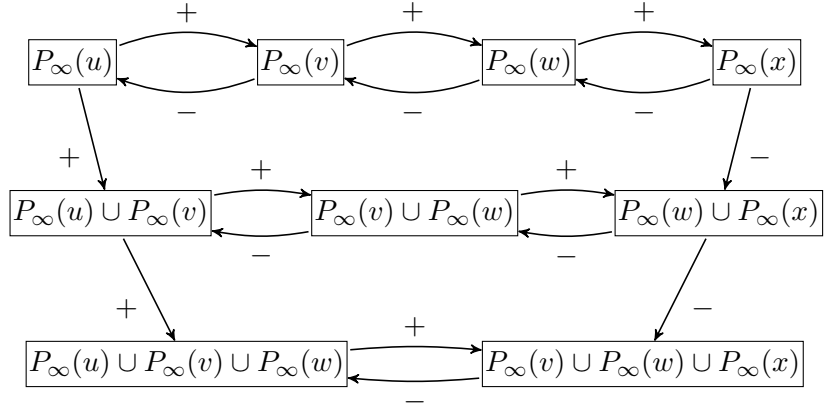

\begin{example}
\label{ex_difference_lkc_psc}
For many shifts (e.g.\ the even shift) the left Krieger cover and the past set cover are equal. To see that this is not always the case, consider the labelled graph in Figure \ref{fig_difference_lkc_psc_I}. The graph is irreducible, left-resolving, and predecessor-separated, so by Theorem \ref{thm_lfc_char}, it is the left Fischer cover of an irreducible sofic shift $X$. Note that $P_\infty(a^\infty) = P_\infty(u_1) \cup P_\infty(u_2) \cup P_\infty(u_3) = P_\infty(u_3)$, and it is easy to check that all other right-rays are also $1$-synchronizing, so the left Krieger cover is equal to the left Fischer cover.
However, $P_\infty(a^n c) = P_\infty(u_1) \cup P_\infty(u_2)$ for each $n \in \N$, and there is clearly no vertex $v$ in the left Fischer cover such that  $P_\infty(a^n c) = P_\infty(v)$, so $P_\infty(a^n c)$ is a vertex in the second layer of the past set cover which is shown in Figure \ref{fig_difference_lkc_psc_II}. This example was inspired by the example from \cite[Section 4]{carlsen_matsumoto}.

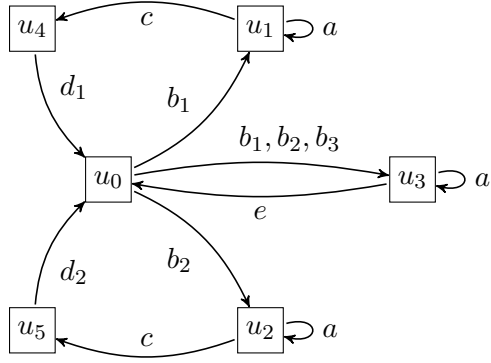
\begin{figure}[t]
\begin{center}
\begin{tikzpicture}
  [bend angle=20,
   knude/.style = {circle, inner sep = 0pt},
   vertex/.style = {circle, draw, minimum size = 1 mm, inner sep =
      0pt},
   textVertex/.style = {rectangle, draw, minimum size = 6 mm, inner sep =
      1pt},
   to/.style = {->, shorten <= 1 pt, >=stealth', semithick}]
  
  \node[textVertex] (u0) at (0,0) {$u_0$};
  \node[textVertex] (u1) at (2,2) {$u_1$};
  \node[textVertex] (u2) at (2,-2) {$u_2$};
  \node[textVertex] (u3) at (4,0) {$u_3$};
  \node[textVertex] (u4) at (-1,2) {$u_4$};
  \node[textVertex] (u5) at (-1,-2) {$u_5$};

  \draw[to, bend right] (u0) to node[auto] {$b_1$} (u1);
  \draw[to, bend left] (u0) to node[auto,swap] {$b_2$} (u2);
  \draw[to, bend left=10] (u0) to node[auto] {\qquad $b_1, b_2, b_3$} (u3);
  \draw[to, bend left=10] (u3) to node[auto] {$e$} (u0);  

  \draw[to, loop right] (u1) to node[auto] {$a$} (u1);
  \draw[to, loop right] (u2) to node[auto] {$a$} (u2);
  \draw[to, loop right] (u3) to node[auto] {$a$} (u3);

  \draw[to, bend right] (u1) to node[auto] {$c$} (u4);
  \draw[to, bend left] (u2) to node[auto,swap] {$c$} (u5);
  \draw[to, bend right] (u4) to node[auto] {$d_1$} (u0);
  \draw[to, bend left] (u5) to node[auto,swap] {$d_2$} (u0);

\end{tikzpicture}
\end{center}
\caption[Left Krieger cover vs. past set cover I.]{Left Krieger cover of the sofic shift considered in Example \ref{ex_difference_lkc_psc}.}
\label{fig_difference_lkc_psc_I}
\end{figure}

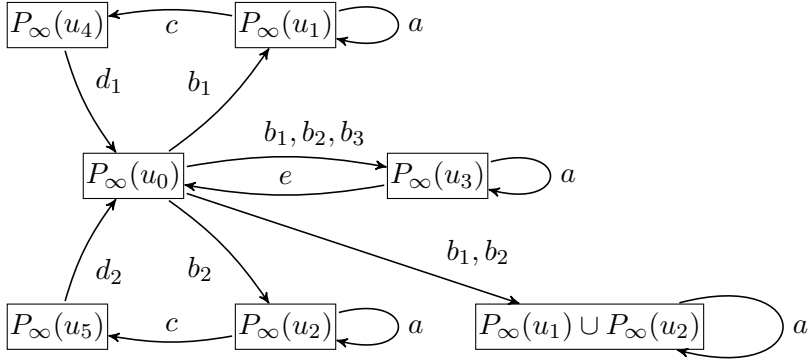
\begin{figure}
\begin{center}
\begin{tikzpicture}
  [bend angle=10,
   knude/.style = {circle, inner sep = 0pt},
   vertex/.style = {circle, draw, minimum size = 1 mm, inner sep =
      0pt},
   textVertex/.style = {rectangle, draw, minimum size = 6 mm, inner sep =
      1pt},
   to/.style = {->, shorten <= 1 pt, >=stealth', semithick}]
  
  \node[textVertex] (u0) at (0,0) {$P_\infty(u_0)$};
  \node[textVertex] (u1) at (2,2) {$P_\infty(u_1)$};
  \node[textVertex] (u2) at (2,-2) {$P_\infty(u_2)$};
  \node[textVertex] (u3) at (4,0) {$P_\infty(u_3)$};
  \node[textVertex] (u4) at (-1,2) {$P_\infty(u_4)$};
  \node[textVertex] (u5) at (-1,-2) {$P_\infty(u_5)$};
  \node[textVertex] (u6) at (6,-2) {$P_\infty(u_1) \cup P_\infty(u_2)$};

  \draw[to, bend right] (u0) to node[auto] {$b_1$} (u1);
  \draw[to, bend left] (u0) to node[auto,swap] {$b_2$} (u2);
  \draw[to, bend left] (u0) to node[auto] {\qquad $b_1, b_2, b_3$} (u3);
  \draw[to, bend left] (u3) to node[auto,swap] {$e$} (u0);  

  \draw[to, loop right] (u1) to node[auto] {$a$} (u1);

  \draw[to, loop right] (u2) to node[auto] {$a$} (u2);
  \draw[to, loop right] (u3) to node[auto] {$a$} (u3);

  \draw[to, bend right] (u1) to node[auto] {$c$} (u4);
  \draw[to, bend left] (u2) to node[auto,swap] {$c$} (u5);
  \draw[to, bend right] (u4) to node[auto] {$d_1$} (u0);
  \draw[to, bend left] (u5) to node[auto,swap] {$d_2$} (u0);

  \draw[to] (u0) to node[auto,near end] {$b_1, b_2$} (u6);
  \draw[to, loop right] (u6) to node[auto] {$a$} (u6);
\end{tikzpicture}
\end{center}
\caption[Left Krieger cover vs. past set cover II.]{Past set cover of the sofic shift considered in Example \ref{ex_difference_lkc_psc}.}
\label{fig_difference_lkc_psc_II}
\end{figure}
\end{example}

\section{The range of a flow invariant}
\label{sec_invariant}
\index{vertex!properly communicate}\index{proper communication set}\index{proper communication graph}\index{PCE@$PC(E)$}
Let $E$ be a directed graph.
Vertices $u,v \in E^0$ \emph{properly communicate} \cite{bates_eilers_pask}
if there are paths
$\mu, \lambda \in E^*$ of length greater than or equal to 1 such that
$s(\mu) = u$, $r(\mu) = v$, $s(\lambda) = v$, and $r(\lambda) =
u$. 
This relation is used to construct maximal disjoint subsets of $E^0$, called \emph{proper communication sets of vertices}, such
that $u,v \in E^0$ properly communicate if and only if they belong to
the same subset. 
The \emph{proper communication graph} $PC(E)$ is defined to be
the directed graph for which the vertices are the proper
communication sets of vertices of $E$ and for which there is an edge
from one proper communication set to another if and only if there is
a path from a vertex in the first set to a vertex in the second. 
The proper communication graph of the left Krieger cover of a sofic shift space is a flow invariant \cite{bates_eilers_pask}.

Let $X$ be an irreducible sofic shift with left Fischer cover $(F, \LL_F)$ and left Krieger cover $(K, \LL_K)$, and let $E$ be the proper communication graph of $K$. By construction, $E$ is finite and contains no circuit. The left Fischer cover  is isomorphic to an irreducible subgraph of $(K,\LL_K)$ corresponding to a root $r \in E^0$, and by definition, there is an edge from $u \in E^0$ to $v \in E^0$ whenever $u > v$. 
The following proposition gives the range of the flow invariant 
by proving that all such graphs can occur. 

\begin{prop}
\label{prop_range_invariant}
\index{proper communication graph!range of}
Let $E$ be a finite directed graph with a root and without circuits. $E$ is the proper communication graph of the left
Krieger cover of an AFT shift if there is an edge from $u \in E^0$ to $v \in E^0$ whenever $u > v$.
\end{prop}

\begin{proof} 
Let $E$ be an arbitrary finite directed graph which contains no circuit and which has a root $r$, and let $\tilde E$ be the directed graph obtained from $E$ by adding an edge from $u \in E^0$ to $v \in E^0$ whenever $u > v$. 
The goal is to construct a labelled graph $(F, \LL_F)$ which is the
left Fischer cover of an irreducible sofic shift with the desired
properties. 
For each $v \in E^0$, let $l(v)$ be the length of the
longest path from $r$ to $v$. This is well-defined since $E$
does not contain any circuits.
For each $v \in E^0$, define $n(v) = 2^{l(v)}$ vertices $v_1, \ldots ,
v_{n(v)}  \in F^0$. The single vertex corresponding to the root $r
\in E^0$ is denoted $r_1$.
For each $v \in E^0$, draw a loop of length 1 labelled $a_v$ at each
of the vertices $v_1, \ldots , v_{n(v)}  \in F^0$. 
If there is an edge from $u \in E^0$ to $v \in E^0$, then $l(v) >
l(u)$. From each vertex $u_1, \ldots , u_{n(u)}$ draw
$n(u,v) = \frac{n(v)}{n(u)} = 2^{l(v)-l(u)} \geq 2$ edges labelled $a_{u,v}^1,
\ldots , a_{u,v}^{n(u,v)}$ such that every vertex $v_1,
\ldots , v_{n(v)}$ receives exactly one of these edges. 
For each sink $v \in E^0$, draw a uniquely labelled edge from each
vertex $v_1, \ldots , v_{n(v)}$ to $r_1$. 
This finishes the construction of $(F, \LL_F)$.

By construction, $F$ is irreducible, right-resolving, and
left-resolving. Additionally, it is 
predecessor-separated because there is a uniquely labelled path to every
vertex in $F^0$ from $r_1$.
Thus,  $(F, \LL_F)$ is the left Fischer
cover of an AFT shift $X$. Let $(K, \LL_K)$ be the
left Krieger cover of $X$. 

\begin{figure}[t]
\begin{center}
\begin{tikzpicture}
  [bend angle=45,
   knude/.style = {circle, inner sep = 0pt},
   vertex/.style = {circle, draw, minimum size = 1 mm, inner sep =
      0pt},
   textVertex/.style = {rectangle, draw, minimum size = 6 mm, inner sep =
      1pt},
   to/.style = {->, shorten <= 1 pt, >=stealth', semithick}]
  
  \node[textVertex] (r) at (0,0) {$r$}; 
  \node[textVertex] (x) at (-1.5,-1.5) {$x$}; 
  \node[textVertex] (y) at (0,-1.5) {$y$}; 
  \node[textVertex] (z) at (1.5,-1.5) {$z$}; 

  \draw[to] (r) to (x);
  \draw[to] (r) to (y);
  \draw[to] (r) to (z);
  \draw[to] (y) to (z);

\end{tikzpicture}
\end{center}
\caption[Construction of a proper communication graph I.]{A directed graph with root $r$ and without circuits.}
\label{fig_ex_construction_E}
\end{figure}
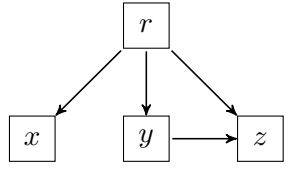  
\begin{figure}[t]
\begin{center}
\begin{tikzpicture}
  [bend angle=5,
   knude/.style = {circle, inner sep = 0pt},
   vertex/.style = {circle, draw, minimum size = 1 mm, inner sep =
      0pt},
   textVertex/.style = {rectangle, draw, minimum size = 6 mm, inner sep =
      1pt},
   to/.style = {->, shorten <= 1 pt, >=stealth', semithick}]
  
  \node[textVertex] (v0) at (0,0) {$r_1$}; 

  \node[textVertex] (vx1) at (-7, -2) {$x_1$}; 
  \node[textVertex] (vx2) at (-6, -2) {$x_2$}; 
  \node[textVertex] (vy1) at (-2.6, -2) {$y_1$}; 
  \node[textVertex] (vy2) at (2.6, -2) {$y_2$}; 

  \node[textVertex] (vz1) at (-3.9, -4) {$z_1$}; 
  \node[textVertex] (vz2) at (-1.3, -4) {$z_2$}; 
  \node[textVertex] (vz3) at (1.3, -4) {$z_3$}; 
  \node[textVertex] (vz4) at (3.9, -4) {$z_4$}; 

  \node[textVertex] (v0') at (0,-7) {$r_1$}; 

  \draw[to, bend right = 20] (v0) to node[near end, above] {$a_{r,x}^1$} (vx1);
  \draw[to, bend right = 18] (v0) to node[near end, below] {$a_{r,x}^2$} (vx2);
  \draw[to, bend right] (v0) to node[near end, right] {$a_{r,y}^1$} (vy1);
  \draw[to, bend left] (v0) to node[near end, left] {$a_{r,y}^2$} (vy2);

  \draw[to, bend right = 40] (v0) to node[near end, left] {$a_{r,z}^1$} (vz1);
  \draw[to, bend right] (v0) to node[near end, right] {$a_{r,z}^2$} (vz2);
  \draw[to, bend left]  (v0) to node[near end, left] {$a_{r,z}^3$} (vz3);
  \draw[to, bend left= 40] (v0) to node[near end, right]  {$a_{r,z}^4$} (vz4);

  \draw[to, bend right] (vy1) to node[very near end, right] {$a_{y,z}^1$} (vz1);
  \draw[to, bend left]  (vy1) to node[very near end, left] {$a_{y,z}^2$} (vz2);
  \draw[to, bend right] (vy2) to node[very near end, right] {$a_{y,z}^1$} (vz3);
  \draw[to, bend left]  (vy2) to node[very near end, left]  {$a_{y,z}^2$} (vz4);

  \draw[to, loop below] (vx1) to node[auto] {$a_x$} (vx1);
  \draw[to, loop below] (vx2) to node[auto] {$a_x$} (vx2) ;
  \draw[to, loop right] (vy1) to node[auto] {$a_y$} (vy1);
  \draw[to, loop left] (vy2) to node[auto] {$a_y$} (vy2);
  \draw[to, loop below] (vz1) to node[auto] {$a_z$} (vz1);
  \draw[to, loop below] (vz2) to node[auto] {$a_z$} (vz2) ;
  \draw[to, loop below] (vz3) to node[auto] {$a_z$} (vz3);
  \draw[to, loop below] (vz4) to node[auto] {$a_z$} (vz4);

  \draw[to, bend right = 35] (vx1) to (v0');
  \draw[to, bend right = 30] (vx2) to (v0');
  \draw[to, bend right = 10] (vz1) to (v0');
  \draw[to, bend left = 10] (vz2) to (v0');
  \draw[to, bend right = 10] (vz3) to (v0');
  \draw[to, bend left = 10] (vz4) to (v0');
\end{tikzpicture}
\end{center}
\caption[Construction of a proper communication graph II.]{Left Fischer cover of the sofic shift $X$ con\-sidered in
  Example \ref{ex_construction}.
 }
\label{fig_ex_construction_F}
\end{figure}

For every $v \in E^0$, the set $\{ v_i \mid 1 \leq i \leq n(v) \}$ satisfies 
$\bigcup_{i=1}^{n(v)} P_\infty(v_i) = P_\infty(a_v^\infty)$ and no smaller set of vertices has this property, so
$P_\infty(a_v^\infty)$ is a vertex in the $n(v)$th layer of the left
Krieger cover. There is clearly a loop labelled $a_v$ at the
vertex $P_\infty(a_v^\infty)$, so it belongs to a proper
communication set of vertices. 
Furthermore, $b a_v^\infty \in X^+$ if and only if $b =
a_v$ or $b = a_{u,v}^i$ for some $u \in E^0$ and $1 \leq i \leq
n(u,v)$. By construction,   
$P_\infty(a_{u,v}^i a_v^\infty) = \bigcup_{j=1}^{n(u)} P_\infty(u_j) =
P_\infty(a_u^\infty)$, so there is an edge from
$P_\infty(a_u^\infty)$ to $P_\infty(a_v^\infty)$ if and only if there
is an edge from $u$ to $v$ in $E$. This proves that $E$, and hence also $\tilde E$, are subgraphs
of the proper communication graph of $K$.

Since the edges which terminate at $r_1$ are uniquely labelled, any
$x^+ \in X^+$ which contains one of these letters must be intrinsically
synchronizing. If $x^+ \in X^+$ does not contain any
of these letters, then $x^+$ must be eventually periodic with $x^+ = w
a_v^\infty$ for some $v \in E^0$ and $w \in \BB(X)$. Thus, $K$ only
has the vertices described above, and therefore the proper
communication graph of $K$ is $\tilde E$. 
\end{proof}

\begin{example}
\label{ex_construction}
\index{proper communication graph!construction of}
To illustrate the construction from
the proof of Proposition \ref{prop_range_invariant}, let
$E$ be the directed graph drawn in Figure
\ref{fig_ex_construction_E}. $E$ has a unique
maximal vertex $r$ and contains no circuit, so 
it is the proper communication graph of the left Krieger cover of an
irreducible sofic shift.
Note that $l(x) = l(y) = 1$ and that $l(z) = 2$. Figure
\ref{fig_ex_construction_F} shows the left Fischer cover of a sofic
shift $X$ constructed using the method from the proof of Proposition
\ref{prop_range_invariant}. Note that the top and bottom vertices
should be identified, and that the labelling of the edges terminating 
at $r_1$ has been suppressed. Figure \ref{fig_ex_construction_K} shows
the left Krieger cover of $X$, but the structure of the irreducible
component corresponding to the left Fischer cover has been suppressed
to emphasise the structure of the higher layers. 
\end{example}

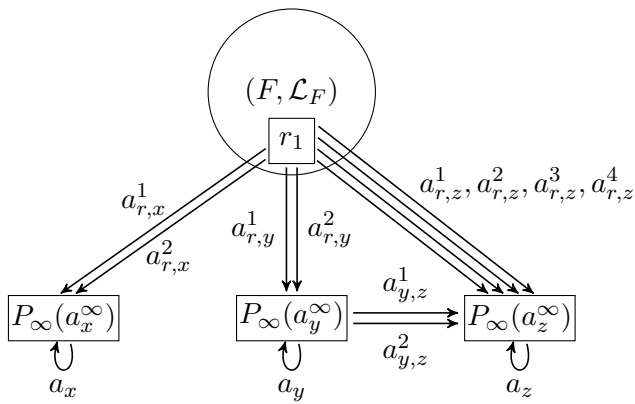
\begin{figure}
\begin{center}
\begin{tikzpicture}
  [bend angle=45,
   knude/.style = {circle, inner sep = 0pt},
   vertex/.style = {circle, draw, minimum size = 1 mm, inner sep =
      0pt},
   textVertex/.style = {rectangle, draw, minimum size = 6 mm, inner sep =
      1pt},
   to/.style = {->, shorten <= 1 pt, >=stealth', semithick}]
  
  \node[circle, draw, minimum size = 22 mm] (F) at
                           (0,0) {$(F, \LL_F)$}; 
  \node[textVertex] (vr) at (0, -0.65) {$r_1$};

  \node[textVertex] (x) at (-3,-3) {$P_\infty (a_x^\infty)$}; 
  \node[textVertex] (y) at (0,-3) {$P_\infty (a_y^\infty)$}; 
  \node[textVertex] (z) at (3,-3) {$P_\infty (a_z^\infty)$}; 

  \draw[->, shorten <= 5 pt, shorten >= 16 pt, >=stealth', semithick]
  ($(vr)+(-0.2,0)$) -> node[midway, above] {$a_{r,x}^1$ \,} ($(x)+(-0.5,0)$);
  \draw[->, shorten <= 12 pt, shorten >= 16 pt, >=stealth', semithick]
  ($(vr)+(0,0)$) -> node[midway, below] {\,$a_{r,x}^2$} ($(x)+(-0.3,0)$) ;

  \draw[->, shorten <= 10 pt, shorten >= 10 pt, >=stealth', semithick]
  ($(vr)+(-0.07,0)$) to node[auto,swap] {$a_{r,y}^1$} ($(y)+(-0.07,0)$);
  \draw[->, shorten <= 10 pt, shorten >= 10 pt, >=stealth', semithick]
  ($(vr)+(0.07,0)$) to node[auto] {$a_{r,y}^2$} ($(y)+(0.07,0)$);

  \draw[->, shorten <= 12 pt, shorten >= 15 pt, >=stealth', semithick]
  ($(vr)+(0.0,0.0)$) -> ($(z)+(0.0,0)$);
  \draw[->, shorten <=  5 pt, shorten >= 15 pt, >=stealth', semithick]
  ($(vr)+(0.2,0.0)$) -> ($(z)+(0.2,0)$);
  \draw[->, shorten <= -2 pt, shorten >= 15 pt, >=stealth', semithick]
  ($(vr)+(0.4,0.0)$) -> ($(z)+(0.4,0)$);
  \draw[->, shorten <= -9 pt, shorten >= 15 pt, >=stealth', semithick]
  ($(vr)+(0.6,0.0)$) -> 
  node[near start, right] {\, $a_{r,z}^1,a_{r,z}^2,a_{r,z}^3,a_{r,z}^4$} ($(z)+(0.6,0)$);

  \draw[->, shorten <= 23 pt, shorten >= 23 pt, >=stealth', semithick]
  ($(y)+(0,0.07)$) -> node[auto] {$a_{y,z}^1$} ($(z)+(0,0.07)$);
  \draw[->, shorten <= 23 pt, shorten >= 23 pt, >=stealth', semithick]
  ($(y)+(0,-0.07)$) -> node[auto,swap] {$a_{y,z}^2$} ($(z)+(0,-0.07)$);

  \draw[to, loop below] (x) to node[auto] {$a_x$} (x);
  \draw[to, loop below] (y) to node[auto] {$a_y$} (y);
  \draw[to, loop below] (z) to node[auto] {$a_z$} (z);

\end{tikzpicture}
\end{center}
\caption[Construction of a proper communication graph III.]{Left Krieger cover of the shift space $X$ considered in
  Example \ref{ex_construction}. 
  The structure of the irreducible component corresponding to the left
  Fischer cover has been suppressed. 
}
\label{fig_ex_construction_K}
\end{figure}  

In \cite{bates_eilers_pask}, it was also remarked that an 
invariant analogous to the one discussed in Proposition
\ref{prop_range_invariant} is obtained by considering the proper
communication graph of the right Krieger  cover. The following example shows that the two invariants may carry different information.

\begin{figure}
\begin{center}
\begin{tikzpicture}
  [bend angle=45,
   knude/.style = {circle, inner sep = 0pt},
   vertex/.style = {circle, draw, minimum size = 1 mm, inner sep =
      0pt},
   textVertex/.style = {rectangle, draw, minimum size = 6 mm, inner sep =
      1pt},
   to/.style = {->, shorten <= 1 pt, >=stealth', semithick}]
  
  \matrix[row sep=10mm, column sep=15mm]{
    & \node[textVertex] (u) {$u$}; &  \\
    \node[textVertex] (v) {$v$}; & \node[textVertex] (w) {$w$}; 
                             & \node[textVertex] (x) {$x$};  \\
    & \node[textVertex] (y) {$y$}; & \\
  };
  \draw[to, loop left]    (v) to node[auto] {$a'$} (v);
  \draw[to, loop above]   (u) to node[auto] {$a$} (u);
  \draw[to, loop below]   (y) to node[auto] {$a$} (y);
  \draw[to, loop above]   (v) to node[auto] {$a$} (v);
  \draw[to, loop left]    (u) to node[auto] {$a'$} (u);
  \draw[to, loop left]    (y) to node[auto] {$a'$} (y);
  \draw[to, bend right=20] (w) to node[auto,swap] {$b$} (v);
  \draw[to] (x) to node[auto,swap] {$f$} (u);
  \draw[to] (w) to node[auto,swap] {$g$} (x);
  \draw[to] (u) to node[auto,swap] {$e$} (w);
  \draw[to, bend right=20] (v) to node[auto,swap] {$c$} (w);
  \draw[to, bend left=20] (y) to node[auto] {$d$} (w);
  \draw[to] (x) to node[auto] {$f$} (y);
  \draw[to, bend left=20] (w) to node[auto] {$b$} (y);
\end{tikzpicture}
\end{center}
\caption[Irreducible left Krieger cover - reducible right Krieger cover I.]{Left Fischer cover of the irreducible sofic shift $X$ discussed
  in Example \ref{ex_2_invariants}.} 
\label{fig_2_invariants_1}
\end{figure}
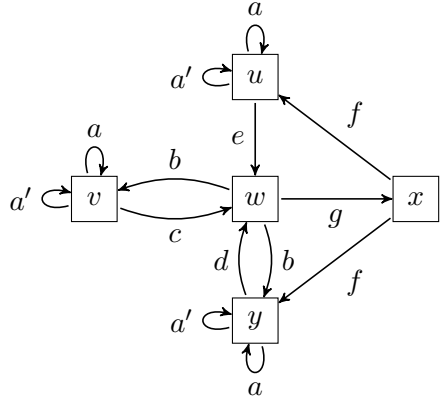
\begin{figure}[t!] 
\begin{center}
\begin{tikzpicture}
  [bend angle=45,
   knude/.style = {circle, inner sep = 0pt},
   vertex/.style = {circle, draw, minimum size = 1 mm, inner sep =
      0pt},
   textVertex/.style = {rectangle, draw, minimum size = 6 mm, inner sep =
      1pt},
   to/.style = {->, shorten <= 1 pt, >=stealth', semithick}]
  
  \matrix[row sep=10mm, column sep=15mm]{
    & \node[textVertex] (u') {$u'$}; &  \\
    \node[textVertex] (v') {$v'$}; & \node[textVertex] (w') {$w'$}; 
                             & \node[textVertex] (x') {$x'$};  \\
  };
  \draw[to, loop left]    (v') to node[auto] {$a'$} (v');
  \draw[to, loop above]   (u') to node[auto] {$a$} (u');
  \draw[to, loop above]   (v') to node[auto] {$a$} (v');
  \draw[to, loop left]    (u') to node[auto] {$a'$} (u');
  \draw[to, bend right=30] (w') to node[auto,swap] {$b$} (v');
  \draw[to] (x') to node[auto,swap] {$f$} (u');
  \draw[to] (w') to node[auto,swap] {$g$} (x');
  \draw[to, bend right=20] (u') to node[auto,swap] {$e$} (w');
  \draw[to, bend left=20] (u') to node[auto] {$d$} (w');
  \draw[to, bend right=30] (v') to node[auto,swap] {$d$} (w');
  \draw[to, bend right=10] (v') to node[auto] {$c$} (w');
\end{tikzpicture}
\end{center}
\caption[Irreducible left Krieger cover - reducible right Krieger cover II.]{Right Fischer cover of the irreducible sofic shift $X$ discussed
  in Example \ref{ex_2_invariants}.}
\label{fig_2_invariants_2}
\end{figure}
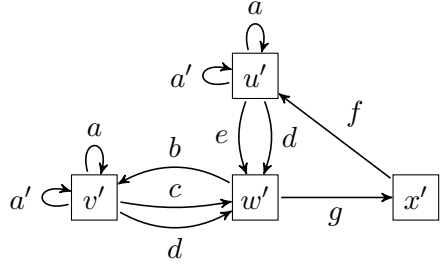

\begin{example}
\label{ex_2_invariants}
The labelled graph in Figure \ref{fig_2_invariants_1} is 
left-resolving, irreducible, and predecessor-separated, so  by Theorem \ref{thm_lfc_char}, it is the
left Fischer 
cover of an irreducible sofic shift. Similarly, the labelled graph in
Figure \ref{fig_2_invariants_2} is irreducible, right-resolving and
follower-separated, so it is the right Fischer cover of an irreducible
sofic shift. By considering the edges labelled $d$, it is easy to see
that the two graphs present the same sofic shift space $X$.  

Every right-ray which contains a letter different from $a$ or $a'$ is
intrinsically synchronizing, so consider a right-ray $x^+ \in 
X^+$ such that $(x^+)_i \in \{a, a'\}$ for all $i \in \N$.
By considering Figure \ref{fig_2_invariants_1}, it is clear that
$P_\infty(x^+) = P_\infty(u) \cup P_\infty(v) \cup 
P_\infty(y) = P_\infty(y)$, so $P(x^+)$ is also in the first layer of the left Krieger cover. Hence, the proper communication graph has only one vertex and no edges.

Every left-ray containing a letter different from $a$ or $a'$ is
intrinsically synchronizing, so consider the left-ray $a^\infty \in
X^-$. Figure \ref{fig_2_invariants_2} shows that $F_\infty(a^\infty)
= F_\infty(u') \cup F_\infty(v')$ and that no single 
vertex $y'$ in the right Fischer cover has $F_\infty(y') =
F_\infty(a^\infty)$, so there is a vertex in the second layer
of the right Krieger cover. Hence, 
the 
proper
communication graph is non-trivial. 
\end{example}

\section{$\Cs$-Algebras associated to sofic shift spaces}
\label{sec_cs} 
\index{Cuntz-Krieger algebra}\index{C*@$\Cs(E)$}
\index{O@$\OO_X$}\index{O@$\OO_{X^\ast}$}
\index{Ca@$\Cs$-algebra!associated to shift space}
\index{O@$\OO_{X^+}$}
Cuntz and Krieger \cite{cuntz_krieger} introduced a class of
$\Cs$-algebras which can naturally be viewed as the universal
$\Cs$-algebras associated to shifts of finite type.
This was generalised by Matsumoto \cite{matsumoto_1997} who associated
two $\Cs$-algebras $\OO_X$ and $\OO_{X^\ast}$ to every shift space
$X$, and these Matsumoto algebras have been studied intensely 
\cite{carlsen, katayama_matsumoto_watatani,
  matsumoto_1997, matsumoto_1998, 
  matsumoto_1999_dimension_group,matsumoto_1999_relations,
matsumoto_1999_simple,
  matsumoto_2000_automorphisms,matsumoto_2000_stabilized,
  matsumoto_2001, matsumoto_2002, matsumoto_watatani_yoshida}.
The two Matsumoto algebras $\OO_X$ and $\OO_{X^\ast}$ are generated
by elements satisfying the same relations, but they are not isomorphic
in general \cite{carlsen_matsumoto}. 
This presentation will follow the approach of Carlsen in \cite{carlsen_2008}
where a universal $\Cs$-algebra $\OO_{\tilde X}$ is associated to every
one-sided shift space $\tilde X$. 
This 
also gives a way to associate $\Cs$-algebras to every two-sided
shift since a two-sided shift $X$ corresponds to two one-sided 
shifts $X^+$ and $X^-$.

\subsection{Ideal lattices}
\index{Ca@$\Cs$-algebra!associated to shift space!ideals of}
Let $X$ be a sofic shift space and let $\OO_{X^+}$ be the universal
$\Cs$-algebra associated to the one-sided shift $X^+$ as defined in
\cite{carlsen_2008}. 
Carlsen proved that $\OO_{X^+}$ is isomorphic to the Cuntz-Krieger
algebra of the left Krieger cover of $X$ \cite{carlsen}, so the
lattice of gauge invariant ideals in $\OO_{X^+}$ is given by the
proper communication graph of the left Krieger cover of $X$
\cite{bates_pask_raeburn_szymanski,kumjian_pask_raeburn_renault}, and all ideals are given in this way if the left Krieger cover satisfies Condition (K) \cite[Theorem 4.9]{raeburn}.
Hence, Proposition \ref{prop_layers} and Theorem \ref{thm_gfc_foundation_lkc} can be used to investigate the ideal
lattice of $\OO_{X^+}$. For a reducible sofic shift, a part of the ideal lattice is given by the structure of the generalised left Fischer cover, which is reducible, but if $X$ is an irreducible sofic shift, and the left Krieger cover of
$X$ satisfies Condition (K), then the fact that the left Krieger cover
has a unique top component implies that $\OO_{X^+}$ will always have
a unique maximal ideal.
The following proposition shows that all these lattices can be
realised.   

\begin{prop}
\label{prop_ideal_lattice_of_irreducible_shift}
Any finite lattice of ideals with a unique maximal ideal is the ideal
lattice of the universal $\Cs$-algebra $\OO_{X^+}$ associated to an
AFT shift $X$.   
\end{prop}

\begin{proof}
Let $E$ be a finite directed graph without circuits and with a unique
maximal vertex.
Consider the following slight modification of the algorithm from the
proof of Proposition \ref{prop_range_invariant}. 
For each $v \in E$, draw two loops of length 1 at each vertex $v_1,
\ldots, v_{n(v)}$ associated to $v$: 
One labelled $a_v$ and one labelled $a_v'$. The rest of the
construction is as before.
Let $(K, \LL_K)$ be the left Krieger cover of the corresponding
sofic shift. As before, the proper communication graph of $K$ is given by $E$, and now $(K, \LL_K)$
satisfies Condition (K),
so there is a bijective correspondence between the hereditary
subsets of $E^0$ and the ideals of $\Cs(K) \cong \OO_{X^+}$.
Since $E$ was arbitrary, any finite ideal lattice with a unique maximal ideal can be obtained in this way.  
\end{proof}

\subsection{The $\Cs$-algebras $\OO_{X^+}$ and $\OO_{X^-}$}
Every two-sided shift space $X$ corresponds to two one-sided shift
spaces $X^+$ and $X^-$, and this gives two natural ways to associate a
universal $\Cs$-algebra to $X$.
The next goal is to show that these two $\Cs$-algebras may carry different information about the shift space.
Let $\OO_{X^-}$ be the universal $\Cs$-algebra associated to the
one-sided shift space $(X^\textrm{T})^+$ as defined in
\cite{carlsen_2008}. The left Krieger cover of $X^\textrm{T}$ is the
transpose of the right Krieger cover of $X$, so by \cite{carlsen},
$\OO_{X^-}$ is isomorphic to the Cuntz-Krieger algebra of the
transpose of the right Krieger cover of $X$.

\begin{example}
\label{ex_+-_not_isomorphic}
\index{Condition (K)}
Let $X$ be the sofic shift from Example \ref{ex_2_invariants}. Note
that the left and right Krieger covers of $X$ both satisfy Condition
(K) from \cite{raeburn}, so the corresponding proper communication graphs
completely determine the ideal lattices of $\OO_{X^+}$ and $\OO_{X^-}$.   
The proper communication graph of the left Krieger cover $(K, \LL_K)$ of $X$ is
trivial, so $\OO_{X^+}$ is simple, while there are precisely two
vertices in the proper communication graph of the right Krieger
cover of $X$, so there is exactly one non-trivial ideal in $\OO_{X^-}$.
In particular, $\OO_{X^+}$ and $\OO_{X^-}$ are not isomorphic.

Consider the edge shift $Y = \X_{K}$. This is an SFT, and the left and
right Krieger covers of $Y$ are both $(K, \LL_{\Id})$,  where
$\LL_{\Id}$ is the identity map on the edge set $K^1$. By \cite{carlsen},
$\OO_{X^+}$ and  $\OO_{Y^+}$ are isomorphic to $\Cs (K)$.
Similarly, $\OO_{Y^-}$ is isomorphic to $\Cs (K^\textrm T)$ and
$K^\textrm T$ is an irreducible graph satisfying Condition (K), so
$\OO_{Y^-}$ is simple.
In particular, $\OO_{Y^-}$ is not isomorphic to $\OO_{X^-}$.  
This shows that the $\Cs$-algebras
associated to $X^+$ and $X^-$ are not always isomorphic, and that
there can exist a shift space $Y$ such that $\OO_{Y^+}$ is
isomorphic to $\OO_{X^+}$ while $\OO_{Y^-}$ is not isomorphic to
$\OO_{X^-}$. 
\end{example}

\subsection{An investigation of Condition ($\ast$)}
\index{Condition ($\ast$)}
In \cite{carlsen_matsumoto}, two $\Cs$-algebras $\OO_X$ and
$\OO_{X^\ast}$ are associated to every two-sided shift space $X$.
The $\Cs$-algebras $\OO_{X}$, $\OO_{X^\ast}$, and $\OO_{X^+}$  
are generated by partial isometries satisfying the same
relations. Unlike $\OO_{X}$, however, $\OO_{X^+}$ is always universal \cite{carlsen_2008}.  
In \cite{carlsen_matsumoto}, it is proved that
$\OO_X$ and $\OO_{X^*}$ are isomorphic when $X$ satisfies a
condition called Condition ($\ast$). 
The example from \cite[Section 4]{carlsen_matsumoto} shows that not
all sofic shift spaces satisfy this condition by constructing a
sofic shift where the left Krieger and the past set cover are
not isomorphic. 
The final result of this chapter further clarifies the relationship between Condition ($\ast$) and the structure of the left Krieger and the past set covers.

\index{l-past@$l$-past equivalence}
For each $l \in \N$ and $w \in \BB(X)$ define 
$P_l(w) = \{ v \in \BB(X) \mid vw \in \BB(X), |v| \leq l \}$. Two
words $v, w \in \BB(X)$ are said to be \emph{$l$-past equivalent} if $P_l(v)
= P_l(w)$. For $x^+ \in X^+$, $P_l(x^+)$ and $l$-past equivalence are
defined analogously.  

\begin{cond}[$\ast$]
For every $l \in \N$ and every infinite $F \subseteq \BB(X)$ such
that $P_l(u) = P_l(v)$ for all $u,v \in F$ there exists $x^+ \in X^+$ such
that $P_l(w) = P_l(x^+)$ for all $w \in F$.
\end{cond}

The result in the following lemma is well known, but a proof is included for completeness.

\begin{lem}
\label{lem_m-past} 
When $X$ is a sofic shift there exists $m \in \N$ such that for all $x_1,x_2 \in X^+ \cup \BB(X)$, $P_m(x_1) = P_m(x_2)$ if and only if $P_\infty(x_1) = P_\infty(x_2)$. 
\end{lem}

\begin{proof}
If $P_\infty(x_1) = P_\infty(x_2)$, then clearly $P_m(x_1) = P_m(x_2)$ for all $m \in \N$. Given $u,v \in \BB(X)$ with $y^- \in P_\infty(u) \setminus P_\infty(v)$, there must exist $w \in \BB(X)$ such that $w$ is a suffix of $y^-$ but not a suffix of any element of $P_\infty(v)$. In particular, it is possible to distinguish the two predecessor sets just by considering the suffixes of length $|w|$.
Define $m$ to be the maximum length needed to distinguish any two predecessor sets of words. This is well-defined since there are only finitely many different predecessor sets of words when $X$ is sofic. By Lemma \ref{lem_prefix_with_same_past}, every predecessor set of a right-ray is also the predecessor set of a word, so this finishes the proof. 
\end{proof}

\begin{lem}
\label{lem_infinitely_many_w}
A vertex $P$ in the past set cover of a sofic shift $X$ is in an essential subgraph
if and only if there exist infinitely many $w \in \BB(X)$ such that
$P_\infty(w) = P$. 
\end{lem}

\begin{proof}
Let $P$ be a vertex in an essential subgraph of the past set cover of
$X$, and let $x^+ \in  X^+$ be a right-ray with a presentation starting at $P$. Given $n \in \N$, there exists $w_n \in \BB(X)$ such that $P = P_\infty(x_1 x_2 \ldots x_n w_n)$. To prove the converse, let $P$ be a vertex in the past set cover for which there exist infinitely many $w \in \BB(X)$ such that $P = P_\infty(w)$. For each $w$, there is a path labelled $w_{[1,\rvert w \lvert -1]}$ starting at $P$. There are no sources in the past set cover, so this implies that $P$ is not stranded.
\end{proof}

\begin{prop}
\label{prop_condition_star}
A sofic shift $X$ satisfies Condition ($\ast$) if and only if the
left Krieger cover is the maximal essential subgraph of the past set
cover. 
\end{prop}

\begin{proof}
Assume that $X$ satisfies Condition ($\ast$). Let $P$ be a vertex in an essential subgraph of the past set cover and define $F = \{ w \in
\BB(X) \mid P_\infty(w) = P \}$. 
Use Lemma \ref{lem_m-past} to 
choose $m \in \N$ such that for all $x,y
\in \BB(X) \cup X^+$, $P_\infty(x) = P_\infty(y)$ if and only if
$P_m(x) = P_m(y)$. 
By Lemma \ref{lem_infinitely_many_w}, $F$ is an infinite set, so
Condition ($\ast$) can be used to choose $x^+ \in X^+$ such that
$P_m(x^+) = P_m(w)$ for all $w \in F$. By the choice of $m$, this means
that $P_\infty(x^+) =  P_\infty(w) = P$ for all $w \in
F$, so $P$ is a vertex in the left Krieger cover.  

To prove the other implication, assume that the left Krieger cover is
the maximal essential subgraph of the past set cover. Let $l \in \N$
be given, and consider an infinite set $F \subseteq \BB(X)$ for which
$P_l(u) = P_l(v)$ for all $u,v \in F$. Since $X$ is sofic, there are
only finitely many different predecessor sets, so there must exist $w
\in F$ such that $P_\infty(w) = P_\infty(v)$ for infinitely many $v
\in F$. By Lemma \ref{lem_infinitely_many_w}, this proves that $P =
P_\infty(w)$ is a vertex in the maximal essential subgraph of the past
set cover. By assumption, this means that it is a vertex in the left
Krieger cover, so there exists $x^+ \in X^+$ such that $P_\infty(w) =
P_\infty(x^+)$. In particular, $P_l(x^+) = P_l(w) = P_l(v)$ for all $v
\in F$, so Condition ($\ast$) is satisfied.
\end{proof}

\noindent
In \cite{bates_pask}, it was proved that $\OO_{X^\ast}$ is isomorphic
to the Cuntz-Krieger algebra of the past set cover of $X$ when $X$
satisfies a condition called Condition (I).
According to Carlsen \cite{carlsen_privat}, a proof similar to the
proof which shows that $\OO_{X^+}$ is isomorphic to the Cuntz-Krieger
algebra of the left Krieger cover of $X$
should prove
that $\OO_{X^\ast}$ is isomorphic to the
Cuntz-Krieger algebra of the subgraph of the past set cover of $X$
induced by the vertices $P$ for which there exist
infinitely many words $w$ such that $P_\infty(w) = P$.
Using Lemma \ref{lem_infinitely_many_w}, this shows that
$\OO_{X^\ast}$ is always isomorphic to the Cuntz-Krieger algebra of
the maximal essential subgraph of the past set cover of $X$.

\section{Perspectives}

The results about the structure of the left Krieger cover and the past set cover developed in Section \ref{sec_foundation} have proved to be useful for constructing sofic shifts with specific properties both in the computation of the range of the flow invariant in Section \ref{sec_invariant} and in the construction of the examples used in Section \ref{sec_cs}. In the following, such arguments about layers will also be used to prove that certain Krieger covers are irreducible.

The existence of the generalised left Fischer cover is interesting because it shows how -- and to what extend -- it is possible to extend the left Fischer cover to reducible sofic shifts. It is mainly useful as a means to provide structure to the left Krieger cover via the layers since there is no way to construct it directly. Indeed, it would be interesting to have an algorithm for the construction of the generalised left Fischer cover that did not rely on a existing construction of the Krieger cover, but it seems unlikely that this can be done, since the definition of the generalised left Fischer cover relies not only on the set of indecomposable predecessor sets, but also on how these predecessor sets sit as vertices in the left Krieger cover. However, if one has the information necessary to construct the Krieger cover, then one also has all the information needed to construct the generalised left Fischer cover, so the generalised left Fischer cover is at least not harder to construct.

%% file: thesis_beta.tex
\label{chap_beta}

Boyle \cite{boyle_personal} has conjectured that irreducible sofic shifts where the covering map of the left or right Fischer cover is $2$ to $1$ can be classified up to flow equivalence by an induced SFT -- called the fiber product cover -- with a $\Z / 2\Z$ action. The present work was motivated by a desire to apply this result to a concrete class of sofic shifts called beta-shifts.

Section \ref{sec_beta_introduction} gives an introduction to the basic definitions and properties of beta-shifts. In Section \ref{sec_beta_covers}, the right Fischer covers of sofic  beta-shifts are determined, and the covering map is shown to be $2$ to $1$ if the shift is strictly sofic. This result is used to show that the right Krieger cover is identical to the right Fischer cover and to construct 
the right fiber product cover. Section \ref{sec_beta_classification} concerns the flow classification of beta-shifts. It is shown that for every $\beta > 1$, there exists $1 < \beta' < 2$ such that the two beta-shifts are flow equivalent and such that the new generating sequence has a form with certain properties. Additionally, the Bowen-Franks groups of the covers considered above are computed.  Finally, the flow class of the fiber product cover of a sofic beta-shift is shown to depend only on a single integer. With the results conjectured by Boyle \cite{boyle_personal}, this will give a complete flow classification of sofic beta-shifts.

\section{Beta-shifts}
\label{sec_beta_introduction}\index{b@$\beta$-expansion} 
\index{generating sequence}\index{b@$\beta$-expansion!finite}\index{b@$\beta$-expansion!eventually periodic}\index{Xbeta@$\X_\beta$}
\index{e*beta@$e(\beta)$}
\index{g*beta@$g(\beta)$}
Here, a short introduction to the basic definitions and properties of beta-shifts are given. For a more detailed treatment of beta-shifts, see \cite{blanchard_beta}.
Let $\beta \in \R$ with $\beta > 1$. For each $t \in [0;1]$ define sequences $(r_n(t))_{n \in \N}$ and $(x_n(t))_{n \in \N}$ by 
\begin{align*}
     r_1(t)  &=  \fracpart{ \beta t }   ,   &r_n(t) = \fracpart{ \beta r_{n-1}(t) }, \\
     x_1(t) &=  \intpart{ \beta t }   ,     &x_n(t) = \intpart{ \beta r_{n-1}(t) },
\end{align*}
where $\intpart y $ is the integer part -- and $\fracpart y$ is the fractional part -- of $y \in \R$. The sequence $x_1(t)x_2(t) \cdots$ is said to be the $\beta$-\emph{expansion} of $t$, and R{\'e}nyi \cite{renyi} has proved that  
\begin{displaymath}
  t = \sum_{n=1}^\infty \frac{ x_n(t) }{ \beta^n }.
\end{displaymath} 
The $\beta$-expansion of $1$ is denoted $e(\beta)$. As an example, consider $\beta = (1+\sqrt 5)/2$ where the $\beta$-expansion of $1$ is $e(\beta) = 11000\cdots$. 
The $\beta$-expansion of $t$ is said to be \emph{finite} if there exists $N \in \N$ such that $x_n(t) = 0$ for all $n \geq N$. It is said to be \emph{eventually periodic} if there exist $N,k \in \N$ such that $x_{n+k}(t) = x_n(t)$ for all $n \geq N$.
Define the \emph{generating sequence of} $\beta$ to be
\begin{displaymath}
g(\beta) = \left \{ \begin{array}{l c l}
        e(\beta)                     & , & e(\beta) \textrm{ is infinite } \\
        (a_1 a_2 \cdots a_{k-1}(a_k-1))^\infty   & , & e(\beta) = a_1  \cdots a_k00\cdots \textrm{ and } a_k \neq 0
      \end{array} \right. .
\end{displaymath}

\index{beta-shift}
\index{golden mean shift!as beta-shift}
\index{beta-shift!order on}
Define $\BB_\beta = \{ x_n(t) \cdots x_m(t) \mid 1 \leq n \leq m, t \in [0;1] \}$. It is easy to check that $\BB_\beta$ is the language of a shift space $\X_\beta$. Such a shift space is called a \emph{beta-shift}. The alphabet $\AA_\beta$ of $\X_\beta$ is either $\{ 0, \ldots, \beta -1\}$ or $ \{0 , \ldots , \intpart \beta \}$ depending on whether $\beta$ is an integer or not. If $\beta \in \N$, then $\X_\beta$ is the full $\{0, \ldots, \beta-1\}$-shift. For $\beta = (1+\sqrt 5)/2$ as considered above, $\X_\beta$ is conjugate to the golden mean shift from Example \ref{ex_golden_mean}.

The words in $\BB(\X_\beta)$ and right-rays in $\X_\beta^+$ are ordered by lexicographical order $\leq$, and the following two theorems use this order to give fundamental descriptions of beta-shifts.

\begin{thm}[{R{\'e}nyi \cite{renyi}}] 
\label{thm_renyi}
Let $\beta >1$, let $g(\beta)$ be the generating sequence, let $\AA_\beta$ be the alphabet of $\X_\beta$, and let $x^+ = x_1 x_2 \ldots \in \AA_\beta^\N$. Then $x^+ \in \X_\beta^+$ if and only if 
$x_k x_{k+1} \cdots \leq g(\beta)$ for all $k \in \N$.
\end{thm}

\begin{thm}[{Parry \cite{parry}}]
\label{thm_beta_existence}
A sequence $a_1 a_2 \cdots$ is the $\beta$-expansion of $1$ for some $\beta > 1$ if and only if $a_k a_{k+1} \cdots < a_1 a_2 \cdots$ for all $k \in \N$. Such a $\beta$ is uniquely given by the expansion of $1$.
\end{thm}

\noindent
This shows that the generating sequence is periodic if and only if the $\beta$-expansion of $1$ is finite.

Beta-shifts provide a link between symbolic dynamics and number theory, and it is natural to explore how the number theoretical properties of $\beta$ influence the dynamical properties of $\X_\beta$. The following two propositions give results of this kind. For detailed treatments of the number theoretical aspects of beta-expansions, see \cite{schmidt} and \cite{bertrand-mathis}.

\index{Pisot number}
\begin{prop}[{Parry \cite{parry}}]
If $\beta > 1$ is a Pisot number (i.e.\ an algebraic integer for which all  conjugates have absolute value less than 1), then $\X_\beta$ is sofic.
\end{prop}

\begin{prop}[{Denker et al. \cite{denker_grillenberger_sigmund}, Lind \cite{lind_beta}}]
Let $\beta > 1$. If $\X_\beta$ is sofic, then $\beta$ is a Perron number.
\end{prop}

\noindent
As a consequence of this, $\X_\beta$ is not sofic if $\beta$ is a non-integer rational number.

\section{Covers of beta-shifts}
\label{sec_beta_covers}

\index{beta-shift!loop graph of}
Let $\beta >1$, let $g(\beta) = g_1 g_2 \cdots$ be the generating sequence of $\X_\beta$, and define an infinite labelled graph $\GG_\beta = (G_\beta, \LL_\beta)$ with  vertices $G_\beta^0 = \{ v_i \mid i \in \N \}$ and edges $G^1 = \{ e_i^{k} \mid i \in \N, 0 \leq k \leq g_i \}$ such that  
\begin{displaymath}
s(e_i^k) = v_i, \qquad
 r(e_i^k) = \left \{ \begin{array}{l c l}
      v_{i+1} & ,  & k = g_i \\
      v_1       & , & k < g_i
\end{array} \right.,
\qquad \textrm{and} \qquad
\LL_\beta(e_i^k) = k
\end{displaymath}
for all $i \in \N$ and $0 \leq k \leq g_i$.
It is easy to check that this is a right-resolving presentation of $\X_\beta$.
The labelled graph $(G_\beta, \LL_\beta)$ is called the \emph{standard loop graph presentation} of $\X_\beta$; it is shown in  Figure \ref{fig_beta_loop}. Note that the structure of the standard loop graph shows that every beta-shift is irreducible. 

\begin{figure}
\begin{center}
\begin{tikzpicture}
  [bend angle=10,
   clearRound/.style = {circle, inner sep = 0pt, minimum size = 17mm},
   clear/.style = {rectangle, minimum width = 10 mm, minimum height = 6 mm, inner sep = 0pt},  
   greyRound/.style = {circle, draw, minimum size = 1 mm, inner sep =
      0pt, fill=black!10},
   grey/.style = {rectangle, draw, minimum size = 6 mm, inner sep =
      1pt, fill=black!10},
   white/.style = {rectangle, draw, minimum size = 6 mm, inner sep =
      1pt},
   to/.style = {->, shorten <= 1 pt, >=stealth', semithick}]
  
  \node[white] (P1) at (0,0) {};
  \node[white] (P2) at (3,0) {};   
  \node[white] (P3) at (6,0) {};
  \node[clear] (P4) at (9,0) {$\cdots$};

  \draw[to] (P1) to node[auto] {$g_1$} (P2); 
  \draw[to] (P2) to node[auto] {$g_2$} (P3);
  \draw[to] (P3) to node[auto] {$g_3$} (P4);

  \draw[to, loop left] (P1) to node[auto] {$0, \ldots , g_1-1$} (P1); 
  \draw[to, bend left = 30] (P2) to node[auto] {\qquad $0, \ldots , g_2-1$} (P1);     
  \draw[to, bend left = 60] (P3) to node[auto] {$0, \ldots , g_3-1$} (P1); 
  
\end{tikzpicture}
\end{center}
\caption[Standard loop graphs of beta-shifts.]{The standard loop graph of a beta-shift with generating sequence $g(\beta) = (g_n)_{n\in \N}$.} 
\label{fig_beta_loop}
\end{figure}
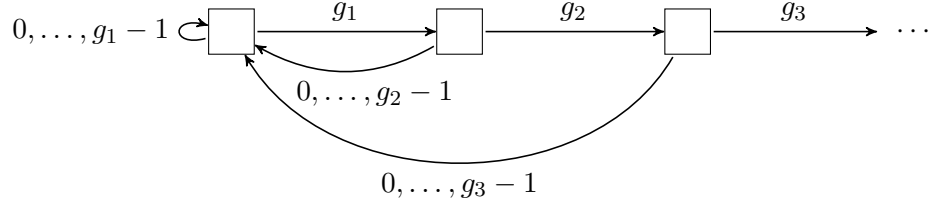

\subsection{Fischer cover}

Parry \cite{parry} proved that $\X_\beta$ is an SFT if and only if the generating sequence is periodic, and Betrand-Mathis \cite{bertrand-mathis_preprint} proved that $\X_\beta$ is sofic if and only if the generating sequence is eventually periodic. The latter result is apparently only available in a preprint, so a proof is given here for completeness.

\index{beta-shift!sofic}\index{beta-shift!right Fischer cover of}
\begin{prop}
\label{prop_beta_sofic}
Given $\beta > 1$, the beta-shift $\X_\beta$ is sofic if and only if the generating sequence $g(\beta)$ is eventually periodic. For minimal $n,p \in \N$ with $g(\beta) = g_1 \cdots g_n (g_{n+1} \cdots g_{n+p})^\infty$, the right Fischer cover of $\X_\beta$ is the labelled graph $(F_\beta, \LL_\beta)$ shown in Figure \ref{fig_beta_rfc} which has $F_\beta^0 = \{ v_1, \ldots v_{n+p} \}$, $F_\beta^1 = \{ e_i^k \mid 1 \leq i \leq n+p, 0 \leq k \leq g_i \}$, and 
\begin{displaymath}
s(e_i^k) = v_i, \textrm{  } 
 r(e_i^k) = \left \{ \begin{array}{l c l}
      v_1       & ,  & k < g_i \\ 
      v_{i+1}  & ,  & k = g_i , i < n+p \\
      v_{n+1} & ,  & k = g_{n+p} , i = n+p      
\end{array} \right.,
\textrm{ and } 
\LL_\beta(e_i^k) = k. 
\end{displaymath}
\end{prop}

\noindent
Note that the right Fischer cover $(F_\beta, \LL_\beta )$ of a sofic beta-shift with generating sequence $g(\beta) = g_1 \cdots g_n (g_{n+1} \cdots g_{n+p})^\infty$ is the labelled graph obtained from the subgraph of the standard loop graph $\GG_\beta$ induced by the first $n+p$ vertices $v_1, \ldots , v_{n+p}$ of $G_\beta^0$ by adding an additional edge labelled $g_{n+p}$ from $v_{n+p}$ to $v_{n+1}$. 

\begin{figure}
\begin{center}
\begin{tikzpicture}
  [bend angle=10,
   clearRound/.style = {circle, inner sep = 0pt, minimum size = 17mm},
   clear/.style = {rectangle, minimum width = 10 mm, minimum height = 6 mm, inner sep = 0pt},  
   greyRound/.style = {circle, draw, minimum size = 1 mm, inner sep =
      0pt, fill=black!10},
   grey/.style = {rectangle, draw, minimum size = 6 mm, inner sep =
      1pt, fill=black!10},
   white/.style = {rectangle, draw, minimum size = 6 mm, inner sep =
      1pt},
   to/.style = {->, shorten <= 1 pt, >=stealth', semithick}]
  
  \node[white] (v1)     at (0,0) {$v_1$};
  \node[white] (v2)     at (2,0) {$v_2$};
  \node[clear] (dots)   at (4,0) {$\cdots$};
  \node[white] (vn+1)     at (6,0) {$v_{n+1}$};
  \node[clear] (dots2) at (8,0) {$\cdots$}; 
  \node[white] (vn+p)     at (10,0) {$v_{n+p}$};

  \draw[to] (v1) to node[auto] {$g_1$} (v2); 
  \draw[to] (v2) to node[auto] {$g_2$} (dots); 
  \draw[to] (dots) to node[auto] {$g_n$} (vn+1);
  \draw[to] (vn+1) to node[auto] {$g_{n+1}$} (dots2);
  \draw[to] (dots2) to node[auto] {$g_{n+p-1}$} (vn+p);  

  \draw[to, loop above] (v1) to node[auto] {$0, \ldots , g_1-1$} (v1); 
  \draw[to, bend left = 30] (v2) to node[auto] {\quad \qquad \qquad $0, \ldots , g_{2}-1$} (v1);    
  \draw[to, bend left = 50] (vn+1) to node[auto] {\qquad $0, \ldots , g_{n+1}-1$} (v1);     
  \draw[to, bend left = 70] (vn+p) to node[auto] {$0, \ldots , g_{n+p}-1$} (v1);     
  \draw[to, bend right = 45] (vn+p) to node[auto,swap] {$g_{n+p}$} (vn+1);     
  
\end{tikzpicture}
\end{center}
\caption[Right Fischer covers of sofic beta-shifts.]{Right Fischer cover of a sofic beta-shift with generating sequence $g(\beta) = g_1 \cdots g_n (g_{n+1} \cdots g_{n+p})^\infty$ for minimal $n,p$.} 
\label{fig_beta_rfc}
\end{figure}
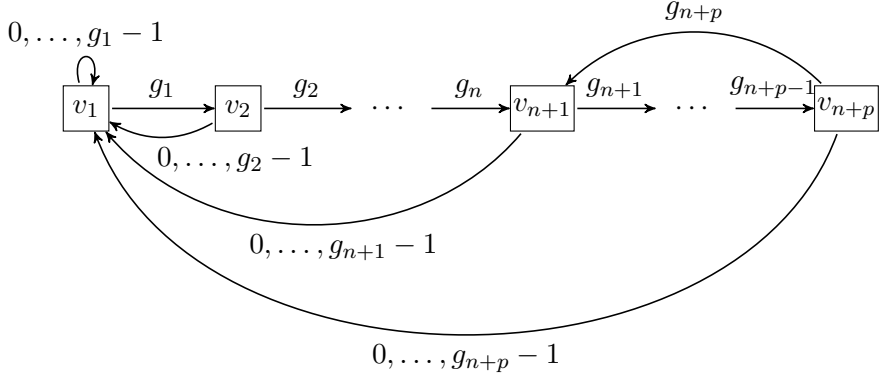

\begin{proof}[Proof of Proposition \ref{prop_beta_sofic}]
Assume that $g(\beta)$ is not eventually periodic and let $v,w$ be finite prefixes of $g(\beta)$ with $v \neq w$. Choose $x_v^+,x_w^+ \in \X_\beta^+$ such that $vx_v^+ = g(\beta) = wx_w^+$.
Since $g(\beta)$ is not eventually periodic, $x_v^+ \neq x_w^+$. Assume without loss of generality that $x_v^+ > x_w^+$. Then $w x_v^+ > wx_w^+ = g(\beta)$, so by Theorem \ref{thm_renyi}, $w x_v^+ \notin \X_\beta^+$. Hence, $P_\infty(x_v^+) \neq P_\infty(x_w^+)$. Since $g(\beta)$ is not eventually periodic, this proves that $\X_\beta$ has infinitely many different predecessor sets, so it is not a sofic shift (see Section \ref{sec_shift_lkc}).
The standard loop graph is a right-resolving presentation of $\X_\beta$, so it follows from the observation preceding this proof that $(F_\beta, \LL_\beta )$ is so as well. It is easy to check that $(F_\beta, \LL_\beta)$ is also follower-separated, so it is the right Fischer cover of $\X_\beta$ by Theorem \ref{thm_lfc_char}. 
\end{proof}

\begin{lem}[Johnson {\cite[Proposition 2.5.1]{johnson_thesis}}]
\label{lem_beta_w_unique_r}
Let $\beta > 1$, let $\X_\beta$ be sofic, and let $(F_\beta, \LL_\beta)$ be the right Fischer cover of $\X_\beta$. If there is a path $\lambda \in F_\beta^*$ with $s(\lambda) = v_1$ such that $\LL_\beta(\lambda) = w$ is not a factor of $g(\beta)$, then every path in $F_\beta^*$ labelled $w$ terminates at $r(\lambda)$.
\end{lem}

\noindent
Some cases are left unchecked in the proof in \cite{johnson_thesis}, so a slightly altered proof is given here for completeness.

\begin{proof}
Let $F_\beta^0 = \{ v_1 , \ldots , v_{n+p} \}$ as in Proposition \ref{prop_beta_sofic}.
Let $\mu \in F_\beta^*$ with $\LL_\beta(\mu) = w$. If $s(\mu) = s(\lambda)$, then $\lambda = \mu$ since $(F_\beta, \LL_\beta)$ is right-resolving, so assume that $s(\mu) = v_k$ for some $k > 1$. Since $w$ is not a factor of $g(\beta)$, $\mu$ must pass through $v_1$, so it is possible to choose minimal $1 \leq j \leq \lvert w \rvert$ such that $r(\mu_j) = v_1$.  
Define $0 \leq l \leq j$ by
\begin{displaymath}
l = \left\{ \begin{array}{l c l}
       0 & , & r(\lambda_i) \neq v_1 \textrm{ for all } 1 \leq i \leq j \\
       \max \{ i \mid 1 \leq i \leq j, r(\lambda_i) = v_1 \}  & , & \textrm{otherwise}\\       
\end{array} \right. .
\end{displaymath}
If $l = j$, then $r(\mu_j) = v_1 = r(\lambda_j)$, so $r(\mu) = r(\lambda)$ since $(F_\beta, \LL_\beta)$ is right-resolving.

The result is obvious if $g(\beta)$ is periodic with period $1$, so assume that this is not the case. 
Then there is no edge labelled $g_1$ terminating at $v_1$, so $\LL_\beta(\lambda_j) = \LL_\beta(\mu_j) < g_1$, and if $l = j-1$, then this implies that $r(\lambda_j) = v_1$ in contradiction with the definition of $l$. Assume therefore that $l < j-1$. Then
\begin{displaymath}
 \LL_\beta(\mu_{[l+1,j]}) =  w_{[l+1,j]} 
            = \LL_\beta(\lambda_{[l+1,j]}) = g_1 \cdots g_{j-l-1},
\end{displaymath}
so $g_1 \cdots g_{j-l-1} g_1 g_2 \cdots \in \X_\beta^+$ since $r(\mu_j) = v_1$.
By Theorems \ref{thm_renyi} and \ref{thm_beta_existence}, this is only possible if $g(\beta) = (g_1 \cdots g_{j-l-1})^\infty$. By the construction of $(F_\beta, \LL_\beta)$, every path labelled $g_1 \cdots g_{j-l-1}$ must then start at $v_1$, so $s(\mu_{l+1}) = v_1$ and $r(\mu_l) = v_1$ in contradiction with the construction of $l$ and the assumptions on $\mu$. 
\end{proof}

\index{beta-shift!of finite type}
\begin{prop}
\label{prop_beta_2_to_1}
Let $\beta > 1$ with eventually periodic $g(\beta)$, let $(F_\beta, \LL_\beta)$ be the right Fischer cover  of $\X_\beta$, and let $\pi \colon \X_{F_\beta} \to \X_\beta$ be the covering map. If $g(\beta)$ is periodic, then $\pi$ is $1$ to $1$, and if not, then it is $2$ to $1$.
\end{prop}

\begin{proof}
Let $g(\beta) =  g_1 \cdots g_n (g_{n+1} \cdots g_{n+p})^\infty$ for minimal $n,p \in \N$, and let $w_p = g_{n+1} \cdots g_{n+p}$. The first goal is to prove that if $x \in \X_\beta$ has more than one presentation in $(F_\beta, \LL_\beta)$, then there exits $k \in \Z$ such that $\cdots x_{k-1} x_k = w_p^\infty$. 

Let $x \in \X_\beta$, and let $\lambda$ be a biinfinite path in $F_\beta$ with $\LL_\beta(\lambda) = x$. 
If there is a lower bound $l$ on the set 
\begin{displaymath}
A = \{ j \in \Z \mid \exists i < j : \LL_\beta(\lambda_{[i,j]}) \textrm{ is not a factor of } g(\beta) \},
\end{displaymath}
then there exists $k \leq l$ such that $\cdots x_{k-1} x_k = w_p^\infty$ since $x$ is biinfinite while $g(\beta)$ is not.
By Proposition \ref{prop_beta_sofic},
the only circuit in $(F_\beta, \LL_\beta)$ that does not pass through $v_1$ is labelled $w_p$, so if there is a lower bound $l$ on the set $B = \{ i \in \Z \mid r(\lambda_i) = v_1 \}$, then there exists $k < l$ such that $\cdots x_{k-1} x_k = w_p^\infty$.
Assume that both $A$ and $B$ are unbounded below, let $\mu \in \X_{F_\beta}$ with $\LL(\mu) = x$, and let $k \in \Z$. Then there exist $i < j < k$ with $s(\lambda_i) = v_1$ such that $w = \LL_\beta(\lambda_{[i,j]})$ is not a factor of $g(\beta)$. By Lemma \ref{lem_beta_w_unique_r}, $r(\mu_j) = r(\lambda_j)$, and $(F_\beta, \LL_\beta)$ is right-resolving, so this implies that $\mu_k = \lambda_k$. Since $k$ was arbitrary, $\lambda = \mu$, and this proves the claim.
 
If $g(\beta)$ is periodic, then every circuit in $F_\beta$ passes through $v_1$, and only one of these is labelled $w_p$ since $p$ is the minimal period. Hence, there is precisely one vertex in $F^0$ where a presentation of $w_p^\infty$ can end, so $\pi_\beta$ is $1$ to $1$ and $\X_\beta$ is an SFT.

Assume that $g(\beta)$ is eventually periodic without being periodic. Then $w_p < g(\beta)_{[1,p]}$, so there is a circuit $\mu$ in $F_\beta$ passing through $v_1$ with $\LL_\beta(\mu) = w_p$. Choose $0 \leq i \leq p$ such that $w_i \cdots w_p w_1 \cdots w_{p-1}$ is maximal among the cyclic permutations of the letters of $w_p$. The number $i$ is unique because $p$ is the minimal period. Now $s(\mu_i) = v_i$, so $\mu$ is unique because $(F_\beta, \LL_\beta)$ is right-resolving.
The only circuit that does not pass through $v_1$ is also labelled $w_p$, so $\pi_\beta$ is $2$ to $1$. 
\end{proof}

\begin{figure}
\begin{center}
\begin{tikzpicture}
  [bend angle=10,
   clearRound/.style = {circle, inner sep = 0pt, minimum size = 17mm},
   clear/.style = {rectangle, minimum width = 17 mm, minimum height = 6 mm, inner sep = 0pt},  
   greyRound/.style = {circle, draw, minimum size = 1 mm, inner sep =
      0pt, fill=black!10},
   grey/.style = {rectangle, draw, minimum size = 6 mm, inner sep =
      1pt, fill=black!10},
   white/.style = {rectangle, draw, minimum size = 6 mm, inner sep =
      1pt},
   to/.style = {->, shorten <= 1 pt, >=stealth', semithick}]
  
  \node[white] (P1) at (0,0) {};
  \node[white] (P2) at (2,0) {};   
  \node[white] (P3) at (4,0) {};
  \node[white] (P4) at (6,0) {};

  \draw[to] (P1) to node[auto] {$1$} (P2); 
  \draw[to] (P2) to node[auto] {$1$} (P3); 
  \draw[to] (P3) to node[auto] {$1$} (P4);

  \draw[to, loop left] (P1) to node[auto] {$0$} (P1);
  \draw[to, bend left = 30] (P2) to node[auto] {$0$} (P1);
  \draw[to, bend left = 60] (P3) to node[auto] {$0$} (P1);
  \draw[to, bend right = 45] (P4) to node[auto,swap] {$0$} (P3);  
   
\end{tikzpicture}
\end{center}
\caption[Right Fischer cover of a sofic beta-shift.]{Right Fischer cover of the sofic beta-shift from Example \ref{ex_beta_2_sofic}.} 
\label{fig_beta_2_sofic}
\end{figure}
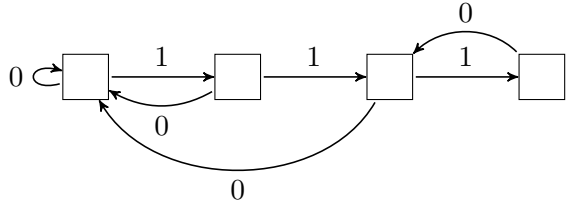

\begin{example}
\label{ex_beta_2_sofic}
Use Theorem \ref{thm_beta_existence} to choose $\beta > 1$ such that $g(\beta) = 11(10)^\infty$. The right Fischer cover of $\X_\beta$ is shown in Figure \ref{fig_beta_2_sofic}. Note that there are two presentations of the biinfinite sequence $(10)^\infty$ and only one presentation of every other sequence.
\end{example}

The following result was proved by Parry \cite{parry}, but it is repeated here since it follows from Propositions \ref{prop_beta_sofic} and \ref{prop_beta_2_to_1} and Corollary \ref{cor_lfc_conj}.

\begin{cor}
\label{cor_beta_sft}
For $\beta > 1$, the beta-shift $\X_\beta$ is an SFT if and only if the generating sequence $g(\beta)$ is periodic.
\end{cor}

A direction was chosen in the construction of $\X_\beta$ which makes right-rays behave qualitatively differently from left-rays. This means that, while the right Fischer cover is easy to construct using Proposition \ref{prop_beta_sofic}, it is generally non-trivial to find the left Fischer cover of a sofic beta-shift. This is illustrated by the following example.

\begin{example}
\label{ex_beta_lfc}\index{beta-shift!left Fischer cover of}
Use Theorem \ref{thm_beta_existence} to find $\beta > 1$ with 
$g(\beta) = (111010)^\infty$. Elementary computations show that the left Fischer cover of $\X_\beta$ is the labelled graph shown in Figure \ref{fig_beta_lfc}. This graph is arguably more complicated than the corresponding right Fischer cover. In particular, the left Fischer cover has non-intersecting loops.
\end{example}

\begin{figure}
\begin{center}
\begin{tikzpicture}
  [bend angle=10,
   clearRound/.style = {circle, inner sep = 0pt, minimum size = 17mm},
   clear/.style = {rectangle, minimum width = 17 mm, minimum height = 6 mm, inner sep = 0pt},  
   greyRound/.style = {circle, draw, minimum size = 1 mm, inner sep =
      0pt, fill=black!10},
   grey/.style = {rectangle, draw, minimum size = 6 mm, inner sep =
      1pt, fill=black!10},
   white/.style = {rectangle, draw, minimum size = 6 mm, inner sep =
      1pt},
   to/.style = {->, shorten <= 1 pt, >=stealth', semithick}]
  
  \node[white] (P1) at (0,0) {};
  \node[white] (P2) at (2,0) {};   
  \node[white] (P3) at (4,0) {};
  \node[white] (P4) at (6,0) {};
  \node[white] (P5) at (8,0) {};
  \node[white] (P6) at (10,0) {};

  \draw[to] (P1) to node[auto] {$1$} (P2); 
  \draw[to] (P2) to node[auto] {$0$} (P3);
  \draw[to, bend right = 30] (P2) to node[auto,swap] {$0$} (P4);    
  \draw[to, bend right = 45] (P3) to node[auto] {$1$} (P1); 
  \draw[to] (P3) to node[auto] {$1$} (P4);
  \draw[to] (P4) to node[auto] {$1$} (P5);
  \draw[to] (P5) to node[auto] {$1$} (P6);
  \draw[to, bend left = 30] (P6) to node[auto] {$0$} (P5);            
  \draw[to, bend left = 45] (P6) to node[auto,swap] {$0$} (P2);
  \draw[to, bend left = 60] (P6) to node[auto,swap] {$0$} (P1); 
  \draw[to, loop right] (P6) to node[auto] {$0$} (P6);                                 
\end{tikzpicture}
\end{center}
\caption[Left Fischer cover of a beta-shift of finite type.]{Left Fischer cover of the beta-shift of finite type considered in Example \ref{ex_beta_lfc}.} 
\label{fig_beta_lfc}
\end{figure}
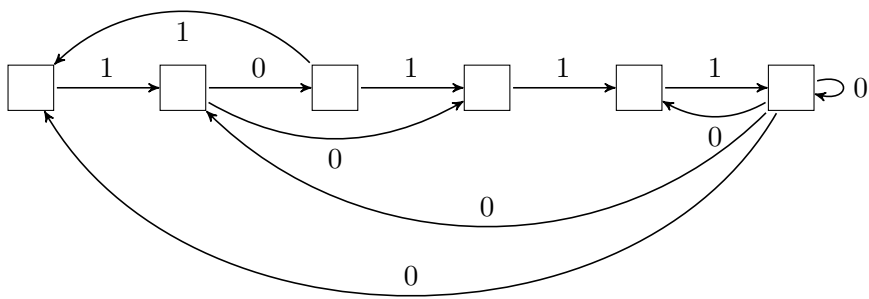

\subsection{Krieger cover}

\index{beta-shift!right Krieger cover of}
\begin{prop}
\label{prop_beta_krieger}
If $\beta > 1$ and the beta-shift $\X_\beta$ is sofic, then the right Krieger cover of $\X_\beta$ is identical to the right Fischer cover $(F_\beta, \LL_\beta)$ of $\X_\beta$.
\end{prop}

\begin{proof}
Let $n,p$ be minimal such that $g(\beta) = g_1 \cdots g_n(g_{n+1} \cdots g_{n+p})^\infty$ and let $F_\beta^0 = \{ v_1 , \ldots, v_{n+p} \}$ as in Proposition \ref{prop_beta_sofic}.
Let $x^- \in \X_\beta^-$, and let $r(x^-) \subseteq F_\beta^0$ be the set of vertices where a presentation of $x^-$ can end. The goal is to prove that $x^-$ is $1$-synchronizing. Assume that $x^- = (g_{n+1} \cdots g_{n+p})^\infty g_{n+1} \cdots g_{n+k}$ for some $k \leq p$. By the proof of Proposition \ref{prop_beta_2_to_1}, there is a unique $i$ such that there is a circuit labelled $w_p' = g_{n+k+1} \cdots g_{n+p} g_{n+1} \cdots g_{n+k}$ passing through $v_1$ and terminating at $v_i$. Similarly, there is a unique $j \neq i$  such that there is a circuit labelled $w_p'$ which terminates at $v_j$ without passing through $v_1$. Now, $r(x^-) = \{ v_i, v_j \}$ and $F_\infty(v_j) \subseteq F_\infty(v_i)$, so $x^-$ is $1$-synchronizing.
If $x^-$ is not of the form considered above, then the proof of Proposition \ref{prop_beta_2_to_1} shows that $\lvert r(x^-) \rvert = 1$, so $x^-$ is $1$-synchronizing. Hence, the right Krieger cover is equal to the right Fischer cover by Theorem \ref{thm_gfc_foundation_lkc}.
\end{proof}

\noindent
This result also follows from \cite{katayama_matsumoto_watatani} where it is shown that the Matsumoto algebra associated to $\X_\beta$ is simple.

\subsection{Fiber product}
This section contains an introduction to fiber products and a construction of the right fiber product covers of sofic beta-shifts.

\begin{definition}
\index{fiber product graph}
Let $X$ be an irreducible sofic shift and let $(F, \LL_F)$ be the right Fischer cover of $X$. The \emph{(right) fiber product graph} of $X$ is defined to be the labelled graph with vertex set $\{ (u,v) \mid u,v \in F^0 \}$ where there is an edge labelled $a$ from $(u_1, v_1)$ to $(u_2, v_2)$ if and only if there are edges labelled $a$ from $u_1$ to $u_2$ and from $v_1$ to $v_2$.
\end{definition}

\begin{lem} 
\label{lem_fiber_structure}
Let $X$ be a sofic shift over $\AA$, let $(F, \LL_F)$ be the right Fischer cover of $X$, let $n = \lvert F^0 \rvert$, and let $A$ be the corresponding symbolic adjacency matrix.
Then the symbolic adjacency matrix of the fiber product graph is $(B(i,j))_{1 \leq i,j \leq n}$ where $B(i,j)$ is the $n \times n$ matrix over formal sums over $\AA$ obtained from $A$ by omitting all symbols not appearing in the entry $A_{ij}$.
\end{lem}

\begin{proof}
This follows directly from the definition of the fiber product graph.
\end{proof}

\index{fiber product cover} \noindent 
The fiber product graph of $X$ is a presentation of $X$, and it contains the right Fischer cover as the subgraph induced by the diagonal vertices $\{ (v,v) \mid v \in F^0 \}$. The fiber product graph is generally not essential, so it is often useful to pass to the maximal essential subgraph. This subgraph will be called the \emph{fiber product cover} in the following.

\index{fiber product cover!$\Z / 2\Z$ action on}
Let $X$ be a sofic shift with right Fischer cover $(F, \LL_F)$ such that the covering map $\pi \colon \X_F \to X$ is $2$ to $1$, and let $(P, \LL_P)$ be the fiber product cover of $X$. 
Let $\lambda \in \X_P$ be a biinfinite sequence, let $i \in \Z$, let $s(\lambda_i) = (u_i,v_i) \in P^0$, and let $ r(\lambda_i) = (u_{i+1},v_{i+1}) \in P^0$. Then $(v_i,u_i), (v_{i+1},u_{i+1}) \in P^0$ as well, and there is a unique edge labelled $\LL_P(\lambda_i)$ from $(v_i,u_i)$ to $(v_{i+1},u_{i+1})$. This defines a continuous and shift commuting map $\varphi \colon \X_P \to \X_P$ with $\varphi^2 = \id$. In this way, the labelling induces a continuous and shift commuting $\Z / 2\Z$ action on the edge shift $\X_P$. This is said to be the $\Z / 2\Z$ action on $\X_P$ \emph{induced by the labels}. This also induces a corresponding continuous $\Z / 2\Z$ action on the suspension $S\X_P$. See \cite{boyle_sullivan} for an investigation of equivariant flow equivalence of shift spaces equipped with group actions.

\begin{conj}[{Boyle \cite{boyle_personal}}]
\label{conj_boyle}
For $i \in \{1,2\}$, let $X_i$ be a sofic shift with right Fischer cover $(F_i, \LL_{F_i})$ for which the covering map $\pi_i \colon \X_{F_i} \to X_i$ is $2$ to $1$, and let $(P_i, \LL_{P_i})$ be the fiber product cover of $X_i$. Then $X_1$ and $X_2$ are flow equivalent if and only if there is a flow equivalence $\Phi  \colon S\X_{P_1} \to S\X_{P_2}$ which commutes with the $\Z /2\Z$ actions induced by the labels.
\end{conj}

\noindent
The present work was motivated by a desire to apply this conjectured result to a concrete class of shift spaces. Strictly sofic beta-shifts are natural for this because the covering maps are always $2$ to $1$ by Proposition \ref{prop_beta_2_to_1}.

\begin{figure}[b]
\begin{center}
\begin{tikzpicture}
  [bend angle=10,
   clearRound/.style = {circle, inner sep = 0pt, minimum size = 17mm},
   clear/.style = {rectangle, minimum width = 10 mm, minimum height = 6 mm, inner sep = 0pt},  
   greyRound/.style = {circle, draw, minimum size = 1 mm, inner sep =
      0pt, fill=black!10},
   grey/.style = {rectangle, draw, minimum size = 6 mm, inner sep =
      1pt, fill=black!10},
   white/.style = {rectangle, draw, minimum size = 6 mm, inner sep =
      1pt},
   to/.style = {->, shorten <= 1 pt, >=stealth', semithick}]
  
  \node[white] (P1) at (0,0) {$(1,1)$};
  \node[white] (P2) at (2,0) {$(2,2)$};   
  \node[white] (P3) at (4,0) {$(3,3)$};

  \node[white] (P12) at (0,1.5) {$(1,2)$};
  \node[white] (P23) at (2,1.5) {$(2,3)$};   

  \node[white] (P21) at (0,-1.5) {$(2,1)$};
  \node[white] (P32) at (2,-1.5) {$(3,2)$};

  \draw[to] (P1) to node[auto] {$1$} (P2); 
  \draw[to, loop left] (P1) to node[auto] {$0$} (P1);   
  \draw[to] (P2) to node[auto] {$1$} (P3);
  \draw[to, bend left = 30] (P2) to node[auto] {$0$} (P1);  
  \draw[to, bend right = 45] (P3) to node[auto,swap] {$0$} (P2);    

  \draw[to,bend left] (P12) to node[auto] {$1$} (P23); 
  \draw[to,bend left] (P23) to node[auto] {$0$} (P12);
  \draw[to] (P12) to node[auto,swap] {$0$} (P1);   
   
  \draw[to,bend left] (P21) to node[auto] {$1$} (P32); 
  \draw[to,bend left] (P32) to node[auto] {$0$} (P21);
  \draw[to] (P21) to node[auto] {$0$} (P1);   

\end{tikzpicture}
\end{center}
\caption[Fiber product cover of a sofic beta-shift.]{Fiber product cover of the sofic shift from Example \ref{ex_beta_fiber}.} 
\label{fig_beta_fiber_ex}
\end{figure}

\begin{example}
\label{ex_beta_fiber}
Use Theorem \ref{thm_beta_existence} to find $1 < \beta < 2$ such that $g(\beta) = 1(10)^\infty$. The symbolic adjacency matrix of the right Fischer cover of the corresponding beta-shift $\X_\beta$ is
\begin{displaymath}
\left( \begin{array}{ c c c }
0 & 1 &      \\
0 &    & 1   \\
   & 0 &      \\
\end{array} \right), 
\end{displaymath}
where a blank entry signifies that there is no edge between the two vertices.
By Lemma \ref{lem_fiber_structure}, the symbolic adjacency matrix of the fiber product graph of $\X_\beta$ is
\begin{equation}
\label{eq_beta_fiber_ex_adj}
A = \left( \begin{array}{ c c c | c c c | c c c }
0 &    & &           &1 &    &               & &    \\
0 &    & &           &   & 1 &              & &    \\
   & 0 & &           &   &    &              & &    \\
\hline   
0 &    & &           &   &    &               &  1 &       \\
0 &    & &           &   &    &               &     & 1    \\
   & 0 & &           &   &    &               &     &       \\
\hline
   &    & &        0 &    &   &                & &    \\
   &    & &        0 &    &   &                & &    \\
   &    & &           & 0 &   &                & &    \\
\end{array} \right)
\begin{array}{l}
(1,1) \\
(1,2) \\
(1,3) \\
(2,1) \\
(2,2) \\
(2,3) \\
(3,1) \\
(3,2) \\
(3,3) 
\end{array}.
\end{equation} 
By disregarding stranded vertices, this can be reduced to 
\begin{displaymath}
A' = \left( \begin{array}{ c c | c c c | c c }
0 &    &    & 1 &    &    &    \\
0 &    &    &    & 1 &    &    \\
\hline
0 &    &    &    &    & 1 &    \\
0 &    &    &    &    &    & 1 \\
   & 0 &    &    &    &    &    \\
\hline
0 &    & 0 &    &    &    &    \\
0 &    &    & 0 &    &    &    \\
\end{array} \right)
\begin{array}{l}
(1,1) \\
(1,2) \\
(2,1) \\
(2,2) \\
(2,3) \\
(3,2) \\
(3,3) \\
\end{array}.
\end{displaymath}
Figure \ref{fig_beta_fiber_ex} shows this fiber product cover. Note that it contains the right Fischer cover as the connected component containing the diagonal vertices.
\end{example}

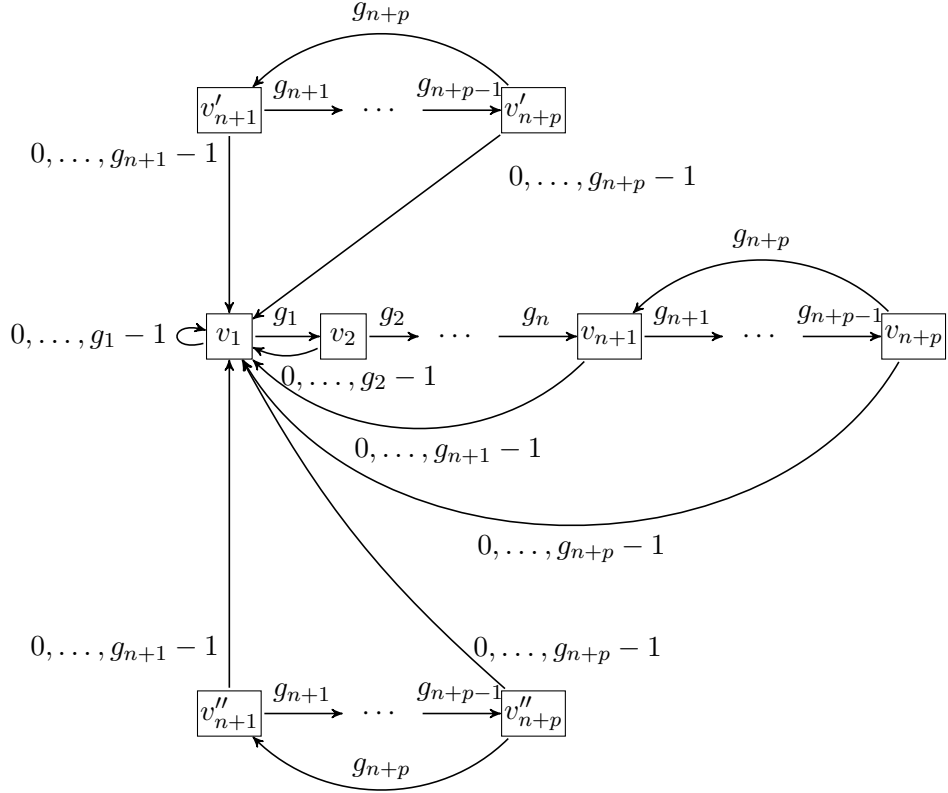
\begin{figure}[ht]
\begin{center}
\begin{tikzpicture}
 [bend angle=10,
   clearRound/.style = {circle, inner sep = 0pt, minimum size = 17mm},
   clear/.style = {rectangle, minimum width = 10 mm, minimum height = 6 mm, inner sep = 0pt},  
   greyRound/.style = {circle, draw, minimum size = 1 mm, inner sep =
      0pt, fill=black!10},
   grey/.style = {rectangle, draw, minimum size = 6 mm, inner sep =
      1pt, fill=black!10},
   white/.style = {rectangle, draw, minimum size = 6 mm, inner sep =
      1pt},
   to/.style = {->, shorten <= 1 pt, >=stealth', semithick}]
  
  \node[white] (v1)     at (0,0) {$v_1$};
  \node[white] (v2)     at (1.5,0) {$v_2$};
  \node[clear] (dots)   at (3,0) {$\cdots$};
  \node[white] (vn+1) at (5,0) {$v_{n+1}$};
  \node[clear] (dots2) at (7,0) {$\cdots$}; 
  \node[white] (vn+p) at (9,0) {$v_{n+p}$};

  \node[white] (vn+1')     at (0,3) {$v_{n+1}'$};
  \node[clear] (dots2')     at (2,3) {$\cdots$}; 
  \node[white] (vn+p')     at (4,3) {$v_{n+p}'$};

  \node[white] (vn+1'')     at (0,-5) {$v_{n+1}''$};
  \node[clear] (dots2'')     at (2,-5) {$\cdots$}; 
  \node[white] (vn+p'')     at (4,-5) {$v_{n+p}''$};

  \draw[to] (v1) to node[auto] {$g_1$} (v2); 
  \draw[to] (v2) to node[auto] {$g_2$} (dots); 
  \draw[to] (dots) to node[auto] {$g_n$} (vn+1);
  \draw[to] (vn+1) to node[auto] {$g_{n+1}$} (dots2);
  \draw[to] (dots2) to node[auto] {$g_{n+p-1}$} (vn+p);  

  \draw[to, loop left] (v1) to node[auto] {$0 ,  \ldots , g_{1}-1$} (v1); 
  \draw[to, bend left = 25] (v2) to node[auto] {\quad \qquad \qquad $0, \ldots , g_{2}-1$} (v1);    
  \draw[to, bend left = 45] (vn+1) to node[auto] {\qquad $0, \ldots , g_{n+1}-1$} (v1);     
  \draw[to, bend left = 60] (vn+p) to node[auto] {$0, \ldots , g_{n+p}-1$} (v1);     
  \draw[to, bend right = 45] (vn+p) to node[auto,swap] {$g_{n+p}$} (vn+1);     

  \draw[to] (vn+1') to node[auto] {$g_{n+1}$} (dots2');
  \draw[to] (dots2') to node[auto] {$g_{n+p-1}$} (vn+p');  
  \draw[to, bend right = 45] (vn+p') to node[auto,swap] {$g_{n+p}$} (vn+1');     
  \draw[to] (vn+1') to node[very near start, left] {$\; 0, \ldots, g_{n+1}-1$} (v1);
  \draw[to] (vn+p') to node[near start, right] {\qquad $0, \ldots, g_{n+p}-1$} (v1);

  \draw[to] (vn+1'') to node[auto] {$g_{n+1}$} (dots2'');
  \draw[to] (dots2'') to node[auto] {$g_{n+p-1}$} (vn+p'');  
  \draw[to, bend left = 45] (vn+p'') to node[auto,swap] {$g_{n+p}$} (vn+1'');     
  \draw[to] (vn+1'') to node[very near start, left] {$\; 0, \ldots, g_{n+1}-1$} (v1);
  \draw[to, bend left = 10] (vn+p'') to node[very near start, right] {$0, \ldots, g_{n+p}-1$} (v1);
  
\end{tikzpicture}
\end{center}
\caption[Fiber product covers of sofic beta-shifts.]{Fiber product cover of a sofic beta-shift with minimal $n,p$ such that $g(\beta) = g_1 \cdots g_n(g_{n+1} \cdots g_{n+p})^\infty$.} 
\label{fig_beta_fiber}
\end{figure}

The following proposition shows that the fiber product cover always has the structure seen in the previous example.

\begin{prop}
\label{prop_beta_fiber}
\index{beta-shift!fiber product cover of}
Let $\beta > 1$ such that $g(\beta)$ is eventually periodic but not periodic, then the fiber product cover of $\X_\beta$ is the graph shown in Figure \ref{fig_beta_fiber}.
\end{prop}

\begin{proof}
Let $(F_\beta, \LL_\beta)$ be the right Fischer cover of $\X_\beta$, and let $n,p$ be minimal such that $g(\beta) = g_1 \cdots g_n (g_{n+1} \cdots g_{n+p})^\infty$. By the proof of Proposition \ref{prop_beta_2_to_1}, a left-ray will have a unique presentation unless it is equal to $w^\infty$ for some cyclic permutation $w$ of the period $g_{n+1} \cdots g_{n+p}$. To find the fiber product cover, it is therefore sufficient to consider such periodic left-rays. 

By the proof of Proposition \ref{prop_beta_2_to_1}, there exist $u_0, \ldots , u_{p-1}, u'_0, \ldots, u'_{p-1} \in F_\beta^0$ with $u_i \neq u_i'$ such that there are edges labelled $g_{n+i+1}$ from $u_i$ to $u_{i+1 \pmod{p}}$ and from $u'_i$ to $u'_{i+1 \pmod{p}}$ for each $0 \leq i \leq p-1$. Now $(u_0,u'_0), \ldots, (u_{p-1},u'_{p-1})$ are the only off-diagonal vertices in the fiber product cover. For each $0 \leq i \leq p-1$, the fiber product cover has an edge labelled $g_{n+i+1}$ from $(u_i,u'_i)$ to $(u_{i+1 \pmod{p}},u'_{i+1 \pmod{p}})$ and edges labelled $0, \ldots, g_{n+i+1}-1$ from $(u_i,u'_i)$ to $(v_1,v_1)$, where $v_1$ is the first vertex of the right Fischer cover. This gives the labelled graph shown in Figure \ref{fig_beta_fiber}.
\end{proof}

\begin{remark}
\label{rem_beta_fiber_action}
Let $\beta > 1$ with $g(\beta) = g_1 \cdots g_n(g_{n+1} \cdots g_{n+p})^\infty$ for minimal $n,p$, and let $(P_\beta, \LL_{P_\beta})$ be the right fiber product cover of $\X_\beta$. Consider the $\Z / 2\Z$ action on $\X_{P_\beta}$ induced by the labelling. The element $1 \in \Z / 2\Z$ acts by fixing the part of the graph that is isomorphic to the Fischer cover and switching the two analogous irreducible components lying above (see Figure \ref{fig_beta_fiber}).
\end{remark}

\section{Flow classification of sofic beta-shifts}
\label{sec_beta_classification}
The aim of this section is to use the acquired knowledge about the structure of the fiber product covers to examine the flow equivalence of the corresponding edge shifts. If Conjecture \ref{conj_boyle} is true, this will give a complete flow classification of sofic beta-shifts.
First, a simple flow classification of beta-shifts of finite type is given. Then the Bowen-Franks groups of the covers of sofic beta-shifts considered in the previous sections are computed and shown to depend only on the sum of the entries in the period of the generating sequence. Finally, it is proved that this number is a complete invariant of flow equivalence of the edge shifts of the fiber product covers of sofic beta-shifts.

\subsection{Flow classification of beta-shifts of finite type}
The characterisation of beta-shifts of finite type given in Corollary \ref{cor_beta_sft} makes it possible to give a complete flow classification of such shifts.

\begin{prop}
\label{prop_beta_sft_bf}\index{beta-shift!of finite type!classification of}\index{beta-shift!of finite type!Bowen-Franks group of}
Given $\beta > 1$ such that $\X_\beta$ is an SFT, choose minimal $p \in \N$ such that the generating sequence is $g(\beta) = (g_1 \cdots g_p)^\infty$, then $\BF_+(\X_\beta) = - \Z / S\Z$ with $S=\sum_{j=1}^p g_j$. In particular, every SFT beta-shift is flow equivalent to a full shift.
\end{prop}

\noindent
Note that $\sum_{j=1}^p g_j = \sum_{i=1}^p a_i - 1$ when $e(\beta) = a_1 \cdots a_p00\cdots$ and $a_p \neq 0$.

\begin{proof}
By Proposition \ref{prop_beta_sofic}, the (non-symbolic) adjacency matrix of the underlying graph of the right Fischer cover of $\X_\beta$ is
\begin{displaymath}
A = \begin{pmatrix}
g_1      & 1 & 0 & \cdots & 0 & 0 \\ 
g_2      & 0 & 1 &             & 0 & 0 \\ 
g_3      & 0 & 0 &             & 0 & 0 \\ 
\vdots   &    &    & \ddots &    &    \\ 
g_{p-1 }& 0 & 0 &          & 0 & 1 \\ 
g_p+1      & 0 & 0 & \cdots & 0 & 0  
\end{pmatrix}.
\end{displaymath} 
Now it is straightforward to compute the complete invariant by finding the Smith normal form and determinant of $\Id - A$.
\end{proof}

\noindent
It is also not hard to construct a concrete flow-equivalence between the beta-shift considered in Proposition \ref{prop_beta_sft_bf} and the full $(S+1)$-shift.

\begin{example}
If $\beta > 1$, then the entropy of $\X_\beta$ is $\log \beta$ \cite{parry,renyi}.
In particular, beta-shifts $\X_{\beta_1}$ and $\X_{\beta_2}$ with $\beta_1 \neq \beta_2$ are never conjugate.
However, by Theorem \ref{thm_beta_existence}, there exist $1 < \beta _1 < 2$ and $2 < \beta_2 < 3$ such that $(110)^\infty$ is the generating sequence of $\X_{\beta_1}$ and $(20)^\infty$ is the generating sequence of $\X_{\beta_2}$, and the beta-shifts $\X_{\beta_1}$ and $\X_{\beta_2}$ are flow equivalent by Proposition \ref{prop_beta_sft_bf}. 
\end{example}

\subsection{Bowen-Franks groups}
The Bowen-Franks groups of the underlying graphs of the covers from Section \ref{sec_beta_covers} will be computed in this section. 

\begin{prop} 
\label{prop_beta_bf}
\index{beta-shift!sofic!Bowen-Franks group of}
Let $\beta > 1$ with sofic $\X_\beta$, and let $n,p$ be minimal such that $g(\beta) = g_1 \cdots g_n ( g_{n+1} \cdots g_{n+p})^\infty$. Let $A_K$
and $A_P$ be the adjacency matrices of the underlying graphs of respectively the Krieger cover
and the fiber product cover. Then $\BF_+(A_K) = -\Z / S\Z$
and $\BF(A_P) = (\Z / S\Z) \oplus \Z \oplus \Z$ where $S = \sum_{i=1}^p g_{n+i}$.
\end{prop}

\begin{proof}
By Proposition \ref{prop_beta_krieger},
\begin{displaymath}
A_K = 
\left(\!\! \begin{array}{ c c c c c c | c c c c c c  }
g_1    & 1 & 0 & \cdots & 0 & 0 &
0        & 0 & 0 & \cdots & 0 & 0  \\
g_2    & 0 & 1 &            & 0 & 0 &
0        & 0 & 0 &            & 0 & 0   \\
g_3    & 0 & 0 &            & 0 & 0 &
0        & 0 & 0 &           & 0 & 0    \\
\vdots &   &    & \ddots &    & \vdots &
\vdots &   &    & \ddots &    & \vdots  \\
g_{n-1}    & 0 & 0 &           & 0 & 1 &
0             & 0 & 0 &            & 0 & 0  \\
g_n         & 0 & 0 & \cdots & 0 & 0 &
1             & 0 & 0 & \cdots & 0 & 0  \\
\hline
g_{n+1} &    &    &            &    &    &
0           & 1 &  0 & \cdots & 0 & 0  \\
g_{n+2} &    &    &            &   &    &
0           & 0 & 1 &            & 0 & 0   \\
g_{n+3} &    &    &            &    &   &
0           & 0 & 0 &           & 0 & 0   \\
\vdots   &   &    & 0         &    &    &
\vdots   &   &    & \ddots &    & \vdots  \\
g_{n+p-1}    &  &  &           &    &   &
0             & 0 & 0 &            & 0 & 1   \\
g_{n+p}  &    &    &            &    &    &
1             & 0 & 0 & \cdots & 0 & 0   
\end{array}\!\! \right).
\end{displaymath}
It is straightforward to find the invariant by computing the Smith form and determinant of $\Id - A_K$. The other Bowen-Franks group is computed in the same manner.
\end{proof}

\noindent
These Bowen-Franks groups only contain information about the sum of the numbers in the periodic part of the generating sequence. This is partially explained by the results of the following section. 
\subsection{Concrete constructions}
\label{sec_beta_constructions}
This section contains recipes for concrete constructions of flow equivalences reducing the complexity of beta-shifts. Let $n \in \N$ and $n-1 < \beta < n$ be given, and let $X = \X_\beta$. Define $\varphi \colon \AA_\beta \to \{ 0 ,1\}^*$ by 
\begin{displaymath}
   \varphi(j) = \overbrace{1 \cdots 1}^{j} 0,
\end{displaymath}
and extend this to $\varphi \colon \AA_\beta^* \to \{ 0 ,1\}^*$ by $\varphi(a_1 \cdots a_n) = \varphi(a_1) \cdots \varphi(a_n)$.
The shift closure of $\varphi(\BB(X))$ is the language of a shift space $X'$ flow equivalent to $X$ by Lemma \ref{lem_replace}. Define the induced map $\varphi \colon X^+ \to X'^+$ by  
$\varphi(x_1 x_2 \cdots ) = \varphi(x_1) \varphi(x_2) \cdots $.

\begin{lem}
\label{lem_beta_order_preserving}
The map $\varphi \colon X^+ \to X'^+$ constructed above is surjective and order preserving.
\end{lem}

\begin{proof}
The map preserves the lexicographic order by construction.
Let $x'^+ \in X'^+$ be given. By construction, there exists $x^+ = x_1x_2 \cdots \in X^+$ such that $x'^+ = w \varphi(x_2) \varphi(x_3) \cdots$ where $w$ is a suffix of $\varphi(x_1)$. Now there exists $1 \leq a \leq x_1$ such that $w = \varphi(a)$. Since $a x_2 x_3 \cdots \leq  x_1 x_2 \cdots \leq g(\beta)$, it follows from Theorem \ref{thm_renyi} that $a x_2 x_3 \cdots \in X^+$, so $x'^+ \in \varphi(X^+)$.  
\end{proof}

\begin{thm}\index{beta-shift!flow equivalence of}
For every $\beta > 1$ there exists $1 < \beta' < 2$ such that $\X_\beta \FE \X_{\beta '}$. 
\end{thm}

\begin{proof}
Let $X = \X_\beta$, and construct $X' \FE X$ and $\varphi \colon X^+ \to X'^+$ as above. Use Theorem \ref{thm_beta_existence} to choose $1 < \beta' < 2$ such that $g(\beta') = \varphi(g(\beta))$. The aim is to prove that $X' = \X_{\beta'}$.
  
Given $x'^+ = x'_1 x'_2 \cdots \in X'^+$ and $k \in \N$, let $x'_{k^+} = x'_k x'_{k+1} \cdots$. Use Lemma \ref{lem_beta_order_preserving} to choose $x^+ \in X^+$ such that $\varphi(x^+) = x'_{k^+}$. Now $x^+ \leq g(\beta)$ and $\varphi$ is order preserving, so $x'_{k^+} \leq g(\beta')$. By Theorem \ref{thm_renyi}, this means that $x'^+ \in \X_{\beta'}^+$. 

Let $x'^+ = x'_1 x'_2 \cdots \in \X_{\beta'}^+$ and let $n = \max \AA_\beta$. Consider the extension $\varphi \colon \{0, \ldots, n\}^\N \to \{0, 1\}^\N$ and note that $x'^+$ does not contain $1^{n+1}$ as a factor, so there exists $x^+ = x_1 x_2 \cdots \in \{0, \ldots, n\}^\N$ such that $\varphi(x^+) = x'^+$.
Let $k \in \N$ be given and let $x_{k^+} = x_k x_{k+1} \cdots$. Then there exists $l \geq k$ such that 
$\varphi(x_{k^+}) =  x'_l x'_{l+1} \cdots \leq g(\beta') = \varphi(g(\beta))$, but $\varphi$ is order preserving, so this means that $x_{k^+} \leq g(\beta)$. Hence, $x^+ \in \X_\beta^+$ and $x'^+ \in X'^+$ as desired.
\end{proof}

\noindent
This shows that it is sufficient to consider $1 < \beta < 2$ when trying to classify sofic beta-shifts up to flow equivalence. The next goal is to find a standard form that any sofic beta-shift can be reduced to.

\begin{lem}
\label{lem_beta_delete_0_in_beginning}
\index{beta-shift!flow equivalence of}
Let $1 < \beta < 2$ such that $g(\beta)$ is aperiodic and let $n$ be the largest number such that $1^n$ is a prefix of $g(\beta)$. Then $\X_\beta \FE \X_\beta^{1^n0 \mapsto 1^n} = \X_{\beta'}$ where $g(\beta')$ is obtained from $g(\beta)$ by deleting a $0$ immediately after each occurrence of $1^n$.
\end{lem} 

\begin{proof}
Note that $1^{n+1} \notin \BB(\X_\beta)$, so each occurrence of $1^n$ in $\X_\beta$ is followed by $0$ and $\X_\beta^{1^n0 \mapsto 1^n}$ is well defined. There is an upper bound on $\{ k \mid (1^n 0)^k \in \BB(\X_\beta) \}$ since $1^n0$ is a prefix of $g(\beta)$ which is aperiodic, and therefore $\X_\beta \FE \X_\beta^{1^n0 \mapsto 1^n}$ by arguments similar to the ones used in Example \ref{ex_gap_shift}. 
Define $\varphi \colon \X_\beta^+ \to \{ 0, 1 \}^\N$ such that $\varphi(x^+)$ is the sequence obtained from $x^+$ by deleting one $0$ immediately after each occurrence of $1^n$. This is an order preserving map. Use Theorem \ref{thm_beta_existence} to choose $\beta'$ such that $g(\beta') = \varphi(g(\beta))$.
Given $y^+ \in \X_\beta^{1^n0 \mapsto 1^n}$ and $k \in \N$, define $x^+$ to be the sequence obtained from $y_k y_{k+1} \cdots$ by inserting $0$ after $y_j$ if there exists $l \in \N$ such that $y_{[j-ln,j]} = 01^{ln}$ or such that $j = k+ln-1$ and $y_{[k,j]} = 1^{ln}$.
Now, $x^+ \in \X_\beta^+$, so $y_k y_{k+1} \cdots = \varphi(x^+) \leq \varphi(g(\beta)) = g(\beta')$, so $y^+ \in \X_{\beta'}^+$ by Theorem \ref{thm_renyi}. A similar argument proves the other inclusion.  
\end{proof}

\begin{cor}
\label{cor_beta_add_0_in_beginning}
\index{beta-shift!flow equivalence of}
Let $1 < \beta < 2$ such that $g(\beta)$ is aperiodic and let $n$ be the largest number such that $1^n$ is a prefix of $g(\beta)$. For each $k > n / 2$, $\X_\beta \FE \X_\beta^{01^k \mapsto 01^k0} = \X_{\beta'}$ where $g(\beta')$ is obtained from $g(\beta)$ by inserting a $0$ immediately after the initial $1^k$ and after each subsequent occurrence of $01^k$.
\end{cor}

\begin{proof}
Apply Lemma \ref{lem_beta_delete_0_in_beginning} to $\X_{\beta'}$. 
\end{proof}

\begin{example}
\label{ex_beta_standard}
Consider $\beta > 1$ such that 
$g(\beta) = 1101101(0101100)^\infty$. Use Lemma \ref{lem_beta_delete_0_in_beginning} to see that $\X_\beta \FE \X_{\beta_1}$ when 
$g(\beta_1) = 11111(010110)^\infty$.
By Corollary \ref{cor_beta_add_0_in_beginning},  
$\X_{\beta_1} \FE \X_{\beta_2}$ when
\begin{displaymath} 
g(\beta_2) = 111110(010110)^\infty = 111(110010)^\infty.
\end{displaymath}
Note how this operation permutes the period of the generating sequence. Use Corollary \ref{cor_beta_add_0_in_beginning} again to show that $\X_{\beta_2} \FE \X_{\beta_3}$ when 
\begin{displaymath}
g(\beta_3) = 1110(110010)^\infty = 11(101100)^\infty.
\end{displaymath}
An additional application of Corollary \ref{cor_beta_add_0_in_beginning} will not reduce the aperiodic beginning of the generating sequence further, since it will also add an extra $0$ inside the period.
Note in particular that the length of neither the period nor the beginning of the generating sequence is a flow invariant. The sum of entries in the period of the generating sequence is a flow invariant by Proposition \ref{prop_beta_bf}, but the same is apparently not true for the sum of entries in the aperiodic beginning. Indeed, it is straightforward to use Lemma \ref{lem_beta_delete_0_in_beginning} and Corollary \ref{cor_beta_add_0_in_beginning} to show that $\X_\beta \FE \X_{\beta'}$ if
\begin{align*}
   g(\beta') &= 1^{3n}(110010)^\infty \textrm{ or } \\ 
   g(\beta') &= 1^{3n+2}(101100)^\infty   
\end{align*}
for some $n \in \N$. However, at this stage it is, for instance, still unclear whether $\X_\beta \FE \X_{\beta''}$ when $g(\beta'') = 1^{3n}(101100)^\infty$.
\end{example}

\begin{example}
\label{ex_beta_period_and_beginning}
Consider $\beta_0 > 1$ such that 
$g(\beta_0) = 11(0110010101)^\infty$. By Lemma \ref{lem_beta_delete_0_in_beginning}, $\X_{\beta_0} \FE \X_{\beta_i}$ when
\begin{align*}
  g(\beta_0) &= 11(0110010101)^\infty   \\
  g(\beta_1) &= 1111(010101011)^\infty       \\
  g(\beta_2) &= 11111(010101101)^\infty       \\
  g(\beta_3) &= 111111(010110101)^\infty       \\
  g(\beta_4) &= 1111111(011010101)^\infty       \\
                   &\vdots  
\end{align*}This illustrates how the length and position of the individual blocks of $1$s in the period are important when using Lemma \ref{lem_beta_delete_0_in_beginning} and Corollary \ref{cor_beta_add_0_in_beginning} to construct concrete flow equivalences. This information is ignored by the Bowen-Franks groups of the underlying graphs of the standard covers which only keep track of the sum of the entries in the period, and later results will suggest that this additional structure is actually \emph{not} preserved under flow equivalence (see Theorem \ref{thm_beta_classification}).
\end{example}

\begin{remark}
\label{rem_beta_standard}
Example \ref{ex_beta_standard} illustrates how the results of Lemma \ref{lem_beta_delete_0_in_beginning} and Corollary \ref{cor_beta_add_0_in_beginning} can be used to reduce a beta-shift $\X_\beta$ with $g(\beta) = w_b w_p^\infty$ to a flow equivalent beta-shift $\X_{\beta'}$ where $g(\beta') = 1^n w_{p'}^\infty$ with $n \leq \lvert w_b \rvert$, $\lvert w_{p'} \rvert \leq \lvert w_p \rvert$, and $n \leq k$ when $01^k0^i$ is a suffix of $w_{p'}$ for some $i$. Note that the number $n$ will depend not only on the sums of the entries in $w_b$ and $w_p$, but also on the length of the blocks of $1$s in $w_b$. 
\end{remark}

\subsection{Flow equivalence of fiber products}
Let $1 < \beta < 2$ with sofic $\X_\beta$, let $n,p$ be minimal such that $g(\beta) = g_1 \cdots g_n (g_{n+1} \cdots  g_{n+p})^\infty$, and let $(P_\beta, \LL_\beta)$ be the fiber product cover of $\X_\beta$. The goal of this section is to study the flow class of the reducible edge shift $\X_{P_\beta}$ defined by the underlying graph. Let $N = \sum_{i=1}^n g_i$ and $S = \sum_{i=1}^p g_{n+i}$. By Remark \ref{rem_beta_standard}, it can be assumed that $N \leq S$ without loss of generality.

Let $P_\beta^0 = \{v_1, \ldots, v_{n+p}, v'_{n+1}, \ldots, v'_{n+p}, v''_{n+1}, \ldots, v'_{n+p} \}$ as in Figure \ref{fig_beta_fiber}. Let $e \in P_\beta^1$ with $\LL_\beta(e) = 0$ and $r(e) \neq v_1$. Then there exists $f \in P_\beta^1$ such that for $\lambda \in \X_{P_\beta}$, $\lambda_0 = f$ if and only if $\lambda_1=e$. Hence, all these edges can be removed using symbol contraction, and this leaves the edge shift of the graph shown in Figure \ref{fig_beta_fiber_fe_I}.
In this graph, the vertices $v_N$ and $v_{N+S}$ both emit one edge to $v_{N+1}$ and one edge to $v_1$. Use in-amalgamation to merge these two vertices. This identifies the edges $e$ and $g$ and the edges $f$ and $h$ marked in Figure \ref{fig_beta_fiber_fe_I}. The result is a graph of the same form, where the size of $N$ is reduced by $1$. Repeat this process $N$ times to show that $\X_{P_\beta}$ is flow equivalent to the graph in Figure \ref{fig_beta_fiber_fe_II}. This leads to the following:

\begin{figure}
\begin{center}
\begin{tikzpicture}
 [bend angle=10,
   clearRound/.style = {circle, inner sep = 0pt, minimum size = 17mm},
   clear/.style = {rectangle, minimum width = 10 mm, minimum height = 6 mm, inner sep = 0pt},  
   greyRound/.style = {circle, draw, minimum size = 1 mm, inner sep = 0pt, fill=black!10},
 grey/.style = {rectangle, draw, minimum size = 6 mm, inner sep = 1pt, fill=black!10},
   white/.style = {rectangle, draw, minimum size = 6 mm, inner sep =1pt},
   to/.style = {->, shorten <= 1 pt, >=stealth', semithick}]
  
  \node[white] (v1)     at (0,0) {$v_1$};
  \node[white] (v2)     at (1.8,0) {$v_2$};
  \node[clear] (dots)   at (3.6,0) {$\cdots$};
  \node[white] (vn) at (5.4,0) {$v_{N}$};  
  \node[white] (vn+1) at (7.2,0) {$v_{N+1}$};
  \node[clear] (dots2) at (9,0) {$\cdots$}; 
  \node[white] (vn+p) at (10.8,0) {$v_{N+S}$};

  \node[white] (vn+1')     at (0,2) {$v_{N+1}'$};
  \node[clear] (dots2')     at (1.5,2) {$\cdots$}; 
  \node[white] (vn+p')     at (3,2) {$v_{N+S}'$};

  \node[white] (vn+1'')     at (0,-3) {$v_{N+1}''$};
  \node[clear] (dots2'')     at (1.5,-3) {$\cdots$}; 
  \node[white] (vn+p'')     at (3,-3) {$v_{N+S}''$};

  \draw[to] (v1) to node[auto] {$$} (v2); 
  \draw[to] (v2) to node[auto] {$$} (dots); 
  \draw[to] (dots) to node[auto] {$$} (vn);
  \draw[to] (vn) to node[auto] {$e$} (vn+1);  
  \draw[to] (vn+1) to node[auto] {$$} (dots2);
  \draw[to] (dots2) to node[auto] {$$} (vn+p);  

  \draw[to, loop left] (v1) to node[auto] {$$} (v1); 
  \draw[to, bend left = 25] (v2) to node[auto] {$$} (v1);    
  \draw[to, bend left = 30] (vn) to node[near start, below] {$f$} (v1);
  \draw[to, bend left = 35] (vn+1) to node[near start, below] {$$} (v1);       
  \draw[to, bend left = 35] (vn+p) to node[auto] {$h$} (v1);     
  \draw[to, bend right = 45] (vn+p) to node[auto,swap] {$g$} (vn+1);     

  \draw[to] (vn+1') to node[auto] {$$} (dots2');
  \draw[to] (dots2') to node[auto] {$$} (vn+p');  
  \draw[to, bend right = 45] (vn+p') to node[auto,swap] {$$} (vn+1');     
  \draw[to] (vn+1') to node[very near start, left] {$$} (v1);
  \draw[to] (vn+p') to node[near start, right] {$$} (v1);

  \draw[to] (vn+1'') to node[auto] {$$} (dots2'');
  \draw[to] (dots2'') to node[auto] {$$} (vn+p'');  
  \draw[to, bend left = 45] (vn+p'') to node[auto,swap] {$$} (vn+1'');     
  \draw[to] (vn+1'') to node[very near start, left] {$$} (v1);
  \draw[to] (vn+p'') to node[very near start, right] {$$} (v1);
\end{tikzpicture}
\end{center}
\caption[Flow equivalence of fiber product covers I.]{An unlabelled graph defining an edge shift flow equivalent to the edge shift of the underlying graph of the fiber product cover of a sofic beta-shift with minimal $n,p$ such that  $g(\beta) =  g_1 \cdots g_n (g_{n+1} \cdots g_{n+p})^\infty$. Here $N = \sum_{i=1}^n g_i$ and $S = \sum_{i=1}^p g_{n+i}$. Every vertex emits precisely two edges; one of which terminates at $v_1$.} 
\label{fig_beta_fiber_fe_I}
\end{figure}
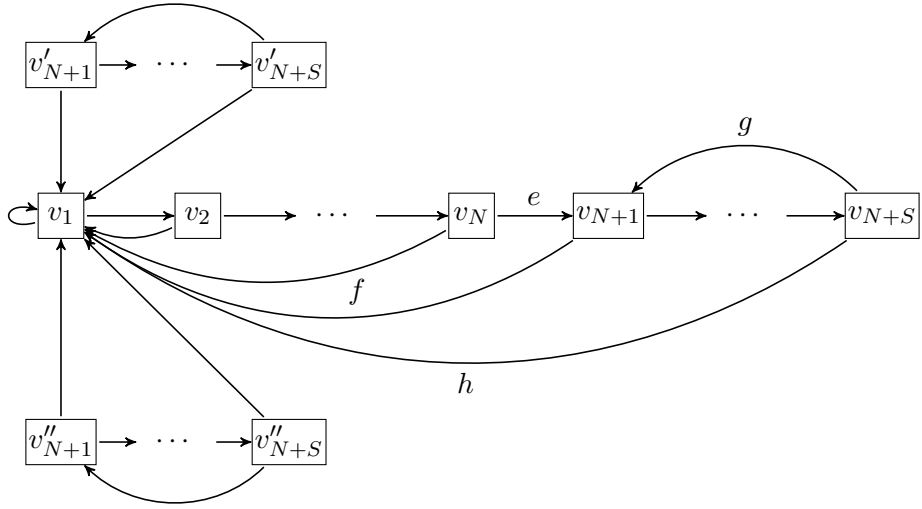

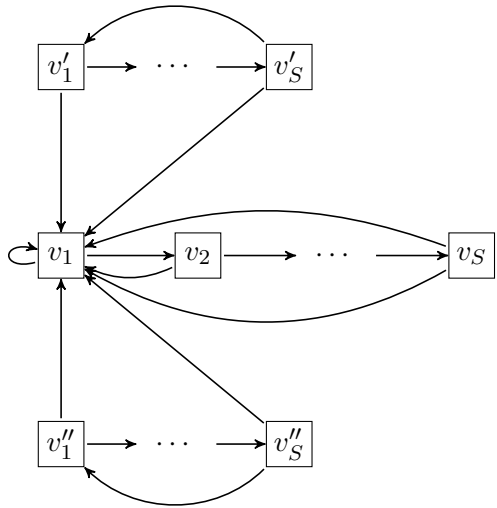
\begin{figure}
\begin{center}
\begin{tikzpicture}
 [bend angle=10,
   clearRound/.style = {circle, inner sep = 0pt, minimum size = 17mm},
   clear/.style = {rectangle, minimum width = 10 mm, minimum height = 6 mm, inner sep = 0pt},  
   greyRound/.style = {circle, draw, minimum size = 1 mm, inner sep =
      0pt, fill=black!10},
   grey/.style = {rectangle, draw, minimum size = 6 mm, inner sep =
      1pt, fill=black!10},
   white/.style = {rectangle, draw, minimum size = 6 mm, inner sep =
      1pt},
   to/.style = {->, shorten <= 1 pt, >=stealth', semithick}]
  
  \node[white] (v1)     at (0,0) {$v_1$};
  \node[white] (v2)     at (1.8,0) {$v_2$};
  \node[clear] (dots)   at (3.6,0) {$\cdots$};
  \node[white] (vp) at (5.4,0) {$v_{S}$};

  \node[white] (vn+1')     at (0,2.5) {$v_{1}'$};
  \node[clear] (dots2')     at (1.5,2.5) {$\cdots$}; 
  \node[white] (vn+p')     at (3,2.5) {$v_{S}'$};

  \node[white] (vn+1'')     at (0,-2.5) {$v_{1}''$};
  \node[clear] (dots2'')     at (1.5,-2.5) {$\cdots$}; 
  \node[white] (vn+p'')     at (3,-2.5) {$v_{S}''$};

  \draw[to] (v1) to node[auto] {$$} (v2); 
  \draw[to] (v2) to node[auto] {$$} (dots); 
  \draw[to] (dots) to node[auto] {$$} (vp);

  \draw[to, loop left] (v1) to node[auto] {$$} (v1); 
  \draw[to, bend left = 25] (v2) to node[auto] {$$} (v1);    
  \draw[to, bend left = 30] (vp) to node[near start, below] {$$} (v1);
  \draw[to, bend right = 20] (vp) to node[near start, below] {$$} (v1);

  \draw[to] (vn+1') to node[auto] {$$} (dots2');
  \draw[to] (dots2') to node[auto] {$$} (vn+p');  
  \draw[to, bend right = 45] (vn+p') to node[auto,swap] {$$} (vn+1');     
  \draw[to] (vn+1') to node[very near start, left] {$$} (v1);
  \draw[to] (vn+p') to node[near start, right] {$$} (v1);

  \draw[to] (vn+1'') to node[auto] {$$} (dots2'');
  \draw[to] (dots2'') to node[auto] {$$} (vn+p'');  
  \draw[to, bend left = 45] (vn+p'') to node[auto,swap] {$$} (vn+1'');     
  \draw[to] (vn+1'') to node[very near start, left] {$$} (v1);
  \draw[to] (vn+p'') to node[very near start, right] {$$} (v1);
\end{tikzpicture}
\end{center}
\caption[Flow equivalence of fiber product covers II.]{An unlabelled graph defining an edge shift flow equivalent to the edge shift of the underlying graph of the fiber product cover of a sofic beta-shift with minimal $n,p$ such that  $g(\beta) =  g_1 \cdots g_n (g_{n+1} \cdots g_{n+p})^\infty$. Here $S = \sum_{i=1}^p g_{n+i}$.} 
\label{fig_beta_fiber_fe_II}
\end{figure}

\begin{prop}
\label{prop_beta_fiber_fe}
For $i \in \{1,2\}$, let $\beta_i > 1$ with minimal $n_i,p_i$ such that
\begin{displaymath} 
g(\beta) =  g^i_1 \cdots g^i_{n_i} (g^i_{n_i+1} \cdots g^i_{n_i+p_i})^\infty,
\end{displaymath}
let $S_i = \sum_{i=1}^{p_i} g_{n_i+i}$, and let $(P_{\beta_i}, \LL_{P_{\beta_i}})$ be the fiber product cover of $\X_{\beta_i}$. Then there is a flow equivalence $\Phi \colon S\X_{P_{\beta_1}} \to S\X_{P_{\beta_2}}$ which commutes with the $\Z /2\Z$ actions induced by the labels if and only if $S_1 = S_2$.
\end{prop}

\begin{proof}
If there is such a flow equivalence, then Theorem \ref{thm_franks} shows that the Bowen-Franks groups of the edge shifts $\X_{P_{\beta_1}}$ and $\X_{P_{\beta_2}}$ must be equal, so $S_1 = S_2$ by Proposition \ref{prop_beta_bf}. Reversely, if $S_1 = S_2$, then the preceding arguments prove that there is a flow equivalence $\Phi \colon S\X_{P_{\beta_1}} \to S\X_{P_{\beta_2}}$. Furthermore, the $\Z / 2\Z$ actions on $\X_{P_{\beta_i}}$ induced by the labels are respected by the conjugacies and symbol reductions used in the construction, so $\Phi$ commutes with the actions. 
\end{proof}

The following theorem gives a complete classification of beta-shifts up to flow equivalence if Conjecture \ref{conj_boyle} is true.

\begin{thm}\label{thm_beta_classification}
\index{beta-shift!of finite type!classification of}
\index{beta-shift!sofic!classification of}
Let $\beta_1, \beta_2 > 1$ such that $\X_{\beta_1}$ and $\X_{\beta_2}$ are both SFTs or both strictly sofic. If Conjecture \ref{conj_boyle} holds, then the following are equivalent
\begin{enumerate}
\item $\X_{\beta_1}  \FE \X_{\beta_2}$.
\item $\X_{P_1} \FE \X_{P_2}$ when $(P_i, \LL_{P_i})$ is the right fiber product cover of $\X_{\beta_i}$ for $i \in \{1,2\}$.
\item $\X_{F_1} \FE \X_{F_2}$ when $(F_i, \LL_{F_i})$ is the right Fischer cover of $\X_{\beta_i}$ for $i \in \{1,2\}$.
\item $S_1 = S_2$ when $S_i = \sum_{k=1}^{p_i} g^i_{n_i+k}$
for minimal $n_i,p_i$ such that $g(\beta_i) = g^i_1 \cdots g^i_{n_i} (g^i_{n_i+1} \cdots g^i_{n_i+p_i})^\infty$ for $i \in \{1,2\}$.
\end{enumerate}
\end{thm}

\begin{proof}
The right Fischer cover and the fiber product cover are flow invariant, so (1) implies (2) and (3) in general. Propositions \ref{prop_beta_sft_bf} and \ref{prop_beta_bf} show that (1) implies (4). By Propositions \ref{prop_beta_sft_bf} and \ref{prop_beta_bf}, (4) implies (3), and by the proof of Proposition \ref{prop_beta_fiber_fe}, (4) implies (2). If Conjecture \ref{conj_boyle} is true, then Proposition \ref{prop_beta_fiber_fe} also shows that (4) implies (1).
\end{proof}

\section{Perspectives} \label{sec_beta_perspectives}
If Conjecture \ref{conj_boyle} is true, then Theorem \ref{thm_beta_classification} gives a very satisfying classification of sofic beta-shifts in terms of a single integer which is easy to compute. A proof of the conjecture will rely on deep results in symbolic dynamics, and it has been beyond the scope of this work to attempt to prove it.
If the conjecture is not true, then there is no obvious route to a flow classification of sofic beta-shifts. The flow invariants considered here can naturally only distinguish beta-shifts with different values of the integer assumed to be a complete invariant, so it would be necessary to examine completely different invariants. It would, for instance, be natural to examine the left Fischer covers of sofic beta-shifts as well, but as mentioned in Example \ref{ex_beta_lfc}, it is generally hard to construct these.

The following example illustrates how hard it is to work with flow equivalence of beta-shifts without the results of Conjecture \ref{conj_boyle}:
The results of Section \ref{sec_beta_constructions} show that for every $\beta > 1$ there exists $1 < \beta' < 2$ such that $\X_\beta \FE \X_{\beta'}$ and such that $g(\beta')$ has a special form. It is still unknown whether it is possible to do further reductions, and there is no known way to construct concrete flow equivalences in the cases where this is not made impossible by the known invariants. Without Conjecture \ref{conj_boyle}, there is for example, no known way to determine whether the beta-shifts $\X_{\beta_1}$ and $\X_{\beta_2}$ with generating sequences 
\begin{displaymath}
  g(\beta_1) = 1(110)^\infty
\textrm{ and }
  g(\beta_2) = 11(110)^\infty
\end{displaymath}
are flow equivalent. It is worth noting the similarity with the problems encountered when attempting to classify gap-shifts in Section \ref{sec_gap_shifts}.

%% file: thesis_renewal.tex
\label{chap_rs}

A renewal system is a shift space consisting of the biinfinite sequences that can be obtained as free concatenations of words from some finite generating list. This simple definition hides a surprisingly rich structure that is in many ways independent of the usual topological and dynamical structure of the shift space. 
The present work was motivated by the following problem raised by Adler: Is every irreducible shift of finite type conjugate to a renewal system? Several attempts have been made to answer this question, and the conjugacy of certain special classes of renewal systems is well understood, but there only exist a few results concerning the general problem.

The aim of this work has been to find the range of the Bowen-Franks invariant over the set of SFT renewal systems in order to answer the corresponding question for flow equivalence. The results are, however, incomplete and it is still unknown whether a large class of pairs of determinants and groups can be achieved. The most general result is obtained by combining a computation of the determinants of renewal systems in a class introduced by Hong and Shin \cite{hong_shin} with a concrete construction of certain non-cyclic Bowen-Franks groups.

\section{Background}
\index{renewal system}
\index{renewal system!generating list of|\\see{generating list}}
\index{generating list}\index{XL@$\X(L)$}
Let $\AA$ be an alphabet, let $L \subseteq \AA^*$ be a finite list of words over $\AA$, and define $\BB(L)$ to be the set of factors of $L^*$. Then $\BB(L)$ satisfies the conditions in Proposition \ref{prop_language}, so it is the language of a shift space $\X(L)$ which is said to be the \emph{renewal system}\footnote{Named by Adler for an analogy with renewal theory from probability theory.} generated by $L$. $L$ is said to be the \emph{generating list} of $\X(L)$. 

\begin{prop}
Every renewal system is an irreducible sofic shift.
\end{prop}

\begin{proof}
Let $L = \{ w_1 , \ldots , w_n\}$, let $(G, \LL)$ be the labelled graph obtained by writing the words $w_1, \ldots, w_n$ on loops starting and ending at a common vertex, and note that $(G, \LL)$ is an irreducible presentation of $\X(L)$.
\end{proof}

\index{renewal system!loop graph of}\index{finitely generated sofic system}\index{flower automata}\index{loop system}\index{even shift!as renewal system}
\noindent
The graph constructed in the proof is called the \emph{standard loop graph presentation} of $\X(L)$, and because of this presentation, renewal systems are called \emph{loop systems} or \emph{flower automata} in automata theory (e.g.\  \cite{berstel_perrin}). Renewal systems are also sometimes called \emph{finitely generated sofic systems} (e.g.\  \cite{restivo}). 
The even shift introduced in Example \ref{ex_even_def} is a renewal system generated by the list $L = \{ 00, 1\}$. In this case, the standard loop graph is equal to the left Fischer cover, but that is not true in general.

\subsection{Adler's problem}

As shown above, every renewal system is an irreducible sofic shift, and some of them are strictly sofic, but is every sofic shift -- or every SFT -- a renewal system? The following example shows that the answer is no.

\begin{example}[{\cite[pp.\ 433]{lind_marcus}}]
\label{ex_sft_not_rs}
Consider the edge shift $X$ of the graph seen in Figure \ref{fig_sft_not_rs}. Assume that $X$ is a renewal system, i.e.\ that there exists a finite list $L$ over $\{ a,b,c,d \}$ such that $X = \X(L)$. Since $a^n, c^n \in \BB(X)$ for all $n \in \N$, there must exist $k,l$ such that $a^k, c^l \in L$, but then $ac \in \BB(X)$ in contradiction with the presentation. This shows that not every irreducible SFT is a renewal system. Note, however, that $X$ is conjugate to the full $2$-shift which is a renewal system.
\end{example}

\begin{figure}
\begin{center}
\begin{tikzpicture}
  [bend angle=30,
   clearRound/.style = {circle, inner sep = 0pt, minimum size = 17mm},
   clear/.style = {rectangle, minimum width = 5 mm, minimum height = 5 mm, inner sep = 0pt},  
   greyRound/.style = {circle, draw, minimum size = 1 mm, inner sep =
      0pt, fill=black!10},
   grey/.style = {rectangle, draw, minimum size = 6 mm, inner sep =
      1pt, fill=black!10},
    white/.style = {rectangle, draw, minimum size = 6 mm, inner sep =
      1pt},
   to/.style = {->, shorten <= 1 pt, >=stealth', semithick}] 
  \node[white] (P1) at (0,0) {};
  \node[white] (P2) at (2,0) {};
  \draw[to,loop left] (P1) to node[auto] {$a$} (P1);
  \draw[to,bend left] (P1) to node[auto] {$b$} (P2);
  \draw[to,loop right] (P2) to node[auto] {$c$} (P2);
  \draw[to,bend left] (P2) to node[auto] {$d$} (P1);
\end{tikzpicture}
\end{center}
\caption[An edge shift that is not a renewal system.]{A graph for which the edge shift is not a renewal system.} 
\label{fig_sft_not_rs}
\end{figure}
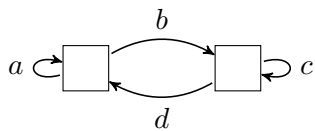

\noindent
Examples such as this naturally raise the following question:

\begin{prob}[Adler]
\index{Adler's problem}
\label{prob_adler}
Is every irreducible shift of finite type conjugate to a renewal system?
\end{prob}

\index{shift equivalence problem!and renewal systems}
\noindent
A positive solution to this problem would reduce the conjugacy problem for irreducible shifts of finite type to a question of conjugacies of SFT renewal systems, and Adler's aim \cite{goldberger_lind_smorodinsky, restivo} was to use this approach to attack the shift equivalence problem which was arguably the most important open problem in symbolic dynamics at the time (see Section \ref{sec_conjugacy}). This intended application is no longer as relevant because different means have long since been used to prove that shift equivalence and strong shift equivalence are indeed different equivalence relations \cite{kim_roush_1,kim_roush_2}, and because it has turned out that it is not particularly easy to determine whether two renewal systems are conjugate (see \cite{restivo_note} for a result about this question in a special case). However, Adler's problem remains open, and it is arguably the most important question concerning renewal systems. Indeed, most of the work done on renewal systems has been motivated by a desire to answer Adler's question \cite{goldberger_lind_smorodinsky,hong_shin_cyclic,hong_shin,johnson_madden,restivo, restivo_note,williams_rs}.

As mentioned above, there exist strictly sofic renewal systems, and the following example answers a natural variation of Adler's question by showing that there exists a strictly sofic shift which is not conjugate to a renewal system.

\begin{example}[{Williams \cite{williams_rs}}]
\index{Adler's problem!for sofic shifts} 
Let $X$ be the strictly sofic shift presented by the labelled graph in Figure \ref{fig_sofic_not_conj_rs}. Assume that there exists a renewal system $\X(L)$ over some alphabet $\AA$ and a conjugacy $\varphi \colon X \to \X(L)$.  Assume without loss of generality that $\varphi = \Phi^\infty_{[n,n]}$ for a map $\Phi \colon \BB_{2n+1}(X) \to \AA$ and that $\varphi^{-1} = \Psi^\infty_{[n,n]}$ for a map $\Psi \colon \BB_{2n+1}(\X(L)) \to \{ a,b,c \}$. The image of the fixpoint $a^\infty \in X$ is a fixpoint in $\X(L)$, so there exists $\alpha \in \AA$ such that $\Phi(a^{2n+1}) = \alpha$, and $m\in \N$ such that $\alpha^m \in L$. For each $0 \leq j \leq 2n$, let $\beta_j =\Phi(a^j b a^{2n-j})$. Then $\varphi( a^\infty b a^\infty) = \alpha^\infty \beta_0 \cdots \beta_{2n} \alpha^\infty$, so there must exist $k,l \in \N$ such that $\alpha^k \beta_0 \cdots \beta_{2n} \alpha^l \in L^*$ and $k+l \geq 2n$.
Let $x = \alpha^\infty \beta_0 \cdots \beta_{2n} \alpha^{k+l} \beta_0 \cdots \beta_{2n} \alpha^\infty \in \X(L)$. Then  $\varphi^{-1}(x) = a^\infty b a^{k+l+4n}b a^\infty$ in contradiction with the definition of $X$.
\label{ex_williams}
\end{example}

\begin{figure}
\begin{center}
\begin{tikzpicture}
  [bend angle=30,
   clearRound/.style = {circle, inner sep = 0pt, minimum size = 17mm},
   clear/.style = {rectangle, minimum width = 5 mm, minimum height = 5 mm, inner sep = 0pt},  
   greyRound/.style = {circle, draw, minimum size = 1 mm, inner sep =
      0pt, fill=black!10},
   grey/.style = {rectangle, draw, minimum size = 6 mm, inner sep =
      1pt, fill=black!10},
    white/.style = {rectangle, draw, minimum size = 6 mm, inner sep =
      1pt},
   to/.style = {->, shorten <= 1 pt, >=stealth', semithick}]
  
  \node[white] (P1) at (0,0) {};
  \node[white] (P2) at (2,0) {};
  
  \draw[to,loop left] (P1) to node[auto] {$a$} (P1);
  \draw[to,bend left] (P1) to node[auto] {$b$} (P2);
  \draw[to,loop right] (P2) to node[auto] {$a$} (P2);
  \draw[to,bend left] (P2) to node[auto] {$c$} (P1);
\end{tikzpicture}
\end{center}
\caption[A sofic shift not conjugate to a renewal system.]{Presentation of a sofic shift not conjugate to a renewal system.} 
\label{fig_sofic_not_conj_rs}
\end{figure}
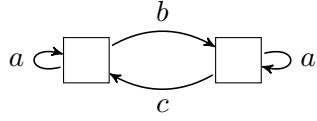

\begin{figure}
\begin{center}
\begin{tikzpicture}
  [bend angle=30,
   clearRound/.style = {circle, inner sep = 0pt, minimum size = 17mm},
   clear/.style = {rectangle, minimum width = 5 mm, minimum height = 5 mm, inner sep = 0pt},  
   greyRound/.style = {circle, draw, minimum size = 1 mm, inner sep = 0pt, fill=black!10},
   grey/.style = {rectangle, draw, minimum size = 6 mm, inner sep =
      1pt, fill=black!10},
  blackRound/.style = {circle, draw, minimum size = 1 mm, inner sep = 0pt, fill=black},
    white/.style = {rectangle, draw, minimum size = 6 mm, inner sep =
      1pt},
  whiteBig/.style = {rectangle, draw, minimum height = 5 cm, minimum width = 8 cm, inner sep =1pt},
   to/.style = {->, shorten <= 1 pt, >=stealth', semithick},
   map/.style = {->, shorten <= 2 pt, shorten >= 2 pt,  >=stealth', semithick}] 

  \node[circle, draw, minimum size = 4cm, fill=black!10] (Fm) at (-3,0) {};
  \node[clear] (rstext) at (-5.5,1.5) {Renewal};
  \node[clear] (rstext2) at ($(rstext)-(0,0.4)$) {systems};  

  \node[rectangle, draw, minimum height = 4.5 cm, minimum width = 6 cm] (sft) at (0,0)   {};
  \node[clear] (rstext) at (3.7,1.5) {SFTs};

  \node[blackRound] (xl) at (-2,0) {};
  \node[clear] (XL) at ($(xl)+(0,0.3)$) {$\X(L)$};  

  \node[blackRound] (even) at (-4,0) {};
  \node[clear] (XL) at ($(even)+(0,0.3)$) {even shift};

  \node[blackRound] (x) at (1.5,-1) {};
  \node[clear] (X) at ($(x)+(0,0.3)$) {$X$};  
  
 \node[blackRound] (xq) at (2.5,0) {};
 \node[clear] (XQ) at ($(xq)+(0,0.3)$) {?};  

    
  \draw[map] (xl) to node[auto] {$\FE$} (x);
  \draw[map] (x) to (xl);  

  \clip (-3,-2.25) rectangle (3,2.25);
  \node[circle, draw, dashed, minimum size = 10cm] (FE) at (-3,0) {};

\end{tikzpicture}
\end{center}
\caption[Flow equivalence problem for renewal systems.]{Flow equivalence problem for renewal systems. $X$ is an irreducible SFT flow equivalent to an SFT renewal system $\X(L)$. The dashed line signifies the possible border between the SFTs that are flow equivalent to renewal systems and those that are not.} 
\label{fig_rs_fe}
\end{figure}
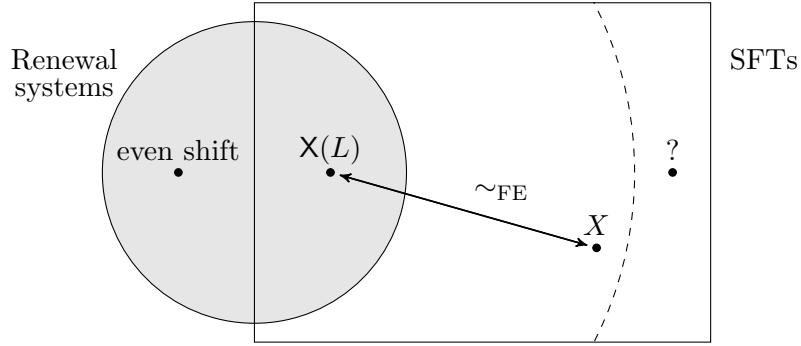

\index{Adler's problem!flow equivalence version}
\index{flow equivalence problem \\for renewal systems}
The aim of the present work has been to answer another natural variation of Adler's question: Is every SFT \emph{flow equivalent} to a renewal system? This question has apparently not been investigated before now.
It is appealing to consider flow equivalence because Theorem \ref{thm_franks} gives a complete flow classification of irreducible SFTs in terms of the Bowen-Franks invariant which is both easy to compute and easy to compare.
To answer the flow equivalence question, it is therefore sufficient to find the range of the Bowen-Franks invariant over the set of SFT renewal systems and check whether it is equal to the range over the set of irreducible SFTs.
It is easy to check that a group $G$ is the Bowen-Franks group of an irreducible SFT if and only if it is a finitely generated abelian group and that any combination of sign and Bowen-Franks group can be achieved by the Bowen-Franks invariant.
Hence, the overall strategy of the investigation of the flow equivalence question has been to attempt to construct all these combinations of groups and signs. However, it has turned out to be surprisingly hard to construct renewal systems with non-cyclic Bowen-Franks groups and/or positive determinants, so there are only partial results so far. The flow equivalence problem for renewal systems is sketched in Figure \ref{fig_rs_fe}.

Since flow equivalence is a weaker equivalence relation than conjugacy, a negative answer to the flow equivalence question will also give a negative answer to Adler's original question, while a positive answer will leave the original question open. Either way, the investigation can be expected to yield insight into the original question as well. 

A part of the investigation of the flow equivalence problem for renewal systems has been experimental, and the Bowen-Franks invariant has been computed for a large number SFT renewal systems. See Appendix \ref{app_programs} for a description of these experiments. It turns out to be difficult to find renewal systems that are not flow equivalent to full shifts in this way, but the process has generated valuable insight and some of the main results of this chapter grew out of the experimental investigation.

\subsection{Properties of generating lists}
Renewal systems have been studied intensely -- both in symbolic dynamics and in coding theory -- so there exists a well developed nomenclature, and several special cases have been studied in detail. This section introduces standard notation and recalls known relations between various standard classes of renewal systems.


\index{generating list!simple}
\index{generating list!minimal}
\index{generating list!prefix}
\index{generating list!suffix}
\index{generating list!uniquely decipherable}
\index{generating list!cyclic}
\index{generating list!pure}
\index{generating list!very pure}
\begin{definition}
A finite list $L \subseteq \AA^*$ is said to be
\begin{description}
\item{\emph{simple}} if no non-trivial concatenation of words from $L$ is an element of $L$.
\item{\emph{minimal}} if it is simple and there is no simple list $\tilde L \neq L$ such that $\X(L) = \X(\tilde L)$ and $L \subseteq \tilde L^*$.
\item{\emph{prefix}} if no word $w \in L$ is a prefix of a word $v \in L \setminus \{ w \}$. 
\item{\emph{suffix}} if no word $w \in L$ is a suffix of a word $v \in L \setminus \{ w \}$.
\item{\emph{uniquely decipherable}} if $v_1, \ldots, v_k, w_1, \ldots, w_l \in L$ and $v_1\cdots v_k = w_1 \cdots w_l$ implies that $k=l$ and $w_i = v_i$ for all $1 \leq i \leq k$.
\item{\emph{cyclic}} 
if, whenever $p,s \in \AA^*$, $v_0, \ldots, v_k, w_1, \ldots, w_l, ps \in L$, $s \neq \varepsilon$, and $v_0 \cdots v_k = sw_1 \cdots w_lp$, then $k=l$, $s= v_0$, $p = \varepsilon$, and $w_i = v_i$ for all $1 \leq i \leq k$.
\item{\emph{pure}} if $w \in L^*$ whenever there exists $n \in \N$ such that $w^n \in L^*$.
\item{\emph{very pure}} if $u,v \in L^*$ whenever there exists $uv,vu \in L^*$.
\end{description}
\end{definition}

\index{code}
\noindent In coding theory, a uniquely decipherable generating list is called a \emph{code}, and most of the terminology introduced above comes from coding theory (cf.\  \cite{berstel_perrin}). If a generating list $L$ is prefix, suffix, or cyclic, then it is automatically uniquely decipherable. The list $\{ a, abc, c \}$ is minimal and uniquely decipherable but neither prefix nor suffix, and the list $\{ aba, bab \}$ is uniquely decipherable but not cyclic, so the reverse is not true.

Restivo \cite{restivo} proved that the language of an irreducible sofic shift has a finite number of maximal monoids, and Hong and Shin \cite{hong_shin_cyclic} have used this to prove that every renewal system has a finite number of minimal generating lists.

\subsection{Irreducible generating lists}
The results of  this section will reduce the flow classification problem by proving that it is sufficient to consider generating lists which have been reduced using symbol contractions and conjugacies. This was useful in the experimental approach described in Appendix \ref{app_programs} since it allowed generating lists to be reduced to a simpler form before they were investigated. First, it is necessary to introduce some more terminology.

\begin{definition}
\label{def_partitioning}
\index{partitioning}\index{partitioning!beginning of}\index{partitioning!end of}\index{partitioning!minimal}
Let $L$ be a generating list.
A triple $(n_b,g,l)$ where $n_b,l \in \N$ and $g$ is an ordered list of words $g_1, \ldots , g_k \in L$ with $\sum_{i=1}^k \lvert g_i \rvert \geq  n_b+l-1$ is said to be a \emph{partitioning} of the factor $v_{[n_b,n_b+l-1]} \in \BB(\X(L))$ of $v = g_1 \cdots g_k$. The \emph{beginning} of the partitioning is the word $v_{[1,n_b-1]}$, and the \emph{end} is the word $v_{[n_b+l,\lvert v \rvert]}$. The partitioning is said to be \emph{minimal} if $n_b \leq \lvert g_1 \rvert$ and $n_b + l-1 > \sum_{i=1}^{k-1} \lvert g_i \rvert$. 
A partitioning of a right-ray $x^+ \in \X(L)^+$ is a pair $p = (n_b,(g_i)_{i\in\N})$ where $n_b \in \N$ and $g_i \in L$ such that $wx^+ = g_1 g_2 \cdots$ when $w$ is the \emph{beginning} consisting of the $n_b-1$ first letters of the concatenation $g_1 g_2 \cdots$. The partitioning is said to be minimal if $n_b \leq \lvert g_1 \rvert$. Partitionings of left-rays are defined analogously.
\end{definition}

Consider the generating list $L = \{ aa ,b \}$. Here $(2,[aa,b],2)$ is a minimal partitioning of the word $ab$. The beginning is the word $a$ and the end is the empty word. Clearly, $(4,[aa,aa,b],2)$ is a non-minimal partitioning of the same word.

\begin{definition}
\index{bordering}
\index{bordering!strongly}
\index{word!bordering}
\index{word!bordering!strongly}
\index{generating list!irreducible}
\index{internal word}
\index{word!internal}
Let $L \subseteq \AA^*$ be a finite list, and let $w \in \BB(\X(L)) \cup \X(L)^+$ be an allowed word or right-ray. Then $w$ is said to be \emph{left-bordering} if there exists a partitioning of $w$ with empty beginning, and \emph{strongly left-bordering} if every partitioning of $w$ has empty beginning. Right-bordering words and left-rays are defined analogously. A word $w \in \BB(\X(L))$ is said to be \emph{internal} if the list $g$ has only one element whenever $(n_b,g,l)$ is a minimal partitioning of $w$.
A generating list $L$ is said to be \emph{irreducible} if it is simple and every internal word has length 1. 
\end{definition}

\noindent
For $L = \{ aa , b \}$, the word $b$ is internal, while $aa$ and $ab$ are not.

\begin{lem}
\label{lem_irr_replace}
Let $L \subseteq \AA^*$ be a finite list of words, and let $w$ be an internal word  for $L$ for which $\X(L)$ does not admit non-trivial $w$-overlaps (i.e.\ there is no $v \in \BB(\X(L))$ with $\lvert w \rvert <\lvert v \rvert < 2 \lvert w \rvert$ such that $w$ is both a prefix and a suffix of $v$). Then $\X(L) \FE \X(M)$, where $M$ is the list obtained by replacing every occurrence of $w$ in $L$ by some $\diamond \notin \AA$.
\end{lem}

\begin{proof}
Apply Lemma \ref{lem_replace} to $\X(L)$ to replace each occurrence of $w$ by $\diamond$ and create the shift space $\X(L)^{w \to  \diamond} \FE \X(L) $. Since $w$ can only be read inside the generating words, $\X(L)^{w \to \diamond}  = \X(M)$.
\end{proof}

\noindent
This result makes it non-trivial to construct a list $L$ for which $\X(L)$ is an SFT not flow equivalent to the full $\{1, \ldots,\lvert L \rvert \}$-shift. 


\begin{prop}
\label{prop_irr}
For every $L \subseteq \AA^*$ there exists an irreducible generating list $M$ such that $\X(L) \FE \X(M)$.
\end{prop}

\begin{proof}
Assume without loss of generality that $L$ is simple.
The lengths of internal words of $L$ are bounded above by $\max_{v \in L} \lvert v \rvert$, so let $w \in \BB(\X(L))$ be an internal word of maximal length, and assume that $|w| \geq 2$. If $\X(L)$ admits a non-trivial $w$-overlap $\tilde w$, then $\tilde w$ must also be internal and $|\tilde w| > |w|$ in contradiction with the assumption on $w$. Apply Lemma \ref{lem_irr_replace} to $L$ and $w$ to produce a list $M$ for which the number of internal words of length $|w|$ is strictly less than the number of internal words of length $|w|$ in $L$, and for which $\X(L) \FE \X(M)$.
$M$ inherits the simplicity of $L$.
Repeat this process until the longest internal word has length 1.
\end{proof}

\noindent
In this way, the study of flow equivalence of renewal systems can be reduced to the study of renewal systems generated by irreducible lists. 

\begin{prop} \label{prop_irr_properties}
Let $L \subseteq \AA^*$ be finite and irreducible.
\begin{enumerate}
  \item For each $a \in \AA(\X(L))$, there exist $u,v \in
  L$ such that $a = \rl(u) = \leftl(v)$.
  \item $\BB_2(\X(L)) = \AA(\X(L))^2$.
 \item If $\X(L)$ is a 1-step shift of finite type, then it is 
 the full $\AA(\X(L))$-shift.  
\end{enumerate}
\end{prop}

\begin{proof}
For $a \in \AA(\X(L))$ there must exist letters $b,c \in \AA(\X(L))$ such that $ab, ca \in \BB(X)$. By assumption, neither of these words are internal, so there must exist $u,v \in L$ such that $a = \rl(u) = \leftl(v)$. The two remaining statements follow easily from this.
\end{proof}


\subsection{Conjugacy and flow equivalence of renewal systems}
In the following, some of the most important results about the conjugacy of renewal systems are recalled and used to derive simple consequences for the flow equivalence of special classes of renewal systems.

\begin{thm}[{B\' eal and Perrin \cite{beal_perrin}}]
\label{thm_cyclic_conjugacy}
Let $L \subseteq \AA^*$ be uniquely decipherable, and let $(G, \LL)$ be the standard loop graph presentation of $\X(L)$. Then $L$ is cyclic if and only if $\LL_\infty \colon \X_G \to \X(L)$ is a conjugacy.
\end{thm}

\noindent
Theorem \ref{thm_cyclic_conjugacy} can be used to prove that every cyclic generating list is also pure. The list $\{a, ab, bc, c \}$ is pure but not uniquely decipherable, and hence not cyclic, so the reverse is not true.

\begin{cor}
\label{cor_cyclic_fe_full}
If $L$ is cyclic, then $\X(L)$ is flow equivalent to a full shift.
\end{cor}

\begin{proof}
Use Theorem \ref{thm_cyclic_conjugacy} to see that $\X(L)$ is conjugate to the edge shift of the underlying graph of the standard loop graph presentation of $\X(L)$, and note that this edge shift can be symbol-reduced to the full $|L|$-shift. 
\end{proof}




The following two results show that various nicely behaved generating lists are automatically cyclic if they generate SFT renewal systems.

\begin{prop}[{Hong and Shin \cite[Proposition 3.2]{hong_shin_cyclic}}]
\label{prop_ud_min_sft_implies_cyclic}
Let $L$ be a minimal and uniquely decipherable generating list for which $\X(L)$ is an SFT, then $L$ is cyclic.
\end{prop}

\begin{thm}[{Hong and Shin \cite[Theorem 3.9]{hong_shin_cyclic}}]
\label{thm_cyclic_equivalent}
Let $L$ be a simple generating list that is either prefix or suffix for which $\X(L)$ is an SFT, then the following are equivalent:
\begin{itemize}
\item $L$ is cyclic.
\item $L$ is pure.
\item $L$ is minimal.
\end{itemize}
\end{thm}

\begin{remark}
\label{rem_nice_lists}
Corollary \ref{cor_cyclic_fe_full}, Proposition \ref{prop_ud_min_sft_implies_cyclic}, and Theorem \ref{thm_cyclic_equivalent} show that many of the nicely behaved classes of generating lists considered above always generate renewal systems flow equivalent to full shifts.
Hence, it is necessary to go beyond these classes in order to find non-trivial results about the range of the Bowen-Franks invariant.
\end{remark}

In an investigation of Adler's original question, it is natural to attempt to compute the ranges of conjugacy invariants over the set of SFT renewal systems. If such a range can be proved to be different from the range over the set of irreducible SFTs, then this will provide a negative answer to the question. The following theorem was the first major result of this kind.

\begin{thm}[{Goldberger, Lind, and Smorodinsky \cite{goldberger_lind_smorodinsky}}]
\index{renewal system!entropy of}
\label{thm_goldberger}
If $X$ is an irreducible SFT, then there exists a renewal system with entropy $h(X)$.
\end{thm}

\noindent
Hong and Shin \cite{hong_shin} improved this result by showing that whenever $\log \lambda$ is the entropy of an irreducible SFT, there exists an \emph{SFT} renewal system with entropy $\log \lambda$. This shows that the range of the entropy invariant over the class of renewal systems cannot be used to answer Adler's question, and it is arguably the most powerful general result obtained in the search for an answer to Adler's question. The class $H$ of generating lists introduced in \cite{hong_shin} will be examined in greater detail in Section \ref{sec_rs_entropy} where the range of the Bowen-Franks invariant over $H$  will be computed.

\subsection{When is a renewal system of finite type?}
\label{sec_rs_when_sft}
In order to attempt to answer Adler's question by computing the range of invariants over the set of SFT renewal systems, it is desirable to have conditions on $L$ which are necessary and sufficient for $\X(L)$ to be an SFT. 

\begin{example}
Consider the lists $L_1 = \{aa, b\}$ and $L_2 = \{ab, ba\}$. $L_1$ generates the even shift which is strictly sofic, and $L_2$ generates a renewal system which is strictly sofic for a similar reason: For each $n \in \N$, $bb(ab)^n, (ab)^nb \in \BB(\X(L_2))$ but $bb(ab)^nb \notin \BB(\X(L_2))$. This kind of behaviour is often seen in strictly sofic renewal systems. However, a list will not necessarily generate a strictly sofic shift simply because it contains a word such as $aa$ or a pair of words such as $ab$ and $ba$. Indeed, if suitable words are added to the lists considered above, then the resulting generating lists will generate shifts of finite type. This is, for instance, the case for $L'_1 = \{aa, aaa, b\}$ and $L'_2 = \{ab, aba, ba, bab\}$.
\end{example}

In general, it is non-trivial to determine whether a list generates an SFT, but there are results in certain special cases, e.g.\  for the class of cyclic renewal systems where the following corollary is a consequence of Theorem \ref{thm_cyclic_conjugacy}.

\begin{cor}
\label{cor_cyclic_implies_sft}
If $L \subseteq \AA^*$ is cyclic, then $\X(L)$ is an SFT.
\end{cor}

\noindent
Consider the list $L = \{ aa, aaa, b \}$ to see that the reverse is not true in general. However, a partial reverse is given by Proposition \ref{prop_ud_min_sft_implies_cyclic}.

\begin{thm}[{Restivo \cite{restivo_question}}]
\label{thm_restivo}A renewal system generated by a very pure list is an SFT. 
\end{thm}

\index{subset construction}
Let $L$ be a generating list, and let $(G, \LL_G)$ be the standard loop graph of $\X(L)$. By using the \emph{subset construction}, it is possible to construct a left-resolving presentation of $\X(L)$ based on  $(G, \LL_G)$ (see e.g.\ \cite[Theorem 3.3.2]{lind_marcus}). Once a left-resolving presentation is known, another algorithm can be used to construct the left Fischer cover (\cite[Theorem 3.4.14]{lind_marcus}). However, the graph constructed by the subset construction has $2^{G^0}-1$ vertices, and even for moderately sized generating lists, it can take a long time to find the left Fischer cover using these algorithms.

Let $(F, \LL)$ be the left Fischer cover of $\X(L)$.
For $w \in \BB(\X(L))$, $\lvert s(w) \rvert = 1$ if and only if $w$ is intrinsically synchronizing, so in order to prove that $\X(L)$ is an $n$-step SFT, it is sufficient to prove that $\lvert s(w) \rvert = 1$ for all $w \in \BB_n(X)$. If $\X(L)$ is an SFT and $\lvert F^0 \rvert = r$, then $\X(L)$ is $(r^2-r)$-step, so only finitely many words have to be examined \cite[Theorem 3.4.17]{lind_marcus}. Hence,
there is an algorithm for checking whether $\X(L)$ is an SFT. However, it is 
generally difficult, to check whether a concrete list generates an SFT unless it belongs to one of the well understood classes considered above.

The following result shows that the reverse problem also has a constructive answer:

\begin{thm}[{Restivo \cite{restivo}}]
If $X$ is an SFT, then it is decidable whether $X$ is a renewal system.\label{thm_restivo_decidable}
\end{thm}

As mentioned in Remark \ref{rem_nice_lists}, the most well understood classes of generating lists are uninteresting when considering the flow equivalence of SFT renewal systems, so the next goal is to develop a new set of conditions which will guarantee that a list generates an $n$-step SFT based only on the knowledge of the partitionings of the allowed words of length less than or equal to $n$.

\begin{definition}
\index{word!strongly synchronizing}
A word $w \in \BB_n(\X(L))$ is said to be \emph{strongly synchronizing} if there exists a prefix $v$ of $w$ such that whenever $(n_b, [g_1, \ldots g_k], n )$ is partitioning of $w$, there exists $1 \leq j \leq k$ such that $(n_b, [g_1, \ldots g_j], \lvert v \rvert )$ is a partitioning of $v$ with empty end.
\end{definition}

\index{word!strongly synchronizing!is synchronizing}
\noindent
I.e.\ $w$ is strongly synchronizing if every partitioning of $w$ has a border between two generating words at a specific place in $w$.
By definition, a strongly synchronizing word $w$ is intrinsically synchronizing, and if $u \in \BB(\X(L))$ contains a strongly synchronizing factor $w$, then $u$ is also strongly synchronizing.
Hence, if every $w \in \BB_n(\X(L))$ is strongly synchronizing, then $\X(L)$ is an $n$-step SFT by Theorem \ref{thm_m-step}. This is precisely the kind of result that is sought, but it is desirable to have a less restrictive condition, and this leads to the following definition.

\begin{definition}
\index{partitioning!extendable}
\index{extendable word}
\index{word!extendable}
A partitioning $p$ of a word $w \in \BB(\X(L))$ is said to be \emph{left-extendable} if whenever $a \in \AA$ and $aw \in \BB(\X(L))$ there exists a partitioning $q$ of $aw$ with the same end as $p$. A word $w$ is said to be \emph{left-extendable} if every minimal partitioning of $w$ is left-extendable. Right extendable partitionings and words are defined analogously.
\end{definition} 

\index{word!strongly synchronizing!is extendable}
\noindent
It is easy to check that a strongly synchronizing word is automatically both left- and right-extendable

\begin{prop} \label{prop_extendable}
If the words in $\BB_n(\X(L))$ 
are all left-extendable or all right-extendable, then $\X(L)$ is an $n$-step SFT.
\end{prop}

\begin{proof}
Assume that all words of length $n$ are left-extendable, and consider $w \in \BB(\X(L))$ with $|w| \geq n$, $a \in \AA$, and $v \in \BB(\X(L))$ such that $aw , wv \in \BB(\X(L))$. Let $p = (n_b, [g_1,\cdots,g_k], \lvert wv \rvert)$ be a minimal partitioning of $wv$. Choose $1 \leq j \leq k$  such that $p' = (n_b, [g_1,\cdots,g_j], n)$ is a minimal partitioning of the word $w' = w_{[1,n]}$. By assumption, $w'$ is left-extendable and $aw' \in \BB(\X(L))$, so there exists a partitioning $q = (m_b, [h_1, \cdots , h_l],n+1)$ of $aw'$  with the same end as $p'$. Now $(m_b, [h_1, \cdots, h_l, g_{j+1}, \cdots, g_k], \lvert wv \rvert +1)$ is a partitioning of $awv$.
Apply this argument repeatedly to see that $uwv \in \BB(\X(L))$ for any $u \in \BB(\X(L))$ for which $uw \in \BB(\X(L))$.
By Theorem \ref{thm_m-step}, this implies that $\X(L)$ is an $n$-step SFT. By symmetry, an analogous result holds for right-extendable words.
\end{proof}

\noindent
Appendix \ref{app_programs} describes how Proposition \ref{prop_extendable} can be used to write a computer program that finds renewal systems of finite type. 

The condition from Proposition \ref{prop_extendable} is generally not necessary for a renewal system to be an SFT, but the following proposition gives a class of renewal systems where it is.

\begin{prop}
\label{prop_extendable_and_sft}
If $L$ has a strongly left-bordering (respectively right-bordering) word then $\X(L)$ is an $n$-step SFT if and only if all allowed words of length $n$ are left-extendable (respectively right-extendable).    
\end{prop}

\begin{proof}
One direction was proved in Proposition \ref{prop_extendable}. Assume that $v \in L^*$ is strongly left-bordering and that $w \in \BB_n(\X(L))$ is not left-extendable. Choose $a \in \AA$ and a partitioning $p$ of $w$ such that $aw \in \BB(\X(L))$ but no partitioning $q$ of $aw$ has the same end as $p$. Let $w_e$ be the end of $p$. Then $a w w_e v$ is forbidden while both $aw$ and $w w_e v$ are allowed. The result follows by Theorem \ref{thm_m-step}, and an analogous result holds for right-extendable words by symmetry. 
\end{proof}

\noindent
Many interesting renewal systems have strongly bordering words (e.g.\  the renewal systems considered in \cite{hong_shin}), and Proposition \ref{prop_extendable_and_sft}  gives a way to determine precisely when these are SFTs. 

\section{The left Fischer cover}
\label{sec_rs_lfc}
In the attempt to find the range of the Bowen-Franks invariant over the set of SFT renewal systems, it is useful to be able to construct complicated renewal systems from simpler building blocks, but in general, it is non-trivial to study the structure of the renewal system $\X(L_1 \cup L_2)$ even if the renewal systems $\X(L_1)$ and $\X(L_2)$ are well understood.
Hence, the goal of this section is to describe the structure of the left Fischer covers of renewal systems in order to give conditions under which the Fischer cover of $\X(L_1 \cup L_2)$ can be constructed when the Fischer covers of $\X(L_1)$ and $\X(L_2)$ are known.

\subsection{Vertices of the left Fischer cover}
\index{renewal system!Fischer cover of!$P_0(L)$}\index{P0@$P_0(L)$}
\index{renewal system!Fischer cover of!construction} 
\label{sec_rs_lfc_construction}
As mentioned in Section \ref{sec_rs_when_sft}, the subset construction can be used to construct the left Fischer cover of a renewal system from the standard loop graph. This section deals with a related way of using the generating list to construct the vertices of the left Fischer cover. In the following, the left Fischer cover will be identified with the top irreducible component of the left Krieger cover via the correspondence described in Section \ref{sec_fischer}. This allows predecessor sets to be used as the vertices of the left Fischer cover.
 
Let $L$ be a generating list and define 
\begin{displaymath}
  P_0(L) = \left\{  \ldots w_{-2} w_{-1} w_{0}  \mid w_i \in L  \right \} \subseteq \X(L)^-.
\end{displaymath}
$P_0(L)$ is the predecessor set of the central vertex in the standard loop graph of $\X(L)$, but it is not necessarily the predecessor set of a right-ray in $\X(L)^+$, so it does not necessarily correspond to a vertex in the left Fischer cover of $\X(L)$.

If $p \in \BB(\X(L))$ is a prefix of some word in $L$, define
 \begin{displaymath}
	P_0(L)p = \left\{  \ldots w_{-2} w_{-1} w_{0} p  \mid w_i \in L  \right \} \subseteq \X(L)^-.
\end{displaymath}
Given a word $w \in \BB(\X(L))$, let $B(w)$ be the set of beginnings of partitionings of $w$. Then $P_\infty(w) = \bigcup_{b \in B(w)} P_0(L)b$. This makes it relatively simple to find the predecessor sets of words and right-rays in $\X(L)$ if all the partitionings are known, so this representation is useful for constructing the left Fischer covers of renewal systems, and it will be used several times in the following.

\subsection{Border points}
\label{sec_border_point}

Two lists may generate the same renewal system even though they have very different behaviour when it comes to the partitionings of allowed words, and the fact that the topological and symbolic dynamical structure of a renewal system does not mirror the behaviour of the partitionings is one of the reasons why renewal systems are hard to examine as shift spaces. The following definition adds a layer of information to the left Fischer cover of a renewal system in order to keep track of this extra structure.

\begin{definition}	\label{def_border_point}
\index{border point}\index{border point!universal}
\index{border point!generator of}
Let $L \subseteq \AA^*$ be finite, and let $(F, \LL_F)$ be the left Fischer cover of $\X(L)$. A vertex $P \in F^0$ is said to be a \emph{(universal) border point} for $L$ if there exists a (strongly) left-bordering $x^+ \in X^+$ such that $P = P_\infty(x^+)$. An intrinsically synchronizing word $w \in L^*$ is said to be a \emph{generator} of the border point $P_\infty(w) = P_\infty(w^\infty)$, and it is said to be a \emph{minimal generator} of $P$ if no prefix of $w$ is a generator of $P$.
\end{definition}

\index{renewal system!border point of}
Two lists generating the same shift space may have different border points, so the border points add information to the Fischer cover about the structure of the generating lists, and this information will be useful for studying $\X(L_1 \cup L_2)$ when the Fischer covers of $\X(L_1)$ and $\X(L_2)$ are known. If $P$ is a (universal) border point of $L$ and there is no ambiguity about which list is generating $X = \X(L)$, then the terminology will be abused slightly by saying that $P$ is a (universal) border point of $X$ or simply of the left Fischer cover.
The following lemma lists a number of simple properties of border points that will be useful in the following. 

\begin{lem}
\label{lem_border_point_general_prop}
Let $L$ be a finite list generating a renewal system with left Fischer cover $(F, \LL_F)$.
\begin{enumerate}
\item 
If $P \in F^0$ is a border point, then $P_0(L) \subseteq P$, and if $P$ is a universal border point, then $P = P_0(L)$.
\item \label{lem_bp_path_to}
If $P_1,P_2 \in F^0$ are border points and if $w_1 \in L^*$ is a generator of $P_1$, then there exists a path with label $w_1$ from $P_1$ to $P_2$. 
\item \label{lem_bp_path_implies}
If $P_1 \in F^0$ is a border point and $w \in L^*$, then there exists a unique border point $P_2 \in F^0$ with a path labelled $w$ from $P_2$ to $P_1$.
\item 
If $\X(L)$ is an SFT, then every border point of $L$ has a generator.
\item \label{lem_bp_strongly_right_implies}
If $L$ has a strongly right-bordering word $w$, then $x^+ \in \X(L)^+$ is left-bordering if and only if $P_\infty(x^+)$ is a border point.
\end{enumerate}
\end{lem} 

\begin{proof} \qquad \\ \vspace{-0.8 em}
\begin{enumerate}
\item 
Choose a left-bordering $x^+ \in \X(L)^+$ such that $P = P_\infty(x^+)$ and note that $y^-x^+ \in \X(L)$ for each $y^- \in P_0(L)$.
\item
Choose a left-bordering $x^+ \in \X(L)^+$ such that $P_2 = P_\infty(x^+)$. Then $P_\infty(w_1 x^+) = P_1$ since $w_1 x^+ \in \X(L)^+$ and  $w_1$ is intrinsically synchronizing, so there is a path labelled $w_1$ from $P_1$ to $P_2$.
\item
Choose a left-bordering $x^+ \in \X(L)^+$ such that $P = P_\infty(x^+)$. Since $w \in L^*$, the right-ray $w x^+$ is also left-bordering.
\item
Let $P = P_\infty(x^+)$ for some left-bordering $x^+ \in \X(L)^+$, and choose an intrinsically synchronizing prefix $w \in L^*$ of $x^+$. Then $P_\infty(x^+) = P_\infty(w)$, so $w$ is a generator of $P$.
\item
If $P_\infty(x^+)$ is a border point, then $w x^+ \in \X(L)^+$, so $x^+$ must be left-bordering. The other implication holds by definition.
\end{enumerate}
\end{proof}

\noindent
In particular, the universal border point is unique when it exists, but not every generating list has a universal border point. A predecessor set $P_\infty(x^+)$ can be a border point even though $x^+$ is not left-bordering (consider e.g.\  the right-ray $(aaaab)^\infty$ in the shift from Example \ref{ex_pos_det_1}).

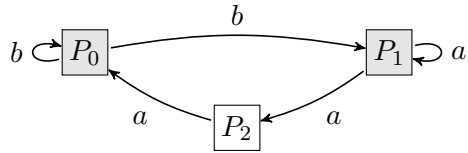
\begin{figure}[htbp]
\begin{center}
\begin{tikzpicture}
  [bend angle=10,
   clearRound/.style = {circle, inner sep = 0pt, minimum size = 17mm},
   clear/.style = {rectangle, minimum width = 5 mm, minimum height = 5 mm, inner sep = 0pt},  
   greyRound/.style = {circle, draw, minimum size = 1 mm, inner sep =
      0pt, fill=black!10},
   grey/.style = {rectangle, draw, minimum size = 6 mm, inner sep =
      1pt, fill=black!10},
    white/.style = {rectangle, draw, minimum size = 6 mm, inner sep =
      1pt},
   to/.style = {->, shorten <= 1 pt, >=stealth', semithick}]
  
  \node[grey] (P0) at (0,0) {$P_0$};
  \node[grey] (P1) at (4,0) {$P_1$};
  \node[white] (P2) at (2,-1) {$P_2$};
  
  \draw[to,loop left] (P0) to node[auto] {$b$} (P0);
  \draw[to,bend left] (P0) to node[auto] {$b$} (P1);
  \draw[to,bend left] (P1) to node[auto] {$a$} (P2);
  \draw[to,bend left] (P2) to node[auto] {$a$} (P0);
  \draw[to,loop right] (P1) to node[auto] {$a$} (P1);

\end{tikzpicture}
\end{center}
\caption[Border points.]{Left Fischer cover of the SFT renewal system $\X(L)$ generated by $L = \{ aa, aaa, b \}$ which is discussed in Example \ref{ex_border_points}. The border points are coloured grey.} 
\label{fig_aa_aaa_b}
\end{figure}

\begin{example} \label{ex_border_points}
Consider the list $L = \{ aa, aaa, b\}$ and the renewal system $\X(L)$. The set of forbidden words is $\FF = \{ bab \}$, so this is an SFT. The left Fischer cover of $\X(L)$ is shown in Figure \ref{fig_aa_aaa_b}. Note that the word $b$ is a generator of the universal border point $P_0 = P_\infty(b)$. Similarly, both $aab$  and $aa$ are generators of the border point $P_1$. On the other hand, every path terminating at $P_2$ has $a$ as a suffix, so $P_0$ is not a subset of $P_2$, and therefore $P_2$ is not a border point. Note the paths labelled $b$ from $P_0$ to $P_1$ and $P_0$, and the paths labelled $aab$ from $P_1$ to $P_0$ and $P_1$.
\end{example}

\begin{lem} \label{lem_cycle_border_point}
Let $L$ be a finite list generating an SFT renewal system with left Fischer cover $(F, \LL_F)$, and let $\gamma = \gamma_1 \cdots \gamma_l$ be a circuit in $F$ with $\LL_F(\lambda) = a_1 \cdots a_l = w$. Then there exist $1 \leq i \leq l$ and $k \in \N$ such that $s(\gamma_i)$ is a border point generated by $w' = a_i \cdots a_l w^k a_1 \cdots a_{i-1}$.
\end{lem}

\begin{proof}
Since $w^\infty \in \X(L)$, there must exist $1 \leq i \leq l$ and $j \in \N_0$ such that $w' = a_i \cdots a_l w^j a_1 \cdots a_{i-1} \in L^*$ is intrinsically synchronizing. Now $s(\lambda_i) = P_\infty(w')$ is a border point.
\end{proof}

Lemma \ref{lem_border_point_general_prop}(\ref{lem_bp_path_to}) and Lemma \ref{lem_cycle_border_point}  force the left Fischer cover of an SFT renewal system to have certain paths connecting the circuits. This is formalised in the following. 

\begin{definition}
A predecessor-separated, left-resolving, and ir\-redu\-cible labelled graph $(G, \LL)$ is said to be \emph{circuit connected}\index{labelled graph!circuit connected} if there exist vertices $P_1, \ldots,  P_n \in G^0$ such that when $\gamma$ is a circuit in $G$, there exist
\begin{itemize}
\item  $1 \leq i \leq n$ and $1 \leq l \leq |\gamma|$ such that $P_i = s(\gamma_l)$, and 
\item an intrinsically synchronizing power $w$ of $\LL(\gamma_l \ldots \gamma_{|\gamma|} \gamma_1 \ldots  \gamma_{l-1})$ for which there is a path labelled $w$ from $P_i$ to $P_j$ for each $1 \leq j \leq n$.
\end{itemize}
An unlabelled graph $G$ is said to be \emph{circuit connected}\index{graph!circuit connected} if there exists a labelling $\LL$ such that $(G, \LL)$ is circuit connected. An edge shift is said to be \emph{circuit connected}\index{edge shift!circuit connected} if the corresponding graph is circuit connected.
\end{definition}

The following proposition generalises Example \ref{ex_sft_not_rs} by showing that there is a class of irreducible SFTs that are not  renewal systems. 

\begin{prop} \label{prop_rs_circuit_connected}
If $(F, \LL_F)$ is the left Fischer cover of an SFT renewal system, then it is circuit connected.
\end{prop}

\begin{proof}
This follows from Lemmas \ref{lem_border_point_general_prop}(\ref{lem_bp_path_to}) and \ref{lem_cycle_border_point}.
\end{proof}

\begin{example} \label{ex_not_connected}
Consider the edge shift $X$ presented by the graph in Figure \ref{fig_edge_shift_without_connections}. There are two loops at each vertex, but only one edge between them in each direction, so the graph is not circuit connected, and hence, there is no way to label it to produce the left Fischer cover of an SFT renewal system. It is, however, unknown whether $X$ is conjugate to a circuit connected edge shift. 
In this context, it is worth noting that the database \cite{sft_database} contains no circuit connected SFT conjugate to $X$. 
The Bowen-Franks group of $X$ is $\Z$, and this value of the complete invariant of flow equivalence of irreducible SFTs can also be achieved by a renewal system (see Theorem \ref{thm_rs_det_range}), so it is easy to check that $X$ is flow equivalent to a renewal system. 
\end{example}

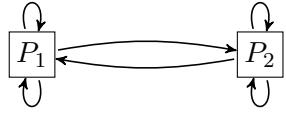
\begin{figure}
\begin{center}
\begin{tikzpicture}
  [bend angle=10,
   clearRound/.style = {circle, inner sep = 0pt, minimum size = 17mm},
   clear/.style = {rectangle, minimum width = 5 mm, minimum height = 5 mm, inner sep = 0pt},  
   greyRound/.style = {circle, draw, minimum size = 1 mm, inner sep =
      0pt, fill=black!10},
   grey/.style = {rectangle, draw, minimum size = 6 mm, inner sep =
      1pt, fill=black!10},
   white/.style = {rectangle, draw, minimum size = 6 mm, inner sep =
      1pt},
   to/.style = {->, shorten <= 1 pt, >=stealth', semithick}]
  
  \node[white] (P1) at (0,0) {$P_1$};
  \node[white] (P2) at (3,0) {$P_2$};
  
  \draw[to,loop above] (P1) to (P1);
  \draw[to,loop below] (P1) to (P1);
  \draw[to,bend left] (P1) to (P2);
  \draw[to,loop above] (P2) to (P2);
  \draw[to,loop below] (P2) to (P2);
  \draw[to,bend left] (P2) to (P1);

\end{tikzpicture}  
\end{center}
\caption[An edge shift that is not the Fischer cover of a renewal system.]{A directed graph presenting the edge shift discussed in Example \ref{ex_not_connected}. There is no way to label this graph to produce the left Fischer cover of an SFT renewal system.} 
\label{fig_edge_shift_without_connections}
\end{figure}

\begin{remark}
The strict structure enforced on the left Fischer cover of an SFT renewal system by Proposition \ref{prop_rs_circuit_connected} may make it difficult for certain SFTs (like the one in Example \ref{ex_not_connected}) to be conjugate to a renewal system, but it has not been possible to use this property to answer Adler's question since the structure is not preserved under conjugacy.
\end{remark}


\subsection{Addition}\label{sec_rs_add}
\index{renewal system!addition of}
Consider two renewal systems $\X(L_1)$ and $\X(L_2)$. The \emph{sum} $\X(L_1) + \X(L_2)$ is the renewal system $\X(L_1 \cup L_2)$. Note that this is generally not the same as the union of the shift spaces $\X(L_1)$ and $\X(L_2)$.

It is easy to check that the left Fischer cover of $\X(L_1 \cup L_2)$ will contain the left Fischer covers of $\X(L_1)$ and $\X(L_2)$ as subgraphs. However, the left Fischer cover of the sum may also contain vertices without analogues in the individual parts, and there will always be a number of connecting edges between the two subgraphs corresponding to the left Fischer covers of $\X(L_1)$ and $\X(L_2)$. Generally, this makes it difficult to study the Fischer covers of such sums even if the Fischer covers of the parts are known. The goal of this section is to develop a condition under which it is easier to keep track of the connecting edges.

%


\index{generating list!modular}
\begin{definition}
Let $L$ be a generating list with universal border point $P_0$ and let $(F, \LL_F)$ be the left Fischer cover of $\X(L)$.  
$L$ is said to be \emph{left-modular} if for all $\lambda \in F^*$ with 
$r(\lambda) = P_0$, $\LL_F(\lambda) \in L^*$ if and only if $s(\lambda)$ is a border point. \emph{Right-modular} generating lists are defined analogously.
\end{definition}

\index{renewal system!modular}
\noindent
Note that one of these implications follows from Lemma \ref{lem_border_point_general_prop}(\ref{lem_bp_path_implies}).
When $L$ is left-modular and there is no doubt about which generating list is used, the renewal system $\X(L)$ will also be said to be \emph{left-modular}. Modular renewal systems are useful building blocks and will be used in several constructions in the following.

\begin{lem} \label{lem_strongly_bordering_implies_modular}
If $L$ is a generating list with a strongly left-bordering word $w_l$ and a strongly right-bordering word $w_r$, then it is both left- and right-modular.
\end{lem}

\begin{proof}
Let $(F, \LL_F)$ be the left Fischer cover of $\X(L)$, let $P \in F^0$ be a border point, and choose $x^+ \in \X(L)^+$ such that $w_lx^+ \in \X(L)^+$. Assume that there is a path from $P$ to $P_0(L) = P_\infty(w_lx^+)$ with label $w$. The word $w_r$ has a partitioning with empty end, so there is a path labelled $w_r$ terminating at $P$. It follows that $w_r w w_l x^+ \in \X(L)^+$, so $w \in L^*$. By symmetry, $L$ is also right-modular.
\end{proof}

\noindent
Consider $L = \{ abb, abbb, bb, bbb \}$ to see that not every left-modular renewal system has a strongly right-bordering word. In general, it is difficult to prove that a renewal system is left-modular, so it is useful to have the stronger condition from Lemma \ref{lem_strongly_bordering_implies_modular}.

Now it is possible to show how to find the left Fischer cover of a sum of two modular renewal systems when the left Fischer covers of the individual terms are known.
For $i \in \{1,2\}$, let $L_i$ be a left-modular generating list and let $X_i = \X(L_i)$ have alphabet $\AA_i$ and left Fischer cover $(F_i, \LL_i)$. Let $P_i \in F_i^0$ be the universal border point of $L_i$. Assume that $\AA_1 \cap \AA_2 = \emptyset$.

The idea is to consider the labelled graph $(F_+, \LL_+)$ obtained by taking the union of $(F_1, \LL_1)$ and $(F_2, \LL_2)$, identifying the two universal border points $P_1$ and $P_2$, and adding certain connecting edges. 
To do this formally, introduce a new vertex $P_+$ and define
\begin{displaymath}
F_+^0 = ( F_1^0 \cup F_2^0 \cup \{P_+\} ) \setminus \{ P_1, P_2 \}.
\end{displaymath}  
Define maps $f_i \colon F_i^0 \to F_+^0$ such that for $v \in F_i^0 \setminus \{ P_i \}$, $f_i(v)$ is the vertex in $F_+^0$ corresponding to $v$ and such that $f_i(P_i) = P_+$. For each $e \in F_i^1$, define an edge $e' \in F_+^1$ such that $s(e') = f_i(s(e))$, $r(e') = f_i(r(e))$, and $\LL_+(e') = \LL_i(e)$. For each $e \in F_1^0$ with $r(e) = P_1$ and each non-universal border point $P \in F_2^0$, draw an additional edge $e' \in F_+^1$ with $s(e') = f_1(s(e))$, $r(e') = f_2(P)$, and $\LL_+(e') = \LL_1(e)$. Draw analogous edges for each $e \in F_2^1$ with $r(e) = P_2$ and every non-universal border point $P \in F_1^0$. 
This construction
is illustrated in Figure \ref{fig_modular_add}.
 
 \begin{figure}
\begin{center}
\begin{tikzpicture}
  [bend angle=10,
   clearRound/.style = {circle, inner sep = 0pt, minimum size = 17mm},
   clear/.style = {rectangle, minimum width = 17 mm, minimum height = 6 mm, inner sep = 0pt},  
   greyRound/.style = {circle, draw, minimum size = 1 mm, inner sep =
      0pt, fill=black!10},
   grey/.style = {rectangle, draw, minimum size = 6 mm, inner sep =
      1pt, fill=black!10},
   white/.style = {rectangle, draw, minimum size = 6 mm, inner sep =
      1pt},
   lfc/.style = {circle, draw, minimum size = 3cm},
   to/.style = {->, shorten <= 1 pt, >=stealth', semithick}]
  
  \node[lfc] (F1) at (-1.5,0) {};
  \node[lfc] (F2) at (1.5,0) {};
  \node[grey] (P+) at (0,0) {$P_+$};
  \node[white] (S) at (-1.5,1) {$v$};
  \node[grey] (P) at (1.5,1) {$P$};  
  \node[clear] (F1text) at (-4,0) {$(F_1,\LL_1)$};
  \node[clear] (F1text) at (4,0) {$(F_2,\LL_2)$};

  \draw[to] (S) to node[auto,swap] {$a$} (P+);
  \draw[to] (S) to node[auto] {$a$} (P);

\end{tikzpicture}
\end{center}
\caption[Addition of modular renewal systems.]{The labelled graph $(F_+, \LL_+)$ constructed from the left Fischer covers of left-modular renewal systems $\X(L_1)$ and $\X(L_2)$. In $(F_1, \LL_1)$, the vertex $v$ emits an edge labelled $a$ to $P_1$, so in $(F_+, \LL_+)$,  the corresponding vertex emits edges labelled $a$ to $P_+$ and to every vertex corresponding to a border point $P \in F_2^0$.} 
\label{fig_modular_add}
\end{figure}
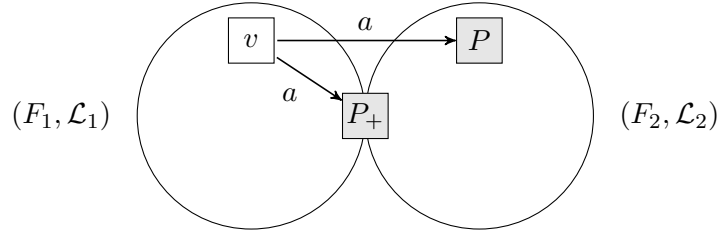

\begin{prop}
\label{prop_addition_modular}
If $L_1$ and $L_2$ are left-modular generating lists with disjoint alphabets, then $L_1 \cup L_2$ is left-modular, the left Fischer cover of $\X(L_1 \cup L_2)$ is the graph $(F_+, \LL_+)$ constructed above, and the vertex $P_+ \in F_+^0$ is the universal border point of $L_1 \cup L_2$. 
\end{prop}

\begin{proof}
By construction, the labelled graph $(F_+, \LL_+)$ is irreducible, left-resolving, and predecessor-separated, so by Theorem \ref{thm_lfc_char}, it is the left Fischer cover of some sofic shift $X_+$. 
Given $w \in L_1^*$, there is a path with label $w$ in the left Fischer cover of $X_1$ from some border point $P \in F_1^0$ to the universal border point $P_1$ by Lemma \ref{lem_border_point_general_prop}(\ref{lem_bp_path_implies}). Hence, there is also a path labelled $w$ in $(F_+, \LL_+)$ from the vertex corresponding to $P$ to the vertex $P_+$. This means that for every border point $Q \in F_2^0$, $(F_+, \LL_+)$ contains a path labelled $w$ from the vertex corresponding to $P$ to the vertex corresponding to $Q$. By symmetry, it follows that every element of $(L_X \cup  L_Y)^*$ has a presentation in $(F_+, \LL_+)$. Hence, $\X(L_1 \cup L_2) \subset X_+$. Note that these arguments hold even for non-modular systems. 

Assume that $awb \in \BB(X_+)$ with $a,b \in \AA_1$ and $w \in \AA_2^*$. Then there must be a path labelled $w$ in $(F_+, \LL_+)$ from a vertex corresponding to a border point $P$ of $L_2$ to $P_+$. By construction, this is only possible if there is also a path labelled $w$ from $P$ to $P_2$ in $(F_2, \LL_2)$, but $L_2$ is left-modular, so this means that $w \in L_2^*$. By symmetry, $\X(L_1 \cup L_2) = X_+$, and $P_+$ is the universal border point by construction.
\end{proof} 

\subsection{Fragmentation}\label{sec_rs_fragmentation}\index{fragmentation}
Let $X$ be a shift space over the alphabet $\AA$. Given $a \in \AA$, $k \in \N$, and new symbols $a_1, \ldots, a_k \notin \AA$  consider the map 
$f_{a,k} \colon (\AA \setminus \{ a \}) \cup \{ a_1, \ldots , a_k\} \to \AA$ defined by
$f_{a,k}(a_i) = a$ for each $1 \leq i \leq k$ and $f_{a,k}(b) = b$ when $b \in \AA \setminus \{a \}$. Let $F_{a,k} \colon ((\AA \setminus \{ a \}) \cup \{ a_1, \ldots , a_k\})^* \to \AA^*$ be the natural extension of $f_{a,k}$.  
If $w \in \AA^*$ contains $l$ copies of the symbol $a$, then the preimage $F_{a,k}^{-1}(\{w \})$ is the set consisting of the $k^l$ words that can be obtained by replacing the $a$s by the symbols $a_1, \ldots, a_k$.

\begin{definition} 
\index{renewal system!fragmentation of}
\index{shift space!fragmentation of}
\index{X*ak@$X_{a,k}$}
\index{F*ak@$F_{a,k}$}
Let $X = \X_\FF$ be a shift space over the alphabet $\AA$, let $a \in \AA$, let $a_1, \ldots, a_k \notin \AA$, and let $F_{a,k}$ be defined as above. Then the shift space $X_{a,k} = \X_{F_{a,k}^{-1}(\FF)}$ is said to be the shift obtained from $X$ by \emph{fragmenting} $a$ into $a_1, \ldots, a_k$.
\end{definition}

\index{L*@$L_{a,k}$}
Note that $X_{a,k}$ is an SFT if and only if $X$ is an SFT, and that $\BB( X_{a,k} ) = F_{a,k}^{-1}(\BB(X))$.
If $X$ is an irreducible sofic shift, then the left and right Fischer and Krieger covers of $X_{a,k}$ are obtained by replacing each edge labelled $a$ in the corresponding cover of $X$ by $k$ edges labelled $a_1, \ldots, a_k$.
Note that $X$ and $X_{a,k}$ are not generally conjugate or even flow equivalent: The full $n$-shift is, for instance, a fragmentation of the trivial shift with one element.
If $X = \X(L)$ is a renewal system, then $X_{a,k}$ is the renewal system generated by the list $L_{a,k} = F_{a,k}^{-1}(L)$. 

\begin{remark}
\label{rem_fragmentation}
\label{rem_sum_frag_commute}
Fragmentation is useful for constructing new renewal systems from known systems. Let $A$ be the symbolic adjacency matrix of the left Fischer cover of an SFT renewal system $\X(L)$ with alphabet $\AA$. Given $a \in \AA$ and $k \in \N$, define $f \colon \AA \to \N$ by $f(a) = k$ and $f(b) = 1$ for $b \neq a$. Extend $f$ to the set of finite formal sums over $\AA$ in the natural way and consider the integer matrix $f(A)$. Then $f(A)$ is the adjacency matrix of the underlying graph of the left Fischer cover of $\X(L_{a,k})$.
This construction will be very useful for constructing renewal systems with specific Bowen-Franks groups and determinants in the following.
For lists over disjoint alphabets, it follows immediately from the definitions that fragmentation and addition commute. 
\end{remark}

\subsection{Exotic determinants} \label{sec_rs_exotic}
An SFT is flow equivalent to a non-trivial full shift if and only if the Bowen-Franks group is cyclic and the determinant is negative. All the SFT renewal systems considered so far have been flow equivalent to full shifts. This is partially explained by the observation following Lemma \ref{lem_irr_replace} and Remark \ref{rem_nice_lists} which show that a generating list must be quite complicated in order to generate an SFT that is not flow equivalent to a full shift, but even outside these classes, it is difficult to find such renewal systems.

Appendix \ref{app_programs} describes an experimental method used to randomly generate and investigate SFT renewal systems, but even when considering only non-trivial irreducible lists it was surprisingly hard to find SFT renewal systems with positive determinants and/or non-cyclic Bowen-Franks groups. In fact, a lengthy search only revealed three such renewal systems -- all of which had cyclic Bowen-Franks groups and positive determinants. These examples will be presented in the following, and the results of the previous sections will be used to show how such examples can be used to generate a class of examples of renewal systems which are not flow equivalent to full shifts. 

\begin{example}\label{ex_pos_det_1}
\index{renewal system!with positive \\determinant}
Consider the irreducible generating list
\begin{displaymath}
L_1 = [ aa, bb, aaa, baa, bba, bbab, bbbbb ].
\end{displaymath}
Via the computer programs described in Appendix \ref{app_programs}, Proposition \ref{prop_extendable} can be used to prove that $\X(L_1)$ is an 8-step SFT with forbidden words 
\begin{displaymath}
\FF_1 = \{ abab, aabaaab, aabbbab, aabbbaaab, aabbbbbab\}.
\end{displaymath}
With this information, it is elementary (if somewhat tedious) to find the Fischer cover of $\X(L_1)$ which is shown in Figure \ref{fig_pos_det_1}. By Corollary \ref{cor_lfc_conj}, $\X(L_1)$ is conjugate to the edge shift of the underlying graph of the left Fischer cover, so the Bowen-Franks invariant of $\X(L_1)$ can be found by computing the determinant and Smith normal form of the corresponding edge shift. The determinant is $1$, so the Bowen-Franks group is trivial.

\begin{figure}
\begin{center}
\begin{tikzpicture}
  [bend angle=10,
   clearRound/.style = {circle, inner sep = 0pt, minimum size = 17mm},
   clear/.style = {rectangle, minimum width = 17 mm, minimum height = 6 mm, inner sep = 0pt},  
   greyRound/.style = {circle, draw, minimum size = 1 mm, inner sep =
      0pt, fill=black!10},
   grey/.style = {rectangle, draw, minimum size = 6 mm, inner sep =
      1pt, fill=black!10},
   white/.style = {rectangle, draw, minimum size = 6 mm, inner sep =
      1pt},
   to/.style = {->, shorten <= 1 pt, >=stealth', semithick}]
  
  \node[grey] (P1) at (-6,0) {};
  \node[grey] (P2) at (-4.5,0) {};
  \node[white] (P3) at (-3,0) {}; 
  \node[grey] (P4) at (-1.5,1.5) {};
  \node[white] (P5) at (-1.5,-1.5) {}; 
  \node[grey] (P6) at (0,1.5) {};
  \node[grey] (P7) at (0,-1.5) {$P_0$}; 
  \node[white] (P8) at (1.5,0) {};
  \node[grey] (P9) at (3,0) {}; 
  \node[white] (P10) at (4.5,0) {};     
  \node[white] (P11) at (6,0) {};

  \draw[to, loop below] (P1) to node[auto] {$a$} (P2); 
  \draw[to, bend left=60] (P1) to node[auto] {$a$} (P6);

  \draw[to, loop below] (P2) to node[auto] {$b$} (P2);
  \draw[to] (P2) to node[auto] {$b$} (P1);
  \draw[to, bend left=20] (P2) to node[auto] {$b$} (P4);  
  \draw[to, bend left=90] (P2) to node[auto] {$b$} (P10);      

  \draw[to] (P3) to node[auto] {$a$} (P2);
  \draw[to, bend left=30] (P3) to node[auto] {$a$} (P7);     
  \draw[->, shorten <= 1 pt, >=stealth', semithick] (P3) .. controls (-1.5,-4) and (1.5,-4) .. node[auto,swap] {$a$} (P9);   
    
  \draw[to] (P4) to node[auto] {$a$} (P3);      

  \draw[to] (P5) to node[auto,swap] {$b$} (P3);     
  \draw[to, bend right=45] (P5) to node[auto,swap] {$b$} (P11);

  \draw[to] (P6) to node[auto] {$a$} (P4);
  
  \draw[to] (P7) to node[auto,swap] {$b$} (P5);
  
  \draw[to] (P8) to node[auto] {$b$} (P6);    
  \draw[to] (P8) to node[auto,swap] {$b$} (P7);    
  
  \draw[to] (P9) to node[auto] {$b$} (P8);    
  \draw[to] (P10) to node[auto] {$b$} (P9);
  
  \draw[to] (P11) to node[auto] {$a$} (P10);  
  \draw[to, bend left=45] (P11) to node[auto] {$a$} (P8);  
  
\end{tikzpicture}
\end{center}
\caption[An exotic example.]{Left Fischer cover of the SFT renewal system $\X(L_1)$ considered in Example \ref{ex_pos_det_1}. The border points are coloured grey. $P_0$ is the universal border point and $bbab$ is a generator of $P_0$.} 
\label{fig_pos_det_1}
\end{figure}
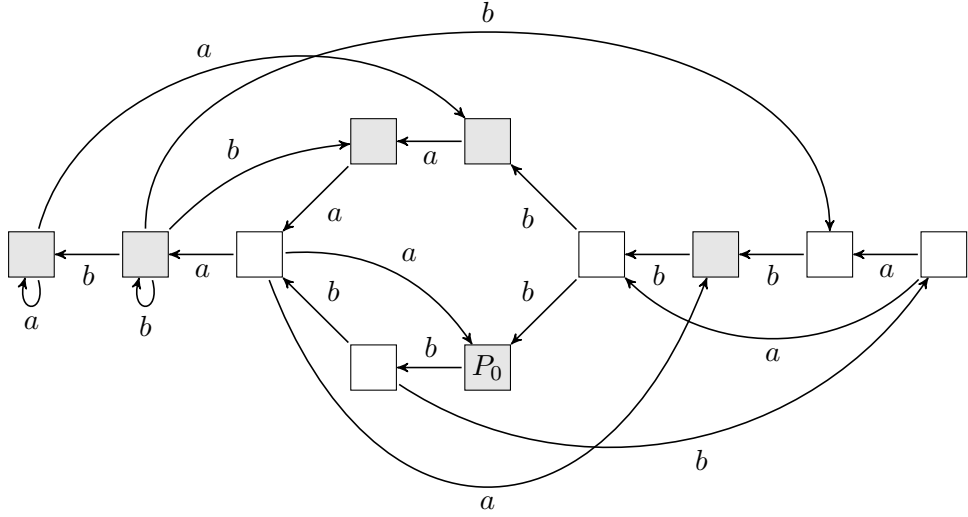
Note that $bbab$ is a strongly left-bordering word for $L_1$, and that $aabaa$ is a strongly right-bordering word for $L_1$, so $L_1$ is left-modular by Proposition \ref{prop_extendable_and_sft}. For a free letter $c$, the renewal system $\X(L_1 \cup \{ c \})$ is an SFT, and by Proposition \ref{prop_addition_modular}, the left Fischer cover of $\X(L_1 \cup \{ c \})$ can be obtained from the left Fischer cover of $\X(L_1)$ by drawing edges labelled $c$ from $P_0$ to each border point (including $P_0$). The symbolic adjacency matrix of this shift is 
\begin{displaymath}
A_1 =  \left( \begin{array}{c c c c c c c c c c c}
a & 0 & 0 & 0 & 0 & a & 0 & 0 & 0 & 0 & 0 \\
b & b & 0 & 0 & b & 0 & 0 & 0 & 0 & b & 0 \\
0 & a & 0 & 0 & 0 & 0 & a & 0 & a & 0 & 0 \\
0 & 0 & b & 0 & 0 & 0 & 0 & 0 & 0 & 0 & b \\
0 & 0 & a & 0 & 0 & 0 & 0 & 0 & 0 & 0 & 0 \\
0 & 0 & 0 & 0 & a & 0 & 0 & 0 & 0 & 0 & 0 \\
c & c & 0 & b & c & c & c & 0 & c & 0 & 0 \\
0 & 0 & 0 & 0 & 0 & b & b & 0 & 0 & 0 & 0 \\
0 & 0 & 0 & 0 & 0 & 0 & 0 & b & 0 & 0 & 0 \\
0 & 0 & 0 & 0 & 0 & 0 & 0 & 0 & b & 0 & 0 \\
0 & 0 & 0 & 0 & 0 & 0 & 0 & a & 0 & a & 0 
\end{array}
\right).
\end{displaymath}
Let $a_1, \ldots, a_{k_a}, b_1, \ldots , b_{k_b}, c_1, \dots, c_{k_c}$ be new symbols and let $L_1'$ be the list obtained by fragmenting $a$, $b$, and $c$ into these $k_a$, $k_b$, and $k_c$ fragments, respectively. Now, $\X(L_1')$ is a shift of finite type, and by Remark \ref{rem_fragmentation}, the left Fischer cover of $\X(L_1')$ can be obtained from the left Fischer cover of $\X(L_1 \cup \{ c\})$ by replacing each edge labelled $a$ by $k_a$ edges labelled $a_1, \ldots, a_{k_a}$ and so on. Hence, if $a$, $b$, and $c$ are replaced by positive integers in the symbolic adjacency matrix $A_1$ above, then the resulting matrix is the adjacency matrix of the underlying graph of the left Fischer cover of the SFT renewal system $\X(L_1')$. Using Maple, or a similar tool, it is easy to find that the determinant of $\Id - A_1$ is the following 10th order polynomial in $a$, $b$, and $c$ \begin{multline*}
p(a,b,c) = 
a^5 b^5  \\
-a^4 b^5
+a^3 b^6
-a^4 b^4 c
\qquad \qquad \qquad \qquad \qquad \qquad \qquad \qquad  \\
+a^5 b^3
+a^4 b^4
+a^4 b^3 c 
\qquad \qquad \qquad \qquad \qquad \qquad \qquad 
\\
-2 a^4 b^3
-a^4 b^2 c  
\qquad \qquad \qquad \qquad
\\
+a^4 b c
+a^3 b^3
+a^4 b^2
+a b^4 c 
 \\
\qquad \qquad
-a^4 b
-a^4 c
-a^3 b^2
-a b^3 c
-b^4 c   \\
\qquad \qquad \qquad \qquad \qquad
+a^3 b
+a^2 b^2
+a b^2 c
+b^3 c  \\
\qquad \qquad \qquad \qquad \qquad \qquad  \qquad
-a^2 b
-a b^2
-a b c
-b^2 c \\
\qquad \qquad \qquad \qquad \qquad \qquad \qquad \qquad \qquad \qquad
+a b
+a c
+b c \\
\qquad \qquad \qquad \qquad \qquad \qquad \qquad \qquad \qquad \qquad \qquad \qquad
-a
-b
-c \\
+1.
\end{multline*}
\noindent
Hence, every integer in the range of $p$ (over the positive integers) is the determinant of an SFT renewal system. In particular, there is no upper bound on the determinant of an SFT renewal system. 
%

The following two irreducible generating lists are included for completeness, and to give an idea of how complicated a generating list needs to be in order to generate an SFT renewal systems that is not flow equivalent to a full shift. Consider first \begin{displaymath}
L_2 = [ aa, ab, aaa, aba, bba, babba ].
\end{displaymath}
Using the computer programs described in Appendix \ref{app_programs}, Proposition \ref{prop_extendable} can be used to show that the renewal system $\X(L_2)$ is the 9-step SFT for which the forbidden words are 
\begin{multline*}
	\FF_2 = \{ bbbb, bbabaa, bbabbb, bbababa, bbaaabaa, bbaaabbb,\\
		         bbaabbaa, bbababbb, bbaaababa, bbaabbaba, bbaaababbb \}.
\end{multline*}
The corresponding determinant is $1$, so the Bowen-Franks group is the trivial group. For
\begin{displaymath}
L_3 = [ aabb, ba, abb, babab, babaa, babb, aaaaa, aa ],
\end{displaymath}
the programs from Appendix \ref{app_programs} can similarly use Proposition \ref{prop_extendable} to prove that $\X(L_3)$ is the 10-step SFT for which the forbidden words are 
\begin{multline*}
	\FF_3 = \{ bbbb, aabaaba, aabbaba, bbbaaba, bbabbaba, bbbaaaaba, \\
			 aababbaba, aabbaaaba. aabaaaaba, bbabbaaaba, aababbaaaba \}.
\end{multline*}
In this case, $\BF_+(\X(L_3)) = +\Z / 4\Z$.

The renewal systems $\X(L_2)$ and $\X(L_3)$ could be investigated in the same manner as $\X(L_1)$, but they are both of higher step than $\X(L_1)$, so it is not as easy to write down the left Fischer covers. As one might expect, it has not been possible to formulate a general hypothesis about when the determinant is positive based on the very limited information available in these examples. 
\end{example}

By adding and fragmenting the renewal systems from Example \ref{ex_pos_det_1} it is possible to obtain additional renewal systems with positive determinants, but this will not be done here, since the examples are not of particular interest. However, the method will be used several times in the following sections to construct complicated renewal systems from simpler building blocks.

\section{Entropy and flow equivalence}
\label{sec_rs_entropy}
As mentioned in Section \ref{sec_rs_exotic}, it is difficult to construct renewal systems with non-cyclic Bowen-Franks groups and/or positive determinants directly, so it is attractive to compute the range of the Bowen-Franks invariant over classes of renewal systems that are known to have rich structure in the hope that this will produce hitherto unseen values of the invariant. Hong and Shin \cite{hong_shin} have constructed a class $H$ of lists generating SFT renewal systems such that $\log \lambda$ is the entropy of an SFT if and only if there exists $L \in H$ with $h(\X(L)) = \log \lambda$, and this is arguably the most powerful general result known about the invariants of SFT renewal systems. In the following, the renewal systems generated by lists from $H$ will be classified up to flow equivalence, and this will produce a class of renewal systems that will eventually serve as building blocks in a construction of renewal systems with more complicated values of the Bowen-Franks invariant.   
The first section recalls the construction from \cite{hong_shin} in order to pave the way for a flow classification of $H$ in the following section.

\subsection{The class $H$}
\label{sec_rs_H}
The construction of the class $H$ of generating lists considered in \cite{hong_shin} will be modified slightly to take advantage of the fact that some of the details of the original construction are invisible up to flow equivalence. In particular, several words from the generating lists can be replaced by single symbols by using symbol reduction. Additionally, there are extra conditions on some of the variables in \cite{hong_shin} which will be omitted here since the larger class can be classified without extra work. The notation has been kept close to the notation from \cite{hong_shin} to allow comparison, but it has been simplified several places by the use of fragmentation.

\index{generating list!in $B$}
\index{B*@$B$}
Let $r \geq 2$ and let $n_1, \ldots, n_r,c_1, \ldots , c_r,d, N \in \N$, and let $W$ be the set consisting of the following words:
\begin{itemize}
\item $\alpha_i = \alpha_{i,1} \cdots \alpha_{i,n_1}$ for $1 \leq i \leq c_1$ 
\item $\tilde \alpha_i = \tilde \alpha_{i,1} \cdots \tilde \alpha_{i, n_1}$ for $1 \leq i \leq c_1$
\item $\gamma_{k,i_k} = \gamma_{k,i_k,1} \cdots \gamma_{k,i_k,n_k}$ for $2 \leq k \leq r$ and $1 \leq i_k \leq c_k$
\item $\alpha_{i_1} \gamma_{2,i_2} \cdots \gamma_{r,i_r} \beta_l^N$
         for $1 \leq i_j \leq c_j$ and $1 \leq l \leq d$
\item $\beta_l^N \tilde \alpha_{i_1} \gamma_{2,i_2} \cdots \gamma_{r,i_r}$
         for $1 \leq i_j \leq c_j$ and $1 \leq l \leq d$.
\end{itemize}
The set of generating lists of this form will be denoted $B$. In \cite{hong_shin} such renewal systems are used as basic building blocks in the construction of renewal systems with specific entropies. 

\begin{remark}
\label{rem_R_def}
\index{generating list!in $R$}
\index{R*@$R$}
Lemmas \ref{lem_a_to_aa} and \ref{lem_replace}
can be used to reduce the words $\alpha_i$, $\tilde \alpha_i$, $\gamma_{k,i_k}$, and $\beta_l^N$ to single letters, so up to flow equivalence, the list $W \in B$ considered above can be replaced by the list $W'$ consisting of the one-letter words $\alpha_i$, $\tilde \alpha_i$, and $\gamma_{k,i}$ as well as the words
\begin{itemize}
\item $\alpha_{i_1} \gamma_{2,i_2} \cdots \gamma_{r,i_r} \beta_l$
         for $1 \leq i_j \leq c_j$ and $1 \leq l \leq d$
\item $\beta_l \tilde \alpha_{i_1} \gamma_{2,i_2} \cdots \gamma_{r,i_r}$
         for $1 \leq i_j \leq c_j$ and $1 \leq l \leq d$.
\end{itemize}
Furthermore, if 
\begin{equation} \label{eq_reduced_list}
L = \{ \alpha, \tilde \alpha, 
 \alpha \gamma_2 \cdots \gamma_r \beta, \beta \tilde \alpha \gamma_2 \cdots \gamma_r \} \cup \{ \gamma_k \mid 2 \leq k \leq r\}, 
\end{equation}
then $\X(W')$ can be obtained from $\X(L)$ by fragmenting $\alpha$ to $\alpha_1, \ldots,  \alpha_{c_1}$, $\beta$ to $\beta_1, \ldots , \beta_l$ and so on.
Let $R$ be the set of generating lists of the form given in Equation \ref{eq_reduced_list}. 
\end{remark}

\index{H@$H$}\index{generating list!in $H$}
Next consider generating lists $W_1, \ldots , W_m \in B$ with disjoint alphabets, and let $W = \bigcup_{j=1}^m W_j$. 
Let $\tilde W$ be a finite set of words that do not share any letters with each other or with the words from $W$, and consider the generating list $W \cup \tilde W$. Let $\tilde H$ be the set of generating lists that can be constructed in this manner.

\begin{remark} \label{rem_H_letter_frag}
If $W \cup \tilde W \in \tilde H$ as above, then Lemma \ref{lem_replace} shows that $\X(W \cup \tilde W)$ is flow equivalent to the renewal system generated by the union of $W$ and $\lvert \tilde W \rvert$ new letters, i.e.\ $\X(W \cup \tilde W)$ is flow equivalent to a fragmentation of $\X(W \cup \{ a \})$ when $a \notin \AA(\X(W))$.
\end{remark}

\begin{lem}[{Hong and Shin \cite{hong_shin}}]
\label{lem_hs}
Let $\mu$ be a Perron  number. Then there exists $\tilde L \in \tilde H$ such that $\X(\tilde L)$ is an SFT and $h(\X(\tilde L)) = \log \mu$.
\end{lem}

Consider a generating list $\tilde L \in \tilde H$ and $p \in \N$. For each letter $a \in  \AA( \X(\tilde L))$, introduce new letters $a_1, \dots , a_p \notin \AA( \X(\tilde L))$, and let $L$ denote the generating list obtained by replacing each occurrence of $a$ in $\tilde L$ by the word $a_1\cdots a_p$.
Let $H$ denote the set of generating list that can be obtained from $\tilde H$ in this manner.

\begin{remark}
\label{rem_H_tilde_fe}
If $L$ is obtained from $\tilde L \in \tilde H$ as above, then $\X(L) \FE \X(\tilde L)$ since the modification can be achieved using symbol expansion of each $a \in \AA( \X(\tilde L))$.
\end{remark}

\begin{thm}[{Hong and Shin \cite{hong_shin}}]
Let $\lambda$ be a weak Perron number. 
Then there exists $L \in H$ such that $\X(L)$ is an SFT and $h(\X(L)) = \log \lambda$.
\label{thm_hs}\index{renewal system!entropy of} 
\end{thm}

\begin{proof}
By assumption, there exists $p \in \N$ such that $\mu = \lambda^p$ is a Perron number. Use Lemma \ref{lem_hs} to find $\tilde L \in \tilde H$ such that $h( \X(\tilde L) ) = \log \mu$. Next replace each letter $a$ in the alphabet by a word $a_1 \cdots a_p$ to obtain the generating list of an SFT renewal system with entropy $\log \lambda$.
\end{proof}

\subsection{Flow classification of $H$}
\label{sec_rs_H_classification}
The aim of this section is to classify the class of renewal systems generated by lists in $H$ up to flow equivalence. The idea is to prove that the building blocks in the class $R$ introduced in Remark \ref{rem_R_def} are left-modular, and to construct the Fischer covers of the corresponding renewal systems. As the following lemma shows, this will allow a classification of the renewal systems generated by lists from $H$ via addition and fragmentation.

\begin{lem}
\label{lem_entropy_fe}
For each $L \in H$, there exist $L_1, \ldots , L_m \in R$ such that $\X(L)$ is flow equivalent to a fragmentation of $\X ( \bigcup_{j=0}^m L_j )$, where $L_0 = \{ a \}$ for some $a$ that does not occur in $L_1, \ldots , L_m$. 
\end{lem}

\begin{proof}
This follows from Remarks \ref{rem_sum_frag_commute},  \ref{rem_R_def}, \ref{rem_H_letter_frag}, and \ref{rem_H_tilde_fe}.
\end{proof}

The following lemma determines the left Fischer covers of the renewal systems generated by lists from $R$, so that Proposition \ref{prop_addition_modular} and Lemma \ref{lem_entropy_fe} can be used to find the left Fischer covers of the renewal systems generated by lists from $H$. 

\begin{lem}
\label{lem_entropy_lfc}
If $L \in R$, then $L$ is left-modular and the left Fischer cover of $\X(L)$ is the labelled graph shown in Figure \ref{fig_entropy_lfc}.
\end{lem}

\begin{proof}
Let
\begin{equation}
\label{eq_R_L}
L = \{ \alpha, \tilde \alpha,  
 \alpha \gamma_2 \cdots \gamma_r \beta, \beta \tilde \alpha \gamma_2 \cdots \gamma_r \} \cup \{ \gamma_k \mid 2 \leq k \leq r\} \in R.
\end{equation}
The word $\alpha \gamma_2 \cdots \gamma_r \beta \beta \tilde \alpha \gamma_2 \cdots \gamma_r$ is strongly left- and right-bordering, so $L$ is left- and right-modular by Lemma \ref{lem_strongly_bordering_implies_modular}. Let $P_0 = P_0(L)$.
If $x^+ \in \X(L)^+$ does not have a suffix of a product of the generating words $\alpha \gamma_2 \cdots \gamma_r \beta$ and $\beta \tilde \alpha \gamma_2 \cdots \gamma_r$ as a prefix, then $x^+$ is strongly left-bordering, so $P_\infty(x^+) = P_0$. Hence, it is sufficient to consider right-rays that do have such a prefix.

Consider first $x^+ \in \X(L)^+$ such that $\beta x^+ \in \X(L)^+$. The letter $\beta$ must come from either $\alpha \gamma_2 \cdots \gamma_r \beta$ or $\beta \tilde \alpha \gamma_2 \cdots \gamma_r$, so the beginning of a partitioning of $\beta x^+$ must be either empty or equal to $\alpha \gamma_2 \cdots \gamma_r$.
Assume first that every partitioning of $\beta x^+$ has beginning $\alpha \gamma_2 \cdots \gamma_r$ (i.e.\ that $\tilde \alpha \gamma_2 \cdots \gamma_r$ is not a prefix of $x^+$).
In this case, $\beta x^+$ must be preceded by $\alpha \gamma_2 \cdots \gamma_r$, and the corresponding predecessor sets are: 
\begin{align}
\label{eq_entropy_predecessor_sets}
  P_\infty( \alpha \gamma_2 \cdots \gamma_r \beta x^+ ) &= P_0 \nonumber \\ 
  P_\infty( \gamma_2 \cdots \gamma_r \beta x^+ ) &= P_0 \alpha = P_1 \nonumber \\ 
  &\;\, \vdots \\
  P_\infty( \gamma_r \beta x^+ ) &= P_0\alpha \gamma_2 \cdots \gamma_{r-1}  = P_{r-1} \nonumber \\
  P_\infty( \beta x^+ ) &= P_0\alpha \gamma_2 \cdots \gamma_{r-1}\gamma_{r}  = P_{r}. \nonumber
\end{align}
Assume now that there exists a partitioning of  $\beta x^+$ with empty beginning (e.g.\  $x^+ = \beta \tilde \alpha \gamma_2 \cdots \gamma_r^\infty$).
The first word used in such a partitioning must be $\beta \tilde \alpha \gamma_2 \cdots \gamma_r$. Replacing this word by the concatenation of the generating words $\alpha \gamma_2 \cdots \gamma_r \beta$, $\tilde \alpha, \gamma_2, \ldots, \gamma_r$ creates a partitioning of $\beta x^+$ with beginning $\alpha \gamma_2 \cdots \gamma_r $, so in this case, the predecessor sets are
\begin{align*}
P_\infty( \alpha \gamma_2 \cdots \gamma_r \beta x^+ ) &= P_0 \\
P_\infty( \gamma_2 \cdots \gamma_r \beta x^+ ) &= P_0 \cup P_0\alpha = P_0\\
 &\;\, \vdots  \\
 P_\infty(  \gamma_r \beta x^+) &= P_0 \cup P_0\alpha \gamma_2 \cdots \gamma_{r-1}  = P_0 \\ 
P_\infty(  \beta x^+) &= P_0 \cup P_0\alpha \gamma_2 \cdots \gamma_{r-1}\gamma_{r}  = P_0.
\end{align*}
The argument above proves that there are no right-rays such that every partitioning of $\beta x^+$ has empty beginning.

It only remains to investigate right-rays that  have a suffix of $\beta \tilde \alpha \gamma_2 \cdots \gamma_r$ as a prefix.
A partitioning of a right-ray $\gamma_r x^+$ may have empty beginning (e.g.\  $x^+ = \gamma_r^\infty$), beginning $\alpha \gamma_2 \cdots \gamma_{r-1}$ (e.g.\ $x^+ =  \beta \beta \tilde \alpha \gamma_2 \cdots \gamma_r \cdots$ or $x^+ = \beta \tilde \alpha \gamma_2 \cdots \gamma_r^\infty$), or beginning $\beta \tilde \alpha \gamma_2 \cdots \gamma_{r-1}$ (e.g.\  $x^+ = \gamma_r^\infty)$. Note that there is a partitioning with empty beginning if and only if there is a partitioning with beginning $\beta \tilde \alpha \gamma_2 \cdots \gamma_{r-1}$.  
If there exists a partitioning of $\gamma_r x^+$ with beginning $\alpha \gamma_2 \cdots \gamma_{r-1}$, then $\beta$ must be a prefix of $x^+$, so the right-ray $\gamma_r x^+$ has already been considered above.
Hence, it suffices to consider the case where there exists a partitioning of $\gamma_r x^+$ with empty beginning and a partitioning with beginning $\beta \tilde \alpha \gamma_2 \cdots \gamma_{r-1}$ but no partitioning with beginning $\alpha \gamma_2 \cdots \gamma_{r-1}$. In this case, the predecessor sets are 
\begin{align*}
  P_\infty( \gamma_{r} x^+ ) &= P_0 \cup P_0\beta \tilde \alpha \gamma_2 \cdots \gamma_{r-1} = P_{2r}\\
&\;\,\vdots \\
P_\infty( \gamma_2 \cdots \gamma_{r} x^+ ) &= P_0 \cup P_0\beta \tilde \alpha = P_{r+2}\\
 P_\infty( \tilde \alpha \gamma_2 \cdots \gamma_{r} x^+ ) &= P_0 \cup P_0\beta = P_{r+1}\\
P_\infty( \beta \tilde \alpha \gamma_2 \cdots \gamma_{r} x^+ ) &= P_0 
  \cup P_0\alpha \gamma_2 \cdots \gamma_{r} = P_0.
\end{align*}

%
%
%
%

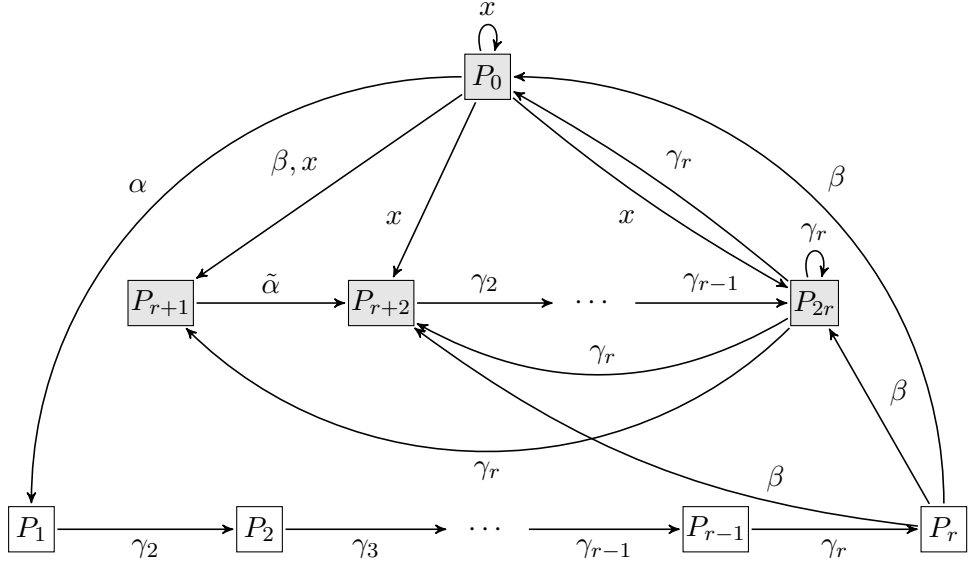
\begin{figure}
\begin{center}
\begin{tikzpicture}
  [bend angle=5,
   clearRound/.style = {circle, inner sep = 0pt, minimum size = 17mm},
   clear/.style = {rectangle, minimum width = 10 mm, minimum height = 6 mm, inner sep = 0pt},  
   greyRound/.style = {circle, draw, minimum size = 1 mm, inner sep =
      0pt, fill=black!10},
  grey/.style = {rectangle, draw, minimum size = 6 mm, inner sep =
    1pt, fill=black!10},
  white/.style = {rectangle, draw, minimum size = 6 mm, inner sep =
    1pt},
   to/.style = {->, shorten <= 1 pt, shorten >= 1 pt, >=stealth', semithick}]
  
 \node[grey] (P0) at (0,0) {$P_0$};

 \node[white] (P1) at (-6,-6) {$P_1$};
 \node[white] (P2) at (-3,-6) {$P_2$};
 \node[clear] (dots1) at (0,-6) {$\cdots$};
  \node[white] (Pr-1) at (3,-6) {$P_{r-1}$};
 \node[white] (Pr) at (6,-6) {$P_r$};

 \node[grey] (Pr+1) at (-4.3,-3) {$P_{r+1}$};
 \node[grey] (Pr+2) at (-1.4,-3) {$P_{r+2}$};
 \node[clear] (dots2) at (1.4,-3) {$\cdots$};
 \node[grey] (P2r) at (4.3,-3) {$P_{2r}$};

  \draw[to, loop above] (P0) to node[auto] {$x$} (P0);

  \draw[to, bend right=45] (P0) to node[auto,swap] {$\alpha$} (P1);
  \draw[to] (P1) to node[auto,swap] {$\gamma_2$} (P2);
  \draw[to] (P2) to node[auto,swap] {$\gamma_3$} (dots1);
  \draw[to] (dots1) to node[auto,swap] {$\gamma_{r-1}$} (Pr-1);
  \draw[to] (Pr-1) to node[auto,swap] {$\gamma_{r}$} (Pr);  
  \draw[to,bend right=45] (Pr) to node[auto,swap] {$\beta$} (P0);
 
  \draw[to, bend left=15] (Pr) to node[near start, above] {$\beta$} (Pr+2);
  \draw[to] (Pr) to node[auto,swap] {$\beta$} (P2r);
  
   \draw[to] (P0) to node[auto,swap] {$\beta,x$} (Pr+1);   
  \draw[to] (Pr+1) to node[auto] {$\tilde \alpha$} (Pr+2);
  \draw[to]  (Pr+2) to node[auto] {$\gamma_{2}$} (dots2); 
  \draw[to] (dots2) to node[auto] {$\gamma_{r-1}$} (P2r); 
  \draw[to,bend right] (P2r) to node[auto,swap] {$\gamma_{r}$} (P0);
  \draw[to,loop above] (P2r) to node[auto] {$\gamma_{r}$} (P2r); 
  \draw[to,bend left=45] (P2r) to node[auto] {$\gamma_{r}$} (Pr+1);
  \draw[to,bend left=30] (P2r) to node[auto,swap] {$\gamma_{r}$} (Pr+2);   

  \draw[to] (P0) to node[auto, swap, near end] {$x$} (Pr+2);
  \draw[to,bend right] (P0) to node[auto,swap] {$x$} (P2r);
    
\end{tikzpicture}
\end{center}
\caption[Building blocks for achieving entropies.]{Left Fischer cover of $\X(L)$ for $L$ defined in Equation \ref{eq_R_L}.
An edge labelled $x$ from a vertex $P$ to a vertex $Q$ represents a collection of edges from $P$ to $Q$ such that $Q$ receives an edge with each label from the set 
$\bigcup_{2\leq j \leq r}  \{ \gamma_j  \} \cup \{  \alpha, \tilde \alpha \}$, i.e.\ the collection fills the gaps left by the edges which are labelled explicitly. The border points are coloured grey.} 
\label{fig_entropy_lfc}
\end{figure}

Now all right-rays have been investigated, so there are exactly $2r +1$ vertices in the left Krieger cover of $\X(L)$. The vertex $P_0$ is the universal border point, and the vertices $P_{r+1}, \ldots, P_{2r}$ are border points, while none of the vertices $P_1, \ldots, P_r$ are border points. 

The equations above give the information needed to draw the left Fischer cover which is shown in Figure \ref{fig_entropy_lfc}. 
To check that this is indeed the left Fischer cover, note that it is irreducible and left-resolving, and that every edge can be inferred from the equations above.
To see that there are no edges missing, note that the vertices $P_0, P_{r+1}, \ldots, P_{2r}$ all receive edges with every letter in the alphabet, while for each $1 \leq i \leq r$, the vertex $P_i$ receives exactly the edges it should according to Equation \ref{eq_entropy_predecessor_sets}.
\end{proof}

\noindent In \cite{hong_shin} it is proved that all renewal systems in the class $B$ are SFTs. That proof will also work for the related class $R$ considered here, but the following lemma shows that the result follows easily from the structure of the left Fischer cover constructed above.

\begin{lem}
\label{lem_entropy_sft}
For each $L \in R$, $\X(L)$ is an SFT.
\end{lem}

\begin{proof}
Let $L$ be defined as in Equation \ref{eq_R_L}. By Corollary \ref{cor_lfc_conj}, $\X(L)$ is an SFT if and only if the covering map of the left Fischer cover $(F, \LL)$ is injective. Assume that $\lambda, \mu$ are biinfinite paths in $F$ such that $\LL(\lambda) = \LL(\mu) = x \in X(L)$.
If there is an upper bound $l$ on the set $\{ i \in \Z \mid x_i = \gamma_r \}$ then $s(\lambda_j) = s(\mu_j) = P_0$ for all $j > l$. Since $(F, \LL)$ is left-resolving, this implies that $\lambda = \mu$.

Assume now that there is no upper bound on the set.
If $x_i = \gamma_r$, then $s(\lambda_i) = s(\mu_i) = P_{2r}$
unless $x_{i-r}\cdots x_{i-1} = \alpha \gamma_2 \ldots \gamma_{r-1}$, in which case $s(\lambda_{i-r}) = s(\mu_{i-r}) = P_0$. The left Fischer cover is left-resolving, so either way it follows that $s(\lambda_{j}) = s(\mu_{j})$ for all $j \leq i-r$. By assumption, $i$ can be arbitrarily large, so $\lambda = \mu$.
\end{proof}

\begin{lem}
\label{lem_entropy_cyclic}
Let $L \in R$ and let $X_f$ be a renewal system obtained from $\X(L)$ by fragmentation. Then the Bowen-Franks group of $X_f$ is cyclic, and the determinant is given by Equation \ref{eq_det_R}.
\end{lem}

\begin{proof}
Let $L \in R$ be defined by Equation \ref{eq_R_L}. The symbolic adjacency matrix of the left Fischer cover of $\X(L)$ (shown in Figure \ref{fig_entropy_lfc}) is 
\begin{displaymath}
A = \left( \begin{array}{c | c c c c c  | c c c c c c}
  \gamma & \alpha & 0 & \cdots  & 0 & 0 & \gamma+\beta & \tilde \alpha' & \gamma'_2 & \cdots  & \gamma'_{r-2} & \gamma'_{r-1} \\
  \hline
  0         & 0         & \gamma_2 & \cdots & 0 & 0    & & & & & &\\
  0         & 0         & 0     &            & 0 & 0    & & & & & &\\
  \vdots & \vdots &        & \ddots &    & \vdots      & & & & 0 & &\\ 
  0         & 0         & 0     &            & 0 &  \gamma_r & & & & & &\\ 
  \beta         & 0         & 0     & \cdots & 0 & 0    & 0  & \beta & \beta & \cdots & \beta & \beta \\   
  \hline
  0         & & &    & & & 0 &  \tilde \alpha & 0     & \cdots & 0 & 0 \\
  0         & & &    & & & 0 & 0     &  \gamma_2 &            & 0 & 0 \\
  0         & & &    & & & 0 & 0     & 0     &            & 0 & 0 \\
  \vdots & & & 0 & & & \vdots   &        &        & \ddots &    & \vdots   \\
  0         & & &    & & & 0 & 0     & 0     &            & 0 &  \gamma_{r-1} \\
   \gamma_r      & & &    & & &  \gamma_r &  \gamma_r     &  \gamma_r     & \cdots &  \gamma_r &  \gamma_r \\   
 \end{array} \right),
 \end{displaymath} 
where $ \gamma = \alpha+\tilde \alpha+ \sum_{k=2}^{r-1}  \gamma_k$, $\tilde \alpha' =  \gamma -  \tilde \alpha$,  and $\gamma'_k =  \gamma -  \gamma_k$. Index the rows and columns of $A$ by $0,\ldots, 2r$ in correspondence with the names used for the vertices above, and note that the column sums of the columns $0, r+1, \ldots, 2r$ are all equal to $\alpha+\tilde\alpha +\beta+ \sum_{k=2}^{r}  \gamma_k$.

If $X_f$ is a fragmentation of $\X(L)$, then the (non-symbolic) adjacency matrix $A_f$ of the underlying graph of the left Fischer cover of $X_f$ is obtained from $A$ by replacing $\alpha, \tilde \alpha, \beta, \gamma_2, \ldots , \gamma_r$ by positive integers (see Remark \ref{rem_fragmentation}).
To put $\Id - A_f$ into Smith normal form, begin by adding each row from number $r+1$ to $2r-1$ to the first row, and subtract the first column from column $r+1, \ldots, 2r$ to obtain
\begin{multline*}
\Id - A_f \rightsquigarrow \\
\left( \begin{array}{c | c c c c c  | c c  c c c}
  1- \gamma & - \alpha & 0 & \cdots  & 0 & 0 & -\beta & 0 & \cdots  & 0 & -1 \\
  \hline
  0         & 1         & -\gamma_2 & \cdots & 0 & 0    & & & & &\\
  0         & 0         & 1     &              & 0 & 0    & & & & &\\
  \vdots & \vdots &        & \ddots &    & \vdots      & &  & 0 & &\\ 
  0         & 0         & 0     &            & 1 & - \gamma_r &  & & & &\\ 
  -\beta         & 0         & 0     & \cdots & 0 & 1    & \beta  & 0 & \cdots & 0 & 0\\   
  \hline
  0         & & &    & & & 1 & -\tilde \alpha & \cdots & 0 & 0 \\
  0         & & &    & & & 0 & 1     &     & 0 & 0 \\
  \vdots & & & 0 & & & \vdots   &         & \ddots &    & \vdots   \\
  0         & & &    & & & 0 & 0         &            & 1 & - \gamma_{r-1} \\
  - \gamma_r      & & &    & & & 0 & 0         & \cdots & 0 &  1 \\   
 \end{array} \right).
 \end{multline*} 
Using row and column addition, this matrix can be further reduced to 
\begin{equation}
\label{eq_entropy_block}
\left( \begin{array}{c | c c c   |   c c c}
  1- \gamma - b & 0 &  \cdots   & 0 & 0 & \cdots  &  t \\
  \hline
  0         & 1         &   \cdots  & 0             &  & &\\
  \vdots & \vdots &   \ddots  & \vdots      & &  0 &\\ 
  0         & 0         &   \cdots  & 1             &   &   & \\   
  \hline
  0         & &    &    & 1          & \cdots   & 0 \\
  \vdots & & 0 &    & \vdots  &  \ddots   & \vdots   \\
  -\gamma_r         & &    &    &  0         &  \cdots   &  1 \\   
 \end{array} \right)
 \begin{array}{l}
\\
 b = \alpha \beta \gamma_2 \cdots \gamma_r \\
 \\
 t = \tilde \alpha \gamma_2 \cdots \gamma_{r-1} (b - \beta) -1. \\
 \\
 \end{array}
\end{equation}
Hence, the Bowen-Franks group of $X_f$ is cyclic, and the determinant is 
\begin{multline} \label{eq_det_R}
\det(\Id - A) = 1- \gamma - b + \tilde \alpha \gamma_2 \cdots \gamma_r (b - \beta) - \gamma_r \\
                   = 1 - \alpha - \tilde \alpha - \sum_{k=2}^r \gamma_k - (\alpha+\tilde \alpha) \beta \gamma_2 \cdots \gamma_r  + \alpha \tilde \alpha \beta(\gamma_2 \cdots \gamma_r)^2.   
\end{multline}
\end{proof}

\noindent
The number in Equation \ref{eq_det_R} may be either positive or negative depending on the values of the variables. In the following section, this will be used to construct a class of renewal systems for which the determinants attain all values in $\Z$.

Now it possible to classify the renewal systems considered in \cite{hong_shin} up to flow equivalence:

\begin{thm}
For each $L \in H$, the renewal system $\X(L)$ has cyclic Bowen-Franks group and determinant given by Equation \ref{eq_entropy_det}. \label{thm_H_classification}
\end{thm}

\begin{proof}
By Lemma \ref{lem_entropy_fe}, there exist $L_1, \ldots , L_m \in R$, $L_0 = \{a \}$ for some letter $a$ that does not appear in any of the lists,  and a fragmentation $Y_f$ of $Y = \X( \bigcup_{j=0}^m L_j )$ such that $Y_f \FE \X(L)$. For $1 \leq j \leq m$, let 
\begin{displaymath}
  L_j = \{ \alpha_j, \tilde \alpha_j, \gamma_{j,k}, 
 \alpha_j \gamma_{j,2} \cdots \gamma_{j,r_j} \beta_j, \beta_j \tilde \alpha_j \gamma_{j,2} \cdots \gamma_{j,r_j} \mid 2 \leq k \leq r_j\} , r_j \in \N.
\end{displaymath}
Each $L_j$ is left-modular by Lemma \ref{lem_entropy_sft},
so $Y$ is an SFT, and the left Fischer cover of $Y$ can be constructed using the technique from Section \ref{sec_rs_add}: Identify the universal border points in the left Fischer covers of $\X(L_0), \ldots , \X(L_m)$, and draw additional edges to the border points corresponding to the edges terminating at the universal border points in the individual left Fischer covers. Hence, the symbolic adjacency matrix $A$ of the left Fischer cover of $Y$ is 
\begin{align*}
&A = \\
&\left( \!\!\! \begin{array}{c | c | c c c c  | c c c c | c | c }
  \gamma &             & \alpha_j & 0 & \cdots  & 0 &  \gamma+\beta_ j & \tilde \alpha'_{ j} & \cdots   & \gamma'_{ j,{r_ j-1}}  & \cdots & \gamma'_{i,k}\\
   \hline
                &\ddots    &             & & & & & & & & &  \\
   \hline
  0            &              & 0           & \gamma_{ j,2} & \cdots  & 0    & & & & & &\\
  0            &              & 0           & 0     &                & 0    & & & & & & \\
  \vdots    &              & \vdots   &        & \ddots      & \vdots     & &  & 0 & & &\\ 
  0            &              & 0           & 0     &                & \gamma_{ j,r} & & & &  & & \\ 
  \beta_ j  &              & 0           & 0     & \cdots  & 0    & 0  & \beta_ j & \cdots & \beta_ j & & \beta_ j \\   
  \hline
  0           &               &             &    & & & 0 & \tilde \alpha_j &  \cdots & 0 &  &\\
  0           &               &             &    & & & 0 & 0     &             & 0 & & \\
  \vdots   &               &             & & 0 & & \vdots   &        &      \ddots &   \vdots   & & \\
  0           &               &             &    & & & 0 & 0     &             &  \gamma_{ j,{r_ j-1}}  & &\\
  \gamma_{ j,r_ j}  & &             &    & & & \gamma_{ j,r_ j} &  \gamma_{ j,r_ j}  & \cdots &  \gamma_{ j,r_ j}  & & \gamma_{ j,r_ j}\\  
  \hline
   &                           & &    & & &  &     &  & &  \ddots & \\     
  \hline 
   &                           & &    & & &  &     &  & &   & \ddots      
\end{array} \!\!\! \right).
\end{align*} 
where $1 \leq j \leq m$,
\begin{align*}
\gamma &= a + \sum_{j=1}^m \Big( \alpha_j+\tilde\alpha_j + \sum_{k=2}^{r_j-1} \gamma_{j,k} \Big), \\
\tilde \alpha_j' &= \gamma - \tilde \alpha_j, \textrm{ and } \\
\gamma_{j,k}' &= \gamma - \gamma_{j,k}.
\end{align*}
This matrix has blocks of the same form as in the $m=1$ case considered in Lemma \ref{lem_entropy_sft}. The $j$th block is shown together with the first row and column of the matrix -- which contain the connections between the $j$th block and the universal border point $P_0$ -- and together with an extra column representing an arbitrary border point in a different block. Such a border point in another block will receive edges from the $j$th block with the same sources and labels as the edges that start in the $j$th block and terminate at the universal border point $P_0$. 

Let $Y_f$ be a fragmentation of $Y$. Then the (non-symbolic) adjacency matrix $A_f$ of the underlying graph of the left Fischer cover of $Y_f$ is obtained by replacing the entries of $A$ by non-negative integers. The entry $a$ can also be replaced by $0$. 
In order to put $\Id - A_f$ into Smith normal form, first add rows $r_j+1$ to $2r_j-1$
in the $j$th block to the first row for each $j$, and then subtract the first column from every column corresponding to a border point in any block. This will remove the entries corresponding to edges between the individual blocks. More precisely, $\Id-A_f$ is transformed into the matrix:
\begin{displaymath}
\left(\!\!\! \begin{array}{c | c | c c c c  | c c c c | c  }
  1-\gamma &    & -\alpha_j & 0 & \cdots  & 0 &  -\beta_j & 0 & \cdots   & -1  & \\
  \hline
   &\!\ddots\!       &         &  &  &   & & & & &  \\
   \hline
  0  &       & 1         & -\gamma_{j,2} & \cdots  & 0    & & & & & \\
  0   &      & 0         & 1     &                & 0    & & & & &  \\
  \vdots& & \vdots &        & \ddots      & \vdots     & &  & 0 & & \\ 
  0  &       & 0         & 0     &                & -\gamma_{j,r_j} & & & &  &  \\ 
  -\beta_j  &       & 0         & 0     & \cdots  & 1    & \beta_j  & 0 & \cdots & 0 &  \\   
  \hline
  0   &    & &    & & & 1 & -\tilde \alpha_j &  \cdots & 0 &  \\
  0    &   & &    & & & 0 & 1     &             & 0 &  \\
  \vdots & & & & 0 & & \vdots   &        &      \ddots &   \vdots   &  \\
  0   &    & &    & & & 0 & 0     &             &  -\gamma_{j,{r_j-1}}  & \\
  -\gamma_{j,r_j}  & & &    & & & 0 & 0 & \cdots &  1  & \\  
  \hline
   &     & &    & & &  &     &  & & \! \ddots  \\     
 \end{array} \!\!\! \right)
 \end{displaymath} 
By using row and column addition, and by disregarding rows and columns where the only non-zero entry is a diagonal $1$, $\Id - A$ can be further reduced to 
\begin{displaymath}
\left( \begin{array}{c c c c c c}
S                           & t_1  & t_2  & \cdots & t_m          \\
-\gamma_{1,r_1}   & 1     & 0     &            &  0               \\
-\gamma_{2,r_2}   & 0     & 1     &            &  0               \\
 \vdots                   &        &        & \ddots &  \vdots       \\
-\gamma_{m,r_m} & 0     & 0     & \cdots &  1 \\
\end{array}\right)
\quad \begin{array}{l}
b_j = \alpha_j \beta_j \gamma_{j,2} \cdots \gamma_{j,r_j} \\ \\
t_j = \tilde \alpha_j \gamma_{j,2} \cdots \gamma_{j,r-1} (b_j-\beta_j) -1 \\ \\
S = 1 - \gamma - \sum_{j=1}^m b_j
\end{array}
.
\end{displaymath}
Hence, the Bowen-Franks group is cyclic and the determinant is 
\begin{equation}
\label{eq_entropy_det}
\det( \Id - A_f ) = 1 - \gamma + \sum_{j=1}^m \left(  \gamma_{j,r_j} t_j  - b_j \right). 
\end{equation}
\end{proof}

\noindent With Theorem \ref{thm_hs}, this gives the following result.

\begin{cor}
When $\log \lambda$ is the entropy of an SFT, there exists an SFT renewal system $X(L)$  with cyclic Bowen-Franks group such that $h(\X(L)) = \log \lambda$.
\end{cor}

\section{The quest for the range of the Bowen-Franks invariant}
\label{sec_rs_range}
All renewal systems considered until now have had cyclic Bowen-Franks groups, and most of them have been flow equivalent to full shifts.
Some of the renewal systems considered in Sections \ref{sec_rs_exotic} and \ref{sec_rs_H_classification} have had positive determinants, but it is still unclear whether every integer can be realised as the determinant of a renewal system. This will be remedied in the following,  were it will be proved that the range of the Bowen-Franks invariant over the class of SFT renewal systems contains a large class of pairs of signs and finitely generated abelian groups. The class is, however, not general enough to show that $\pm G$ is the signed Bowen-Franks group of some SFT renewal system for every finitely generated abelian group $G$, so Adler's question remains open both for conjugacy and for flow equivalence.

\subsection{Determinants}

\begin{figure}
\begin{center}
\begin{tikzpicture}
  [bend angle=10,
   clearRound/.style = {circle, inner sep = 0pt, minimum size = 17mm},
   clear/.style = {rectangle, minimum width = 17 mm, minimum height = 6 mm, inner sep = 0pt},  
   greyRound/.style = {circle, draw, minimum size = 1 mm, inner sep =
      0pt, fill=black!10},
   grey/.style = {rectangle, draw, minimum size = 6 mm, inner sep =
      1pt, fill=black!10},
   white/.style = {rectangle, draw, minimum size = 6 mm, inner sep =
      1pt},
   to/.style = {->, shorten <= 1 pt, >=stealth', semithick}]
  
  \node[grey] (P0) at (0,0) {$P_0$};
  \node[white] (P1) at (-1,-2) {$P_1$}; 
  \node[white] (P2) at (4,-2) {$P_2$}; 
  \node[grey] (P3) at (-1,2) {$P_3$};
  \node[grey] (P4) at (4,2) {$P_4$}; 

  \draw[to, loop left] (P0) to node[auto] {$\alpha, \tilde \alpha, a$} (P0); 
  \draw[to] (P0) to node[auto,swap] {$\alpha$} (P1);
  \draw[to] (P1) to node[auto,swap]  {$\gamma$} (P2);
  \draw[to] (P2) to node[auto]  {$\beta$} (P0);  
  \draw[to] (P2) to node[auto,swap]  {$\beta$} (P4);  
  \draw[to] (P0) to node[auto]  {$\alpha, \tilde \alpha, \beta, a$} (P3);  
  \draw[to,bend left] (P3) to node[auto]  {$\tilde \alpha$} (P4);  
  \draw[to,bend left] (P4) to node[auto]  {$\gamma$} (P3);  
  \draw[to,loop right] (P4) to node[auto]  {$\gamma$} (P4);
  \draw[to,bend left=20] (P4) to node[auto]  {$\gamma$} (P0);
  \draw[to] (P0) to node[near start, above]  {$\alpha,a$} (P4);      
\end{tikzpicture}
\end{center}
\caption[Renewal systems with positive determinants.]{Left Fischer cover of $X(L)$ for $L = \{ a, \alpha, \tilde \alpha, \beta, \gamma, \alpha \gamma \beta, \beta \tilde \alpha \gamma \}$. The border points are coloured grey.} 
\label{fig_LFC_pos_det_range}
\end{figure}
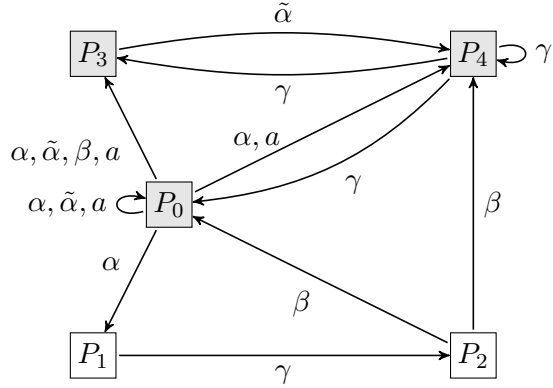

\label{sec_pos_det}
\index{renewal system!with positive \\determinant}
As mentioned in the previous section, the SFT renewal systems constructed by Hong and Shin \cite{hong_shin} exhibit positive determinants, and unlike the examples given in Section \ref{sec_rs_exotic}, these renewal systems are easy to study systematically. The following example is a special case of the renewal systems considered in Theorem \ref{thm_H_classification}, but it is presented again here, since this particular renewal system will be used as an important building block in the following.


\begin{example}
\label{ex_pos_det}
Consider the generating list 
\begin{equation}
   L = \{ a, \alpha, \tilde \alpha, \beta, \gamma, \alpha \gamma \beta, \beta \tilde \alpha \gamma \}.
\end{equation}
By Lemmas \ref{lem_entropy_lfc} and \ref{lem_entropy_sft}, $L$ is left-modular, $\X(L)$ is an SFT, and the left Fischer cover of $\X(L)$ is the labelled graph shown shown in Figure \ref{fig_LFC_pos_det_range}. The corresponding symbolic adjacency matrix is  
\begin{equation}
\label{eq_pos_det_adj}
A = \left( \begin{array}{c | c c | c c}
       a +\alpha+\tilde \alpha & \alpha    & 0 & a+\alpha +\tilde \alpha+\beta & a+\alpha \\
       \hline
       0                   &   0      & \gamma &  0                         &   0          \\
       \beta             &   0      & 0            &  0                         &   \beta     \\
\hline
       0                   &   0      & 0            &  0                         &   \tilde \alpha        \\
       \gamma        &   0      & 0            &  \gamma              &  \gamma                                      
\end{array} \right).
\end{equation}
\noindent
By fragmenting $\X(L)$, it is possible to construct an SFT renewal system for which the (non-symbolic) adjacency matrix of the underlying graph of the left Fischer cover has this form with $\alpha, \tilde \alpha, \beta, \gamma \in \N$ and $a \in \N_0$. Let $A_f$ be such a matrix. This is a special case of the shift spaces considered in Theorem \ref{thm_H_classification}, so the Bowen-Franks group is cyclic and the determinant is 
\begin{displaymath}
   \det(\Id - A_f) =  \beta \alpha \tilde \alpha \gamma^2 - \alpha \beta  \gamma - \tilde \alpha \beta  \gamma -\alpha - \tilde \alpha-\gamma-a +1.
\end{displaymath} 
\end{example} 


\begin{thm}\label{thm_rs_det_range}
Any $k \in \Z$ is the determinant of an SFT renewal system with cyclic Bowen-Franks group.
\end{thm}

\begin{proof}
Consider the renewal system from Example \ref{ex_pos_det} in the case $\alpha = \tilde \alpha = \beta = 1$, where the determinant is
\begin{displaymath}
   \det(\Id - A_f)   =  \gamma^2 - 3\gamma - a -1,
\end{displaymath}
and note that the range of this polynomial is $\Z$.
\end{proof}
\subsection{Non-cyclic Bowen-Franks groups}\label{sec_nc_bf}
\index{renewal system!with non-cyclic Bowen-Franks \\group}
The first known examples of SFT renewal systems with non-cyclic Bowen-Franks groups are given in this section. These groups are achieved by a class of renewal systems where the only forbidden words are powers of the individual letters, and these renewal systems are classified in the following.
The investigation 
was motivated by results of the experiments described in Appendix \ref{app_programs} which suggested that such systems have adjacency matrices with interesting structure.

\index{X@$\Xd{n_1,\dots,n_k}$}
Let $k \geq 2$, $\AA = \{ a_1, \ldots , a_k \}$, and let $n_1, \ldots , n_k \geq 2$ with $\max_{i}\{n_{i}\} > 2$. The goal is to define an SFT renewal system for which the set of forbidden words is $\FF = \{ a_i^{n_{i}}  \}$. For each $1 \leq i \leq k$, define
\begin{multline}
\label{eq_Ld}
  L_i = \{  a_j a_i^l \mid j \neq i \textrm{ and } 0 < l < n_i -1\}  \\ \cup 
            \{  a_m a_j a_i^l \mid m \neq j \neq i \textrm{ and } 0 < l < n_i -1\}.
\end{multline}
Define $L = \bigcup_{i=1}^k L_i \neq \emptyset$, and $\Xd{n_1,\dots,n_k} = \X(L)$.

\begin{lem}
\label{lem_Xd}
For the renewal system $\Xd{n_1,\dots,n_k}$ introduced above, the set of forbidden words is $\{ a_i^{n_{i}}  \}$, so $\Xd{n_1,\dots,n_k}$ is an SFT. The symbolic adjacency matrix of the left Fischer cover of $\Xd{n_1,\dots,n_k}$ is the matrix in Equation \ref{eq_diag_A}. 
\end{lem}

\begin{proof}
Let $\FF$ be the set of forbidden words for $\Xd{n_1,\dots,n_k}$ and note that $\{ a_i^{n_{i}}  \} \subseteq \FF$ by construction. For $1 < l < n_i-1$ and $j \neq i$ the word $a_j a_i^{l}$ has a partitioning in $\Xd{n_1,\dots,n_k}$ with empty beginning and end. Hence, $a_{i_1} a_{i_2}^{l_2} a_{i_3}^{l_3} \cdots a_{i_m}^{l_m}$ has a partitioning with empty beginning and end whenever $i_j \neq i_{j+1}$, $1 < l_j < n_{i_j}$ for all $1 < j < m$, and $0 < l_m < n_{i_m} -1$. Given $i_1, \ldots , i_m \in \{1, \ldots, k\}$ with $i_j \neq i_{j+1}$ and $m \geq 2$, the word $a_{i_1} a_{i_2} \cdots a_{i_m}$ has a partitioning with empty beginning and end. Hence, every word that does not contain one of the words $ a_i^{n_{i}}$ has a partitioning.

\begin{figure}
\begin{center}
\begin{tikzpicture}
 [bend angle=10,
   clearRound/.style = {circle, inner sep = 0pt, minimum size = 17mm},
   clear/.style = {rectangle, minimum width = 17 mm, minimum height = 6 mm, inner sep = 0pt},  
   greyRound/.style = {circle, draw, minimum size = 1 mm, inner sep =
      0pt, fill=black!10},
   grey/.style = {rectangle, draw, minimum size = 6 mm, minimum height = 8mm, inner sep =
      1pt, fill=black!10},
   white/.style = {rectangle, draw, minimum size = 6 mm, minimum height = 8mm, inner sep =
      1pt},
   to/.style = {->, shorten <= 1 pt, >=stealth', semithick}]
  
  \node[white] (an1) at (0,3) {$P_\infty \left(a_i^{n_i-1} \right)$};
  \node[white] (an2) at (3,3) {$P_\infty \left(a_i^{n_i-2}\right)$};
  \node[clear] (dots) at (6,3) {$\cdots$};  
  \node[grey] (a) at (9,3) {$P_\infty \left(a_i\right)$}; 

  \node[white] (bn1) at (0, 0) {$P_\infty \left(a_j^{n_j-1}\right)$};
  \node[white] (bn2) at (3,0) {$P_\infty \left(a_j^{n_j-2}\right)$};
  \node[clear] (dots2) at (6,0) {$\cdots$};  
  \node[grey] (b) at (9,0) {$P_\infty \left(a_j \right)$};

  \draw[to] (an1) to node[auto] {$a_i$} (an2);
  \draw[to] (an2) to node[auto] {$a_i$} (dots);
  \draw[to] (dots) to node[auto] {$a_i$} (a);  

  \draw[to] (bn1) to node[auto,swap] {$a_j$} (bn2);
  \draw[to] (bn2) to node[auto,swap] {$a_j$} (dots2);
  \draw[to] (dots2) to node[auto,swap] {$a_j$} (b);  

  \draw[to, bend left] (a) to node[auto] {$a_i$} (b);
  \draw[to] (a) to node[auto, near end] {$a_i$} (bn2);
  \draw[to] (a) to node[auto,swap, very near end] {$a_i$} (bn1);  

  \draw[to, bend left] (b) to node[auto] {$a_j$} (a);
  \draw[to] (b) to node[auto,swap, near end] {$a_j$} (an2);
  \draw[to] (b) to node[auto,very near end] {$a_j$} (an1);  

\end{tikzpicture}
\end{center}
\caption[Building non-cyclic Bowen-Franks groups.]{Part of the left Fischer cover of $\Xd{n_1,\dots,n_k}$.
The entire graph can be found by varying $i$ and $j$.
The border points are  coloured grey.} 
\label{fig_Xd_LFC}
\end{figure}
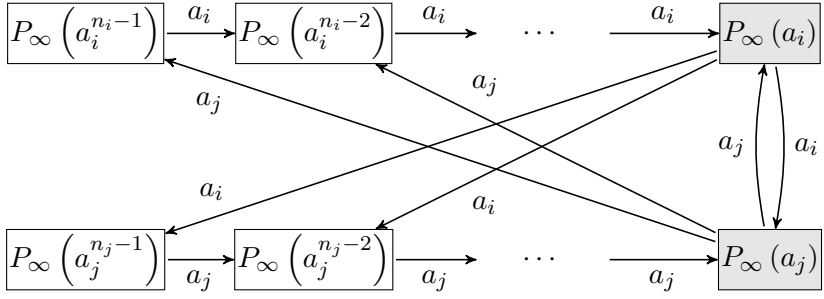

To find the left Fischer cover of  $\Xd{n_1,\dots,n_k}$, it is first necessary to determine the predecessor sets. Given $1 \leq i \leq k$ and $j \neq i$ 
\begin{align}
	P_\infty(a_i a_j \cdots) &= \{ x^- \in \Xd{n_1,\dots,n_k}^-  |  x_{-n_i +1} \cdots x_0  \neq a_i^{n_i-1} \} \nonumber \\
	P_\infty(a_i^2 a_j \cdots) &= \{ x^- \in \Xd{n_1,\dots,n_k}^-  |  x_{-n_i +2} \cdots x_0  \neq a_i^{n_i-2} \} \label{eq_rs_diag}\\
	&\vdots  \nonumber \\
	P_\infty(a_i^{n_i-1} a_j \cdots) &= \{ x^- \in \Xd{n_1,\dots,n_k}^-  |  x_0  \neq a_i \} \nonumber .
\end{align}
Only the first of these predecessor sets is a border point (this is not important in the present construction, but it will allow the construction of the left Fischer covers of sums involving renewal systems of this form in Section \ref{sec_pos_det+nc_bf}). Equation \ref{eq_rs_diag} gives all the information necessary to draw the left Fischer cover of $\Xd{n_1,\dots,n_k}$. A part of the left Fischer cover is shown in Figure \ref{fig_Xd_LFC},
and the corresponding symbolic adjacency is: \vspace{-0.5 em}\begin{align}
 	\nonumber
	&\hspace{14 pt}
	\begin{array}{c c c c}
       \overbrace{
       	\phantom{\begin{matrix}
			a_1 &  \cdots & a_1 & a_1
		\end{matrix}}
	}^{n_1-1}		
	&  
  	\overbrace{
       	\phantom{\begin{matrix}
			a_2 &  \cdots & a_2 & a_2
		\end{matrix}}
	}^{n_2-1}
	&
	\phantom{\cdots}
	&		
	\overbrace{
       	\phantom{\begin{matrix}
			a_k &  \cdots & a_k & a_k
		\end{matrix}}
	}^{n_k-1}
	\end{array} \\[-1.5em]
       &\left( \begin{array}{ c  | c |  c |  c } 
	\begin{matrix}
	0 &   \cdots & 0 & 0 \\
	a_1 &  \cdots & 0 & 0 \\
	   &     \ddots &    &    \\
	0 &   \cdots & a_1 & 0		 
	\end{matrix}		
	&
	\begin{matrix}
	a_1 &   \cdots & a_1 & a_1\\
	0 &   \cdots & 0 & 0	\\
	   &   \ddots &    &    \\
	0 &  \cdots & 0 & 0	
	\end{matrix}		
	&
	\cdots
	&
	\begin{matrix}
	a_1 &  \cdots & a_1 &  a_1\\
	0 &  \cdots & 0 & 0	\\
	   &  \ddots &    &    \\
	0 &  \cdots & 0 & 0	
	\end{matrix}		\\
	\hline
	\begin{matrix}
	a_2 &  \cdots & a_2 &  a_2\\
	0 &   \cdots & 0 & 0	\\
	   &   \ddots &    &    \\
	0 &  \cdots & 0 & 0	
	\end{matrix}		
	&
	\begin{matrix}
	0 &  \cdots & 0 & 0 \\
	a_2 &  \cdots & 0 & 0 \\
	   &   \ddots &    &    \\
	0 &  \cdots & a_2 & 0		 
	\end{matrix}		
	&
	\cdots
	&
	\begin{matrix}
	a_2 & \cdots & a_2 &  a_2\\
	0 & \cdots & 0 & 0	\\
	   &  \ddots &    &    \\
	0 &  \cdots & 0 & 0	
	\end{matrix}		\\
	\hline
	\vdots & \vdots & \ddots & \vdots \\
	\hline
	\begin{matrix}
	a_k &   \cdots & a_k &  a_k\\
	0 &  \cdots & 0 & 0	\\
	   &   \ddots &    &    \\
	0 &  \cdots & 0 & 0	
	\end{matrix}		
	&
	\begin{matrix}
	a_k &  \cdots & a_k &  a_k\\
	0 &  \cdots & 0 & 0	\\
	   &  \ddots &    &    \\
	0 & \cdots & 0 & 0	
	\end{matrix}		
	&
	\cdots
	&
	\begin{matrix}
	0 &  \cdots & 0 & 0 \\
	a_k &  \cdots & 0 & 0 \\
	   &     \ddots &    &    \\
	0 &   \cdots & a_k & 0		 
	\end{matrix}		\\
\end{array} \right). \label{eq_diag_A} 
\end{align} \end{proof}

Let $A$ be the (non-symbolic) adjacency matrix of the underlying graph of the left Fischer cover of $\Xd{n_1,\dots,n_k}$ constructed above.
Then it is possible to do the following transformation by row and column addition  
\begin{align*}
	\Id - A  &\rightsquigarrow 
	\begin{pmatrix}
	1        & 1-n_2 & 1-n_3 & \cdots  & 1-n_k  \\
	1-n_1 & 1        & 1-n_3 & \cdots  & 1-n_k  \\
	1-n_1 & 1-n_2 & 1        & \cdots  & 1-n_k  \\
       \vdots & \vdots & \vdots & \ddots & \vdots     \\
	1-n_1 & 1-n_2 & 1-n_3 & \cdots  & 1 		 
	\end{pmatrix} \\ 
	&\rightsquigarrow
	\begin{pmatrix}
	x      & 1       & 1       & \cdots  & 1  \\
	-n_1 & n_2  & 0       & \cdots  & 0  \\
	-n_1 & 0       & n_3  & \cdots  & 0  \\
       \vdots & \vdots & \vdots & \ddots & \vdots     \\
	-n_1 & 0       & 0 & \cdots  & n_k 		 
	\end{pmatrix}
	\quad x = 1-(k-1) n_1.		
\end{align*}
The determinant of this matrix is
\begin{multline*}
\det(\Id - A) = n_2 \cdots n_k \left( x + \sum_{i=2}^k \frac{n_1}{n_i} \right) \\
                   = - n_1 n_2 \cdots n_k \left( k-1 - \sum_{i=1}^k \frac{1}{n_i} \right)
                   < 0.
\end{multline*}
The inequality is strict since $k-1 - \sum_{i=1}^k \frac{1}{n_i} > \frac{k}{2} -1 \geq 0$.
Given concrete $n_1, \ldots, n_k$, it is straightforward to compute the Bowen-Franks group of $\Xd{n_1, \ldots, n_k}$, but it has not been possible to derive a general closed form for this group. For this reason, the following proposition concerns only a subclass of the renewal systems considered above, and this gives the first known examples of SFT renewal systems with non-cyclic Bowen-Franks groups. 

\begin{prop}
\label{prop_non-cyclic_BF}
Let $n_1, \ldots, n_k \geq 2$ with $n_i | n_{i-1}$ for $2 \leq i \leq k$ and $n_1 > 2$. Then $\BF_+(\Xd{n_1, \ldots, n_k}) = - \Z /m\Z \oplus \Z /n_3\Z \oplus \cdots \oplus \Z / n_k\Z$ for $m = n_1 n_2 ( k-1 - \sum_{i=1}^k \frac{1}{n_i} )$.
\end{prop}

\begin{proof}
By the arguments above, $\Xd{n_1, \ldots , n_k}$ is conjugate to an edge shift with adjacency matrix $A$ such that the following transformation can be carried out by row and column addition
\begin{displaymath}
	\Id - A \rightsquigarrow 
	\begin{pmatrix}
	y   & 1       & 1       & \cdots  & 1  \\
	0 & n_2  & 0       & \cdots  & 0  \\
	 0   & 0   & n_3  & \cdots  & 0  \\
       \vdots & \vdots & \vdots & \ddots & \vdots     \\
	 0   & 0       & 0 & \cdots  & n_k 		 
	\end{pmatrix}
	\rightsquigarrow
	\begin{pmatrix}
	 0   & 1       & 0   & \cdots  & 0  \\
	m &  0  & 0       & \cdots  & 0  \\
	 0   & 0   & n_3  & \cdots  & 0  \\
       \vdots & \vdots & \vdots & \ddots & \vdots     \\
	 0   & 0       & 0 & \cdots  & n_k 		 
	\end{pmatrix} ,	
\end{displaymath}
where $y = -n_1 \left( k-1- \sum_{i=1}^k 1/n_i \right)$.	
It follows that the Smith normal form of $\Id - A$ is $\diag(m, n_3, \ldots , n_k)$, and $\det(\Id - A) < 0$.
\end{proof}

Let  $G$ be a finite direct sum of finite cyclic groups. Then Proposition \ref{prop_non-cyclic_BF} shows that $G$ is a subgroup of the Bowen-Franks group of some SFT renewal system, but it is still unclear whether $G$ itself is also the Bowen-Franks group of a renewal system since the term $\Z / m\Z$ in the statement of Proposition \ref{prop_non-cyclic_BF}  is determined by the other terms. Furthermore, the groups constructed in Proposition \ref{prop_non-cyclic_BF} are all finite, so it is also unknown whether groups such as $\Z \oplus \Z$ can occur as subgroups of the Bowen-Franks group of an SFT renewal system. The following example shows that at least some such infinite groups can occur.

\begin{example} \label{ex_0_in_bf}
Consider the renewal system $\X(L)$ generated by the list 
$ L = \{ aa, aaa, baa, baaa, bb , bbb, abb, abbb, c \}$.
Via the computer programs described in Appendix \ref{app_programs}, Proposition \ref{prop_extendable} can be used to prove that $\X(L)$ is a $3$-step SFT and it is not hard to check that the set of forbidden words is $\FF = \{ abc, bac, cac, cbc, abab, baba, cbab \}$. The symbolic adjacency matrix of the left Fischer cover of $\X(L)$ is 
\begin{displaymath}
A = \left( \begin{array}{c | c c c | c c c}
 c & c & 0  & c & c & 0 & c \\
 \hline
 0 & a & a & 0 & 0 & 0 & a \\
 a & 0 & 0 & a & 0 & 0 & 0 \\
 0 & b & 0 & 0 & 0 & 0 & 0 \\
 \hline
 0 & 0 & 0 & b & b & b & 0 \\
 b & 0 & 0 & 0 & 0 & 0 & b \\
 0 & 0 & 0 & 0 & a & 0 & 0 
\end{array} \right). 
\end{displaymath}
Fragmenting the letter $c$ to $c_1, \ldots, c_n$ produces an SFT renewal system for which the (non-symbolic) adjacency matrix of the underlying graph of the left Fischer cover is obtained from $A$  by replacing $a$ and $b$ by $1$, and $c$ by $n$. It is easy to check that this system has determinant $0$ and Bowen-Franks group $\Z/(n+1) \oplus \Z$.
Appendix \ref{app_programs} contains further examples of renewal systems where $\Z$ is a subgroup of the Bowen-Franks group.
\end{example}

\subsection{Positive determinants and non-cyclic groups}
\label{sec_pos_det+nc_bf}
\index{renewal system!with positive \\determinant}
\index{renewal system!with non-cyclic Bowen-Franks \\group}
The determinants of all the renewal systems with non-cyclic Bowen-Franks groups investigated in the previous section were negative or zero, so the goal of this section is to construct a class of SFT renewal systems with positive determinants and non-cyclic Bowen-Franks groups by adding the renewal systems considered in Sections \ref{sec_pos_det} and \ref{sec_nc_bf}.

\begin{lem}
\label{lem_addition_Xd}
Let $L_d$ be the generating list of $\Xd{n_1, \ldots , n_k}$ as defined in Equation \ref{eq_Ld}, and let $(F_d, \LL_d)$ be the left Fischer cover of $\Xd{n_1, \ldots , n_k}$.
Let $L_m$ be a left-modular generating list for which $\X(L_m)$ is an SFT with left Fischer cover $(F_m, \LL_m)$. 
For $L_{d+m} = L_d \cup L_m \cup_{i=1}^k \{ a_i w \mid w \in L_m \}$,   $\X(L_{d+m})$ is an SFT for which the left Fischer cover is obtained by adding the following connecting edges to the disjoint union of $(F_d, \LL_d)$ and $(F_m, \LL_m)$:
\begin{itemize}
\item For each $1 \leq i \leq k$ and each $e \in F_m^0$ with $r(e) = P_0(L_m)$ draw an edge $e_i$ with $s(e_i) = s(e)$ and $r(e_i) = P_\infty(a_i a_j \ldots)$  labelled $\LL_m(e)$.
\item For each $1 \leq i \leq k$ and each border point $P \in F_m^0$ draw an edge labelled $a_i$ from $P_\infty(a_i a_j \ldots)$ to  $P$.
\end{itemize}
\end{lem}

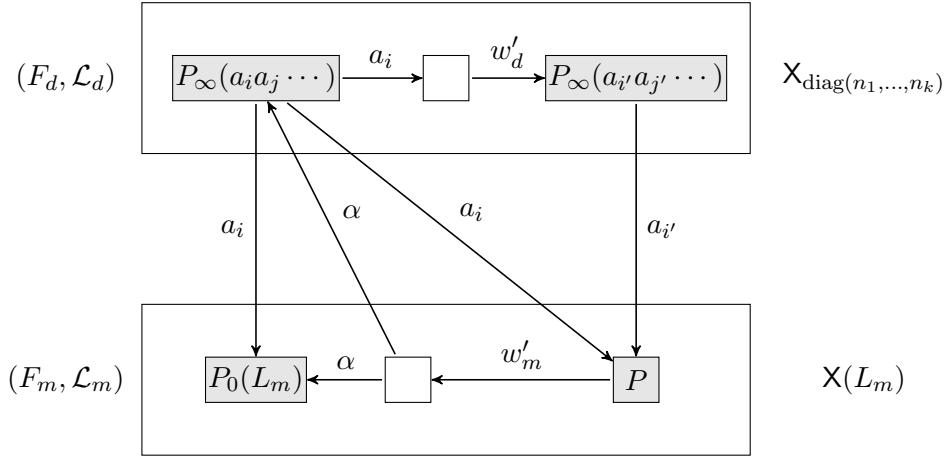
\begin{figure}
\begin{center}
\begin{tikzpicture}
  [bend angle=30,
   clearRound/.style = {circle, inner sep = 0pt, minimum size = 17mm},
   clear/.style = {rectangle, minimum width = 5 mm, minimum height = 5 mm, inner sep = 0pt},  
   greyRound/.style = {circle, draw, minimum size = 1 mm, inner sep =
      0pt, fill=black!10},
   grey/.style = {rectangle, draw, minimum size = 6 mm, inner sep =
      1pt, fill=black!10},
    white/.style = {rectangle, draw, minimum size = 6 mm, inner sep =
      1pt},
  whiteBig/.style = {rectangle, draw, minimum height = 2 cm, minimum width = 8 cm, inner sep =
      1pt},
   to/.style = {->, shorten <= 1 pt, >=stealth', semithick}] 

  \node[whiteBig] (Fd) at (0,2)   {};
  \node[whiteBig] (Fm) at (0,-2) {};
  
  \node[grey] (Pi) at (-2.5,2) {$P_\infty(a_ia_j \cdots)$};
  \node[white] (Q') at (0,2) {};  
  \node[grey] (Pi') at (2.5,2) {$P_\infty(a_{i'}a_{j'} \cdots)$};
  
  \node[grey] (P0) at (-2.5,-2) {$P_0(L_m)$};
  \node[grey] (P) at (2.5,-2) {$P$};
  \node[white] (Q) at (-0.5,-2) {};  
  \node[clear] (Fdtext) at (-5,2) {$(F_d, \LL_d)$};  
  \node[clear] (Fmtext) at (-5,-2) {$(F_m, \LL_m)$};    
  \node[clear] (Xdtext) at (5.5,2) {$\Xd{n_1, \ldots, n_k}$};  
  \node[clear] (Lmtext) at (5.5,-2) {$\X(L_m)$};    
   
  \draw[to] (Pi) to node[auto] {$a_i$} (P);
  \draw[to] (Pi) to node[auto,swap] {$a_i$} (P0);
  \draw[to] (Pi) to node[auto] {$a_i$} (Q');
  \draw[to] (Q') to node[auto] {$w'_d$} (Pi');
  \draw[to] (Pi') to node[auto] {$a_{i'}$} (P);      
  \draw[to] (P) to node[auto,swap] {$w'_m$} (Q);
  \draw[to] (Q) to node[auto,swap] {$\alpha$} (P0);
  \draw[to] (Q) to node[auto,swap] {$\alpha$} (Pi);

\end{tikzpicture}
\end{center}
\caption[Construction of the Fischer cover of a sum.]{Construction of the left Fischer cover considered in Lemma \ref{lem_addition_Xd}. Here, $w'_m \alpha = w_m \in L_m^*$ and $w'_d a_{i'} = w_d \in \BB(\Xd{n_1, \ldots, n_k})$ with $\leftl(w_d) \neq a_i$. Border points are coloured grey.} 
\label{fig_addition_Xd}
\end{figure}

\noindent 
Figure \ref{fig_addition_Xd} shows the construction of the left Fischer cover defined in Lemma \ref{lem_addition_Xd}.
The addition of $\{ a_i w \mid w \in L_m \}$ to $L_d \cup L_m$ ensures that a word $w_m \in L_m^*$ can be preceded by $a_i^{n_i-1}$ in $\X(L_{d+m})$. This will allow a simpler analysis than in the renewal system generated by $L_d \cup L_m$. 

\begin{proof}[Proof of Lemma \ref{lem_addition_Xd}]
 Let $(F_{d+m}, \LL_{d+m})$ be the labelled graph defined in the statement of the lemma and sketched in Figure \ref{fig_addition_Xd}.
The graph is left-resolving, predecessor-separated, and irreducible by construction, so by Theorem \ref{thm_lfc_char}, it is the left Fischer cover of some sofic shift $X$. The first goal is to prove that $X = \X(L_{d+m})$.
By the arguments used in the proof of Lemma \ref{lem_Xd}, any word of the form $a_{i_0}w_m a_{i_1} a_{i_2}^{l_i} \ldots a_{i_p}^{l_p}$ where 
$w_m \in L_m^*$,
$p \in \N$, 
$i_j \neq i_{j+1}$ and $l_j < n_{i_j}$ for $1 < j < p$,
and $1 \leq l_p < n_{i_p} -1$
has a partitioning with empty beginning and end in $\X(L_{d+m})$. 
Hence, the language of $\X(L_{d+m})$ is the set of factors of concatenations of words from
\begin{displaymath}
\left\{ w_m a_i w_d \mid w_m \in L_m^*, 1 \leq i \leq k, w_d \in \BB(\Xd{n_1, \ldots, n_k}), \leftl(w_d) \neq a_i \right \}.
\end{displaymath}
Since $L_m$ is left-modular, a path $\lambda \in F_m^*$ with $r(\lambda) = P_0(L_m)$ has $\LL_m(\lambda) \in L_m^*$ if and only if $s(\lambda)$ is a border point in $F_m$.
Hence, the language recognised by the left Fischer cover $(F_{d+m}, \LL_{d+m})$ is precisely the language of $\X(L_{d+m})$. This is illustrated in Figure \ref{fig_addition_Xd}.
 
It remains to show that $(F_{d+m}, \LL_{d+m})$ presents an SFT. 
Let $1 \leq i \leq k$ and let $\alpha \in \BB(\X(L_m))$, then any labelled path in $(F_{d+m}, \LL_{d+m})$ with $a_i \alpha$ as a prefix must start at $P_\infty(a_ia_j\cdots)$.
Similarly, if there is path $\lambda \in F_{d+m}^*$ with $\alpha a_i$ as a prefix of $\LL_{d+m}(\lambda)$, then there must be unique vertex $v$ emitting an edge labelled $\alpha$ to $P_0(L)$, and $s(\lambda) = v$.
Let $x \in \X_{(F_{d+m},\LL_{d+m})}$. If there is no upper bound on set of $i \in \Z$ such that $x_i \in \{a_1, \ldots, a_k\}$ and $x_{i+1} \in \AA(\X(L_m))$ or vice versa, then the arguments above and the fact that the graph is left-resolving prove that there is only one path in $(F_{d+m},\LL_{d+m})$ labelled $x$. If there is an upper bound on the set considered above, then a presentation of $x$ is eventually contained in either $F_{d}$ or $F_{m}$. It follows that the covering map of $(F_{d+m}, \LL_{d+m})$ is injective, so it presents an SFT by Corollary \ref{cor_lfc_conj}.
\end{proof}

The next step is to use Lemma \ref{lem_addition_Xd} to construct renewal systems that share features with both $\Xd{n_1, \dots, n_k}$ and the renewal systems with positive determinants considered in Example \ref{ex_pos_det}.

\begin{example}
\label{ex_pos_det+nc_bf}
Given $n_1, \ldots , n_k \geq 2$ with $\max_j n_j > 2$, consider the list $L_d$ defined in Equation \ref{eq_Ld} which generates the renewal system $\Xd{n_1, \ldots, n_k}$,
and the list 
\begin{displaymath}
L = \{ a, \alpha, \tilde \alpha, \beta, \gamma, \alpha \gamma \beta, \beta \tilde \alpha \gamma \}
\end{displaymath}
introduced in Example \ref{ex_pos_det}. $L$ is left-modular, and $\X(L)$ is an SFT, so Lemma \ref{lem_addition_Xd} can be used to find the left Fischer cover of the SFT renewal system $X_+$ generated by
\begin{displaymath}
L_{+} = L_d \cup L \cup_{i=1}^k \{ a_i w \mid w \in L \}
,\end{displaymath}
and the corresponding symbolic adjacency matrix $A_+$ is
\begin{displaymath}
\left( \!\!\! \begin{array}{c | c c c c | c c c c c | c | c c c c c  }
       b & \alpha    & 0 & b +\beta & a+\alpha                        &  b & 0 &\! \cdots \!& 0 & 0 & \!\cdots\! & b & 0  &\! \cdots \!& 0 & 0 
       \\
       \hline
       0            &   0      & \gamma &  0             & 0                 & 0 & 0           &\! \cdots \!& 0 & 0 &   & 0 & 0 &\! \cdots \!& 0 & 0      \\
       \beta      &   0      & 0            &  0             & \beta           & \beta & 0     &\! \cdots \!& 0 & 0 &  & \beta & 0 &\! \cdots \!& 0 & 0              \\
       0            &   0      & 0            &  0             & \tilde \alpha & 0 & 0           &\! \cdots \!& 0 & 0 &  & 0 & 0 &\! \cdots \!& 0 & 0      \\
       \gamma &   0      & 0            &  \gamma  & \gamma      & \gamma & 0 &\! \cdots \!& 0 & 0 &  & \gamma & 0 &\! \cdots \!& 0 & 0   \\
       \hline                          
       a_1     & 0 & 0 & a_1 & a_1 & 0     & 0  &  & 0 & 0 &  & a_1  & a_1 &\! \cdots \!& a_1 & a_1   \\       
       0         & 0 & 0 & 0     & 0     & a_1 & 0  &  & 0 & 0 &  & 0      & 0     &  & 0     & 0   \\ 
       \vdots & \vdots   & \vdots  & \vdots   & \vdots        &        &     &\! \ddots \!&    &    &  &         &        &\! \ddots \!&        &      \\          
       0         & 0 & 0 & 0     & 0     & 0     & 0  &  & 0 & 0 &  & 0      & 0     &  & 0     & 0   \\ 
       0         & 0 & 0 & 0     & 0     & 0     & 0  &  & a_1 & 0 &  & 0      & 0     &  & 0     & 0   \\
       \hline 
       \vdots &    &    &        &        &        &     &              &        &        & \!\ddots\!  &         &        &  &        &      \\                 
       \hline
       a_k     & 0 & 0 & a_k & a_k & a_k & a_k&\! \cdots \!& a_k & a_k &  & 0      & 0     & & 0     & 0   \\       
       0         & 0 & 0 & 0     & 0     & 0     & 0    &  & 0     & 0     &  & a_k      & 0     & & 0     & 0   \\ 
       \vdots & \vdots   & \vdots  & \vdots   & \vdots     &        &       &\! \ddots \!&        &        &  &         &        &\! \ddots \!&        &      \\          
       0         & 0 & 0 & 0     & 0     & 0     & 0    &  & 0     & 0     &  & 0      & 0     & & 0     & 0   \\ 
       0         & 0 & 0 & 0     & 0     & 0     & 0    &  & 0 & 0     &  & 0      & 0     & & a_k     & 0   \\
\end{array} \!\!\! \right)
\end{displaymath}
where $b = a + \alpha + \tilde \alpha$. Let $Y_+$ be a renewal system obtained from $X_+$ by a fragmentation of $a$, $\alpha$, $\tilde \alpha$, $\beta$, and $\gamma$. Then the (non-symbolic) adjacency matrix of the left Fischer cover of $Y_+$ is obtained from the matrix $A_+$ above by replacing $a_1, \ldots, a_k$ by $1$, and replacing $a$, $\alpha$, $\tilde \alpha$, $\beta$, and $\gamma$ by positive integers. Let $B_+$ be a matrix obtained in this manner. By doing row and column operations as in the construction that leads to the proof Proposition \ref{prop_non-cyclic_BF}, and by disregarding rows and columns where the only non-zero entry is a diagonal $1$, it follows that
\begin{multline*}
\Id - B_+ \rightsquigarrow \\
\left( \!\!\! \begin{array}{c | c c c c | c c c c  }
       1-b            & -\alpha & 0            & -b -\beta   & -a-\alpha     &  -b          & -b             &\! \cdots \! & -b \\
       \hline
       0               & 1          & -\gamma &  0             & 0                 & 0            & 0              &\! \cdots \! & 0  \\  
       -\beta         & 0          & 1             &  0            & -\beta          & -\beta     & -\beta      &\! \cdots \! & -\beta  \\
       0              & 0          & 0              &  1             & -\tilde \alpha & 0            & 0            &\! \cdots \! & 0 \\
       -\gamma  & 0          & 0              & - \gamma & 1-\gamma     & -\gamma & -\gamma &\! \cdots \! & -\gamma \\
       \hline                          
       -1              & 0         & 0            & -1              & -1                 & 1            & 1-n_2     & \! \cdots \! & 1-n_k \\  
       -1              & 0         & 0            & -1              & -1                 & 1-n_1     & 1            &                  & 1-n_k \\
        \vdots      & \vdots & \vdots    & \vdots       & \vdots          &  \vdots    &               & \! \ddots \! & \vdots \\
        -1             & 0         & 0            & -1              & -1                 & 1-n_1     & 1-n_2     & \! \cdots \! & 1 \\           
\end{array} \!\!\! \right).
\end{multline*}
Add the third row to the first and subtract the first column from columns $4, \ldots , k+4$ as in the proof of Lemma \ref{lem_entropy_cyclic} to obtain
\begin{multline*}
\Id - B_+ \rightsquigarrow \\
\left( \begin{array}{c | c c  c c | c c c c  }
       1-b            & -\alpha & 0            & -\beta        & -1     &  -1          & -1             &\! \cdots \! & -1 \\
       \hline
       0               & 1          & -\gamma &  0             & 0                  &             &               &       &   \\  
       -\beta         & 0          & 1             &  \beta             & 0           &             &               &       &  \\
       0              & 0          & 0              &  1             & -\tilde \alpha &             &               &       &  \\
       -\gamma  & 0          & 0              &  0               & 1                &             &               &       &  \\
       \hline                          
       -1              &            &               &                  &                    & 2            & 2-n_2     & \cdots      & 2-n_k \\  
       -1              &            &               &                 &                    & 2-n_1     & 2            &                  & 2-n_k \\
        \vdots      &            &               &                 &                    & \vdots     &               & \ddots  &  \vdots\\
        -1             &            &               &                 &                     & 2-n_1     & 2-n_2    & \cdots  & 2 \\           
\end{array}  \right).
\end{multline*}
By choosing the variables $a$, $\alpha$, $\tilde \alpha$, $\beta$, and $\gamma$ as in the proof of Theorem \ref{thm_rs_det_range}, this can be further reduced to 
\begin{displaymath}
\Id - B_+ \rightsquigarrow 
\left( \begin{array}{c | c c c c c }
x         & -1       & -1     & -1     & \cdots & -1 \\
\hline
2x-1    &  0       & -n_2 & -n_3 & \cdots & -n_k \\
0         & -n_1   &  n_2 &  0     &            & 0  \\
0         & -n_1   &  0     & n_3  &            & 0  \\
\vdots & \vdots &        &         & \ddots  & \vdots  \\
0        & -n_1     &  0    & 0      & \cdots  & n_k  \\ 
\end{array}  \right),
\end{displaymath}
where $x \in \Z$ is arbitrary. Assume that $n_i | n_{i-1}$ for $2 \leq i \leq k$. Then 
\begin{displaymath}
\Id - B_+ \rightsquigarrow 
\left( \begin{array}{c c c c c c }
x         & -\sum_{i=1}^k \frac{n_1}{n_i} & -1     & 0      & \cdots & 0 \\
2x-1    &  -(k-1)n_1                               & 0      &  0     & \cdots & 0 \\
0         &  0                                            &  n_2 &  0     &            & 0  \\
0         &  0                                            &  0     & n_3  &            & 0  \\
\vdots & \vdots                                     &         &         & \ddots  & \vdots   \\
0        &    0                                           &  0    & 0      &  \cdots  & n_k  \\ 
\end{array}  \right).
\end{displaymath}
Hence, the determinant is 
\begin{equation}
\label{eq_pos_det+nc_bf_det}
  \det(\Id - B_+) =  n_2 \cdots n_k  \left( (2x-1)\sum_{i=1}^k \frac{n_1}{n_i}   -  x(k-1)n_1 \right),
  \end{equation}
and there exists an abelian group $G$ with at most two generators such that the Bowen-Franks group of the corresponding SFT is $G \oplus \Z / n_3 \Z \oplus \cdots \oplus \Z / n_k \Z$.
For $x = 0$, the determinant is negative and the Bowen-Franks group is $\Z \big/ \big(\sum_{i=1}^k \frac{n_2 n_1}{n_i} \big)\Z \oplus \Z /n_3\Z \oplus \cdots \oplus \Z /n_k\Z$.
\end{example}

The example was motivated by a hope that the presence of the arbitrary $x$ would make it possible to prove that any direct sum of finite abelian groups is the Bowen-Franks group of an SFT renewal system, but the construction is not general enough for this. It does, however, give the first example of SFT renewal systems that simultaneously have positive determinants and non-cyclic Bowen-Franks groups.

\begin{thm}\label{thm_nc_bf_det}
Given $n_1, \ldots, n_k \geq 2$ with $n_i | n_{i-1}$ for $2 \leq i \leq k$ there exist abelian groups $G_\pm$ with at most two generators and SFT renewal systems $\X(L_\pm)$ such that $\BF_+(\X(L_\pm)) = \pm G_\pm \oplus \Z /n_1\Z \oplus \cdots \oplus \Z /n_k \Z$.
\end{thm}

\begin{proof}
Consider the renewal system from Example \ref{ex_pos_det+nc_bf}. Equation \ref{eq_pos_det+nc_bf_det} shows that no matter the values of the other variables, $x$ can be chosen such that the determinant is either positive or negative.
\end{proof}

\noindent
This is the most general result about non-cyclic groups obtained in the search for the range of the Bowen-Franks invariant. 

\section{Perspectives} \label{sec_rs_perspectives}
The question raised by Adler, and the related question concerning the flow equivalence of renewal systems remain unanswered, and a significant amount of work remains before they can be solved. Indeed, one of the main conclusions of this work must be these questions are very difficult to answer despite the simple formulations.
 While it has not been possible to answer these questions by finding the range of the Bowen-Franks invariant over the set of SFT renewal systems, significant progress has been made in the search: Theorem \ref{thm_rs_det_range} shows that every $k \in \Z$ is the determinant of an SFT renewal system, Proposition \ref{prop_non-cyclic_BF} shows that a large class of non-cyclic groups appear as the Bowen-Franks groups of SFT renewal systems, and Theorem \ref{thm_nc_bf_det} shows that non-cyclic groups and positive determinants can appear simultaneously.
The complexity of the constructions used to derive these results, and the investigation of the class $H$ in Section \ref{sec_rs_entropy} illustrate the problems involved in constructing renewal systems with complicated values of the Bowen-Franks invariant. This is further illustrated by the modest number of SFT renewal systems not flow equivalent to full shifts that came out of the experimental investigation discussed in Appendix \ref{app_programs}.

The main obstruction to constructing SFT renewal systems with a specific value of the Bowen-Franks invariant is seen in the strict structure enforced on the Fischer cover by Proposition \ref{prop_rs_circuit_connected}. This complicates the construction because the connecting edges between circuits in the Fischer cover make it difficult have non-trivial diagonal elements in the Smith normal form of the associated matrices. This is seen in Proposition \ref{prop_non-cyclic_BF} and Theorem \ref{thm_nc_bf_det} where the connecting edges are responsible for the undesired extra terms in the groups.
For an investigation of Adler's original question, it seems interesting to look at this structure in SFTs such as the one given in Example \ref{ex_not_connected}. Circuit connection is not preserved by conjugacy, but if this property implies some deeper property that is preserved, then this might be used to answer the question. However, there is no evidence to suggest that this is the case. 

In this study, the techniques for adding and fragmenting renewal systems introduced in Sections \ref{sec_rs_add} and \ref{sec_rs_fragmentation} have been very useful in the construction of complicated renewal systems from simpler building blocks, and it is reasonable to assume that they can continue to play this role in additional investigations of the range of the Bowen-Franks invariant.

There is little hope that further experimental investigations of randomly generated renewal systems will yield interesting results at this stage. A continued investigation will have to deal with very complicated renewal systems, and the already significant computational difficulties of the experimental approach described in Appendix \ref{app_programs} will grow exponentially with the step of the SFT renewal systems considered.

Proposition \ref{prop_extendable} gives a useful testable condition guaranteeing that a renewal system is an SFT. The condition is necessary for renewal systems with strongly left- or right-bordering words as proved in Proposition \ref{prop_extendable_and_sft}, but not in general, so some SFTs will be missed by a search relying on this test. It would be desirable to have a comparison of the complexity of this algorithm with the complexity of the algorithm based on the subset construction sketched in Section \ref{sec_rs_when_sft}, but these computations have not been done yet. 
In theory, the algorithm based on the subset construction is far superior because it can determine precisely when a renewal system is an SFT, but in a practical investigation of a left Fischer cover with $r$ vertices it will only be possible to check all the words of length $r^2-r$ for small values of $r$. So for practical purposes, this method cannot generally be used to prove that a renewal system is strictly sofic, and in this way, it shares the fundamental problem of the method based on Proposition \ref{prop_extendable}.

One could also try to approach the problem from the other direction since Theorem \ref{thm_restivo_decidable} shows that it is decidable whether an SFT is a renewal system. This would make it possible to test whether SFTs in a class with as yet unrealised values of the Bowen-Franks invariant are in fact renewal systems. However, this is probably not a useful practical approach because of the complexity of the algorithm and the rarity of renewal systems in the class of general irreducible SFTs.

The group $\Z \oplus \Z$ is arguably the least complicated group that has not yet been proved to be the Bowen-Franks group of an SFT renewal system, so an obvious next step in the investigation of the range of the complete invariant would be to attempt to construct an SFT renewal system with this Bowen-Franks group. More generally, a reasonable strategy in the investigation would be to always aim to construct a renewal system with the least complicated combination of group and determinant that is not yet known to be achieved by a renewal system. This may solve the problem either by showing that the range of the Bowen-Franks invariant is the same over the set of SFT renewal systems as it is over the entire class of irreducible SFTs, or by providing insight enough to show that some combination of sign and Bowen-Franks group cannot be achieved by an SFT renewal system.
 
At this point, it is reasonable to expect a conjecture about the solution to Adler's question. Several such conjectures have been formulated during the investigation only to be abandoned later when acquired knowledge failed to support them or simply proved them wrong. Currently, it seems that the solutions could go either way, so the following bold statements are really based on intuition:

\begin{conj}
There exists an irreducible SFT $X$ such that no renewal system is conjugate to $X$.
\end{conj}

\begin{conj}
Every irreducible SFT is flow equivalent to a renewal system.
\end{conj}

%% file: thesis_programs.tex
\label{app_programs}

This appendix describes the computer experiments used to investigate the flow equivalence problem for renewal systems discussed in Chapter \ref{chap_rs}.
Examples \ref{ex_pos_det_1} and \ref{ex_0_in_bf} rely on these programs to show that certain renewal systems are SFTs with specific forbidden words, but apart from this, the results generated by the programs have only been used as inspiration and none of the other results in Chapter \ref{chap_rs} rely on computations. The experimental investigation was primarily done using a number of programs written in C++ which will be described in the following.

Section \ref{sec_exp_preliminary} gives a description of a preliminary investigation which guided the rest of the experimental process, while Section \ref{sec_exp_programs} contains an overview of all the C++ programs used, 
and a more detailed description of the most important classes and programs in order to explain the general structure and ideas.
However, this is not meant as a proper documentation of the programs, but rather as a broad description with focus on the mathematically interesting features, so only a few details are given about the programming techniques used. The complexities of the algorithms have not been determined, but some comments are given about the parts of the programs where the majority of the computations are carried out.
Section \ref{sec_exp_results} contains the main results of the investigation, most of which have already been discussed in Chapter \ref{chap_rs} where they served as inspiration for the development of the theory. However, some of them are only presented here because they had no natural place in the previous exposition. Finally, the results and methods are discussed in Section \ref{sec_exp_discussion}.

Files with the source code of the programs used are available at: 
\begin{center}
 \url{http://www.math.ku.dk/~rune/renewal/}
\end{center}
\noindent
The programs have been compiled using \texttt{g++} under Mac OS X 10.5.8, but the portability has not been tested, so adjustments may be necessary in order to compile them on other systems. 
In particular, some of the programs execute system commands in order to create directories where logs can be stored, and these commands will only be recognised by operating systems such as Unix and Linux. Additionally, the programs rely on Maple to do linear algebra, so they will only function if this program is available. Please send any questions concerning the programs to \url{rune@math.ku.dk}.

\section{Preliminary investigation}
\label{sec_exp_preliminary}

For the preliminary investigation of the flow equivalence problem for renewal systems, a collection of procedures were constructed in Maple. Given a generating list $L$, these would inductively construct the language and set of forbidden words of the corresponding renewal system $\X(L)$. For each length $n$, the procedures would check whether there were any forbidden words of length $n$ that did not contain any known forbidden word of shorter length as a factor. If a renewal system did not exhibit any such new forbidden words for a number $k$ of steps after the $n$th step, then it would be taken as an indication that $\X(L)$ could be an $n$-step SFT. Since the allowed words of length $n$ were known, it was then straightforward to construct the higher block shift as in Proposition \ref{prop_higher_block} and to compute the Bowen-Franks invariant. 

This method clearly risks to falsely report certain renewal systems as $n$-step SFTs since it has no way to determine whether there are new forbidden words of length greater than $n+k$. This was not as big a problem as it would seem since all higher block shifts constructed in this way turned out to be flow equivalent to full shifts. I.e.\ even by considering a larger class than the one of interest, it was only possible to achieve a small part of the values of the Bowen-Franks invariants achieved by general irreducible SFTs.

A more significant problem was that it took a long time to carry out the $k$ extra steps of the induction used to filter out renewal systems that where strictly sofic or of too high step to manage.
This prevented the investigation of more complicated renewal systems where one might hope to find non-trivial values of the Bowen-Franks invariant. Indeed, results such as Lemma \ref{lem_irr_replace}, Corollary \ref{cor_cyclic_fe_full}, and Theorem \ref{thm_H_classification} suggest that a generating list must be quite complicated in order to generate an SFT renewal system that is not flow equivalent to a full shift. The desire to solve this problem led 
to the proof of Proposition \ref{prop_extendable}.

\section{Methods}
\label{sec_exp_programs}

Based on the insight gained from the preliminary investigation,
it was decided to use the following overall strategy for the investigation of a renewal system $\X(L)$: Construct the allowed words and their minimal partitionings inductively, and check at each length $n$ whether the conditions of Proposition \ref{prop_extendable} are satisfied. If this is the case, then $\X(L)$ is an $n$-step SFT, so the knowledge of the allowed words of length $n$ can be used to construct the higher block shift as in Proposition \ref{prop_higher_block}. Finally, the Bowen-Franks invariant of $\X(L)$ can be found by computing the determinant and Smith normal form of $\Id -A$ for the adjacency matrix $A$ of the higher block shift. 
It was decided to found the investigation on Proposition \ref{prop_extendable} rather than the algorithm based on the subset construction, which is described in Section \ref{sec_rs_when_sft}, since the latter was thought to require too many computations.

To implement this strategy, a number of programs were written in C++, and these will be described in the following. Focus will be on the places where novel algorithms have been constructed, so this will not be a complete documentation of the programs, and many technical details will be omitted. Section \ref{sec_programs_overview} gives a quick overview of all the classes and programs used, while the following sections give more detailed accounts of the most important classes and programs. 

\subsection{Overview over classes and programs}
\label{sec_programs_overview}

In order to study the flow equivalence of renewal systems using the strategy described above, a number of \emph{classes} were used to represent mathematical objects (such as renewal systems and partitionings) in the programs. These classes are designed to store information that defines the corresponding mathematical object (e.g.\ the generating list of a renewal system), and they contain functions that allow natural operations to be carried out on the objects (e.g.\ a function that finds the allowed words of length $n$ in a renewal system).

The classes are defined in so called \emph{header files} which are available at the address given above. The contents of the header files are described in the following list:

\begin{description}
\item[\program{renewalsystem.h}] defines the important classes
used to represent words, partitionings, and renewal systems. They are called \cclass{allowedWord}, \cclass{partitioning}, and \cclass{renewalSystem}, respectively. This file contains the bulk of the code.
\item[\program{adjacencymatrix.h}] defines the class \cclass{adjMatrix} which is used to give a compact representation of the large and sparse integer matrices that appear as the adjacency matrices of the higher block shifts of renewal systems. 
\item[\program{generator.h}] contains the class \cclass{generator} which can be used to represent the generating list of a renewal system. It defines a standard file format that is used by all programs, so that output written by one program can be read by another (the files \program{sft.txt} and \program{sofic.txt}, which are available online, give examples of this format).
\item[\program{collection.h}] defines the class \cclass{collection} which is used to represent a collection of renewal systems as a vector of objects from the class \cclass{generator}. It allows collections to be generated from a file or directly from user input. It contains functions that make it possible to carry out the operations defined in \cclass{renewalsystem.h} on all members at once.
\item[\program{symmetricRS.h}] contains the class \cclass{symmetricRS} which can be used to represent the specific class of renewal systems where the only forbidden words are powers of the individual letters as defined in Section \ref{sec_nc_bf}.
\item[\program{log.h}] contains the class \cclass{logBook} which is used to print information about the progress of a program to a file. This class contains no structures specific to renewal systems, so it could be used to generate logs in other contexts.
\item[\program{option.h}] contains the class \cclass{option} which represents optional arguments given to a program when it is run from the command line. This class contains no structures specific to renewal systems, so it could be used to add options to programs used in other contexts.
\item[\program{tools.h}] contains various general tools used by the other classes (e.g.\ to print all values contained in a list).
\end{description}

Based on these classes, a number of programs have been written in order to allow users to investigate renewal systems. They are available at the address given above and described in the following list:

\begin{description}
\item[\program{irs.cpp} \normalfont{(\textbf{I}nvestigate \textbf{r}enewal \textbf{s}ystem)}] is the main interface which allows  users to investigate renewal systems based on the classes described above. It examines all the renewal systems in a collection and prints the Bowen-Franks invariants of the ones that can be proved to be SFTs. 
\item[\program{reduce.cpp}] takes a generating list $L$ as input and uses the algorithm from the proof of Proposition \ref{prop_irr} to construct an irreducible generating list $M$ such that $\X(L) \FE \X(M)$.
\item[\program{add.cpp}] takes two files containing collections of generating lists $\{L_i \mid 1 \leq i \leq k\}$ and $\{M_j \mid 1 \leq j \leq l\}$ as input and outputs a file with the collection of generating lists $\{L_i \cup M_j \mid 1 \leq i \leq k, 1 \leq j \leq l\}$.
\item[\program{d0.cpp}] uses the class \cclass{symmetricRS} to construct a collection of generating lists for which the renewal systems only have powers of the individual letters as forbidden words as defined in Section \ref{sec_nc_bf}.
\item[\program{rename.cpp}] takes a file containing generating lists and gives each of them a name based on a text-string supplied by the user. 
\item[\program{testGenerator.cpp}] is a tutorial for the file format defined by the class \cclass{generator}.
\end{description}

The following sections will describe the main features of the classes \cclass{allowedWord}, \cclass{partitioning}, \cclass{renewalSystem}, and \cclass{adjMatrix} which contain the core of the experimental setup. The functionality of the programs \program{irs.cpp} and \program{reduce.cpp}, which allow users to use the classes to work experimentally with renewal systems, will also be described, but the remaining classes and programs are less central to the experimental investigation and contain no mathematically interesting features, so they will not be examined further.
              
\subsection{The class \cclass{allowedWord}}

The class \cclass{allowedWord} is used to represent a word $w \in \BB_n(\X(L))$ in the programs, and a key feature of the class is that it can be used to keep track of all the minimal partitionings of $w$. This is important, because the idea is to apply Proposition \ref{prop_extendable}, which relies on  the structure of the partitionings to test whether $\X(L)$ is an SFT.

The most important function in the class \cclass{allowedWord} checks whether $w$ is strongly synchronizing, left-extendable, and/or right-extendable. This is done via a pairwise comparison of all the partitionings of $w$, and it can be a time consuming computation if $w$ is a long word with many partitionings.
These comparisons are skipped if $w$ is known to have a factor that is strongly synchronizing since $w$ is automatically strongly synchronizing, left-, and right-extendable in this case.
In most renewal  systems, this allows a significant reduction of the total number of comparisons needed because a large fraction of the allowed words are strongly synchronizing. 
In the description of the class \cclass{partitioning} below, it is shown how information about the strongly synchronizing factors is passed to  \cclass{allowedWord} when a representation of a new word is created. 

\subsection{The class \cclass{partitioning}}
\label{sec_programs_partitioning}

Let $p = (n_b,[g_1, \ldots,g_k],l)$ be a minimal partitioning of $w \in \BB_l(\X(L))$. In the programs, $p$ is represented by the class \cclass{partitioning} which keeps track of $n_b$, $[g_1, \ldots,g_k]$, $l$, and $w$.

The main feature of the class \cclass{partitioning} is a function which extends $p$ in the following way: If $n_b + l-1 < \sum_{i=1}^k \lvert g_i \rvert$, define a new partitioning $p' = (n_b,[g_1, \ldots,g_k],l+1)$. If $n_b + l-1 = \sum_{i=1}^k \lvert g_i \rvert$, define a new partitioning $p' = (n_b,[g_1, \ldots,g_k,g],l+1)$ for every $g \in L$. 
When this process is carried out for all minimal partitionings of allowed words of length $l$, it will produce all minimal partitionings of allowed words of length $l+1$. In this way, the function allows an inductive construction of the allowed words and their minimal partitionings.

Let $p'$ be a new partitioning obtained by extending $p$ in this way, and let $w' \in \BB_{l+1}(\X(L))$ be the corresponding word.
When $p'$ is first constructed, it is checked whether the corresponding word $w'$ is already known to belong to $\BB_{l+1}(\X(L))$. If so, the partitioning is added to the list of partitionings of $w$. If not, then $w$ is created as a new member of the class \cclass{allowedWord}. The word $w$ is a factor of $w'$, so if $w$ is strongly synchronizing then so is $w'$. Therefore, if $w$ is known to be strongly synchronizing, then this information is passed to the representation $w'$ when it is created. As mentioned in the description of the class \cclass{allowedWord} above, this reduces the number of comparisons needed when checking whether the words of length $l+1$ are left- and/or right-extendable.

\subsection{The class \cclass{renewalSystem}}

In the class \cclass{renewalSystem}, a renewal system $\X(L)$ is naturally represented by the generating list $L$. The class contains a large number of functions that allow the investigation and manipulation of renewal systems, so only the most important of these will be described here.

The main feature of the class \cclass{renewalSystem} is a collection of functions that inductively construct lists of all the allowed words and their minimal partitionings by using the functions in the classes \cclass{partitioning} and \cclass{allowedWord} described above. The class \cclass{renewalSystem} will first find all minimal partitionings of allowed words in $\BB_1(\X(L))$ and represent these using the classes \cclass{partitioning} and \cclass{allowedWord}.
Now the partitionings of words of greater length can be constructed inductively by applying the functions from the class \cclass{partitioning} described above, so assume that all the allowed words of length $n-1$, as well as the minimal partitionings of these words, have been constructed. Then the following steps can be carried out:
\begin{description}
\item[Extend partitionings:] For each minimal partitioning $p$ of length $n-1$, use the functions from the class \cclass{partitioning} to extend $p$. This gives a list of all the minimal partitionings of length $n$, and a list of all the allowed words of length $n$.
\item[Find extendable words:] For each $w \in \BB_n(\X(L))$, use the functions from the class \cclass{allowedWord} to check whether $w$ is strongly synchronizing, left-extendable and/or right-extendable. 
\item[Repeat:] If all words of length $n$ are left-extendable or if they are all right-extendable, then the induction stops. It also stops if the total number of allowed words exceeds a predefined maximum. Otherwise, the previous steps are repeated for $n$.
\end{description}
\noindent
If the induction stops because all words of length $n$ are left-extendable or all right-extendable, then Proposition \ref{prop_extendable} proves that $\X(L)$ is an $n$-step SFT, and the Bowen-Franks invariant is then computed using functions from the class \cclass{adjMatrix} described below. 
If the induction stops because the predefined maximal number of allowed words has been reached, then the investigation of $\X(L)$ is abandoned.

\subsection{The class \cclass{adjMatrix}}

The class \cclass{adjMatrix} is invoked to find the Bowen-Franks invariant when the functions from the class \cclass{renewalSystem} described above have been used to show that $\X(L)$ is an  $n$-step SFT. In this case, the allowed words of length $n$ are known, so it is straightforward to construct the adjacency matrix $A$ of the corresponding higher block shift as in Proposition \ref{prop_higher_block}. The Bowen-Franks invariant of $\X(L)$ can then be computed by computing the Smith normal form and determinant of $\Id -A$.  However, $A$ will generally be a large and sparse matrix, so it is very inefficient to do these computations directly.

The class \cclass{adjMatrix} essentially represents $A$ as list of triples $(i,j,A_{i,j})$ where $A_{i,j} \neq 0$, and this leads to a drastic decrease in the memory needed to store $A$ when $A$ is large and sparse.
The class \cclass{adjMatrix} contains a function that finds all rows which are equal. If row number $r_1, \ldots , r_k$ are equal, then the corresponding vertices in the edge shift $\X_A$ can be merged using a state-amalgamation, and the class \cclass{adjMatrix} contains a function that transforms $A$ into the adjacency matrix of this reduced edge shift. State-amalgamations are conjugacies, so this does not change the value of the Bowen-Franks invariant.  This process is repeated until all rows are different. Generally, this leads to a drastic reduction in the size of the matrix and a corresponding increase in the sizes of the entries. After carrying out this reduction, the matrix can be effectively represented in the normal way and passed to Maple where standard tools are used to compute the determinant and Smith normal form. For large matrices, the reduction is naturally very time consuming, but it greatly increases the size of the matrices that can be treated.

\subsection{The program \program{irs.cpp}}

The program \program{irs.cpp} allows a user to apply the tools in the classes described above to an investigation of a collection of renewal systems.
If the program is executed without arguments, it will prompt the user to input generating lists of renewal systems. If the program is executed with a single argument, then this argument will be assumed to be a file name, and the collection will be generated from the contents of the file. If the program is executed with more than one argument, then the arguments will be treated as the words of a generating list, and the collection will be defined to consist of the corresponding single renewal system.

Each renewal system in the collection is investigated with the functions from the class \cclass{renewalSystem} described above. If the renewal system can be proved to be an SFT, then the adjacency matrix is represented by the class \cclass{adjMatrix}, and the Bowen-Franks invariant is calculated using the functions from \cclass{adjMatrix} as described above. The results of these computations are written to the screen and to a file. If the renewal system cannot be proved to be an SFT, information about the extendable words will be printed instead. The following options may be given to the program:
\begin{description}
\item[-n=\textit{x}:] Sets the maximal number allowed words to $x$. The investigation of a renewal system will be abandoned when the allowed words of length $m$ have been constructed if $| \BB_m(\X(L)) | \geq x$ and the conditions of Proposition \ref{prop_extendable} are not satisfied. The default is 10000.
\item[-i:] Interactive mode. If the preset maximal number of allowed words is reached before the conditions of Proposition \ref{prop_extendable} are satisfied, then the user is asked whether to continue the investigation or not. By default, the program abandons the investigation without prompting when the maximal number of allowed words is reached.
\item[-s:] Makes the program print information about strongly synchronizing, left-extendable, and right-extendable words to allow closer investigation of renewal systems that are thought to be strictly sofic. This can produce a lot of extra output and it is turned off by default.
\item[-f=\textit{x}:] Saves the results in a file named $x$. The default is a file name based on the current date and time.
\end{description}

\noindent
The files \program{sft.txt} and \program{sofic.txt}, which are available at the address given above, provide samples of the output from this program.

\subsection{The program \program{reduce.cpp}}

Given a generating list $L$, the program \program{reduce.cpp} is used to find an irreducible generating list $M$ such that $\X(L) \FE \X(M)$. This is done using a function in the class \cclass{renewalsystem} which implements the algorithm given in the proof of Proposition \ref{prop_irr}. In order to do this, the function uses the tools from the class \cclass{renewalsystem} described above to find all partitionings of all factors of all $w \in L$. With this information, it is elementary to construct a list $I$ which contains every internal word that is not contained in any longer internal word. A new generating list $L'$ is then constructed from $L$ by replacing every occurrence of an internal word $w_i \in I$ by a new symbol $a_i$. This process is repeated until all internal words have length $1$. Finally, the symbols are renamed and the list is sorted to make it easier to compare different irreducible lists output by the program.

\section{Results}
\label{sec_exp_results}

Thousands of randomly generated SFT renewal systems were examined using the methods described in the previous section. The precise number of examined lists is unknown because the lists which could not be proved to generate SFTs have not been counted and because the randomly generated lists were reduced using the program \program{reduce.cpp} and only examined if the reduced form was not already known. 
The file \program{sft.txt}, which is available at the address given above, contains the output generated by \program{irs.cpp} for a collection of $488$ such random irreducible lists generating SFT renewal systems. For a list $L$ which generates an $n$-step SFT with determinant $d$ and Bowen-Franks group $\Z / d_1 \Z \oplus \cdots \oplus  \Z / d_k \Z$, the output has the following format:
\begin{displaymath}
\textrm{Name of renewal system: } L ; n ; d ; [d_1, \ldots, d_k].
\end{displaymath}
The lists are ordered by the sum of the lengths of the generating words.
All of the randomly generated SFT renewal systems examined in this way turned out to have negative determinants and cyclic Bowen-Franks groups, so they were all flow equivalent to full shifts. This was the same pattern as seen in the preliminary investigation even though it was now possible to examine a wider range of renewal systems, so no new information about the range of the Bowen-Franks invariant was gained this way.

The file \program{sofic.txt}, which is available at the address given above, contains the output generated by \program{irs.cpp} for a collection of generating lists for which it was not possible to prove that the conditions of Proposition \ref{prop_extendable} were satisfied. For each list, the maximum length of words that have been examined is shown. Additionally, a list is given showing a sequence of the minimum of the numbers of left-extendable and right-extendable words. This gives the user data to decide whether it is feasible to continue the investigation. There is for instance no indication that this minimum will reach $0$ in any of the lists shown in \program{sofic.txt}.

Next, a less random approach was used in an attempt to find SFT renewal systems with more complicated values of the Bowen-Franks invariant. The goal was to construct SFT renewal systems with specific forbidden words and to find out why it was so hard to obtain positive determinants and non-cyclic groups. This lead to the discovery of the renewal systems with non-cyclic Bowen-Franks groups considered in Section \ref{sec_nc_bf}. The programs described in the previous section were used to examine the renewal systems and to establish hypotheses about the invariant. These hypotheses were then tested using the programs,  before a proof of Proposition \ref{prop_non-cyclic_BF} was formulated. 

It was expected that generating lists with complicated behaviour must exhibit a lot of entanglement (i.e.\ that each word must have many different partitionings), and in order to construct such lists randomly, the number of letters was decreased to two while the maximal number of words used in the generating list was increased. A lengthy investigation of random generating lists constructed in this way lead to the three exotic examples of SFT renewal systems with positive determinants considered in Section \ref{sec_rs_exotic}. The study of these systems lead to the theory of border points, and hence, to the results about the structure of the Fischer covers of sums of modular renewal systems developed in Section \ref{sec_rs_add} and used to construct the composite renewal systems in Sections \ref{sec_rs_entropy} and \ref{sec_rs_range}.

The investigation also lead to a number of SFT renewal systems with Bowen-Franks group $\Z$ (and hence determinant $0$). Note that these renewal systems are not flow equivalent to the full shift with only one symbol, since Theorem \ref{thm_franks} only gives a classification of shifts that are not in the trivial flow-class. All the known renewal systems of this kind are listed in Table \ref{tab_Z}.

\begin{table}
\begin{center}
\begin{tabular}{| l l l l |}
\hline
Generating list & Step & BF-group & Determinant \\
\hline
$\{ a, aba, bab \}$ &  3  &  $\Z$  &  0 \\
$\{ aa, aaa, abb, abbb, baa, baaa, bb, bbb \}$  &  3  &  $\Z$  &  0  \\
$\{ aa, ba, bb, aaa, aba, bbb \}$  &  5  &  $\Z$  &   0 \\
$\{ aa, ab, bb, aaa, bab, bbb \}$  &  5  &  $\Z$  &   0 \\
$\{ ab, bb, aba, bbb, abaa, aabbb \}$  &  8  &  $\Z$  &   0  \\
$\{ aa, aaa, baa, bba, abaa, bbab \}$  &  8  &  $\Z$  &   0 \\
$\{ ab, baa, bba, abba \}$  &  9  &  $\Z$  &  0 \\
$\{ a, bc, bbcbb, cbcbb \}$  &  9  &  $\Z$  &   0 \\
$\{ ab, bba, bbaa, babab, bbaaa \}$  &  9  &  $\Z$  &  0 \\
$\{ aa, ab, aaa, bab, abba, bbab \}$ & 10 & $\Z$ &  0 \\
$\{ ab, bb, aaa, aab, bbb, aaaa, baab \}$ & 10 & $\Z$ &  0  \\
\hline
\end{tabular}
\end{center}
\caption{Renewal systems with Bowen-Franks group $\Z$.}
\label{tab_Z}
\end{table}

In Example \ref{ex_0_in_bf}, it was shown that there exist SFT renewal systems for which the Bowen-Franks group has the form $\Z / n\Z \oplus \Z$. This example was obtained by adding a free letter to the list
\begin{displaymath}  
L = \{ aa, aaa, abb, abbb, baa, baaa, bb, bbb \}
\end{displaymath}
and fragmenting. Another class of interesting Bowen-Franks groups is obtained by considering the renewal systems generated by the disjoint union of $L$ with one or more copies of itself. These renewal systems are listed in Table \ref{tab_nc_bf}.
It is easy to guess how this sequence of Bowen-Franks groups can be continued, but the computations have not been done beyond this step.

\begin{table}
\begin{center}
\begin{tabular}{| l l l l |}
\hline
Generating list & Step & BF-group & Determinant \\
\hline
$L$                               & 3 & $\Z$  &  0 \\
$L \sqcup L$                & 3 & $ \Z / 3\Z \oplus \Z \oplus \Z$ & 0  \\
$L \sqcup L \sqcup L$  & 3 & $ \Z / 5\Z \oplus \Z \oplus \Z \oplus \Z$ & 0  \\
\hline
\end{tabular}
\end{center}
\caption{Renewal systems with generating lists constructed as disjoint unions of $L = \{ aa, aaa, abb, abbb, baa, baaa, bb, bbb \}$ with itself.}
\label{tab_nc_bf}
\end{table}

Addition and fragmentation of renewal systems of the form considered in Section \ref{sec_nc_bf} can be used to produce new renewal systems with non-cyclic Bowen-Franks groups, and such sums have been investigated experimentally, but it has not been possible 
to use this method to construct more general Bowen-Franks groups than the ones obtained in Proposition \ref{prop_non-cyclic_BF}.

\section{Discussion}
\label{sec_exp_discussion}

The first conclusion of the experimental investigation is that it is very hard to randomly construct an SFT renewal system that is not flow equivalent to a full shift. This is surprising, but it gives an indication of why Adler's question is hard to answer.

The computer programs described in Section \ref{sec_exp_programs} have been very useful in the investigation of concrete renewal systems, and many of the main ideas of Chapter \ref{chap_rs} grew out of a desire to understand concrete examples that turned up in the experimental investigation. As mentioned in Section \ref{sec_rs_perspectives}, a reasonable next step in the search for the range of the Bowen-Franks invariant over the set of SFT renewal systems would be to attempt to construct an SFT renewal system with Bowen-Franks group $\Z \oplus \Z$. Here, the programs could be useful for testing candidates.

The programs rely on Proposition \ref{prop_extendable} to check whether a renewal system is an SFT, and since the conditions are sufficient but not necessary, this will fail to detect some SFTs. Proposition \ref{prop_extendable_and_sft} shows that the conditions are necessary for a large class of renewal systems, but it is unknown how well the algorithm performs outside this class, and it would be interesting to examine this in detail.
As mentioned in Section \ref{sec_rs_perspectives}, it would also be interesting to know the complexities of the algorithms and to compare them with algorithms based on the subset construction. 

The class \cclass{adjMatrix} uses a combination of reductions carried out by C++ and linear algebra computations carried out in Maple to investigate the large and sparse integer matrices that appear as the adjacency matrices of the higher block shifts of renewal systems. It would be more ideal to carry out all the computations in C++, but it has not been possible to find a suitable linear algebra package, and it would take a significant amount of time to construct such a package from the bottom.

Depending on the length of the words in $L$, the program \program{reduce.cpp} can take a long time to find the corresponding irreducible list, so for many renewal systems, it is faster to simply investigate the original list using the program \program{irs.cpp}. The reduction is, however, useful if the same renewal system is going to be investigated several times, i.e.\ if it is later going to be used as a term in a sum of renewal systems. Generally, it is also easier to identify the interesting features of the irreducible generating list than of the original list. The generating lists considered in Example \ref{ex_pos_det_1} are, for example, irreducible.

If the programs were remade from the bottom, it would be useful to have them compute the left Fischer cover and the border points instead of the higher block shift. As mentioned in section \ref{sec_rs_lfc_construction}, this is possible once the partitionings of the allowed words are known. This would make it possible to compare the Fischer covers where it is probably easier to gather relevant information than in the higher block shifts. In particular, information about the border points would be useful for constructing more complicated renewal systems from simple building blocks as in the constructions used in Sections \ref{sec_nc_bf} and \ref{sec_pos_det+nc_bf}.